\documentclass[reqno, 11pt, twoside]{amsart}

\makeatletter
\def\specialsection{\@startsection{section}{1}%
  \z@{\linespacing\@plus\linespacing}{.5\linespacing}%
  {\normalfont}}% NEW
\def\section{\@startsection{section}{1}%
  \z@{.7\linespacing\@plus\linespacing}{.5\linespacing}%
  {\normalfont\scshape}}% NEW
\makeatother

\usepackage{etoolbox}
\patchcmd{\section}{\scshape}{\bfseries}{}{}
\makeatletter
\renewcommand{\@secnumfont}{}
\makeatother

\makeatletter
\def\subsection{\@startsection{subsection}{1}%
  \z@{.5\linespacing\@plus.7\linespacing}{-.5em}%
  {\normalfont\itshape}}
	\def\@sect#1#2#3#4#5#6[#7]#8{%
  	\edef\@toclevel{\ifnum#2=\@m 0\else\number#2\fi}%
  	\ifnum #2>\c@secnumdepth \let\@secnumber\@empty
  	\else \@xp\let\@xp\@secnumber\csname the#1\endcsname\fi
  	\@tempskipa #5\relax
  	\ifnum #2>\c@secnumdepth
  	  \let\@svsec\@empty
  	\else
  	  \refstepcounter{#1}%
    \edef\@secnumpunct{%
      \ifdim\@tempskipa>\z@ % not a run-in section heading
        \@ifnotempty{#8}{.\@nx\enspace}%
      \else
        \@ifempty{#8}{.}{.\@nx\enspace}%
      \fi
    }%
    \@ifempty{#8}{%
      \ifnum #2=\tw@ \def\@secnumfont{\bfseries}\fi}{}%
    \protected@edef\@svsec{%
      \ifnum#2<\@m
        \@ifundefined{#1name}{}{%
          \ignorespaces\csname #1name\endcsname\space
        }%
      \fi
      \@seccntformat{#1}%
    }%
  \fi
  \ifdim \@tempskipa>\z@ % then this is not a run-in section heading
    \begingroup #6\relax
    \@hangfrom{\hskip #3\relax\@svsec}{\interlinepenalty\@M #8\par}%
    \endgroup
    \ifnum#2>\@m \else \@tocwrite{#1}{#8}\fi
  \else
  \def\@svsechd{#6\hskip #3\@svsec
    \@ifnotempty{#8}{\ignorespaces#8\unskip
       }%
    \ifnum#2>\@m \else \@tocwrite{#1}{#8}\fi
  }%
  \fi
  \global\@nobreaktrue
  \@xsect{#5}}
\makeatother

\makeatletter
\def\abstract#1{ \gdef\@abstract{#1}}
\def\@abstract{\@latex@error{No \noexpand\abstract given}\@ehc}
\makeatother

\makeatletter
\def\keywords#1{ \gdef\@keywords{#1}}
\def\@keywords{\@latex@error{No \noexpand\keywords given}\@ehc}
\makeatother

\makeatletter
\def\MSC#1{ \gdef\@MSC{#1}}
\def\@MSC{\@latex@error{No \noexpand\MSC given}\@ehc}
\makeatother

\makeatletter
\def\abstract#1{ \gdef\@abstract{#1}}
\def\@abstract{\@latex@error{No \noexpand\abstract given}\@ehc}
\makeatother
\usepackage{verbatim}
\usepackage[letterspace=150]{microtype}
\usepackage{multicol}
\usepackage[affil-it]{authblk}
\usepackage{titling}
\usepackage[latin1]{inputenc}
\usepackage{calligra}
\usepackage[OT1]{fontenc}
\usepackage[english]{babel}
\usepackage{amsfonts}
\usepackage{amsmath}
\usepackage{amssymb}
\usepackage{amsthm}
\usepackage{faktor}
\usepackage{float}
\usepackage{enumerate}
\usepackage{color}
\usepackage{pdfpages}
\usepackage{esint}
\usepackage{cite}
\usepackage{fancyhdr}
\usepackage{mathtools}
\usepackage{mathrsfs}
\DeclareMathOperator*{\Id}{Id}

\newcommand{\tr}{{}^\mathrm{t} }
\newcommand{\Div}{\mathrm{div}}
\newcommand{\trc}{\mathrm{tr}}

\newcommand{\Pp}{\mathcal{P}}
\newcommand{\ee}{\varepsilon}

\newcommand{\Aa}{\mathcal{A}}
\newcommand{\Cc}{\mathcal{C}}
\newcommand{\Bb}{\mathcal{B}}

\newcommand{\dd}{\mathrm{d}}

\newcommand{\pre}{ \mathrm{p}}
\newcommand{\uu}{ \mathsf{u}}
\newcommand{\nn}{ \mathsf{n}}

\newcommand{\e}{ \mathrm{e}}
\newcommand{\Tt}{\vartheta}

\newcommand{\X}{\mathfrak{X}}

\newcommand{\SSS}{\mathbb{S}}

\newcommand{\J}{\mathcal{J}}
\newcommand{\I}{\mathcal{I}}

\newcommand{\Ff}{\mathscr{F}}
\newcommand{\FF}{\mathcal{F}}

\newcommand{\DD}{\mathfrak{D}}
\newcommand{\RR}{\mathbb{R}}
\newcommand{\ZZ}{\mathbb{Z}}
\newcommand{\NN}{\mathbb{N}}

\newcommand{\BB}{\dot{B}}

\newcommand{\Dd}{\dot{\Delta}}
\newcommand{\Sd}{\dot{S}}

\newcommand{\vv}{{\rm v}}
\newcommand{\h}{\mathbb{H}}

\newtheorem{theorem}{Theorem}[section]
\newtheorem{cor}{Corollary}[theorem]
\newtheorem{prop}[theorem]{Proposition}
\newtheorem{lemma}[theorem]{Lemma}
\newtheorem{definition}[theorem]{Definition}
\newtheorem{remark}[theorem]{Remark}

\DeclareFontFamily{OT1}{rsfs}{}
\DeclareFontShape{OT1}{rsfs}{m}{n}{ <-7> rsfs5 <7-10> rsfs7 <10-> rsfs10}{}
\DeclareMathAlphabet{\mycal}{OT1}{rsfs}{m}{n}

\def\q{{\bf q}}
\def\w{\omega}
\def\m{{\mathbf{m}}}

\def\mcN{{\mycal N}}

\DeclareSymbolFont{letters}{OML}{ztmcm}{m}{it}

\begin{document}
 \begin{center}%
  \noindent\rule{155mm}{0.01cm}
  \let \footnote \thanks
  {\LARGE  \textsc{ Non-isothermal general Ericksen-Leslie system: derivation, analysis and thermodynamics-consistency} \par}%
  \noindent\rule{155mm}{0.01cm}
  \vskip 1.em
  \today	
 \end{center}
 \vskip 1.em%
 {\let \footnote \thanks  Francesco De Anna$\,^1$,$\hspace{0.2cm}$Chun Liu$\,^2$ \par}
 {\textit{\footnotesize{$\,^1$ Department of Mathematics, Penn State University, University Park, PA 16802, US\\
			e-mail: fzd16@psu.edu\vspace{0.1cm} \\			
			$\,^2\,$ Department of Applied Mathematics, Illinois Institute of Technology, Chicago, IL 60616, US\\
			e-mail: cliu124@iit.edu} } }
 \par
 \vskip 1.em
 \noindent\rule{155mm}{0.01cm} 
 \vskip 1.em
 {\small
  \begin{minipage}{0.3\linewidth}
  \noindent\hspace{-0.6cm} \textsc{article info}
  
  \noindent\hspace{-0.4cm}\rule{4.3cm}{0.01cm} 
  
  \end{minipage}
  \begin{minipage}{0.6\linewidth}
   \noindent \noindent \textsc{ abstract}%\lsstyle
   
   \noindent\rule{10.32cm}{0.01cm} 
  
  \end{minipage}
	
  \smallskip
  \hspace{-0.4cm}
  \begin{minipage}{0.29\linewidth}
  
  \noindent \hspace{-0.0cm}\textit{Keywords:}
  Ericksen-Leslie, nematic liquid crystals, thermodynamics consistency, global well-posedness, Besov regularity.
  
  \vspace{0.4cm}
  \noindent \hspace{-0.0cm}\textit{MSC:}
  35Q30, 35Q35, 	35Q79, 76A15.
  \end{minipage}
  \hspace{0.34cm}
  \begin{minipage}{0.6\linewidth}
  
  \noindent $$\,$$
We derive a model describing the evolution of a nematic liquid-crystal material under the action of thermal effects. The first and second laws of thermodynamics lead
to an extension of the general Ericksen-Leslie system where the Leslie stress tensor and the Oseen-Frank energy density are considered in their general forms. 
The work postulate proposed by Ericksen-Leslie is traduced in terms of entropy production. We finally analyze the global-in-time well-posedness of the system for small initial data in the framework of Besov spaces.  
$$\,$$
  
  \end{minipage}
  
  \noindent\rule{155mm}{0.01cm} 
  } 	
\makeatother
\allowdisplaybreaks{}

\tableofcontents

\section{Introduction}
\noindent 
The main aim of this paper is to derive and analyze an evolutionary PDE's-system modeling the dynamics of nematic liquid crystals. The model we are interested in extends the general Ericksen-Leslie theory, allowing a non-constant temperature. We derive such a model in accordance with the main laws of thermodynamics.

\smallskip\noindent
Nowadays, the engineering
and mathematical community is familiar with the concept of nematic liquid crystals. A nematic medium is a compound of fluid molecules, which has a state of matter between an ordinary liquid and a crystal solid. Although the centers of mass can freely translate as in a common fluid, the constitutive molecules present a privileged orientation. This alignment strongly interacts with the underlying flow of the nematic.

\smallskip\noindent
Reinitzer discovered one of these materials in the 1888 and since then there have been numerous attempts to formulate continuum theories describing the time behavior of the flow. Ericksen and Leslie developed the most widely recognaized model during the 1960's in their pioneeric papers \cite{Ericksen, Ericksen2, Leslie},  generalizing the Oseen-Frank theory for the static case \cite{Frank}. 
From the first mathematical success in analyzying the model performed by Lin and Liu \cite{Lin-Liu} in the 1991, the dynamics theory of liquid crystal has become the new El Dorado of theoretical studies motivated by real-world applications. The well-posedness analysis in bounded domains \cite{H-N-P-S, Lin-Wang, Cavaterra-Rocca-Wu} as well as in the whole space \cite{Hineman-wang, Dea}, has especially received high interest in the recent decades, both for what concerns the director theory and the $\mathbb{Q}$-tensor framework. These results are often inspired by the ample literature concerning the Navier-Stokes equations, since any derivation of a consistent model for these anisotropic materials usually starts from the well-known conservation of mass and balances of linear and angular momentum. 

\smallskip
\noindent
Despite a wide literature concerning the dynamics of nematic liquid crystals, to the best of our knowledge there are few papers dealing also with thermodynamic effects. Liquid crystals are mostly considered in an isothermal environment, which is sometimes unnatural. Indeed, as described by Stewart in \cite{Stewart}, liquid phases are mainly induced by changing the temperature  (\textit{thermotropic} LC) or the concentration of a solvent (\textit{lyotropic} LC). Moreover, only few articles allows thermal effects and simoltaneously treat the \textit{consistency} of their models with respect to the main laws of thermodynamics. We refer for instance to the recent work of Hieber and Pr\"uss \cite{H2} as well as of Fereisl \cite{F} and Feireisl, Rocca and Schimperna \cite{F-R-S}, where the authors deal with  the time-evolution of incompressible non-isothermal nematics in the Ericksen-Leslie formalism. We mention also the pioneristic works of Feireisl, Rocca, Schimperna and Zarnescu \cite{FRSZ,FRSZ2}, concerning the flow of incompressible non-isothermal nematics, under the $\mathbb{Q}$-tensor formalism. 

\smallskip
\noindent
This paper deals with the time-evolution of nematic liquid crystals assuming a non-constant temperature. 
We indeed perform  a thermodynamics-consistent model which extends the widespread \textit{non-simplified} Ericksen-Leslie theory. 
%This extension should be treated as a direct consequence of the first and second law of thermodynamics.
Moreover, this extension is performed both in the case of a compressible as well as an imcompressible nematic material.

\subsection{The equations of motion}$\,$

\smallskip
\noindent
We begin with introducing the main continuum variables describing the evolution of the medium. We denote by $\rho(t,x)$ the density of the liquid crystal and by $\uu = \uu(\,t,\,x\,)\in \RR^\dd$ the velocity field of the flow in the Eulerian reference system, for a fixed time $t\in\RR_+$ and a position $x\in \RR^\dd$. The time-evolution of the flow is described by a Navier-Stokes-type equation, under the action of thermodynamic effects.
We denote by $\nn = \nn(\,t,\,x\,)$  the so-called \textit{director} field, returning values into the sphere $\SSS^{\,\dd-1} $. 
The unit vector $\nn$ represents the direction of the preferred long-range orientation of the constitutive molecules in a neighborhood of any point. 
The evolution of the director field $\nn(t,\,x)$ is driven by a convection-diffusion equation, whose simplest form reduces to the heat flow of harmonic map into the sphere 
(cf. \cite{Lin-Liu}).
In this paper, the director equation is supported by the usual constraint $|\,\nn\,|=1$, producing an high-order non-linearity in the system.
It is common in literature to relax such a non-linearity, introducing a Ginzburg-Landau penalization term in the free energy of the system (cf. \cite{FFRS, FRSZ, FRSZ2, Lin-Liu, Wu-Xu-Liu}).
%
%The dynamics of $\nn$ is usually described through a parabolic-type equation arising from the balance of angular momentum. 
%Allowing thermodynamic effects will perturb this balance with respect to the case of an isothermal nematic. 
Finally, we denote by $\Tt=\Tt(\,t,\,x\,)>0$ the so-called absolute temperature, and  we are interested in the range of temperatures such that the nematic phase occurs. For instance, as explained by Stewart in \cite{Stewart} (see Figure $1.5$), a \textrm{PAA} exhibit a nematic behavior when its temperature is between $391\,\mathrm{K}$ and $408\, \mathrm{K}$, while the $\bar{10}S5$ becomes nematic as the temperature increases from $353\,\mathrm{K}$ to $359\,\mathrm{K}$.

\smallskip\noindent
We denote by $\FF$ the Helmholtz free energy density of the system, we assume depending on the set of variables $(\rho,\,\Tt,\,\nn,\,\nabla \nn)$ and 
we expect a strong correlation between $\FF$ and the classical Oseen-Frank energy density for isothermal nematic media.

\smallskip\noindent
Furthermore, the non-isothermal environment we take into account generally gives rises to a non-constant entropy. Thus, the first essential relation we take into account in this paper connects the definition of the local entropy, we denote by $\eta$, and the Helmholtz energy density $\FF$. More precisely, the Maxwell's identity (cf. \cite{Malek-Prusa}, section $2.3$) insures $\eta$ to be defined by means of
\begin{equation}\label{Maxwell_rel}
	\eta :=- \frac{\partial \FF}{\partial \Tt}.
\end{equation}

\begin{comment}
The above identity is in accordance with the formalism presented in , where proceeding as in section $2.3$, the internal energy of the system ${\rm e}^{\rm int}\,=\,\FF\,+\,\Tt\eta$ depends on the set of state variable $(\rho,\,\eta,\,\nn,\,\nabla \nn)$. 
Then the connection between absolute temperature $\Tt$, local entropy $\eta$ and internal energy ${\rm e}^{\rm int}$ is given through 
\begin{equation*}
	\Tt(t,\,x) := \frac{\,\,\partial {\rm e}^{\rm int}}{\partial \eta}(\,\rho(t,\,x),\,\eta(t,\,x),\,\nn(t,\,x),\,\nabla \nn(t,\,x)\,).
\end{equation*}
Thanks to the relation between the Helmholtz free energy density and the internal energy densities $\FF= {\rm e}^{\rm int}-\Tt \eta$, we deduce the thermodynamics identity
\begin{equation*}
	\frac{\partial \Tt}{\partial \eta}\eta + \frac{\partial \FF}{\partial \eta} = 0,
\end{equation*}
which corresponds to \eqref{Maxwell_rel}, under the assumption that $\partial_{\eta} \Tt\neq 0$. 
\end{comment}

\smallskip\noindent
We begin with introducing the main balance laws that drive the evolution of a non-isothermal and compressible liquid crystal material. Their pointwise forms read as follows:
\begin{equation}\label{main_system}
\left\{\hspace{0.2cm}
	\begin{alignedat}{2}
		&\,\rho\Big[\,\partial_t \uu  + \uu\cdot \nabla \uu\,\Big]\,   =
		\Div\,\, \Big[ \, -\pre\,\Id\,+ \,\sigma^{\mathrm{E}}\,+\,\sigma^{\mathrm{L}}\,\Big]\hspace{2cm}
																		&& \RR_+\times \RR^\dd, \vspace{0.1cm}\\
		&\partial_t \rho \,+\,\Div\,(\,\rho\,\uu\,)\,=\,0					&&\RR_+\times \RR^\dd, \vspace{0.1cm}\\
		&\,\mathrm{g}\, + \mathrm{h}\, =\,\beta\,\nn , 
																		&&\RR_+\times \RR^\dd, \vspace{0.1cm}	\\	
		&\nn\,\cdot \nn\,= 1									&&\RR_+\times \RR^\dd, \vspace{0.1cm}	\\																	
		&\Tt\, \Big[\,\partial_t  \eta 
		+ \Div\,(\,\uu \,\eta\,)\,
		\Big] = \Tt \,\Div\,  \frac{\bf q}{\Tt} +\Tt\,\Delta^*
		&&\RR_+\times \RR^\dd, \vspace{0.1cm}\\																		
		%&(\,\rho,\,\uu,\,\Tt,\,\nn)_{t=0} = (\,\rho_0,\,\uu_0,\,\Tt_0,\,\nn_0\,) &&\hspace{1.02cm}\RR^\dd ,										
	\end{alignedat}
	\right.
\end{equation}
\noindent
In system \eqref{main_system} we have introduced the balance of linear momentum, the conservation of mass, the balance of angular momentum, the unitary-constraint on the director field $\nn$, and the so-called Clausius-Duhem inequality, respectively. We assume that the considered liquid crystal occupies the whole space $\RR^\dd$, with a dimension $\dd\geq 3$. The large dimension is essential for the global-in-time well-posedness result we perform in section \ref{sec-wp}. However, we claim that the thermodynamics consistency of our model holds also in the bi-dimensional case. 

\noindent A peculiarity of system \eqref{main_system} with respect to the general Ericksen-Leslie system relies in the last equation, the Clausius-Duhem inequality also known as the second law of thermodynamics. 

\smallskip\noindent
We begin with describing the main terms driving the time-evolution of the nematic medium. The tensor $\sigma^{\rm E}$ in $\RR^{\dd\times \dd}$ in the momentum equation stands for the well-known Ericksen tensor, defined by means of
\begin{equation}\label{def:E-stress}
	\sigma^{\rm E}:=-\tr\Big[ \nabla \nn\Big]\frac{\partial \FF}{\partial \nabla \nn}, \hspace{1cm}\text{i.e.} \quad\quad
	\sigma^{\rm E}_{ij}:=-\nn_{k,i}\frac{\partial \FF}{\,\,\partial \nn_{k,j}}
\end{equation}
where we have used the Einstein summation convention of summation over repeated indices.
Furthermore, the tensor $\sigma^{\rm L}$ in $\RR^{\dd\times\dd}$ denotes the Leslie stress tensor. We initially assume $\sigma^{\rm L}$ to be only an isotropic tensor in $\RR^{\dd\times \dd}$ depending on the set of variables $(\rho,\,\Tt,\,\,\mcN,\,\mathbb{D},\,\nn)$. Here, $\mathbb{D}$ stands for the symmetric part of $\nabla \uu$ and $\Omega$ for the skew-adjoint part:
\begin{equation*}
	\mathbb{D}\,: =\, \frac{\nabla \uu  + \tr \nabla \uu}{2}\quad\quad \quad\text{and}\quad \quad\quad
	\Omega 	\,:=\, \frac{\nabla \uu  - \tr \nabla \uu}{2}.
\end{equation*}

\smallskip\noindent
The notation $\mcN$ identifies the so-called co-rotational time flux of the director field $\nn$, whose formula is determined by
\begin{equation}\label{def-corot}
	\mcN \,:=\,\dot \nn\,-\,\Omega\,\nn\,=\,\partial_{\,t}\,\nn\,+\,\uu\cdot \nabla \nn \,-\,\Omega\,\nn. 
\end{equation}
It is worth to remark that a consistent number of papers in literature relax the co-rotational time flux through the identity $\mcN\,\sim\,\dot \nn\,=\,\partial_t\nn\,+\,\uu\cdot\nabla\nn$. This starts from the pioneristic work of Lin and Liu \cite{Lin-Liu2}, as analysis of a simplified version for the Ericksen-Leslie theory. In this work we preserve the genuine structure given by \eqref{def-corot}.

\smallskip
\noindent
The balance of angular momentum in \eqref{main_system} is expressed in terms of the molecular field $\rm h$ and the kinematic transport $\mathrm{g}$ of the director $\nn$, which represents the effect of the macroscopic flow field on the microscopic structure. We consider the following formulations:
\begin{equation}\label{def_hg}
	\mathrm{h}:=\frac{\delta \FF}{\delta \nn}= \frac{\partial \FF}{\partial \nn}-\Div\,\frac{\partial \FF}{\partial \nabla \nn},
	\hspace{2cm}
	\mathrm{g}\otimes \nn -\nn\otimes {\rm g}=\,-\,\Big[\,\sigma^{\rm \,L}-\tr \sigma^{\rm\, L}\,\Big].
\end{equation}
The above expression of the molecular field $\rm h$ is common in literature, when replacing the free energy density $\FF$ by the well-known Oseen-Frank energy density (cf. definition \eqref{def_OF_intro}). The kinematic transport $\rm g$ is usually formulated as the orthogonal-projection with respect to $\nn$ of $\tilde {\rm g}\,=\,\gamma_1\,\mcN\,+\,\gamma_2\,\mathbb{D}\nn$, that is
\begin{equation}\label{def2g}
	{\rm g}\,=\,\gamma_1\,\mcN\,+\,\gamma_2\,\Big[\,\mathbb{D}\nn\,-[\,\nn\cdot \mathbb{D}\nn\,]\,\nn\,\Big],
\end{equation}
where the coefficient $\gamma_1$ represents the co-rotational behavior of the nematics, reflecting the molecular shape (Jeffery's orbit \cite{Jeffery}), and $\gamma_2$ determines the stretching of the molecules by the flow. We refer for instance the reader to identity $(4.123)$ in \cite{Stewart}. The definition of $\rm g$ in \eqref{def_hg} extends the one of \eqref{def2g} when preserving a general structure of the isotropic Leslie stress tensor $\sigma^{\,\rm L}$. However, it is worth to remark that whenever $\sigma^{\rm\, L}$ coincides with its widespread formulation, as for instance in \eqref{Leslie-stress}, then $\rm g$ in \eqref{def_hg} coincides with \eqref{def2g},
up to the following relations
\begin{equation*}
	\gamma_1\,=\,\alpha_3\,-\,\alpha_2\quad\text{and}\quad\gamma_2\,=\,\alpha_6\,-\,\alpha_5.
\end{equation*}
Assuming moreover that $\rm g$ is perpendicular to $\nn$, we can explicitly identify the kinematic transport multiplying \eqref{def_hg} by $\nn$, namely
\begin{equation}\label{def_g2}
	{\rm g} = -\Big[\,\sigma^{\,\rm L}\,-\,\tr \sigma^{\,\rm L}\,\Big]\nn.
\end{equation}

\smallskip\noindent
The $\beta$-term in the main system \eqref{main_system} stands for the Lagrangian multiplier which insures the unitary constraint on the director field $\nn$, namely $|\nn|^2=1$. An explicit formula for $\beta$ can be achieved multiplying the angular momentum equation by $\nn$, more precisely $\beta\,=\,\mathrm{h}\cdot \nn$.

\smallskip\noindent
Finally we denote by $\bf q$ the so-called heat flux and by $\Tt \Delta^*$ one of the key element for the thermodynamic consistency of our model: the entropy production.  
We recall that according to the second law of thermodynamics, the entropy production must be always semi-positive defined.
The structure of the heat $\q$ can depend on the medium and we assume it to have a phenomenological derivation, namely to be a function of the state variables.

\medskip\noindent
It is worth to remark that system \eqref{main_system} coincides to the classical general Ericksen-Leslie system for the evolution of an incompressible nematic, whenever $\FF$ reduces to the classical Oseen-Frank energy density (as in \eqref{def_OF_intro}, below), $\sigma^{\rm L}$ stands for the the classical Leslie tensor (as in \eqref{Leslie-stress}, below) and moreover $(\rho,\,\Tt,\,\eta)$ are assumed to be constant. 

\subsection{The pressure}$\,$ 

\noindent
As already pointed out, our model investigates both the case of a compressible liquid crystal as well as an incompressible nematic. Whenever the density is assumed to be constant (we impose equal to $1$ for the sake of clarity), the conservation of mass reduces to the classical divergence-free condition on the velocity field $\uu$. In section \ref{sec-wp} we prove a well-posedness result for system \eqref{main_system} under such a condition.
In this framework, the pressure $\pre$ stands for the Lagrangian multiplier insuring the incompressible condition of the material. 
Furthermore, as additional assumption, $\FF$ does not depend on the density $\rho$. We point out that this is not a consequence of a constant density. 

\smallskip
\noindent On the other hand, when deriving our model, we take into account both the case of compressible and incompressible materials. Whenever the density is not constant, we assume the free energy density to depends on the density, $\FF\,=\,\FF(\,\rho,\,\Tt,\,\nn,\,\nabla \nn)$. These conditions lead the pressure $\pre =\pre (\,\rho,\,\Tt,\,\nn,\,\nabla \nn)$ to be defined by means of the Maxwell's relation
\begin{equation}\label{pre:compressible-case}
	\pre(\,\rho,\,\Tt,\,\nn,\,\nabla \nn)\,=\,\rho\partial_\rho\FF\,-\,\FF\,=\rho^2\frac{\partial}{\partial\rho}\Big[\frac{\FF}{\rho}\Big],
\end{equation}
where $\FF/\rho$ is the free energy density per unit mass.

\medskip\noindent
We now state our three main results: Theorem \ref{main-thm1}, Theorem \ref{main-thm2} and Theorem \ref{main-thm3}. We collect them into three subsections. The first and second theorems treat the consistency of system \eqref{main_system} with respect to the first and second law of thermodynamics, while the third theorem deals with the global-in-time well-posedness of our model.

\subsection{The first law of thermodynamics}$\,$

\smallskip\noindent
The first law of thermodynamics ensures that the  rate of change of the total energy given by the sum of the internal energy and the kinetic energy
${\rm e}^{\rm tot}={\rm e}^{\rm int}\,+\,\rho|\uu|^2/2$, is totally transformed into work $\Sigma$ or heat $\bf q$. We can write this postulate as follows:
\begin{equation}\label{first-law-thm}
	\partial_t\mathrm{e}^{\rm tot} = \Div\, {\rm \Sigma}\,+\,\Div\,{\bf q}.
\end{equation}
In this paper, we assume the work density $\Sigma(t,\,x)\in \RR^\dd$ to have a specific structure. More precisely, we consider an arbitrary smooth domain $U$ which is not moving under the action the flow $\uu$. Then, denoting by $\nu$ the normal vector to the boundary $\partial \, U$, we define the work produced by the system to the environment through the relation
\begin{equation}\label{def_work}
	\int_{\partial\, U} \Sigma \cdot \nu \,\dd x\,=\, \int_{\partial\, U} \Big[\,	\mathbb{T}\,\uu\,+\,\tr \left[\frac{\partial \FF }{\partial\nabla \nn}\right]\,\dot \nn\,-\uu\,{\rm e}^{\rm tot}\,	\Big]\cdot \nu\,	\dd x,
\end{equation}
where $\mathbb{T}\,= \,-\mathbb{\pre}\Id \,+\,\sigma^{\rm\,E}\,+\,\sigma^{\rm\,L}$ is the total stress tensor. 
The second term on the right-hands side can be seen as an extension of the standard angular work defined on a three dimensional spatial domain: If $\mathbb{L}$ and $\rm w$ are the couple stress tensor and $\mathrm{w}\in\RR^3$ is the local angular velocity of the director $\nn\in\SSS^2$ given by $\dot \nn\,=\,{\rm w}\wedge \nn$, then the following identity holds
\begin{equation}\label{def_couple_stress}
	\mathbb{L}{\rm w}\,=\,\,\tr\left[\frac{\partial\,\FF}{\partial \nabla \nn}\right](\,{\rm w}\wedge \nn)\,=\,
	\tr\left[\frac{\partial\,\FF}{\partial \nabla \nn}\right]\dot \nn.
\end{equation}
For further details we refer the reader to \cite{Stewart}, formula $(4.57)$. 
%Although the first two terms of the work are common in literature, the last term is due to the fact that the domain $U$ does not move together with the underlying flow.

\smallskip
\noindent
As a first main result of this article, we want to show that whenever explicit formulas for the heat $\bf q$ and and the Leslie-stress tensor $\sigma^{\rm \,L}$ are provided, then we automatically identify the entropy production $\Tt\Delta^*$. More precisely, we will prove the following statement: 
\begin{theorem}\label{main-thm1}
	Assuming the first law of thermodynamics \eqref{first-law-thm} to be satisfied, then the entropy production must fulfill the following identity
	\begin{equation}\label{entropy-production}
		\Tt\Delta^*= \sigma^{\,\rm L}:\mathbb{D}\, +\, \rm{g}\cdot \mcN \,+\,{\bf q}\cdot\frac{ \nabla \Tt }{\Tt}.
	\end{equation}
\end{theorem}
\begin{remark}
Whenever the temperature $\Tt$ is constant, the entropy production given by \eqref{entropy-production} reduces to the
viscous  dissipation  $ \DD	:= \sigma^{\,\rm L}:\mathbb{D} \,+\,\mathrm{g}\cdot \mcN$ introduced  by Ericksen and Leslie in their rate-of-work postulate (see for instance \cite{Stewart}, identity $(4.82)$). In this article the rate-of-work is replaced by the second law of thermodynamics, namely the Clausius-Duhem inequality. The viscous dissipation contributes to determine the formulation of the entropy production $\Tt\Delta^*$. 
\end{remark}

\subsection{The second law of thermodynamics}$\,$

\smallskip
\noindent
In section \ref{sec-2law} we then consider an explicit formulation of the Leslie stress tensor $\sigma^{\rm L}$ and an explicit definition of the heat flux $\bf q$, depending on the state variables. 
More precisely, we first assume $\sigma^{\rm L}=\sigma^{\rm L}(\rho,\,\Tt,\,\mcN,\,\mathbb{D})$ given by means of
\begin{equation}\label{Leslie-stress-compressible}
	\begin{aligned}
	\sigma^{\,\rm L}= 
	\alpha_0\,\,\big[\nn\cdot \mathbb{D}\nn\big]\,&\Id\,+\,
	\alpha_1\,	\big[\,\nn\cdot \mathbb{D}\nn\,\big]\,\nn\otimes\nn\,+\,
	\alpha_2\,	\mcN\otimes \nn\,	+\,
	\alpha_3\,	\,\nn\otimes \mcN\,+\,
	\alpha_4\,	\mathbb{D}\,+\,
	\\&+\,
	\alpha_5\,	\mathbb{D}\nn\otimes \nn\,+\,
	\alpha_6\,	\nn\otimes \mathbb{D}\nn\,+\,
	\alpha_7\,	\trc\,	\mathbb{D}\,\Id\,
	+\,
	\alpha_8\,	\,\trc\,\mathbb{D}\,\,\nn\otimes \nn,
	%+\,\\\,&
	%+\,
	%\alpha_9 (\nn\cdot \nabla \Tt)\Id\,+\,
	%\alpha_{10} (\nn\cdot(\nn\wedge \nabla \Tt))\nn\otimes \nn\,\,+
	%\alpha_{11}(\nn\wedge \nabla \Tt)\otimes \nn\,+\,
	%\alpha_{12}\nn\otimes (\nn\wedge\nabla \Tt),
\end{aligned}
\end{equation}
which corresponds to the general structure of the  Leslie stress tensor, including all the terms taking into account a \textit{compressible} behavior of the nematic liquid crystal. The Leslie viscosity coefficients $\alpha_0,\dots,\alpha_8$ are smooth functions depending on the temperature $\Tt$ and the density $\rho$. The coefficients: $\alpha_0$, $\alpha_7$ and $\alpha_8$ are strictly related to the compressible assumption, indeed whenever the $\trc\,\mathbb{D}$ is null, the terms related to $\alpha_7$ and $\alpha_8$ disappear, while the $\alpha_0$-term can be absorbed by the definition of the pressure.

\smallskip
\noindent
The heat flux ${\bf q}$ we consider in section \ref{sec-2law} is a vector-function depending on the set $(\rho,\,\Tt,\,\nn,\,\nabla \Tt)$. More precisely, we extend the widespread Fourier's laws for ${\bf q}$ as
\begin{equation}\label{heat}
	{\bf q} = \lambda_1\,\nabla \Tt\,+\,\lambda_2\,[\,\nn\cdot \nabla \Tt\,]\,\nn.
\end{equation}
The coefficients $\lambda_1$ and $\lambda_2$ are smooth functions depending on the couple $(\rho,\,\Tt)$.
It is interesting to remark that in \cite{S-V}, section $3.1.5$, Virga and Sonnet derives an heat flux perturbed also by the co-rotational time flux $\mcN$ as well as by the stretching term $\mathbb{D}\nn$:
\begin{equation*}
	{\bf q} = \lambda_1\,\nabla \Tt\,+\,\lambda_2\,[\,\nn\cdot \nabla \Tt\,]\,\nn\,\,+\lambda_3\,\,\nn\wedge \mcN\,+\,\lambda_4\, [\,\mathbb{D}\,\nn\,]\wedge\nn.
\end{equation*}
In this work we preserve the linearity of $\q$ with respect to $\nabla \Tt$.

\medskip
\noindent
The second law of thermodynamics asserts that the entropy production $\Tt \Delta^*$ given by Theorem \ref{main-thm1} must be semi-positive defined. Since $\sigma^{\rm L}$ does not depend on $\nabla \Tt$, then \eqref{entropy-production} reduces to
\begin{equation*}
	\sigma^{\rm L}:\mathbb{D}\,+\,{\rm g}\cdot\mcN \geq 0\quad\text{and}\quad {\bf q}\cdot \nabla \Tt\geq 0.
\end{equation*}
We then perform the most general conditions on the $\alpha$-coefficients as well as on the $\lambda$-ones for the above inequality to hold:
\begin{theorem}\label{main-thm2}
	Let us assume that definitions \eqref{Leslie-stress} and \eqref{heat} determine the entropy production by \eqref{entropy-production}. 
	Then the second law of thermodynamics holds if and only if the following inequalities are fulfilled
	\begin{equation}\label{inequalities_viscous_diss}
	\begin{alignedat}{8}
	\lambda_1\geq 0,\hspace{4cm}
	\lambda_1\,+\,\lambda_2&\geq 0, \\
	\alpha_3-\alpha_2 \geq 0, 
	\hspace{3.97cm}
	\alpha_4 &\geq 0,\\
	2\alpha_4+\alpha_5 + \alpha_6 \geq 0, \hspace{0.57cm}
	2\alpha_1+2\alpha_5 + 2\alpha_6 + 3\alpha_4&\geq 0, \\
	4(\alpha_2-\alpha_3)(2\alpha_4\,+\,2\alpha_5\,+\,2\alpha_6\,+\,3\alpha_4\,)\,-\, 
	(\alpha_2\,+\,\alpha_3\,+\,\alpha_5\,-\,\alpha_6)^2&\geq 0\\
	(\,\alpha_4\,+\,\alpha_7\,+\,\alpha_8\,)(\,(\dd-1)\alpha_1\,+\,(\dd-1)\alpha_5\,+\,(\dd-1)\alpha_6\,+\,
	(2\dd-3)\alpha_4\,)&\geq \frac{(\alpha_8+\alpha_1)^2 }{\dd}.
\end{alignedat}
	\end{equation}
\end{theorem}
\noindent
The first inequalities reflects the same restrictions of the $\alpha$-coefficients imposed by Ericksen and Leslie when considering incompressible isothermic nematic liquid crystals. The novelty of the theorem must be seen in the last inequality which is essential for $\Tt\Delta^*$ to be semi-positive definite when the liquid crystals has a non-constant density. 
Indeed, we disclose that the compressible condition we can impose on our nematic materials, together with the coefficients $\alpha_7$ and $\alpha_8$ perturb the viscous dissipation $\sigma^{\rm L}:\mathbb{D}+\mathrm{g}\cdot\mcN$ by means of non-trivial quadratic terms depending on both $\trc \,\mathbb{D}$ and $\mathbb{D}\nn$.  
\begin{remark}
As natural approach to prove Theorem \eqref{main-thm2}, we can split the entropy production \eqref{entropy-production} into two main terms. More precisely we can take separately into account the free-divergence component $\Pp\uu $ of $\uu$ and its orthogonal projection $\Pp^{\perp} \uu$ in the viscous dissipation.  This is a standard approach to the compressible Navier-Stokes equation, where the viscous dissipation reduces to
\begin{equation*}
	\int_{\RR^{\dd}}\,	\mathbb{S}(t,\,x):\mathbb{D}(t,\,x)\dd x\,=\,\mu\|\, \nabla\Pp \uu\, \|^2_{L^2(\RR^{\dd})} + (\mu + \lambda)\|\, \Div\,\Pp^{\perp} \uu\, \|^2_{L^2(\RR^{\dd})},
\end{equation*}
where $\mathbb{S}$ stands for the classical Cauchy stress tensor while $\mu$ and $\lambda$ are viscous coefficients. We can observe that the two projections $\Pp \uu$ and $\Pp^\perp \uu$ do not interact in the above dissipation.
It is then sufficient to separately analyze the term depending on $\Pp\uu$ and the one on $\Pp^{\perp}\uu$.

\noindent 
However, when considering the general Leslie stress tensor \eqref{Leslie-stress}, the anisotropic peculiarity of nematic materials leads some term in the viscous dissipation to depend both on $\Pp\uu$ and $\Pp^{\perp  }\uu$. We then analyze the entropy production in its general formulation. 
Nevertheless, it is worth to remark that whenever $\alpha_8$ and $\alpha_1$ are null (i.e $\Pp \uu$ and $\Pp^\perp\uu$ do not interact), then the last inequality of \eqref{inequalities_viscous_diss} reduces to 
\begin{equation*}
	\alpha_4\,+\,\alpha_7\geq 0,
\end{equation*}
namely the classical assumption for the viscous coefficients of a compressible isotropic fluid. 
\end{remark}

\smallskip
\subsection{Well-posedness}$\,$

\smallskip\noindent
In section \ref{sec-wp}, we finally deal with the well-posedness of the system \eqref{main_system} when the density is assumed to be constant. The stress tensors and the free energy density $\FF$ are supposed to depend on the set of variables $(\,\Tt,\,\nn,\,\nabla \nn\,)$. Moreover, since the conservation of mass reduces to a divergence free condition for the velocity field $\uu$, the Leslie stress tensor $\sigma^{\rm\,L} $ in \eqref{Leslie-stress-compressible} reduces to
\begin{equation}\label{Leslie-stress}
	\begin{aligned}
	\sigma^{\,\rm L}= 
	\alpha_1\,	\big[\,\nn\cdot \mathbb{D}\nn\,\big]\,\nn\otimes\nn\,+\,
	\alpha_2\,	\mcN\otimes \nn\,	+\,
	\alpha_3\,	\,\nn\otimes \mcN\,+\,
	\alpha_4\,	\mathbb{D}\,+\,
	\alpha_5\,	\mathbb{D}\nn\otimes \nn\,+\,
	\alpha_6\,	\nn\otimes \mathbb{D}\nn,
\end{aligned}
\end{equation}
where the Leslie viscosity coefficients $\alpha_1,\dots,\alpha_6$ are smooth functions depending on the absolute temperature $\Tt$.
Defining the heat flux $\q$ as  in \eqref{heat}, we then assume that the hypotheses of Theorem \ref{main-thm1} and Theorem \ref{main-thm2} are fulfilled. 

\noindent We recall that the kinematic transport $\mathrm{g}$ is defined by means of 
\begin{equation*}
	\mathrm{g} \,=\,-\Big[\,\sigma^{\rm\, L}-\tr\,\sigma^{\rm\, L}\,\Big]\nn\,= \,\gamma_1 \mcN \,+\, \gamma_2\,\Big[\,\mathbb{D}\nn\,-\,[\,\nn\cdot \mathbb{D}\nn\,]\,\nn\,\Big].
\end{equation*}
where $\gamma_1$ and $\gamma_2$ depend on the absolute temperature $\Tt$ and they are defined by
\begin{equation*}
	\gamma_1\,:=\alpha_3\,-\alpha_2\quad\text{and}\quad\gamma_2:=\alpha_6\,-\,\alpha_5.
\end{equation*} 
Applying Theorem \ref{main-thm1} to the system \eqref{main_system} finally leads to the following model for the time evolution of an incompressible non-isothermal nematic liquid crystal.
\begin{equation}\label{intro:main_system-wp}
\left\{\hspace{0.2cm}
	\begin{alignedat}{8}
		&\,\partial_t \uu  + \uu\cdot \nabla \uu\,  =
		\Div\,\, \Big[ \, -\pre\,\Id\,+ \,\sigma^{\mathrm{E}}\,+\,\sigma^{\mathrm{L}}\,\Big]
																		&&&&&&&&\RR_+\times \RR^\dd,\vspace{0.1cm}\\
		&\Div\,\uu = 0 									&&&&&&&&\RR_+\times \RR^\dd,\vspace{0.1cm} \\
		&\,\gamma_1 \mcN\, +\,\gamma_2\,\Big[\,\mathbb{D}\nn\,-\,[\nn\cdot \mathbb{D}\nn\,]\nn\,\Big]\, +\,\frac{\delta \FF}{\delta \nn}\, =\,\beta\,\nn\,
																		&&&&&&&&\RR_+\times \RR^\dd, \vspace{0.1cm}	\\	
		&\nn\,\cdot \nn\,= 1											&&&&&&&&\RR_+\times \RR^\dd, \vspace{0.1cm}	\\																	
		&-\Tt \Big[\,\partial_t  \frac{\partial \FF}{\partial \Tt}
		+ \uu\cdot \nabla\, \frac{\partial \FF}{\partial \Tt}\,
		\Big] = \Div\,\Big[\,\lambda_1 \nabla \Tt\,+\,\lambda_2\,[\,\nn\cdot \nabla \Tt\,]\nn \,\Big]\,+ &&&&&&&&\,\\
		&\hspace{5.4cm}+\,\sigma^{\rm\,L}:\mathbb{D}
		\,+\,\gamma_1|\,\mcN\,|^2\,\,+\gamma_2\mcN\cdot \mathbb{D}\nn
		\hspace{0.5cm}
																		&&&&&&&&\RR_+\times \RR^\dd, \vspace{0.1cm}\\																		
		&(\,\uu,\,\Tt,\,\nn)\big|_{t=0} = (\,\uu_0,\,\Tt_0,\,\nn_0) 			&&&&&&&&\hspace{1.02cm}\RR^\dd.									
	\end{alignedat}
	\right.
\end{equation}
The free energy density $\FF=\FF(\,\Tt,\,\nn,\,\nabla \nn)$ is then defined as a non-trivial perturbation of the classical Oseen-Frank energy density:
\begin{equation*}
	\FF(\,\Tt,\,\nn,\,\nabla \nn\,)\,=\,
	-\,\Tt\,\ln \Tt\,+\, \mathrm{W}_{\rm F}(\Tt,\,\nn,\,\nabla \nn)\,.
\end{equation*}
The term $-\Tt\ln \Tt$ we have introduced in our definition leads to a parabolic behavior of the temperature equation, while 
the temperature-dependent density ${\rm W}_{\rm F}\,=\,{\rm W}_{\rm F}(\Tt,\,\nn,\,\nabla \nn)$ stands for the classical Oseen-Frank energy density with non-isothermal coefficients, namely
\begin{equation}\label{def_OF_intro}
	\begin{aligned}
	 {\rm W}_{\rm F}(\,\Tt,\,\nn,\,\nabla \nn)\,=\,
	 \frac{k_{22}(\Tt)}{2}\,|\,\nabla \,\nn\,|^2\,&+\,
	 \,\frac{k_{11}(\Tt)-k_{22}(\Tt)-k_{24}(\Tt)}{2}\left|\,\Div\,\nn\,\right|^2 +\\&+ 
	\frac{k_{33}(\Tt)-k_{22}(\Tt)}{2}\left| \,\nn\cdot\nabla \nn\, \right|^2 + k_{24}(\Tt)\,\trc\{(\nabla \nn)^2\}.
\end{aligned}
\end{equation}
The coefficients $k_{ij}$ are assumed to be smooth functions depending on the temperature $\Tt$ and satisfying specific inequalities (we refer to \eqref{ass-3}, Section \ref{sec-wp}) in order to ensure a parabolic behavior of the director equation.

\smallskip
\noindent
Finally, the initial data we take into account belong to suitable homogeneous Besov spaces, more precisely:
\begin{equation}\label{initial_condition}
\begin{alignedat}{32}
	\uu_0 								\in \BB_{2,1}^{\frac{\dd}{2}-1}							\cap 	\BB_{2,1}^{\frac{\dd}{2}},\quad\quad
	\Tt_0 	-\bar{\Tt}						\in \BB^{\frac{\dd}{2}-2}_{2,1}	\hspace{0.24cm}			\cap	\BB^{\frac{\dd}{2}}_{2,1},\quad\quad
	\nn_0 - \bar{\nn} \,	\in\, \BB_{2,1}^{\frac{\dd}{2}}					\cap 	\BB_{2,1}^{\frac{\dd}{2}+1},
\end{alignedat}
\end{equation}
where $\bar{\theta}$ is a fixed positive  temperature and $\bar{\nn}$ is a given unit vector in $\RR^\dd$. 

\noindent
In section \ref{sec:hom-besov} we provide a brief overview of the main properties of these functional spaces. We mention that Besov regularities of solutions for the Ericksen-Leslie model have already been introduced by the first author in \cite{Dea}.

\begin{remark}
 The main system \eqref{intro:main_system-wp} does not present the standard scaling of the Ericksen-Leslie model with constant coefficients: if $(\uu,\,\nn,\,\Tt)$ solves \eqref{intro:main_system-wp}, then  
 \begin{equation*}
 	(\uu_\lambda(t,x),\,\nn_\lambda(t,x),\,\Tt_\lambda(t,x))=(\lambda\uu(\lambda^2t,\lambda x),\,\nn(\lambda^2t,\lambda x),\,\lambda^2\Tt(\lambda^2t,\lambda x)),\quad\text{with}\quad \lambda>0,
 \end{equation*}
is not necessarily a new solution of \eqref{intro:main_system-wp}, since any coefficient of the system is a smooth function which depends on the temperature $\Tt$.  There is no meaning to define a critical regularity for system \eqref{intro:main_system-wp} and the regularities we assume on our initial data in \eqref{initial_condition} allow to well-define the composition between any viscous coefficients and the temperature function $\Tt(t,x)$. 

\noindent Furthermore, We remark that the intersection between homogeneous Besov spaces reminds the functional framework of the so-called hybrid Besov spaces (cf. \cite{D2}).
\end{remark}

\smallskip
\noindent
We can state our well-posedness result:
\begin{theorem}\label{main-thm3}
	Let us assume that $ (\,\uu_0,\,\Tt_0,\,\nn_0)$ fulfills the initial conditions given by \eqref{initial_condition}, 
	and moreover that they satisfy the smallness condition
	\begin{equation}\label{ineq:realsmallness-condition-main-thm}
	\begin{alignedat}{4}
		\| \,\,\nn_0- \bar{\nn}\,\,\|_{\BB_{2,1}^{\frac{\dd}{2}}\cap \BB_{2,1}^{\frac{\dd}{2}+1}} +
		\| \, \Tt_0\,-\,\bar{\Tt}\,\,\|_{\BB^{\frac{\dd}{2}}_{2,1}\cap \BB_{2,1}^{\frac{\dd}{2}-2}}+
		\| \,\,\uu_0\,\, \|_{\BB_{2,1}^{\frac{\dd}{2}-1}\cap \BB_{2,1}^{\frac{d}{2}}}  \, 
		\,\leq  \,\ee^4,
	\end{alignedat}
	\end{equation}
	for a suitable small positive constant $\ee>0$. Assume that the Leslie coefficients $\bar{\alpha}_4:=\alpha_4(\bar{ \Tt})$ is large enough.
	%Whenever $(\bar \rho,\,\bar \Tt)$ satisfies
	%Relaxing the molecular field on the 
	%director equation (see section \ref{iperjuve}),  
	Then system \ref{main_system}
	admits a unique global-in-time strong solution, 
	under the following class-affinity:
	\begin{equation}\label{defX}
	\begin{alignedat}{500}
	&\Big\{\;
		&&\uu 				&&&&\in L^\infty(\,\,\RR_+,\,\,
											\BB_{2,1}^{\frac{\dd}{2}-1}	
											&&&&&&&&\cap\, \BB_{2,1}^{\frac{\dd}{2}}\,\,
											&&&&&&&&&&&&&&&&)
											\quad\text{with}\quad 
		(\,\partial_t\uu,\,&&&&&&&&&&&&&&&&&&&&&&&&&&&&&&&&\Delta\uu\,
		&&&&&&&&&&&&&&&&&&&&&&&&&&&&&&&&&&&&&&&&&&&&&&&&&&&&&&&&&&&&&&&&)\,\in\,									
											L^1(\,\,\RR_+,\,\,\BB_{2,1}^{\frac{\dd}{2}-1}\cap \BB_{2,1}^\frac{\dd}{2}\,
																					)\;
&&&&&&&&&&&&&&&&&&&&&&&&&&&&&&&&&&&&&&&&&&&&&&&&&&&&&&&&&&&&&&&&&&&&&&&&&&&&&&&&&&&&&&&&&&&&&&&&&&&&&&&&&&&&&&&&&&&&&&&&&&&&&&&&	
\,
&&&&&&&&&&&&&&&&&&&&&&&&&&&&&&&&&&&&&&&&&&&&&&&&&&&&&&&&&&&&&&&&&&&&&&&&&&&&&&&&&&&&&&&&&&&&&&&&&&&&&&&&&&&&&&&&&&&&&&&&&&&&&&&&&&&&&&&&&&&&&&&&&&&&&&&&&&&&&&&&&&&&&&&&&&&&&&&&&&&&&&&&&&&&&&&&&&&&&&&&&&&&&&&&&&&&&&&&&&&&&&&&&&&&&&&&&&&&&&&&&&&&&&&&&&&&&&&&
	\Big\}									=: \X_1,\\	
	&\Big\{\;
		\Tt-&&\bar{\Tt}   	&&&&\in L^\infty
											(\,\,\RR_+,\,\,
											\BB_{2,1}^{\frac{d}{2}-2}
											&&&&&&&&\cap\, \BB_{2,1}^{\frac{d}{2}}\,\,
											&&&&&&&&&&&&&&&&)
											\quad\text{with}\quad
	(\,\partial_t\Tt,\,&&&&&&&&&&&&&&&&&&&&&&&&&&&&&&&&\Delta\Tt\,
	&&&&&&&&&&&&&&&&&&&&&&&&&&&&&&&&&&&&&&&&&&&&&&&&&&&&&&&&&&&&&&&&)\,\in\,											
											 L^1(\,\,\RR_+,\,\,\BB_{2,1}^{\frac{\dd}{2}-2}											 
											\cap\BB_{2,1}^{\frac{\dd}{2}}\,)\;
&&&&&&&&&&&&&&&&&&&&&&&&&&&&&&&&&&&&&&&&&&&&&&&&&&&&&&&&&&&&&&&&&&&&&&&&&&&&&&&&&&&&&&&&&&&&&&&&&&&&&&&&&&&&&&&&&&&&&&&&&&&&&&&&
\,	&&&&&&&&&&&&&&&&&&&&&&&&&&&&&&&&&&&&&&&&&&&&&&&&&&&&&&&&&&&&&&&&&&&&&&&&&&&&&&&&&&&&&&&&&&&&&&&&&&&&&&&&&&&&&&&&&&&&&&&&&&&&&&&&&&&&&&&&&&&&&&&&&&&&&&&&&&&&&&&&&&&&&&&&&&&&&&&&&&&&&&&&&&&&&&&&&&&&&&&&&&&&&&&&&&&&&&&&&&&&&&&&&&&&&&&&&&&&&&&&&&&&&&&&&&&&&&&&
	\Big\}												=:\X_2,\\
	&\Big\{\;
		\nn	-&&\bar{\nn}\, 	&&&&\in L^\infty
											(\,\,\RR_+,\,\,
											\BB_{2,1}^{\frac{\dd}{p}}
											&&&&&&&&\cap\, \BB_{2,1}^{\frac{\dd}{2}+1}\,		
											&&&&&&&&&&&&&&&&)
											\quad\text{with}\quad
	(\,\partial_t\nn,\,&&&&&&&&&&&&&&&&&&&&&&&&&&&&&&&&\Delta\nn\,
	&&&&&&&&&&&&&&&&&&&&&&&&&&&&&&&&&&&&&&&&&&&&&&&&&&&&&&&&&&&&&&&&)\,\in\,										
											L^1
											(\,\,\RR_+,\,\,\BB_{2,1}^{\frac{\dd}{2}}\cap \BB_{2,1}^{\frac{\dd}{2}+1}\,)\;
&&&&&&&&&&&&&&&&&&&&&&&&&&&&&&&&&&&&&&&&&&&&&&&&&&&&&&&&&&&&&&&&&&&&&&&&&&&&&&&&&&&&&&&&&&&&&&&&&&&&&&&&&&&&&&&&&&&&&&&&&&&&&&&&
\,
&&&&&&&&&&&&&&&&&&&&&&&&&&&&&&&&&&&&&&&&&&&&&&&&&&&&&&&&&&&&&&&&&&&&&&&&&&&&&&&&&&&&&&&&&&&&&&&&&&&&&&&&&&&&&&&&&&&&&&&&&&&&&&&&&&&&&&&&&&&&&&&&&&&&&&&&&&&&&&&&&&&&&&&&&&&&&&&&&&&&&&&&&&&&&&&&&&&&&&&&&&&&&&&&&&&&&&&&&&&&&&&&&&&&&&&&&&&&&&&&&&&&&&&&&&&&&&&&
	\Big\}									=: \X_3.
	\end{alignedat}
	\end{equation}
	Furthermore, the following inequalities hold:
	\begin{equation}\label{ineq:smallness-condition-main-thm}
	\begin{alignedat}{16}
		\|\,\;\uu\,\;\|_{\X_1}\,\leq\,\ee^2\quad\quad
		\|\,\Tt-\bar \Tt  \,\|_{\X_2}\,+\,
		\,\|\,\,\nn\,-\bar \nn\,\,\|_{\X_3} \, 
		\leq \ee^3,
	\end{alignedat}
	\end{equation}
	where the following norms are defined:
	\begin{equation*}
	\begin{alignedat}{64}
		&\|\,\hspace{0.43cm}\uu\,&&\|_{\X_1}\,&&&&:=\,\|\,\hspace{0.43cm}\uu\,&&&&&&&&\|_{L^\infty_t\BB_{2,1}^{\frac{\dd}{2}-1}\cap \BB_{2,1}^\frac{\dd}{2}}\,+\,
		\|\,\partial_t \uu\,&&&&&&&&&&&&&&&&\|_{L^1_t\BB_{2,1}^{\frac{\dd}{2}-1}\cap \BB_{2,1}^\frac{\dd}{2}}\,+\,
		\bar{\alpha}_4&&&&&&&&&&&&&&&&&&&&&&&&&&&&&&&&\|\,\Delta \uu\,&&&&&&&&&&&&&&&&&&&&&&&&&&&&&&&&&&&&&&&&&&&&&&&&&&&&&&&&&&&&&&&&\|_{L^1_t\BB_{2,1}^{\frac{\dd}{2}-1}\cap \BB_{2,1}^\frac{\dd}{2}},\\
		%-----------------------------------------------------------------------------------------------------------
		&\|\,\Tt-\bar{\Tt}\,&&\|_{\X_2}\,&&&&:=\,\|\,\Tt-\bar{\Tt}\,&&&&&&&&\|_{L^\infty_t\BB_{2,1}^{\frac{\dd}{2}-2}\cap \BB_{2,1}^\frac{\dd}{2}}\,+\,
		\|\,\partial_t \Tt\,&&&&&&&&&&&&&&&&\|_{L^1_t\BB_{2,1}^{\frac{\dd}{2}-2}\cap \BB_{2,1}^\frac{\dd}{2}}\,+\,
		&&&&&&&&&&&&&&&&&&&&&&&&&&&&&&&&\|\,\Delta \Tt\,&&&&&&&&&&&&&&&&&&&&&&&&&&&&&&&&&&&&&&&&&&&&&&&&&&&&&&&&&&&&&&&&\|_{L^1_t\BB_{2,1}^{\frac{\dd}{2}-2}\cap \BB_{2,1}^\frac{\dd}{2}},\\
		%-----------------------------------------------------------------------------------------------------------
		&\|\,\nn-\bar{\nn}\,&&\|_{\X_3}\,&&&&:=\,\|\,\,\nn-\bar{\nn}\,&&&&&&&&\|_{L^\infty_t\BB_{2,1}^{\frac{\dd}{2}}\cap \BB_{2,1}^{\frac{\dd}{2}+1}}\,+\,
		\|\,\partial_t \nn\,&&&&&&&&&&&&&&&&\|_{L^1_t\BB_{2,1}^{\frac{\dd}{2}}\cap \BB_{2,1}^{\frac{\dd}{2}+1}}\,+\,
		&&&&&&&&&&&&&&&&&&&&&&&&&&&&&&&&\|\,\Delta \nn\,&&&&&&&&&&&&&&&&&&&&&&&&&&&&&&&&&&&&&&&&&&&&&&&&&&&&&&&&&&&&&&&&\|_{L^1_t\BB_{2,1}^{\frac{\dd}{2}}\cap \BB_{2,1}^{\frac{\dd}{2}+1}}.
	\end{alignedat}
	\end{equation*}
\end{theorem}
\begin{remark}
For the definition of the homogeneous Besov space $\BB_{2,1}^s$, we refer the reader to Section \ref{sec:hom-besov}. We point out that the natural and widespread definition of these spaces requires the index of regularity $s$ to be bounded by $s\leq \dd/2$. The case of $s>\dd/2$ is usually treated in literature slightly modifying the Definition \ref{def:hom-bes}. For the sake of clarity, we then remark that in the above statements we have introduced an abuse of notation: a function $f$ belongs to the space $\BB_{2,1}^{\dd/2}\cap \BB_{2,1}^{\dd/2+1}$ if and only if $f$ and $\nabla f$ belong to $\BB_{2,1}^{\dd/2}$.
\end{remark}
\begin{remark}
	The anisotropic smallness condition we have introduced in \eqref{ineq:realsmallness-condition-main-thm} for the initial data 
	plays a fundamental rule when determining the uniform-in-time bound in \eqref{ineq:smallness-condition-main-thm}.  
	This is mainly due to the presence of linear terms in the main equations. 
	For the sake of clarity, we anticipate that any non-linear term allows the smallness condition in \eqref{ineq:smallness-condition-main-thm}. An example is given by the 
	the non-linear term of the Navier-Stokes-type equation:
	\begin{equation*}
		\|\,\uu\cdot\nabla \uu\,\|_{Y}\leq \|\,\uu\,\|_{\X_1}^2 \leq C \ee^4,
	\end{equation*}
	for a suitable norm $\|\cdot\|_Y$ we will introduce in section \ref{sec-wp}. Thus, assuming $\ee$ small enough, the above term is bounded by $\ee^2$. This property does not hold anymore whenever a linear term arises in the equation. We can 
	refer for instance to the tensor $\mcN\otimes \bar\nn$ related to the the Leslie stress $\sigma^{{\rm L}}$, for which we require a stronger smallness condition to the norm of the director field $\nn$.
	
	\noindent
	In order to solve such a challenging, we impose a large viscosity $\bar \alpha_4$ in the balance of linear momentum, which allows to close our uniform estimates.
\end{remark}

\begin{remark}
	The regularity of the initial velocity in Theorem \ref{main-thm3} is sufficient to generate a Lipschitz velocity fields. 
	We then claim that a  similar result to Theorem \ref{main-thm3} holds also assuming a non-constant density, inspired by the Lagrangian approach used in \cite{DM}. 
	This article does not treat such a situation, since the main system \ref{intro:main_system-wp} already presents non-trivial analytic challenges.
\end{remark}

\begin{remark}
In this paper, we will denote by $C$ any "harmless" constant, and we will sometimes use the notation $A\lesssim B$ equivalently to $A\leq \,C B$.
\end{remark}

\noindent
Before going on, let us give an overview of the paper. In section \ref{sec-1law} we prove Theorem \ref{main-thm1}, where the entropy production is expressed in terms of the viscous dissipation. In section \ref{sec-2law} we then establish Theorem \ref{main-thm2}, concerning the general conditions for the entropy production to be semi-positive defined. In section \ref{sec:hom-besov} we then recall some important tools concerning the Besov formalism we will use in section \ref{sec-wp}, when proving the global-in-time well-posedness of our main system \eqref{intro:main_system-wp}.

\section{The first law of thermodynamics for nematic liquid crystals}\label{sec-1law}
\noindent
In this section we aim to prove Theorem \ref{main-thm1}, where the entropy production \eqref{entropy-production} is explicitly determined by  the first law of thermodynamics. We recall that
the first laws of thermodynamics says that the rate of change of the total energy is totally transformed into work and heat. 
We identify the heat as the feedback of the environment returning on the system. We point-wise formulate the first law of thermodynamics by means of
\begin{equation*}
	\partial_t\,{\rm e}^{\,\rm{tot}} =\partial_t\left[ \, {\rm e}^{\rm int } + \frac{1}{2}\,\rho\,|\,\uu\, |^2\,\right] = \Div\, \Sigma \,+\, \Div\,{\bf q},
\end{equation*} 
where $\Sigma=\Sigma(\,\rho,\,\Tt,\, n,\nabla n)$ stands for the specific work done to the system and ${\bf q}$ is the heat flux lost by the system (from which the choice of a positive sign in front of it). 

\smallskip
\noindent
It is worth to remark that whenever a nematic occupies a bounded domain, the above identity requires our system to be local. More precisely there is no interactions between our material and the environment but the one passing through the boundary.  

\smallskip 
\noindent
Let us recall the statement we aim to prove:
\begin{theorem}\label{main-thm1b}
	The first law of thermodynamics leads to the following relation between the entropy production $\Tt\Delta^*$, 
	the heat flux $\mathbf{q}$, the co-rotational time flux $\mcN$ and the Leslie stress tensor $\sigma^{\rm \,L}$:
	\begin{equation}\label{entropy-productionb}
		\Tt\Delta^*=\, \sigma^{\,\rm L}:\mathbb{D} \,+\, \rm{g}\cdot \mcN \,+\,\frac{ {\bf q}\cdot\nabla \Tt }{\Tt}.
	\end{equation}
\end{theorem}
\begin{proof}
Since we assume the internal energy density ${\rm e}^{\rm int}\,=\,\FF\,+\,\Tt\eta$ of the system to smoothly depend on the state variables $(\,\rho,\,\Tt,\,\nn\,,\,\nabla \nn)$, the chain rule leads to
\begin{equation*}
	\partial_t\left[ \, \frac{1}{2}\,\rho\,|\,\uu\, |^2\,+\,{\rm e}^{\,\rm int }\,\right] \,=\,\partial_t \rho\frac{1}{2} |\,\uu\,|^2 	\,+\,
	 \,\rho \,\uu\cdot \partial_t\,\uu \,+ \,\frac{\partial \FF}{\partial \nn}\cdot\partial_t \nn	 \,+\,
	 \frac{\partial \FF}{\partial \nabla \nn}:\partial_t \nabla \nn + 
	 \Tt\, \partial_t \eta,
\end{equation*}
where we recall the definition of the local entropy $\eta\,=\,-\,\partial \FF/\partial \Tt$, given by the Maxwell's relation in \eqref{Maxwell_rel}. 
Let us remark that in the case of an incompressible fluids with constant density, we can neglect the time derivative of the density in the above identity. We then gather by the 
conservation of mass and the balance of linear momentum that 
\begin{equation*}
\begin{aligned}
	\partial_t\left[ \, \frac{1}{2}\,\rho\,|\,\uu\, |^2\,+\,{\rm e}^{\,\rm int }\,\right]	\,=\,-\,\Div\,(\,\rho\uu\,)\,\frac{1}{2} |\,\uu\,|^2 	\,&-\,\rho \,\uu\cdot\nabla \frac{\, |\,\uu\,|^2}{2}
	\,+\,\Div\,\mathbb{T}\,\cdot \uu\,+\\&+\,\frac{\partial \FF}{\partial \rho}\,\partial_t\rho\,+ \,\frac{\partial \FF}{\partial \nn}\cdot\partial_t \nn	 \,+\,
	 \frac{\partial \FF}{\partial \nabla \nn}:\partial_t \nabla \nn + 
	 \Tt\, \partial_t \eta,
\end{aligned}
\end{equation*}
where $\mathbb{T}=-\pre \Id\,+\,\sigma^{\rm\, E}\,+\,\sigma^{\rm\, L}$ is the total stress tensor. A further development leads the rate of the total energy to fulfill
\begin{equation*}
\begin{aligned}
	\partial_t\left[ \, \frac{1}{2}\,\rho\,|\,\uu\, |^2\,+\,{\rm e}^{\,\rm int }\,\right]	\,=\,\Div\Big[\,\mathbb{T}\,\uu -\rho \uu\,\frac{\,|\,\uu\,|^2}{2}\,\Big]
	\,&-\,\mathbb{T}:\nabla  \uu\,-\,\frac{\partial \FF}{\partial \rho}\Div(\, \rho\,\uu\,)	\,+\\&+ \,\frac{\partial \FF}{\partial \nn}\cdot\partial_t \nn	 \,+\,
	 \frac{\partial \FF}{\partial \nabla \nn}:\partial_t \nabla \nn + 
	 \Tt\, \partial_t \eta,
\end{aligned}
\end{equation*}
\noindent 
We now reformulate the term with the higher number of derivatives into a divergence form. We proceed as follows
%One of the most important step underlying our proof relies on rewriting the term with the highest number of derivatives through the relation
\begin{equation*}
		\frac{\partial \FF}{\partial \nabla \nn}:\partial_t \nabla \nn = \Div\Big[\,\tr \Big[\,\frac{\partial \FF}{\partial \nabla \nn}\,\Big]\partial_t \nn\,\Big]-
		\Div\, \Big[\,\frac{\partial \FF}{\partial \nabla \nn}\,\Big] \cdot \partial_t \nn,
\end{equation*}
The core of our model releases in the above identity: the divergence term on the right-hand side will be absorbed by the definition of the rate of work $\Sigma$, hence
the term with the highest number of derivatives will not contribute to the structure of the main system.
Thanks to the entropy equation of system \eqref{main_system}, the rate of the total energy density $\rm e^{\rm tot}$ satisfies
\begin{equation}\label{ineq:sec1law-rateenergy}
\begin{aligned}
	\partial_t\,{\rm e}^{\,\rm{tot}} = 
	\Div\bigg[\,\mathbb{T}\,\uu -\rho \uu\,|\,\uu\,|^2 \,+\, 
	\tr \Big[\,\frac{\partial \FF}{\partial \nabla \nn}\,&\Big]\partial_t \nn\,\,\bigg] \,-\,\mathbb{T}:\nabla \uu 
	\,-\,\mathbb{T}:\nabla  \uu\,-\,\frac{\partial \FF}{\partial \rho}\Div(\, \rho\uu\,)\,+
	\\&+\,	\frac{\delta\FF}{\delta \nn}\,\partial_t \nn\,
	-\,\Tt\,\Div(\,\uu\,\eta\,)\,
	+ 	\Tt\,\Div\left[\,\frac{\mathrm{\bf q}}{\Tt}\,\right]
	+\Tt \Delta^* ,
\end{aligned} 
\end{equation}
where we recall that the Volterra derivative $\delta \FF/\delta n$ stands for $\delta \FF/\delta n:= \partial_n\FF-\Div\,\partial_{\,\nabla n}\FF$. We remark that whenever the material is incompressible, then $\Div(\,\rho\,\uu\,)\partial \FF/\partial \rho$ is identically null. We now take into account the following relation:
%Now, observing that the incompressible condition  $\Div\,\uu\,=\,0$ leads to
\begin{equation*}
\begin{aligned}
	\Div\,\Big[\,\FF\,\uu\,\Big]=
	\nabla \FF \cdot \uu \,+\,\FF\,\Div\,\uu=\frac{\partial \FF}{\partial \rho}\,\uu\cdot \nabla \rho\,+\,\FF\,\Div\,\uu  \,+ \,\frac{\partial \FF}{\partial \Tt}\,\uu\cdot \nabla \Tt\, 
	+  \,\frac{\partial \FF}{\partial \nn}\cdot \Big(\,\uu\cdot \nabla \nn\,\Big)  + \\+
	\frac{\partial \FF}{\partial \nabla \nn}:\Big(\,\uu\cdot \nabla^2 \nn\,\Big)\,=\,-\,\Big[\,\rho \partial_\rho \FF\,-\,\FF\, \Big]\,\Div\,\uu\,+\,
	\frac{\partial \FF}{\partial \rho}\Div(\,\rho\,\uu\,)\,+ \,\Div(\,\eta\,\uu\,)\Tt\,-\\-\, \Div(\,\eta\,\uu\,\Tt\,)\,
	+  \,\frac{\partial \FF}{\partial \nn}\cdot \Big(\,\uu\cdot \nabla \nn\,\Big) \, + \,
	\frac{\partial \FF}{\partial \nabla \nn}:\Big(\,\uu\cdot \nabla^2 \nn\,\Big).
\end{aligned}
\end{equation*}
Replacing the above identity into the rate of the total energy density $\rm e^{\rm tot}$ in \eqref{ineq:sec1law-rateenergy}, we finally gather
\begin{equation}\label{ineq:sec1law-rateenergy2}
\begin{aligned}
		\partial_t\,&{\rm e}^{\,\rm{tot}} = 
		\Div\bigg[\,\mathbb{T}\,\uu - \uu\,\Big(\,|\,\uu\,|^2+\FF\,-\,\Tt\eta\Big)
		+ \tr\left[\frac{\partial \FF}{\partial \nabla \nn}\right]\partial_t \nn\,  \bigg] \,-\,
		\left(\,\mathbb{T}\,+\,\pre\,\Id\,\right):\nabla \uu \,+\\&
		+\,\frac{\delta \FF}{\delta \nn}\cdot \partial_t \nn 
		+  \frac{\partial \FF}{\partial \nn}\cdot \Big(\uu\cdot \nabla \nn\Big) +
		\frac{\partial \FF}{\partial \nabla \nn}:\Big(\uu\cdot\nabla^2 \nn\,\Big)+\Div\,{\bf q} 
		-\frac{{\bf q}\cdot\nabla \Tt}{\Tt}\,  +\,\Tt\Delta^*.
\end{aligned}
\end{equation}
we specify that the last relation holds both for the compressible and incompressible cases. Indeed whenever the density is non-constant then we recall that the pressure is defined by means of $\pre = \partial_\rho\FF - \FF$. On the other hand a free-divergence condition on the velocity field $\uu$ yields both $(\partial_\rho\FF - \FF)\Div\,\uu\,=\,0$ and $\pre \Id:\nabla \uu$ to be null. 

\smallskip\noindent
We now analyze the contribution of the free energy $\FF$ on the right-hand side of \eqref{ineq:sec1law-rateenergy2}. We first remark that
\begin{equation*}
\begin{aligned}
	\frac{\delta \FF}{\delta \nn}\cdot \partial_t \nn\,& +\,  \frac{\partial \FF}{\partial \nn}\cdot \Big(\,\uu\cdot \nabla \nn\,\Big)\, +\,
	\frac{\partial \FF}{\partial \nabla \nn}:\Big(\,[\,\uu\cdot\nabla\,]\nabla \nn\,\Big)\, =\\&=\,
	\,\frac{\delta \FF}{\delta \nn}\,\dot \nn\,+\,
	\Div\,\left[\,\tr\left[\frac{\partial \FF}{\partial \nabla \nn} \,\right]\uu\cdot \nabla \nn \,\right]
	\,-\,\Big[\,\tr \nabla \nn\, \frac{\partial \FF}{\partial \nabla \nn}\,\Big] :\nabla \uu\,,	
\end{aligned}
\end{equation*}
thus, replacing the above result into the last identity for the rate of the total energy $\rm e^{\rm\, tot}$ in \eqref{ineq:sec1law-rateenergy2}, we deduce
\begin{equation}\label{rate0a}
\begin{aligned}
		\partial_t\,{\rm e}^{\,\rm{tot}} = 
		\Div\,
		\Big[
			\,\mathbb{T}\,\uu - \uu\,\e^{\rm tot}\,+ \,\tr\Big[\frac{\partial \FF}{\partial \nabla \nn}\,\Big]\,\dot \nn\, 
		\Big] \,&-\,
		\Big(
			\,\mathbb{T}\,+\,\pre\,\Id\,+\,\tr \nabla \nn \, \frac{\partial \FF}{\partial \nabla \nn}\,
		\Big):\nabla \uu \,+\,	\\
		&+\,\frac{\delta \FF}{\delta \nn}\cdot \dot \nn\,+\,\Div\,\q 
		\,-\,\frac{{\q}\cdot\nabla \Tt}{\Tt} 
	 +\,\Tt\Delta^*.
\end{aligned}
\end{equation}
Thanks to the definition of the Ericksen stress tensor and the total stress tensor
\begin{equation*}
	\sigma^{\rm\, E}\,=\,-\Big[\,\tr \nabla \nn \,\Big]\, \frac{\partial \FF}{\partial \nabla \nn}\quad\quad\text{and}\quad\quad
	\mathbb{T}\,=\, - \pre \Id\,+\, \sigma^{\rm\, E}\,+\, \sigma^{\rm\, L},
\end{equation*}
we finally get that
\begin{equation}\label{rate0}
\begin{aligned}
		\partial_t\,{\rm e}^{\,\rm{tot}} = 
		\Div\,
		\Big[
			\,\mathbb{T}\,\uu - \uu\,\e^{\rm tot}\,+ \,\tr\Big[\frac{\partial \FF}{\partial \nabla \nn}\,\Big]\,\dot \nn\, 
		\Big] \,&-\,
		\sigma^{\rm\,L}:\nabla \uu \,
		+\,\frac{\delta \FF}{\delta \nn}\cdot \dot \nn\,+\,\Div\,\q 
		\,-\,\frac{{\q}\cdot\nabla \Tt}{\Tt} 
	 +\,\Tt\Delta^*.
\end{aligned}
\end{equation}
We now recall the work-postulate we have introduced in \eqref{def_work}. We denote by $\Sigma $  the rate at which the system do work on the nematic material, namely
\begin{equation*}
	\Sigma\, = \,\mathbb{T}\uu \,+ \tr\left[\frac{\partial \FF}{\partial \nabla \nn}\right]\,\dot \nn\,
	- \uu\,\e^{\rm tot}.
\end{equation*}
In the three-dimensional case $\Sigma$ reduces to the rate at which linear and angular moments do work on a nematic, that is
 \begin{equation*}
 	\Sigma\,=	
 	\mathbb{T}\uu\,+\, \mathbb{L}\,{\rm w}\,-\,\uu\,{\rm e}^{\rm tot}\,, 
 \end{equation*}
where ${\rm w}$ is the local angular velocity of the director $\nn$ and $\mathbb{L}$ is the couple stress tensor.
Inserting the above identity into the rate of the total energy density $\partial_t\rm e^{\rm tot}$ in \eqref{rate0} yields that
 \begin{equation}\label{rate1}
 \begin{aligned}
 	\partial_t\,{\rm e}^{\,\rm{tot}} \,=\,\Div\,\Sigma 
 	\,-\,  	
 	{\bf\sigma^{\rm\, L}}:\mathbb{D}\,-\,{\bf\sigma^{\rm\, L}}:\Omega
 	\,&\,+\,\frac{\delta \FF}{\delta \nn}\cdot \dot \nn
 	\,+\, \,\Div\,\,\mathrm{\bf q} \,-\frac{\mathrm{\bf q}\cdot\nabla \Tt}{\Tt} \,+\,
	 \Tt\,{\Delta}^*.
\end{aligned}
\end{equation}
The first law of thermodynamics holds if the rate of the total energy is totally transformed into work and heat, more precisely if and only if
\begin{equation*}
	\partial_t\,{\rm e}^{\,\rm{tot}} \,=\,\Div\,\Sigma \,+\,\Div\,\,\mathrm{\bf q}.
\end{equation*}
We then impose the extra term in \eqref{rate1} to be identically null. 
This yields the following balance between the entropy production $\Tt\Delta^*$, the heat flux $\mathbf{q}$, the Leslie stress tensor $\sigma^{\rm \,L}$ and the free energy density $\FF$:
 \begin{equation}\label{entropy-production-thm1-proof}
 	\Tt\,{ \Delta}^*= {\bf \sigma^{\rm \,L}}:\mathbb{D}\,+\,{\bf \sigma^{\rm \,L}}:\Omega\,-\frac{\delta \FF}{\delta \nn}\cdot \dot \nn\,+\,\frac{\mathrm{\bf q}\cdot\nabla \Tt}{\Tt} ,
 \end{equation}

\noindent\smallskip
We first reformulate the molecular-field term $\I:=\,\dot \nn\cdot\delta \FF/\delta \nn$ by means of the corotational time flux $\mcN$ and the kinematic transport $\rm g$. 
%This will arise to a contribution of the entropy production in terms of dissipation. 
The angular momentum equation of system \eqref{main_system} yields that 
\begin{equation}\label{I}
\I \,=\,-\,\mathrm{g}\cdot  \Big[ \partial_t \nn +\uu\cdot\nabla \nn\,\Big] +\underbrace{ \beta \nn\cdot \Big[ \partial_t \nn +\uu\cdot\nabla \nn\,\Big]}_{=0} = \,-\,\mathrm{g}\cdot \mcN\,-\,\mathrm{g}\cdot\Omega \,\nn, 
\end{equation}
where $\beta$ is the Lagrangian multiplier driving the constriction of unit-modulus $|\nn|^2= 1$. We recall moreover that the co-rotational time flux is defined by $\mcN \,=\, \partial_t\nn+\uu\cdot\nabla \nn\,-\,\Omega\, \nn$. 
Thanks to the definition \eqref{def_hg} we have $\mathrm{g}\otimes\nn-\nn\otimes {\rm g}=-[\,\sigma^{\rm\, L}-\tr \sigma^{\rm\, L}\,]$, which implies
\begin{equation*}
	\mathrm{g}\cdot \Omega \,\nn\, +\, \sigma^{\rm L}:\Omega\,= \,0,
\end{equation*}
Replacing the above identity into the molecular-field term $\I$ in \eqref{I} leads to $\I\,=\,\mathrm{g}\cdot \mcN\,-\,\sigma^{\rm\,L}:\Omega$. Hence the entropy production in \eqref{entropy-production-thm1-proof} can  be rewritten as follows:
 \begin{equation*}
 \begin{aligned}
 	\Tt\,{ \Delta}^*\,=\, \sigma^{\,\rm L}:\mathbb{D} \,+\, \rm{g}\cdot \mcN \,+\,\frac{ {\bf q}\cdot\nabla \Tt }{\Tt},
\end{aligned}
\end{equation*}
which corresponds to \eqref{entropy-productionb}. This concludes the proof of Theorem \ref{main-thm1b}.
\end{proof}
%%%%%%%%%%%%%%%%%%%%%%%%%%%%%%%%%%%%%%%%%%%%%%%%%%%%%%%%%%%%%%%%%%%%%%%%%%%%%%%%%%%%%%%%%%%%%%%%%%%%%%%%%%%%%%%%%%%%%%%%%%%%%%%%%%%%%%%%%%%%%%%%%%%%%%%%%%%%%%%%
%%%%%%%%%%%%%%%%%%%%%%%%%%%%%%%%%%%%%%%%%%%%%%%%%%%%%%%%%%%%%%%%%%%%%%%%%%%%%%%%%%%%%%%%%%%%%%%%%%%%%%%%%%%%%%%%%%%%%%%%%%%%%%%%%%%%%%%%%%%%%%%%%%%%%%%%%%%%%%%%
%%%%%%%%%%%%%%%%%%%%%%%%%%%%%%%%%%%%%%%%				SECOND LAW OF THEORMODYNAMICS	%%%%%%%%%%%%%%%%%%%%%%%%%%%%%%%%%%%%%%%%%%%%%%%%%%%%%%%%%%%%%%%%%%%%%%%%%%
%%%%%%%%%%%%%%%%%%%%%%%%%%%%%%%%%%%%%%%%%%%%%%%%%%%%%%%%%%%%%%%%%%%%%%%%%%%%%%%%%%%%%%%%%%%%%%%%%%%%%%%%%%%%%%%%%%%%%%%%%%%%%%%%%%%%%%%%%%%%%%%%%%%%%%%%%%%%%%%%
%%%%%%%%%%%%%%%%%%%%%%%%%%%%%%%%%%%%%%%%%%%%%%%%%%%%%%%%%%%%%%%%%%%%%%%%%%%%%%%%%%%%%%%%%%%%%%%%%%%%%%%%%%%%%%%%%%%%%%%%%%%%%%%%%%%%%%%%%%%%%%%%%%%%%%%%%%%%%%%%
\section{The second law of thermodynamics for nematic liquid crystals}\label{sec-2law}
\noindent In this section we deal with the second law of thermodynamics for compressible nematic liquid crystals. This principle is known as the Clausius-Duhem inequality and it infers that the entropy production $\Tt\Delta^*$ must be semi-positive defined. 
Hence, according to Theorem \eqref{main-thm1} and the balance for the entropy production in \eqref{entropy-production}, the second law of thermodynamics can be split into two parts:
\begin{equation}\label{dissipation}
	\sigma^{\rm \,L}:\mathbb{D}\,+\,\mathrm{g}\cdot\mcN \geq 0,\quad\quad\text{and}\quad\quad
	\frac{{\bf q}\cdot \nabla \Tt}{\Tt} \geq 0.
\end{equation}
Because of the material-frame indifferent, the Leslie stress tensor $\sigma^{\rm L}$ should be generally taken as
a smooth tensor depending on $(\rho,\,\Tt)$ and as a smooth isotropic tensor depending on $(\nn,\,\mcN,\,\mathbb{D})$. 
Thus, in order to better develope the Clausius-Duhem inequality, in this section we keep the assumptions made by Ericksen and Leslie of $\sigma^{\rm\, L}$ to be linearly dependent upon the couple $(\mcN$, $\mathbb{D})$. With similar arguments as the one reported by Stewart in \cite{Stewart} (we refer the reader to section $4.2.3$), the explicit formula for the Leslie tensor 
$\sigma^{\rm L}=\sigma^{\,\rm L}(\,\rho,\,\Tt,\,\,\nn,\,\mcN,\,\mathbb{D}\,)$ given by
\begin{equation}\label{Leslie-stressb}
	\begin{aligned}
	\sigma^{\,\rm L}= 
	\alpha_0\,\,\big[\nn\cdot \mathbb{D}\nn\big]\,&\Id\,+\,
	\alpha_1\,	\big[\,\nn\cdot \mathbb{D}\nn\,\big]\,\nn\otimes\nn\,+\,
	\alpha_2\,	\mcN\otimes \nn\,	+\,
	\alpha_3\,	\,\nn\otimes \mcN\,+\,
	\alpha_4\,	\mathbb{D}\,+\,
	\\&+\,
	\alpha_5\,	\nn\otimes \mathbb{D}\nn\,+\,
	\alpha_6\,	\mathbb{D}\nn\otimes \nn\,+\,
	\alpha_7\,	\trc\,	\mathbb{D}\,\Id\,
	+\,
	\alpha_8\,	\,\trc\,\mathbb{D}\,\,\nn\otimes \nn
	%+\,\\\,&
	%+\,
	%\alpha_9 (\nn\cdot \nabla \Tt)\Id\,+\,
	%\alpha_{10} (\nn\cdot \nabla \Tt)\nn\otimes \nn\,\,+
	%\alpha_{11}\nabla \Tt\otimes \nn\,+\,
	%\alpha_{12}\nn\otimes \nabla \Tt
	%%\\&+
	%\beta_1(\nabla \Tt\cdot \mathbb{D}\nn)\nn\otimes \nn+
	%\beta_2 \nabla \Tt\otimes \mcN\,+\,
	%\beta_3	\mcN\otimes \nabla \Tt\,+\\\,&+
	%\beta_4\,\nabla \Tt\otimes \mathbb{D}	\nn\,+\,
	%\beta_5 \mathbb{D}\nn\otimes \nabla \Tt\,+\,
	%\beta_6\trc\,\mathbb{D}\;\nn\otimes \nabla \Tt\,+
	%\beta_7\trc\,\mathbb{D}\;\nabla \Tt\otimes \nn
	.
\end{aligned}
\end{equation}
In this section, any $\alpha$-coefficient is considered to be a smooth function depending on the density $\rho$ and the absolute temperature $\Tt$.
The coefficients from $\alpha_1$ until $\alpha_6$ are in a one to one relation with the classical Leslie viscosities. The coefficient $\alpha_0$ is not new in the incompressible Ericksen-Leslie theory (we refer for instance to \cite{Stewart}, term $\mu_9$ in $(4.74)$), however it is usually neglected since absorbed by the definition of the pressure. Assuming a compressible nematics, we preserve such a term in the definition of the Leslie tensor. 

\smallskip\noindent
Because of the compressible condition, we have also introduced the new terms $\alpha_7$ and $\alpha_8$ that disappear whenever a free-divergence condition is imposed to the velocity field.  
It is worth to remark that $\alpha_7$ coincides with the non-homogeneous viscosity of the Cauchy tensor for isotropic fluid, while the $\alpha_8$-term is necessary to keep $\sigma^{\rm L}$ as the most general transversely isotropic tensor with respect to $\nn$.

\smallskip
\noindent
We assume the heat flux ${\bf q}={\bf q}(\,\rho,\,\Tt,\,\nn,\,\nabla \Tt)$ to be smooth on $(\rho,\,\Tt)$ and linear isotropic on $\nabla \Tt$. Thus we can write it as
\begin{equation}\label{heat_flux2}
	{\bf q}\,=\,
	\lambda_1(\,\rho,\,\Tt)\,\nabla \Tt\,
	+\lambda_2(\,\rho,\,\Tt)\,\big[\nn\cdot \nabla \Tt\big]\nn
	%-\lambda_3\big[\mathbb{D}\nn\cdot \nabla \Tt\big]\mathbb{D}\nn
	%-\lambda_4 \big[\mcN\cdot \nabla \Tt\big]\mcN,
\end{equation}
with $\lambda_1$ and $\lambda_2$ two smooth functions depending on the couple $(\rho,\,\Tt)$. Under the explicit formulas given by \eqref{Leslie-stressb} and \eqref{heat_flux2}, the Clausius-Duhem inequality \eqref{dissipation} becomes
\begin{equation}\label{ineq1a}
\begin{aligned}
	&\alpha_0\,(\, \nn\cdot  \mathbb{D}\nn\,) \trc\,\mathbb{D}	\,+\,
	\alpha_1 \,|\,\nn\cdot \mathbb{D}\nn\,|^2\,+\,(\,\alpha_2\,\,\alpha_3\,+\,\alpha_5\,-\,\alpha_6\,)\,\mcN\cdot \mathbb{D}\nn\,+\,
	\alpha_4\,|\,\mathbb{D}\,|^2\,+\\&+\,(\,\alpha_5\,+\,\alpha_6\,) \,|\,\mathbb{D}\nn\,|^2\,+\,
	(\,\alpha_3-\alpha_2\,)\,|\,\mcN\,|^2\,+\,(\,\alpha_7\,+\,\alpha_0\,)\,\trc^2\, \mathbb{D} \,+\,\alpha_8\,\trc\,\mathbb{D}\,(\,\nn\cdot \mathbb{D}\nn\,)\,\,\geq \,0,
\end{aligned}
\end{equation}
together with 
\begin{equation}\label{ineq2a}
	\lambda_1|\,\nabla \Tt\,|^2\,+\,\lambda_2|\,\nn\cdot \nabla \Tt\,|^2
	\geq\, 0.
\end{equation}
We first observe that from \eqref{ineq2a} the $\lambda$-coefficients must satisfy 
\begin{equation*}
	\lambda_1 \geq 0,\quad \text{and}\quad \lambda_1\,+\,\lambda_2\geq 0,
\end{equation*}
which correspond to the first two inequalities in \eqref{inequalities_viscous_diss} of Theorem \ref{main-thm2}. 

\smallskip\noindent
We then focus on the condition given by \eqref{ineq1a}. As for a classical isotropic fluids, it is worth to analyze this dissipation by means of $\tilde {\mathbb{D}}:=\mathbb{D}-\trc\,\mathbb{D}\Id/\dd$, the projection of $\mathbb{D}$ into the traceless-matrices vector space, and its orthogonal matrix $\trc\,\mathbb{D}\Id/\dd$. Hence, \eqref{ineq1a} reduces to
\begin{equation}\label{ineq2}
\begin{aligned}
	&\alpha_1 \,|\,\nn\cdot \tilde{ \mathbb{D}}\nn\,|^2+(\,\alpha_2\,+\,\alpha_3\,+\,\alpha_5\,-\,\alpha_6\,)\,\mcN\cdot \tilde{ \mathbb{D}}\nn+
	\alpha_4\,|\,\tilde{ \mathbb{D}}\,|^2 +(\,\alpha_5\,+\,\alpha_6\,) \,|\,\tilde{ \mathbb{D}}\nn\,|^2+\\&+\,
	(\,\alpha_3-\alpha_2\,)\,|\,\mcN\,|^2\,+\,(\,\alpha_0\,+\,\alpha_4\,+\,\alpha_7\,+\,\alpha_8\,)\,\trc^2\, \mathbb{D}\,+\,
	(\,\alpha_0\,+\alpha_1\,+\,\alpha_8\,)\trc\,\mathbb{D}\,\,(\,\nn\cdot \tilde{\mathbb{D}}\nn\,)\,
	\geq\, 0.
\end{aligned}
\end{equation}
We then localize two main terms: the classical Ericksen-Leslie viscous dissipation and an additional contribution related to the compressible condition  $\trc\,\mathbb{D}\neq 0$. Seeking for the most general condition on the $\alpha$-coefficients, let us remark that we cannot separately analyze the dissipation given by $\tilde{\mathbb{D}}$ and the one given by $\trc\,\mathbb{D}\Id$. The anisotropic structure of the total stress tensor is reflected by a non-trivial interaction between $\mathbb{D}$ and $\tilde{\mathbb{D}}$. We then need to take into consideration the entire set of $\alpha$-coefficients.

\smallskip\noindent
We aim to prove the following statement:
\begin{theorem}\label{thm2bbb}
Inequality \eqref{ineq2} is satisfied if and only if
\begin{equation}\label{EL-ineq}
\begin{aligned}
	\alpha_3-\alpha_2 \geq 0 ,\quad
	\alpha_4 \geq 0,\quad\quad
	2\alpha_4+\alpha_5 + \alpha_6 \geq 0,\quad
	N\,(\,\alpha_1\,+\,\alpha_5 \,+\, \alpha_6\,)\, +\, (\,N\,+\,1\,)\alpha_4\geq 0, \\
	\alpha_4(\,\alpha_0\,+\,\alpha_4\,+\,\alpha_7\,+\,\alpha_8\,)\big\{
	(\,\dd\,-\,1\,)\,(\,\alpha_1\,+\,\alpha_5\,+\,\alpha_6\,)\,+\,\dd\,\alpha_4 \big\}
	 \geq\,(\dd\,-\,1)\frac{(\,\alpha_0\,+\,\alpha_1\,+\,\alpha_8\,)^2}{4}
\end{aligned}
\end{equation}
for any $N=2,\dots,\dd-1$.
\end{theorem}
\begin{proof}
	For any fixed $\nn\in\RR^\dd$ and $\mcN\in\RR^\dd$, we consider an orthonormal basis $( e_1,\,e_2,\,e_3)$ fulfilling
	\begin{equation*}
		\nn\,=\,e_1\quad\quad \mcN\,=\mathcal{N}\,e_2\quad\text{and}\quad \mathbb{D}\,=\,\mathbb{A}_{ij}e_i\otimes e_j.
	\end{equation*}	  
	Hence, the Duhem-Clausius inequality \eqref{ineq2} readily reduces to
	\begin{equation*}
	\begin{aligned}
			\alpha_1\tilde{\mathbb{A}}_{11}^2	\,&+\,
			(\alpha_2\,+\,\alpha_3 \,+\,\alpha_5\,-\alpha_6\,)\mathcal{N}\tilde{\mathbb{A}}_{12}\,+\,
			\alpha_4\,\tilde{ \mathbb{A}}_{ij}\tilde{\mathbb{A}}_{ij}\,+\,
			(\,\alpha_5\,+\,\alpha_6\,)
			\tilde{\mathbb{A}}_{1j}\tilde{\mathbb{A}}_{1j}\,+\,\\\,&+\,
			(\,\alpha_3\,-\,\alpha_2\,)
			\mathcal{N}^2
			\,+\,(\alpha_0\,+\alpha_4\,+\,\alpha_7\,+\,\alpha_8\,)\trc^2 \mathbb{A}\,+\,
			(\,\alpha_0\,+\,\alpha_1\,+\,\alpha_8\,)(\,\trc\,\mathbb{A}\,)\,\mathbb{A}_{11}\,\geq\, 0.
	\end{aligned}
	\end{equation*}
	Using the free-trace property of $\tilde{\mathbb{A}}$, we can replace the term $\tilde{ \mathbb{A}}_{\dd\dd}$ 
	by $-(\tilde{\mathbb{A}}_{11}\,+\,\dots\,+\tilde{ \mathbb{A}}_{(\dd-1),(\dd-1)})$, which allows to split the above inequality into three independent parts:
	\begin{equation*}
		(\,\alpha_2\,-\,\alpha_3\,)\mathcal{N}^2\,+ \,(\alpha_2\,+\alpha_3\,+\,\alpha_5\,-\alpha_6\,)\mathcal{N}\tilde{\mathbb{A}}_{12}
		\,+\,(\,\alpha_5\,+\,\alpha_6\,+\,2\alpha_4\,)\tilde{\mathbb{A}}_{12}^2\,\geq\, 0,
	\end{equation*}
	together with
	\begin{equation*}
		(\,2\alpha_4\,+\,\alpha_5\,+\,\alpha_6\,)\tilde{\mathbb{A}}_{1\dd}^2\,+\,
		2\alpha_4\sum_{2 \leq i< j\leq \dd-1}\tilde{\mathbb{A}}_{ij}^2\geq 0.
	\end{equation*}
	and finally
	\begin{equation}\label{second_ineq}
	\begin{aligned}
		(\,\alpha_1\,&+\,\alpha_5\,+\,\alpha_6\,+2\alpha_4)\tilde{\mathbb{A}}_{11}^2\,+\,
		(\alpha_0\,+\alpha_4\,+\,\alpha_7\,+\,\alpha_8\,)\trc^2 \mathbb{A}\,+\\&+\,2\alpha_4\sum_{i=2}^{\dd-1}
		\tilde{\mathbb{A}}_{ii}^2\,+
		2\alpha_4\sum_{1\leq i<j\leq \dd-1} \tilde{\mathbb{A}}_{ii}\tilde{\mathbb{A}}_{jj}\,+\,
		(\,\alpha_0\,+\,\alpha_1\,+\,\alpha_8\,)(\,\trc\,\mathbb{A}\,)\,\mathbb{A}_{11}\geq 0.
	\end{aligned}
	\end{equation}
	The first two inequalities reduce to the standard conditions for the Leslie coefficients (cf. \cite{Stewart}, inequalities $(4.91)-(4.95)$): 
	\begin{equation*}
		\alpha_3\,-\,\alpha_2\geq 0,\quad\quad\alpha_4 \geq 0,\quad\quad \alpha_5\,+\,\alpha_6\,+\,2\alpha_4\geq 0.
	\end{equation*}
	The main peculiarity of our model relies on the third inequality \eqref{second_ineq}, where the term	$\trc\,\mathbb{A}$ is not necessarily null.	
	This inequality holds whenever the following symmetric matrix is semi-positive defined:
	\begin{equation*}
	\mathbb{M}\,:=\,
		\left(\,
		\begin{matrix}
			(\alpha_1\,+\,\alpha_5\,+\,\alpha_6\,+\,2\alpha_4)&\alpha_4 &\alpha_4
			&\dots&\alpha_4&\frac{1}{2}(\,\alpha_0\,+\,\alpha_1\,+\,\alpha_8\,)\\
			\alpha_4&2\alpha_4&\alpha_4&\dots&\alpha_4&0\\
			\alpha_4&\alpha_4&2\alpha_4&\dots&\alpha_4&0\\
			\vdots&	\vdots&\vdots&\ddots&\vdots&\vdots\\
			\alpha_4&\alpha_4& \alpha_4&\dots&2\alpha_4&0\\
			\frac{1}{2}(\,\alpha_0\,+\,\alpha_1\,+\,\alpha_8\,)&0&\dots&\dots&0&(\,\alpha_0\,+\,\alpha_4\,+\,\alpha_7\,+\,\alpha_8\,)
		\end{matrix}
		\,\right)\,\in\,\RR^{\dd\,\times\,\dd}.
	\end{equation*}
	We then we apply the 	Sylvester's criterion, for which $\mathbb{M}$ is semi-positive defined whenever any
	leading principal minor has positive determinant. The first leading principal minor we consider is
	\begin{equation}\label{juve1}
	\left(
		\begin{matrix}
			(\alpha_1\,+\,\alpha_5\,+\,\alpha_6\,+\,2\alpha_4)&\alpha_4\\
			\alpha_4&2\alpha_4&
		\end{matrix}
	\right)\quad \Rightarrow\quad 
	\left\{
	\begin{aligned}
		\alpha_4 &\geq 0,\\
		2(\alpha_1\,+\,\alpha_5\,+\,\alpha_6\,)\,+\,3\alpha_4 &\geq 0  .
	\end{aligned}	
	\right.
	\end{equation}
	Next, for any $N\in{3,\dots,\dd-1}$, we denote by $\mathbb{M}_N$ the $N\times N$-matrix  defined by the first $N$ rows and 
	$N$ columns of $\mathbb{M}$ . In order to prove that $\det \mathbb{M}_N$ is semi-positive, 
	we make use of the following lemma, whose proof is postponed to the end of this section:
	\begin{lemma}\label{lemma_matrix1}
	Let $x$, $y$ and $z$ be three real numbers and let $A$ be am $N\times N$ matrix, with $N\geq 2$ defined by
	\begin{equation*}
		A\,:=\,\left(\,
		\begin{matrix}
		x & y & y & \dots & y\\
		y & z & y & \dots & y\\
		y & y & z & \dots & y\\	
		\vdots &\vdots  &\vdots &\ddots	  & \vdots\\
		y &y &y & \dots & z
		\end{matrix}
		\,\right).
	\end{equation*}
	Then the determinant of $A$ satisfies
	\begin{equation}\label{det_appx}
		\det\,A \,=\,(z-y)^{N-2}\big[\,x\,z\,+\,(N-2)x\,y\,-\,(N\,-\,1)y^2\,\big].
	\end{equation}
	\end{lemma}
	
	\smallskip\noindent
	Replacing $x=\alpha_1\,+\,\alpha_5\,+\,\alpha_6\,+\,2\alpha_4$, $y=\alpha_4$  and $z=2\alpha_4$ in the above lemma, we obtain that
	\begin{equation*}
		\det \mathbb{M}_N\,=\,\alpha_4^{N-1}\Big[\,N(\alpha_1\,+\,\alpha_5\,+\alpha_6)\,+\,(N\,+\,1)\alpha_4\Big],
	\end{equation*}
	for any $N=3,\dots,\dd-1$. Thus, $\mathbb{M}_{N} $ is semi-positive defined, if and only if
	\begin{equation}\label{juve2}
	N(\alpha_1\,+\,\alpha_5\,+\alpha_6)\,+\,(N\,+\,1)\alpha_4\geq 0,
	\end{equation}
	I remains to impose the determinant of $\mathbb{M}$ to be semi-positive.
	We claim that such a determinant is characterized by the following formula:
	\begin{equation}\label{juve5}
	\begin{aligned}
		\det\,\mathbb{M}\,=\,\alpha_4^{\dd\,-\,2\	}	
		\Big[
			\alpha_4(\,\alpha_0\,+\,\alpha_4\,+\,\alpha_7\,&+\,\alpha_8\,)\big\{
			(\,\dd\,-\,1\,)\,(\,\alpha_1\,+\,\alpha_5\,+\,\alpha_6\,)\,+\\&+\,\,\dd\,\,\alpha_4 \big\}
			-\,(\dd\,-\,1)\frac{(\,\alpha_0\,+\,\alpha_1\,+\,\alpha_8\,)^2}{4}
		\Big].
	\end{aligned}
	\end{equation}
	Indeed, applying twice the Leibenitz formula leads to
	\begin{equation}\label{juve4}
		\det\,\mathbb{M}\,=\,
		(\,\alpha_0\,+\,\alpha_4\,+\,\alpha_7\,+\,\alpha_8\,)
		\det\,\mathbb{B}_1\,-\,
		\frac{(\,\alpha_0\,+\,\alpha_1\,+\,\alpha_8\,)^2}{4}
		\det\,\mathbb{B}_2,
	\end{equation}
	where $\mathbb{B}_1$ and $\mathbb{B}_2$ stand for the matrices
	\begin{equation*}
	\begin{alignedat}{4}
		\mathbb{B}_1&:=
		\left(\,
		\begin{matrix}
			\alpha_0\,+\,\alpha_1\,+\,\alpha_5\,+\,\alpha_6\,+\,2\alpha_4&\alpha_4&\dots&\alpha_4\\
			\alpha_4&2\alpha_4&\dots&\alpha_4\\
			\vdots&\vdots&\ddots&\vdots\\
			\alpha_4& \alpha_4&\dots  &2\alpha_4
		\end{matrix}
		\,\right)\,&&\in\,\RR^{(\dd-1)\times (\dd-1)},\\
		\mathbb{B}_2&:=\hspace{1cm}
		\left(\,
		\begin{matrix}
			2\alpha_4	&	\alpha_4	&\dots	&	\alpha_4			\\
			\alpha_4	&	2\alpha_4	&\dots	&	\alpha_4			\\
			\vdots		&	\vdots		&\ddots	&	\vdots				\\
			\alpha_4	&	\alpha_4	&\dots	&	2\alpha4 	
		\end{matrix}
		\,\right)&&\in \RR^{(\dd-2)\times (\dd-2)}.
	\end{alignedat}
	\end{equation*}
	Hence, making use of Lemma \ref{lemma_matrix1}, we achieve that
	\begin{equation}\label{juve3}
	\begin{aligned}
		\det	\, 	\mathbb{B}_1	\,&=\,\alpha_4^{\dd\,-\,3}
		\left[\, 
			\,(\,\dd\,-\,1\,)\,(\,\alpha_0\,+\,\alpha_1\,+\,\alpha_5\,+\,\alpha_6\,)\,+\,\,\dd\,\,\alpha_4
		\,\right],		\\
		\det 	\,	\mathbb{B}_2	\,&=\, \alpha_4^{\dd\,-\,2}\,(\,\dd \,-\,1\,).
	\end{aligned}
	\end{equation}
	Combining \eqref{juve3} together with \eqref{juve4} finally leads to the identity
	\eqref{juve5}, from which we deduce that $\det\,\mathbb{M}>0$ if and only if
	\begin{equation*}
		\alpha_4(\,\alpha_0\,+\,\alpha_4\,+\,\alpha_7\,+\,\alpha_8\,)\big\{
			(\,\dd\,-\,1\,)\,(\,\alpha_1\,+\,\alpha_5\,+\,\alpha_6\,)\,+\,\dd\,\,\alpha_4 \big\}
			\geq(\dd\,-\,1)\frac{(\,\alpha_0\,+\,\alpha_1\,+\,\alpha_8\,)^2}{4}.
	\end{equation*}
	Summarizing the above inequality together with \eqref{juve1} and \eqref{juve2}, finally leads to \eqref{EL-ineq}, which
	concludes the proof of Theorem \ref{thm2bbb}.
\end{proof}
\noindent
We now perform the proof of Lemma \ref{lemma_matrix1}.
\begin{proof}[Proof of Lemma \ref{lemma_matrix1}]
	We proceed by induction. When $N=2$ the matrix $A$ reduces to 
	\begin{equation*}
		A\,=\,\left(\,
		\begin{matrix}
			x&y\\
			y&z\,
		\end{matrix}\right)
		\quad \text{with}\quad 
		\det\,A = x\,z-y^2,
	\end{equation*}
	from which we achieve the base case. Let us assume that \eqref{det_appx} is true for $N-1$. Then the Leibenitz formula together 
	with the induction hypotheses yields
	\begin{equation}\label{app-rel1}
	\begin{aligned}
		\det A\,&=\,
		x\left|\,
		\begin{matrix}
		 z & y & \dots & y\\
		 y & z & \dots & y\\
		\vdots &   &\ddots	  & \vdots\\
			y  &\dots & \dots & z
		\end{matrix}\,
		\right|\,+\,
		y\sum_{i=1}^{N-1}(-1)^{i}
		\det
		\,{\rm B_i}\\
		&=(\,z\,-\,y\,)^{N-3}\big[\,x\,z^2\,+\,(N-3)x\,z\,y\,-\,(N\,-\,2)\,x\,y^2\,\big]\,+\,
		y\sum_{i=1}^{N-1}(-1)^{i}
		\det
		\,{\rm B_i}\\
		&=(\,z\,-\,y\,)^{N-2}\big[\,x\,z\,+\,(N\,-\,2)\,x\,y\,\big]	\,+\,
		y\sum_{i=1}^{N-1}(-1)^{i}
		\det
		\,{\rm B_i},	
		\end{aligned}
	\end{equation}
	where ${\rm  B_i}$ is the $(\,N-1\,)\times (\,N-1\,)$ matrix defined as follows:
	\begin{itemize}
		\item		the components of the $i$-th row are all equal to $y$,
		\item		the $j$-th row, with $j<i$ is composed by $z$ in the $i$-th column and by $y$ elsewhere,
		\item		the $j$-th row, with $j>i$ is equal to $z$ in the $i+1$-th column and $y$ elsewhere.
	\end{itemize}	
	Applying $i-1$ permutations, the matrix ${\rm B_i}$ always reduces to
	\begin{equation*}
	\left(\,
	\begin{matrix}	
		 y & y & \dots & y\\
		 y & z & \dots & y\\
		\vdots &   &\ddots	  & \vdots\\
			y  &\dots & \dots & z
		\end{matrix}\,
	\right), \quad \text{hence}\quad
	(-1)^i\det {\rm B_i}\,=\,-(\,z\,-y\,)^{N-3}(\,yz\,-y^2\,)\,=\,-(\,z\,-y\,)^{N-2}y.
	\end{equation*}	 
	Thus, replacing the above identity into the relation \eqref{app-rel1}, we gather
	\begin{equation*}
		\det\,A\,=\,(\,z\,-\,y\,)^{N-2}\,\big[\,x\,z\,+(N-2)x\,y\,-(N\,-\,1\,)\,y^2\,
		\big]
	\end{equation*}
	which concludes the proof of the lemma.
\end{proof}

\section{Besov Spaces}
\noindent
The purpose of this section is to recall some important tools of the Littlewood-Paley decomposition we will use in section \ref{sec-wp}, when proving the well-posedness of system \eqref{intro:main_system-wp}. We first recall some product and composition rules that play a key role when estimating some suitable approximate solutions. Then, we deal with some results concerning the propagation of Besov regularities for linear parabolic PDE's. We refer the reader to \cite{B-C-D}, for more specifics.

\subsection{Homogeneous Besov spaces}\label{sec:hom-besov}$\,$

\smallskip
\noindent
We  define $\mathcal{C}$ to be the ring of center $0$, of small radius $1/2$ and great radius $2$. There exist
a non-negative radial function $\varphi$ belonging to ${\mathcal{D}} (\mathcal{C}) $ that decomposes the unity as follows
\begin{equation*}
\sum_{q\in \ZZ} \varphi \Big(\frac{\xi}{2^q}\Big) = 1,\quad\text{for any}\quad \xi\in\mathbb{R}^\dd,
\end{equation*}
and such that any couple with large distance between indexes do not interact, in the following sense: for any $p\in\ZZ$ and 
$q\in \ZZ$ with distance $|\,p\,-\,q\,|\geq 5$, we get
\begin{equation*}
{\rm Supp}\,\, \varphi(2^{-q}\cdot)\cap {\rm Supp}\,\, \varphi(2^{-p}\cdot)=\emptyset.
\end{equation*}
We denote by $\Ff$ the Fourier transform acting on $\mathbb{R}^\dd$. Then we define the homogeneous dyadic block
$\Dd_q$ and the operator $\Sd_q$ through the relations
\begin{equation*}
\begin{aligned}
	\Dd_q\, u \,&= \Ff^{-1}(\varphi(2^{-q}\xi)\Ff u), \\
	\Sd_{q}\,u \,&=\,\sum_{j\,\leq\, q-1}\,\Dd_j\,u,
\end{aligned}
\end{equation*}
for any integer $q$.
We recall that for two appropriately smooth functions $a$ and $b$ we have the so-called Bony's decomposition
\cite{B-C-D}:
\begin{equation*}
	ab\,=\,\dot T_a\, b\,+\,\dot T_b\, a\,+\,\dot R(a,b)
\end{equation*} 
where 
\begin{equation*}
	\dot T_a\, b=\sum_{q\in\ZZ}\Sd_{q-1} a\Dd_{q}b,\quad 
	\dot T_b\, a=\sum_{q\in\ZZ}\Sd_{q-1} b\Dd_{q}b\quad\textrm{ and }\quad \dot R(a,b)=\sum_{\substack{q\in\ZZ,\\ i\in\{0,\pm 1\}} }\Dd_{q} a\Dd_{q+i} b.
\end{equation*}
Then the homogeneous Besov space $\BB_{p,r}^s$ is identified by the following definition:
\begin{definition}\label{def:hom-bes}
The homogeneous Besov spaces $\dot B^s_{2,1}$ with a real $s\in \RR$, $p,r\in [1,\infty]^2$ fulfilling
\begin{equation*}
	s\,<\,\frac{\dd}{2}\quad\text{if}\quad r>1,\quad\quad s\leq \frac{\dd}{2}\quad \text{if}\quad r=1,
\end{equation*}
consists of all homogeneous tempered distributions $u$ such that:
\begin{equation*}
	\|\,u\,\|_{\BB_{p,r}^s}\,:=\,\big\|\,\big(\,2^{qs}\,\Dd_q\,u\,\big)_{q\in\ZZ}\,\big\|_{\ell^r(\ZZ)}<\infty.
\end{equation*}
\end{definition} 
\noindent
We recall that this definition reduces to the classical homogeneous Sobolev space $\dot H^s$ whenever $p=r=2$. 
In this work we deal with functions with low oscillations, that is we consider the case $p=2$ and $r=1$. Moreover it is worth to remark that $\BB_{2,1}^s$ is continuously embedded in $\dot H^s(\RR^\dd)$ and moreover $\BB_{2,1}^{\dd/2}$ is continuously included in $\FF(L^1(\RR^\dd))$ and thus in the space of continuous bounded functions.

\smallskip 
\noindent
The following product rule between homogeneous Besov spaces is satisfied (cf. \cite{D2}):
\begin{prop}\label{prop-hom-bes}
	Let $u$ be in $\BB^{s_1}_{2,1} $ and $v $ be in $\BB^{s_2}_{2,1}$ with $s_1,\,s_2\leq \dd/2$. If $s_1\,+\,s_2\,>0$ then
	the product $u\,v$ belongs to $\BB_{2,1}^{s_1\,+\,s_2-\dd/2}$ and the following inequality holds
	\begin{equation*}
		\| \,u\,v\,\|_{\BB_{2,1}^{s_1\,+\,s_2-\frac{\dd}{2}}}\,\leq \,C\,\|\,u\,\|_{\BB^{s_1}_{2,1}}\,\|\,v\,\|_{\BB^{s_2}_{2,1}},
	\end{equation*}
	for an harmless constant $C$ depending on $s$ and $\dd$.
\end{prop}
\noindent
Fixing $s_1=s_2=\dd/2$ in the above proposition leads to the algebra structure of the space $\BB_{2,1}^{\dd/2}$. 
In the next section we will repeatedly make use of the following sub-cases of Proposition \ref{prop-hom-bes}
\begin{equation}\label{pruseful}
\begin{alignedat}{8}
	&	\BB_{2,1}^\frac{\dd}{2}\,	 &&\times\, \BB_{2,1}^{\frac{\dd}{2}-1}\,\rightarrow\,\BB_{2,1}^{\frac{\dd}{2}-1},\quad
	 	\BB_{2,1}^\frac{\dd}{2}\,	 &&\times\, \BB_{2,1}^{\frac{\dd}{2}-2}\,\rightarrow\,\BB_{2,1}^{\frac{\dd}{2}-2},\quad
	 	\BB_{2,1}^{\frac{\dd}{2}-1}\,&&\times\, \BB_{2,1}^{\frac{\dd}{2}-1}\,\rightarrow\,\BB_{2,1}^{\frac{\dd}{2}-2}.
\end{alignedat}
\end{equation}
We recall that the domain $\RR^\dd$ is assumed at least three dimensional, so that the regularity 
$\dd/2-2$ is strictly positive. 
%%%%%%%%%%%%%%%%%%%%%%%%%%%%%%%%%%%%%%%%%%%%%%%%%%%%%%%%%%%%%%%%%%%%%%%%%%%%%%%%%%%%%%%%%%%%%%%%%%%%%%%%%%%%%%%%%%%%%%%%%%%%%%
%%%%%%%%%%%%%%%%%%%%%%%%%%%%%%%%%%%%%%%%%%%%%%%%%%%%%%%%%%%%%%%%%%%%%%%%%%%%%%%%%%%%%%%%%%%%%%%%%%%%%%%%%%%%%%%%%%%%%%%%%%%%%%
%%%%%%%%%%%%%%%%%%%%%%%%%%%%%%%%%%%%%%%%%%%%%%%%%%%%%%%%%%%%%%%%%%%%%%%%%%%%%%%%%%%%%%%%%%%%%%%%%%%%%%%%%%%%%%%%%%%%%%%%%%%%%%
\subsection{Regularizing effects}$\,$

\smallskip
\noindent In this section we establish regularizing effects for parabolic-type equations in the framework of the homogeneous Besov spaces.  
%These results can be seen as a generalization of the regularities theorem for the heat flow in Besov spaces reported for instance in \cite{}. 
%We perform a different  details of the proof, for the sake of completeness.
We begin with the following definition:
\begin{definition}\label{def_elliptic_oper}
	An operator $\Lambda$ from $\DD(\,\RR^\dd,\,\RR^M)$ to itself is a second-order strong elliptic operator, if  there exists a positive constant $\lambda_0$ such that
	\begin{equation*}
		-\int_{\RR^\dd}\Lambda \vv(x)\cdot \vv(x)\dd x\, \geq\, \lambda_0 \|\,\nabla \vv \,\|_{L^2(\RR^\dd)}
	\end{equation*}
	for any smooth vector-function with compact support $\vv\in \DD(\,\RR^\dd,\,\RR^M)$.
\end{definition}
\noindent The following classical result for parabolic equation in Besov spaces holds:
\begin{theorem}\label{thm-superjuve}
	Let us consider an initial-vector $\varphi_0$ in $(\,\BB_{2,1}^s(\RR^\dd)\,)^{\times M}$ with regularity $s\leq \dd/2$, and for an integer
	$M\geq 1$. Introducing a driving force 
	$f$ in $L^1_t (\,\BB_{2,1}^s\,)^M$, we denote by $\varphi$ be the unique solution of 
	the following linear parabolic PDE's
	\begin{equation}\label{BS:linearsyst}
	\left\{\;
	\begin{alignedat}{8}
		&\partial_t \varphi\,-\,\Lambda\,\varphi\,=\,f\hspace{2cm} 	&&\RR_+\times	&&&&\RR^\dd,\\
		&\varphi\big|_{t=0}\,=\,\varphi_0						  		&& 				&&&&\RR^\dd,
	\end{alignedat}
	\right.
	\end{equation}
	where $\Lambda$ is a strongly second order elliptic operator as in Definition \ref{def_elliptic_oper}. Then $\varphi$ belongs to $L^\infty_t (\,\BB_{2,1}^s\,)^M$ and the couple
	$(\partial_t \varphi,\,\Delta \varphi)$ to  $L^1_t (\,\BB_{2,1}^s\,)^M$. Furthermore, there exists an harmless positive 
	constant $C$ such that
	\begin{equation}\label{superjuve}
		\|\,\varphi\,\|_{L^\infty_t \BB_{2,1}^s}\,+\,
		\|\,\partial_t \varphi\,\|_{L^1_t \BB_{2,1}^{s}} \,+\,\lambda_0\| \,\Delta \varphi\,\|_{L^1_t \BB_{2,1}^{s}}
		\,\leq\,
		C\Big[\,\|\,\varphi_0\,\|_{\BB_{2,1}^s}\,+\,\|\,f\,\|_{L^1_t \BB_{2,1}^{s}}\,\Big].
	\end{equation}
\end{theorem}
\begin{proof}
	We first apply the homogeneous dyadic bloc $\Dd_q$, to the main equation, we multiply by 
	$\Dd_q \varphi$ both the left-hand the right-hand sides, and we integrate over $\RR^d$. We get
	\begin{equation*}
		2\|\,\Dd_q\,\varphi\,\|_{L^2_x}\frac{\dd}{\dd\,t}\|\,\Dd_q\,\varphi\,\|_{\BB_{2,1}^s}\,
		-\int_{\RR^d}\Lambda\Dd_q \varphi\,\Dd_q\varphi\,\dd x\,=\int_{\RR^d}\Dd_q\,f\,\Dd_q\varphi\,\dd x.
	\end{equation*}
	Since $\Lambda$ is a strong elliptic operator, we deduce that
	\begin{equation*}
		\|\,\Dd_q\,\varphi\,\|_{L^2_x}\frac{\dd}{\dd\,t}\|\,\Dd_q\,\varphi\,\|_{\BB_{2,1}^s}\,+\,
		\lambda_0\|\,\nabla\Dd_q \varphi\,\|^2_{L^2_x}\,\lesssim\,\|\,\Dd_q\,f\,\|_{L^2_x}\,\|\,\,\Dd_q\,\varphi\,\|_{L^2_x}.
	\end{equation*}
	A Bernstein-type inequality yields that there exists a constant $c>0$ such that $\|\,\nabla\Dd_q \varphi\,\|_{L^2_x}\,\geq\,c 2^q\|\Dd_q \varphi\,\|_{L^2_x} $, 
	from which we gather
	\begin{equation*}
		\frac{\dd}{\dd\,t}\|\,\Dd_q\,\varphi\,\|_{L^2_x}\,+\,\lambda_0 2^{2q}\|\,\Dd_q \varphi\,\|_{L^2_x}\,\lesssim
		\,\|\,\Dd_q\,f\,\|_{L^2_x}.
	\end{equation*}
	Hence, multiplying by $2^{qs}$ both the left and the right-hand sides, 
	taking the sum as $q\in \ZZ$ and integrating in time over $(0,\,t)$, we deduce
	\begin{equation*}
		\|\,\,\varphi(t)\,\|_{\BB_{2,1}^{s}}\,+\,\int_0^t\,\|\,\varphi(\tau)\,\|_{\BB_{2,1}^{s\,+\,2}}\dd \tau\,\lesssim
		\|\,\varphi_0\,\|_{\BB_{2,1}^s}\,+\,
		\,\int_0^t\,\|\,f(\tau)\,\|_{\BB_{2,1}^s}\dd \tau.
	\end{equation*}
	To finally achieve a bound for $\partial_t \varphi$ we apply $(-\Delta)^{-1}$ to the main system \eqref{BS:linearsyst}. Since
	$(-\Delta)^{-1}\Lambda$ is a Fourier multiplier of degree $0$, we get
	\begin{equation*}
		\|\, \partial_t\varphi\,\|_{\BB_{2,1}^{s}}\,\lesssim\, \| \,f\,\|_{\BB_{2,1}^s}\,+\,
		\|\,	\Delta \varphi\,\|_{\BB_{2,1}^{s}},
	\end{equation*}
	which finally leads to $\partial_t \varphi\,\in\,L^1_t\BB_{2,1}^s$ and to inequality \eqref{superjuve}.
\end{proof}

\noindent
When estimating the co-rotational time flux $\mcN$, we will to control the $L^2_t\BB_{2,1}^{\dd/2}$-norm of $\partial_t \nn$. The $L^2$-integrability in time is given by the following theorem:
\begin{theorem}\label{thm-superjuve2}
	Let us consider an initial-vector $\varphi_0$ such that $\nabla \varphi_0$ belongs to $(\,\BB_{2,1}^s(\RR^\dd)\,)^{\times M^2}$ with regularity $s\leq \dd/2$, and for an integer
	$M\geq 1$. Introducing a driving force 
	$f$ in $L^2_t (\,\BB_{2,1}^s\,)^M$, we denote by $\varphi$ be the unique solution of 
	the following linear parabolic PDE's
	\begin{equation*}
	\left\{\;
	\begin{alignedat}{8}
		&\partial_t \varphi\,-\,\Lambda\,\varphi\,=\,f\hspace{2cm} 	&&\RR_+\times	&&&&\RR^\dd,\\
		&\varphi\big|_{t=0}\,=\,\varphi_0						  		&& 				&&&&\RR^\dd,
	\end{alignedat}
	\right.
	\end{equation*}
	where the  second-order elliptic operator $\Lambda$ is symmetric.
	\begin{comment}
	\begin{equation*}
		\int_{\RR^\dd}\Lambda \vv(x)\cdot \omega(x)\,\dd x =\int_{\RR^\dd}\Lambda \omega(x)\cdot \vv(x)\dd x,\quad
		\text{for any} \quad\vv,\,\omega\quad \text{in} \quad \mathfrak{D}(\RR^\dd,\,\RR^M).
	\end{equation*}
	\end{comment} 
	Then there exists an harmless positive constant $C$ such that
	\begin{equation}\label{saw-l'enigmista}
	\|\,	(\partial_t\varphi,\,\Delta \varphi)\,\|_{ L^2_t\BB_{2,1}^{s}}\,+\,\|\,\nabla \varphi\,\|_{L^\infty_t \BB_{2,1}^s}\,\leq 
	C\,\Big[\,\|\,\nabla \varphi_0\,\|_{\BB_{2,1}^s}\,+\,\|\,f\,\|_{ L^2_t\BB_{2,1}^{s}}	\,\Big].
	\end{equation}
\end{theorem}
\begin{proof}
	Applying a standard energy estimate to the localized function $\Dd_q \varphi$ yields
	\begin{equation*}
		\|\,\varphi\,\|_{L^\infty_t\BB_{2,1}^s}\,+\,\|\,\varphi\,\|_{L^2_t\BB_{2,1}^{s+1}}\,\leq 
		C\,\Big[\,\|\,\varphi_0\,\|_{\BB_{2,1}^s}\,+\,\|\,f\,\|_{L^2_t\BB_{2,1}^{s-1}}	\,\Big].
	\end{equation*}
	Furthermore, the time derivative $\Dd_q\partial_t \varphi$ satisfies:
	\begin{equation}\label{legalegalisasion}
		\|\,\Dd_q\partial_t \varphi\,\|^2_{L^2_x}\,-\,\int_{\RR^\dd}\Lambda \Dd_q \varphi\cdot \partial_t\Dd_q \varphi \dd x\,\leq\,
		\|\, \Dd_qf\,\|_{L^2_x}^2.
	\end{equation}
	Thanks to the symmetry of $\Lambda$, we also deduce
	\begin{equation*}
			-\frac{1}{2}\frac{\dd}{\dd t}\int_{\RR^\dd}\Big[\,\Lambda \Dd_q\varphi\cdot\Dd_q\varphi\,\Big]\dd x\,=\,
			-\,\int_{\RR^\dd}\Lambda \Dd_q\varphi\cdot \Dd_q\,\partial_t\varphi\dd x.
	\end{equation*}
	Hence integrating  \eqref{legalegalisasion} in time, leads to
	\begin{equation*}
		\int_0^t \|\, \Dd_q\partial_t\varphi(\tau)\,\|_{L^2_x}^2\dd\tau\,+\,
		\|\,\Dd_q \nabla \varphi(t)\,\|_{L^2_x}^2\,\lesssim \,\|\,\Dd_q\nabla \varphi_0\,\|_{L^2_x}^2\,+ \,\int_0^t\,\|\, \Dd_q f(\tau)\,\|_{L^2_x}^2\dd \tau,
	\end{equation*}
	from which, multiplying by $2^{2qs}$ and taking the sum as $q\in\ZZ$, we finally achieve \eqref{saw-l'enigmista}.
	%\begin{equation*}
	%	\int_0^t \|\, \partial_t\varphi(\tau)\,\|_{\BB_{2,1}^{s-1}}^2\dd\tau\,+\,
	%	2^{2q}\|\,\Dd_q \varphi(t)\,\|_{L^2_x}^2\,\lesssim \,2^{2q}\|\,\Dd_q \varphi_0\,\|_{L^2_x}^2\,+ \,\int_0^t\,\|\, \Dd_q f(\tau)\,\|_{L^2_x}^2\dd \tau.
	%\end{equation*}
\end{proof}

\noindent A similar result holds for the following Stokes-type system:
\begin{theorem}\label{thm-superjuve-stokes}
	Let $\uu(t,x)\in \RR^\dd$ be a solution to the system
	\begin{equation*}
	\left\{
	\begin{alignedat}{8}
		&\partial_t\uu - \Aa[\,\uu\,] + \nabla \pre = f,\hspace{1cm}				&&\RR_+\times	&&&&\RR^\dd,\\
		&\Div\,\uu = 0														&&\RR_+\times	&&&&\RR^\dd,\\
		&\uu\big|_{t=0}\,=\,\uu_0						  					&& 	\,			&&&&\RR^\dd,
	\end{alignedat}
	\right.	
	\end{equation*}
	where the initial datum $\uu_0$ belongs to $\BB_{2,1}^s$, $s\leq \dd/2$, $f\in L^1_t\BB_{2,1}^s$ and the linear operator $\Aa$ is strongly elliptic.
	Then a pressure $\pre$ is defined by means of
	\begin{equation*}
		\pre  = (-\Delta)^{-1}\Big[\,-\Div\,\Aa[\,\uu\,]\, +\, \Div \,f\,\Big]\,\in\, \mathfrak{D}'(\RR_+\times \RR^\dd),\quad\text{with}\quad \nabla \pre \in L^1_t\BB_{2,1}^{s}
	\end{equation*}
	and the following inequality is satisfied:
	\begin{equation*}
		\| \,\uu \,\|_{L^\infty_t\BB_{2,1}^s} + \| \,(\partial_t\uu,\,\nabla \pre)\, \|_{L^1_t\BB_{2,1}^s}+\lambda_0\| \,\Delta \uu \,\|_{L^1_t\BB_{2,1}^s}
		\leq C\Big[ \, \|\,\uu_0\,\|_{\BB_{2,1}^s} + \| \,f\,\|_{L^1_t\BB_{2,1}^s}\,\Big],
	\end{equation*}
	for a positive constant $C$ and a constant $\lambda_0>0$ introduced in Definition \ref{def_elliptic_oper}. 
\end{theorem}

%%%%%%%%%%%%%%%%%%%%%%%%%%%%%%%%%%%%%%%%%%%%%%%%%%%%%%%%%%%%%%%%%%%%%%%%%%%%%%%%%%%%%%%%%%%%%%%%%%%%
%%%%%%%%%%%%%%%%%%%%%%%%%%%%%%%%%%%%%%%%%%%%%%%%%%%%%%%%%%%%%%%%%%%%%%%%%%%%%%%%%%%%%%%%%%%%%%%%%%%%
%																	Composition under smooth function
%%%%%%%%%%%%%%%%%%%%%%%%%%%%%%%%%%%%%%%%%%%%%%%%%%%%%%%%%%%%%%%%%%%%%%%%%%%%%%%%%%%%%%%%%%%%%%%%%%%%
%%%%%%%%%%%%%%%%%%%%%%%%%%%%%%%%%%%%%%%%%%%%%%%%%%%%%%%%%%%%%%%%%%%%%%%%%%%%%%%%%%%%%%%%%%%%%%%%%%%%
\subsection{Composition under smooth function}$\,$

\smallskip
\noindent
We conclude this section considering the action of smooth functions on the Besov space $\BB_{2,1}^{\dd/2}$. 
We refer the reader to Theorem $2.61$ in \cite{B-C-D}, for a detailed proof of the next lemma. 
\begin{lemma}\label{thm:composition-besov}
	Let $f$ be a smooth function on $\RR$ which vanishes at 
	$0$. For any  functions $\vv$ in $\BB_{2,1}^{\dd/2}$
	% \cap \BB^s_{2,1}$, with $s>\dd/2 
	the function $f\,(\,\vv\,)\,$ still belongs to $\BB_{2,1}^{\dd/2}$ and 
	the following inequality is satisfied,
	\begin{equation*}
	\begin{aligned}
		\|\,f\,(\,\vv\,)\,\|_{\,\BB_{2,1}^\frac{\dd}{2}\,}\,\leq	
		\mathrm{Q}\,\big(\,f,\,\|\,\vv\,\|_{L^\infty_x}\,\big)
			\|\,\vv\,\|_{\BB_{2,1}^\frac{\dd}{2}}
	\end{aligned}
	\end{equation*}
	where $Q$ is a smooth function depending on the value of $f$ and its derivative 
\end{lemma}
\noindent We readily obtain the following corollary
\begin{cor}
	Let $\rm v_1$ be a function in $\BB_{2,1}^{\dd/2}\cap \BB_{2,1}^{\dd/2+1}$ and $\rm v_2$ be in $\BB_{2,1}^s\cap \BB_{2,1}^{s+1}$ such that the product is continuous in
	\begin{equation*}
		\BB_{2,1}^\frac{\dd}{2}\times \BB_{2,1}^s\quad\rightarrow\quad \BB_{2,1}^s.
	\end{equation*}
	Let $f$ be a smooth function on $\RR$, then $f(\rm v_1)\rm v_2$ belongs to $\BB_{2,1}^s\cap \BB_{2,1}^{s+1}$ and the following inequalities are satisfied
	\begin{equation*}
	\begin{aligned}
		\|\,		f(\mathrm{v}_1)\,{\rm v}_2	\,\|_{\BB_{2,1}^s				}
		&\lesssim
		\mathrm{Q}\,\big(\,f,\,\|\,\vv\,\|_{L^\infty_x}\,\big)
		\|\,		{\rm v}_1					\,\|_{\BB_{2,1}^\frac{\dd}{2}	}
		\|\,		{\rm v}_2		 			\,\|_{\BB_{2,1}^s				},
		\\
		\Big\|\,	\nabla\Big(	f(\mathrm{v}_1)\,{\rm v}_2\Big)	\,\Big\|_{\BB_{2,1}^s			}
		&\lesssim
		\mathrm{Q}\,\big(\,f,\,\|\,\vv\,\|_{L^\infty_x}\,\big)
		\Big\{\,
			\|\,		\nabla {\rm v}_1				\,\|_{\BB_{2,1}^\frac{\dd}{2}	}
			\|\,		{\rm v}_2		 			\,\|_{\BB_{2,1}^s				}
			\,+\,
			\|\,		 {\rm v}_1					\,\|_{\BB_{2,1}^\frac{\dd}{2}	}
			\|\,		\nabla{\rm v}_2				\,\|_{\BB_{2,1}^s				}
		\,\Big\},
	\end{aligned}
	\end{equation*}
	where $Q$ is a smooth function depending on the value of $f$ and its derivative.
\end{cor}
\begin{proof}
	We begin with the first inequality, applying a translation in order to have $f$ vanishing at $0$:
	\begin{equation*}
	\begin{aligned}
		\|\,		f(\mathrm{v}_1)\,{\rm v}_2					\,\|_{\BB_{2,1}^s				}
		\,&\leq\,
		\|\,	\Big(\,f(\mathrm{v}_1)\,-\,f(0)\Big)\,{\rm v}_2	\,\|_{\BB_{2,1}^s				}
		\,+\,
		|\,f(0)\,|\|\,{\rm v}_2								\,\|_{\BB_{2,1}^s				}
		\\
		&\leq
		\|\,	\,f(\mathrm{v}_1)\,-\,f(0)						\,\|_{\BB_{2,1}^\frac{\dd}{2}	}
		\|\,	 \,{\rm v}_2										\,\|_{\BB_{2,1}^s				}
		\,+\,
		|\,f(0)\,|\|\,{\rm v}_2								\,\|_{\BB_{2,1}^s				},
	\end{aligned}
	\end{equation*}
	hence, applying Lemma \ref{thm:composition-besov} with $f(x)-f(0)$ as smooth function, we deduce that
	\begin{equation*}
		\|\,		f(\mathrm{v}_1)\,{\rm v}_2					\,\|_{\BB_{2,1}^s				}
		\lesssim
		\mathrm{Q}\,\big(\,f,\,\|\,\vv\,\|_{L^\infty_x}\,\big)
		\|\,		{\rm v}_1					\,\|_{\BB_{2,1}^\frac{\dd}{2}	}
		\|\,		{\rm v}_2		 			\,\|_{\BB_{2,1}^s				}.
	\end{equation*}
	The second inequality turns out from a straightforward computations:	
	\begin{equation*}
	\begin{aligned}
		\|\,	\nabla\Big(	f(\mathrm{v}_1)\,{\rm v}_2\Big)					\,\|_{\BB_{2,1}^s			}
		&\leq 
		 \|\,(\,	f'(\mathrm{v}_1)-f'(0)\,)\nabla{\rm v}_1 \,{\rm v}_2		\,\|_{\BB_{2,1}^s			}
		\,+\,
		|f'(0)|	\|\,\nabla{\rm v}_1 \,{\rm v}_2							\,\|_{\BB_{2,1}^s			}
		\,+
		\\&\hspace{2cm}
		+\,
		\|\,(\,	f(\mathrm{v}_1)-f(0)\,){\rm v}_1 \,\nabla{\rm v}_2		\,\|_{\BB_{2,1}^s			}
		\,+\,
		|f'(0)|	\|\,{\rm v}_1 \,\nabla{\rm v}_2							\,\|_{\BB_{2,1}^s			}
		\\
		&\lesssim\,
		\mathrm{Q}\,\big(\,f,\,\|\,\vv\,\|_{L^\infty_x}\,\big)
		\Big\{\,
			\|\,		\nabla {\rm v}_1				\,\|_{\BB_{2,1}^\frac{\dd}{2}	}
			\|\,		{\rm v}_2		 			\,\|_{\BB_{2,1}^s				}
			\,+\,
			\|\,		 {\rm v}_1					\,\|_{\BB_{2,1}^\frac{\dd}{2}	}
			\|\,		\nabla{\rm v}_2				\,\|_{\BB_{2,1}^s				}
		\,\Big\},
	\end{aligned}		
	\end{equation*}
	which concludes the proof of the Corollary.
\end{proof}
\section{Well-posedness}\label{sec-wp}
\noindent
In this section we deal with the well-posedness of the system \eqref{intro:main_system-wp}, where the nematic material has a constant density.  Thanks to Theorem \ref{main-thm1} and Theorem \ref{main-thm2}, such a model arises when the heat flux ${\bf q}$, the Leslie stress tensor $\sigma^{\rm L}$ and the free energy $\FF$ are defined by means of \eqref{heat}, \eqref{Leslie-stress} and \eqref{def_OF_intro}.
The free energy $\FF=\FF(\Tt,\,\nn,\,\nabla \nn)$ we consider in this section does not depend on the density and it generalizes the 
Oseen-Frank energy density for Nematic liquid crystal with variable temperature:
\begin{equation}\label{wp-def_F}
	\FF(\Tt,\,\nn,\,\nabla \nn)\,:=\,-\,\Tt\ln \Tt\,+\,{\rm W_F}(\Tt,\,\nn,\,\nabla \nn), 
\end{equation}
where $\rm W_{F}$ stands for
\begin{equation}\label{def:OF-density}
	{\rm W_F}\,=\,
	\frac{ \,\,K_{1}(\Tt)}{2}\,|\,		  \nabla 	\nn\, |^2 \, +\,
	\frac{ \,\,K_{2}(\Tt)}{2}\,|\,			\Div \,	\nn\, |^2 \, +\,
	\frac{ \,\,K_{3}(\Tt)}{2}\,|\,\nn\cdot \nabla 	\nn\, |^2 \, +\,
	\frac{ \,\,K_{4}(\Tt)}{2}\,\trc\{\,(\nabla \nn)^2\,\}.
\end{equation}
The coefficients $K_1,\dots K_4$ are smooth functions depending on the temperature $\Tt$ and they are related to $k_1,\dots, k_4$ in \eqref{def_OF_intro} by
\begin{equation*}
\begin{alignedat}{8}
	K_1\,&=\,k_{12},\hspace{4cm}
	&&K_2\,&&&&=\,k_{11}-k_{22}-k_{24},\\
	K_3\,&=\,k_{33}-k_{22},
	&&K_4\,&&&&=\, k_{24}.
\end{alignedat}
\end{equation*}
The local entropy $\eta$ is thus explicitly defined by means of
\begin{equation*}
	\eta\,=\,-\frac{\partial \FF}{\partial \Tt}\,=\,\frac{\partial}{\partial\Tt}\Big[\,\Tt\ln\Tt\,-\,{\rm W}_{\rm F}\,\Big]\,=\,-\,(\,1\,+\,\ln\,\Tt)\,-\,\frac{\partial {\rm W}_{\rm F}}{\partial \Tt}.
\end{equation*}
Then, the rate of increase of the local entropy as well its convection in the last equation of \eqref{intro:main_system-wp} reduces to a material derivative on the temperature $\Tt$:
\begin{equation*}
	\Tt\Big[\,\partial_t \eta\,+\,\,\uu\cdot \nabla \eta\,\Big]\,=\,\,\partial_t \Tt\,+\,\uu\cdot\nabla \Tt\,
	\,-\,\Tt\,\Big[\,\partial_t \frac{\partial {\rm W}_{\rm F}}{\partial \Tt}\,+\,\,\uu\cdot\nabla \,\frac{\partial {\rm W}_{\rm F}}{\partial \Tt}\,\Big].
\end{equation*}

\smallskip\noindent
In this section we impose a parabolic behavior to the momentum equation, the heat equation and the balance of angular momentum. 
Recalling that $2\mathbb{D}_{ij}=\uu_{i,j}+\uu_{j,i}$ and $2\Omega_{ij}=\uu_{i,j}-\uu_{j,i}$, we define the linear operator $\Aa$, by
\begin{equation}\label{def-A}
	\Aa\,[\,\uu\,] := \Div\,\left[ \sigma^{\,\rm L}\Big(\bar{\Tt},\,\bar{\nn},\,-\Omega \bar{n},\,\mathbb{D}\Big)\right],
\end{equation}
and we denote by $\tilde \sigma^{\rm L}\,=\sigma^{\rm L}\,-\,\sigma^{\,\rm L}(\,\bar{\Tt},\,\bar{\nn},\,-\Omega \bar{n},\,\mathbb{D})$ the perturbed Leslie stress tensor. Without loss of generality, we assume $\Aa$ to be strongly elliptic, since the Leslie coefficient $\bar{\alpha}_4$ is supposed large enough.
Furthermore we decompose the $\lambda$-coefficients in the heat flux $\q$ through
\begin{equation}\label{ass-2}
\begin{alignedat}{8}
	\lambda_1(\,\Tt\,)\,&=\,\bar \lambda_1\,+\,\tilde\lambda_1(\,\Tt\,-\bar{\Tt}\,),\quad\quad\lambda_2(\,\Tt\,)\,&=\,\bar \lambda_2\,+\,\tilde\lambda_2(\,\Tt\,-\bar{\Tt}\,),
\end{alignedat}
\end{equation}
where the couple $(\,\bar{\lambda}_1,\,\bar{\lambda}_2\,)\in \RR^2$ corresponds to the value $(\lambda_1(\bar{\Tt}),\lambda_2(\bar{\Tt}))$. We assume $(\bar \lambda_1,\,\bar \lambda_2)$ fulfilling $\bar{\lambda}_1>0$ and $\bar{\lambda}_2>-\bar{\lambda}_1$ so that the linear operator $\Bb$ defined by
\begin{equation}\label{def-B}
	\Bb\,[\,\Tt\,]\,=\,\bar{\lambda}_1\Delta \Tt\,+\,\bar{\lambda}_2[\,\bar{\nn}\cdot \nabla \Tt\,]\bar{\nn}
\end{equation}
is strongly elliptic. Finally, we decompose the $K$-coefficients through $K_i(\,\Tt\,)\,=\,\bar{K}_i\,+\,\tilde K_i(\,\Tt-\bar{\Tt}\,)$ for any $i\,=\,1\,\dots,\,4$,
where $\bar{K}_i\in\RR\,$ and the smooth functions $\tilde K_i$ are null in $\Tt-\bar{\Tt}=0$. We then introduce the linear operator $\Cc$
\begin{equation}\label{def-C}
	\Cc[\,\nn\,]\,=\,\bar K_1\Delta \nn\,+\,(\bar K_2\,+\,\bar K_4)(\Id\,+\,\bar \nn\otimes \bar \nn)\nabla \Div\,\nn\,+\,
	\bar K_3\, \Div\,\big[(\bar \nn\cdot \nabla \nn)\otimes \bar \nn\big],
\end{equation}
which is strongly elliptic when the following conditions hold:
\begin{equation}\label{ass-3}
	\bar K_1\,>0,\quad \, 2(\,\bar K_2\,+\,\bar K_4\,)\,\,+\,\bar K_3\,\,<\, \bar K_1.
\end{equation}

\smallskip\noindent
For the sake of notation, we introduce the variables 	$\w\,=\,\Tt\,-\,\bar{\Tt}$ and $\m\,:=\,\nn\,-\,\bar \nn$, describing the perturbation of the temperature and the director with respect to $(\bar \Tt,\,\bar \nn)$.
Thus, the main system \eqref{main_system} reduces to
\begin{equation}\label{main_system2}
\left\{
\begin{aligned}
	\hspace{0.2cm}
	&\partial_t \uu\,+\uu\cdot\nabla \uu\,-\,\Aa[\,\uu\,]\,+\nabla \pre\,=\,\Div\{ \sigma^{\rm E}\,+\,\tilde \sigma^{\rm L}\},\\
	&\Div\,\uu\,=0\,\\
	&\partial_t\m\,+\,\uu \cdot \nabla \m\,-\Cc[\,\m\,]=\,
	\tilde \gamma_1\Big[\,\partial_t\m\,+\,\uu \cdot \nabla \m\,
	\Big]\,+\,\gamma_1\,\Omega\m\,+\,\gamma_1\,\Omega\bar\nn\,+\\&\hspace{1.5cm}+\,\gamma_2\Big[\,\mathbb{D}(\,\bar{\nn}\,+\,\m\,)\,
	-\,[\,(\,\bar{\nn}\,+\,\m)\cdot \mathbb{D}(\,\bar{\nn}\,+\,\m)\,]\,(\bar{\nn}\,+\,\m\,)\,\Big]\,-\tilde {\mathrm{h}}\,+\,\tilde \beta \nn\,+\,f,\\
	&|\,\bar{\nn}\,+\,\m\,|^2=1,\\
	&\partial_t \w\,+\,\uu\cdot\nabla \w\,-\Bb[\, \w\,]\,=
	\,(\,\bar \Tt \,+\w\,)\Big[\,\partial_t \frac{\partial {\rm W_F}}{\partial\w}\,+\,\uu\cdot\nabla\,\,{\rm \frac{\partial W_F}{\partial\w}}\,\Big]\,+
	\\&\hspace{1.cm}+\,\Div\{\,\tilde \lambda_1\nabla \w\}
	\,+\,\Div\{\tilde \lambda_2[\,(\bar \nn\,+\m)\cdot \nabla \w\,](\,\bar{\nn}\,+\,\m\,)\}\,+\,\sigma^{\rm L}:\mathbb{D}\,+\,{\rm g}\cdot \mcN.
\end{aligned}
\right.
\end{equation}
First, we have defined the perturbed molecular field $\tilde{\mathrm{h}}$ by means of
\begin{equation*}
	\begin{aligned}
		\tilde{ \mathrm{h}}\,=\,K_3(\,\nabla \m\odot \nabla \m\,)\nn\,&-\,
								\Div
								\Big\{\,
									(K_1\,-\,\bar K_1)\nabla \m\,+\,
									(K_2\,-\,\bar K_2\,)\, \Div\,\m\,\Id\,+\\&+\,
									(K_3\,-\,\bar K_3\,)\,[\,\nn\cdot \nabla \m\otimes\nn\,+\,
									(K_4\,-\,\bar K_4\,)\,\tr \nabla \m \,
								\Big\}\,+\\&+\,
								\bar K_3\,
								\Div
								\Big\{\,
									[\,\m\cdot \nabla \m\,	]\otimes\nn	\,+\,
									[\,\nn\cdot\nabla \m\,	]\otimes\m	\,+\,
									[\m\cdot \nabla \m\,	]\otimes\m\,
								\Big\}.
	\end{aligned}
\end{equation*}
Indeed, thanks to the definition for the free energy $\FF$, the Lagrangian multiplier for the constraint $|\,\nn\,|^2=1$ can be formulated as
\begin{equation*}
\begin{aligned}
	\beta\,:=\,\frac{\delta F}{\delta \nn}\cdot \nn\,=\,
	\underbrace{
	K_1|\,\nabla \m\,|^2\,+\,
	(\,K_2\,+\,K_4)\,|\,\Div\,\m\,|^2+\, 
	K_3\,|\,\nn\cdot\nabla \m\,|^2
	}_{\tilde \beta}\,-\\-\,
	\Div\,\Big\{\,(\,K_2\,+\,K_4\,)\,(\Div\,\m)\,\nn\,\Big\}.
\end{aligned}
\end{equation*}
Thus, defining the driving force $f$ by means of the following identity
\begin{equation*}
\begin{aligned}
	f\,:=\,\Div\,\Big\{\,(\,K_2-\bar K_2\,&+\,K_4\,-\bar K_4\,)\,(\Div\,\m)\,\nn\,\Big\}\nn\,+\,
			(\,\bar K_2\,+\,\bar K_4\,)\Div\,\Big\{\,\,(\Div\,\m)\,\m\,\Big\}\m\,+\\&+\,
			(\,\bar K_2\,+\,\bar K_4\,)\Div\,\Big\{\,\,(\Div\,\m)\,\bar \nn\,\Big\}\m\,+\,
			(\,\bar K_2\,+\,\bar K_4\,)\Div\,\Big\{\,\,(\Div\,\m)\,\m\,\Big\}\bar \nn,
\end{aligned}
\end{equation*}
one gets the following equality
\begin{equation*}
	\mathrm{h}\,-\,\beta\,\nn\,=-\Cc[\,\m\,]\,+\,\tilde{\mathrm{h}}\,-\tilde{\beta}\nn\,-\,f.
\end{equation*}

\subsection{Sketch of the proof}$\,$

\noindent
In this section we overview the proof of Theorem \ref{main-thm3}. We proceed with a standard method: we use an iterative scheme to build a sequence of solutions for a linear approximation of system \eqref{intro:main_system-wp}.  
We set the first term of the sequence $(\,\uu^0(t,\,x),\,\w^0(t,\,x),\,\m^0(t,\,x),\,)$ to be null everywhere in $ \RR_+\times\RR^\dd$. Then, we choose   
 $(\,\uu^k(t,\,x),\,a^k(t,\,x),\,\w^k(t,\,x),\,\m^k(t,\,x)\,)$ as the solution of the following linear system:
\begin{equation}\label{appx-system-nu}
\left\{\;
\begin{alignedat}{64}
						\partial_t\, 	
	&						\uu^{k+1}\,
	&&						+\,\,\,\uu^{k}\cdot\nabla \uu^{k+1}	
	&&&&					-\,		
	&&&&&&					\Aa	
	&&&&&&&&				[\,			
	&&&&&&&&&&				\uu^{k+1}
	&&&&&&&&&&&&			]\,	
	&&&&&&&&&&&&&&		+\,\nabla \pre^{k+1}\,	=\,F_{1}^k
	&&&&&&&&&&&&&&&&		\RR_+\times			
	&&&&&&&&&&&&&&&&&&	\RR^{\dd},\\
	&\hspace{-0.4cm}\Div\,
	&&\hspace{-0.8cm}\uu^{k+1}=0
	&&&&
	&&&&&&
	&&&&&&&&
	&&&&&&&&&&
	&&&&&&&&&&&&
	&&&&&&&&&&&&&&
	&&&&&&&&&&&&&&&&			\RR_+\times	
	&&&&&&&&&&&&&&&&&&		\RR^{\dd},\\
							\partial_t\, 	
	&						\m^{k+1}\,
	&&						+\,\,\,\uu^{k}\cdot\nabla \m^{k+1}		
	&&&&					-\,		
	&&&&&&					\Bb	
	&&&&&&&&				[\,			
	&&&&&&&&&&				\m^{k+1}\,
	&&&&&&&&&&&&			]\,		
	&&&&&&&&&&&&&&			=\,F_{2}^k
	&&&&&&&&&&&&&&&&		\RR_+\times			
	&&&&&&&&&&&&&&&&&&		\RR^{\dd},\\
							\partial_t\,	
	&						\w^{k+1}\,	\,
	&&						+\,\,\,\uu^{k}\cdot\nabla \w^{k+1}
	&&&&					-\,		
	&&&&&&					\Cc	
	&&&&&&&&				[\,			
	&&&&&&&&&&				\w^{k+1}\,
	&&&&&&&&&&&&			]\,		
	&&&&&&&&&&&&&&			=\,F_{3}^k
	&&&&&&&&&&&&&&&&		\RR_+\times			
	&&&&&&&&&&&&&&&&&&		\RR^{\dd},\\
	&						\,
	&&						\,
	&&&&						\,
	&&&&&&					\,
	&&&&&&&&					\,
	&&&&&&&&&&				\,
	&&&&&&&&&&&&				\,
	&&&&&&&&&&&&&&			\,\hspace{-5.9cm}(\,\uu^{k+1},\,\w^{k+1},\,\m^{k+1}\,)
							\big|_{\,t=0}=(\,\uu_{0},\,\w_{0},\,\m_{0}\,)
	\hspace{2.5cm}
	&&&&&&&&&&&&&&&&		\,					
	&&&&&&&&&&&&&&&&&&		\RR^\dd,
\end{alignedat}
\right.
\end{equation}
for any $k\in\NN$. We have introduced the following notation: for any function $f$ depending on the variable $(\uu,\,\w,\,\m)$, we denote by $f^k$ the function
$
	f^k(t,x)\,:=\,f\,(\,\uu^{k}(t,x),\w^{k}(t,x),\,\m^{k}(t,x)\,)
$.
Furthermore, we denote by $\delta f^k$ the difference between two consecutive functions: $ 
	\delta f^k
	:=\,f^{k+1}-
		f^k.$
The driving force $F^{k}_{i}$ are then defined by means of
\begin{equation*}
\begin{alignedat}{8}
	F^{k}_{1}\,&:=		\,\Div\big[\, \sigma^{{\rm E},\,k}\,+\,\tilde \sigma^{{\rm L},k}\,\big],\\
	F^{k}_{2} \,&:=	\,(\gamma_{1}^k\,-\,1)\Big[\,\partial_t\m^k\,
						+\,\uu^k \cdot \nabla \m^k\,\Big]\,
						-\,\gamma^k_{1}\,\Omega^k(\bar{\nn}\,+\,\m^k)\,+\\
						&\hspace{1.7cm}+\,\gamma^k_{2}\Big[\,\mathbb{D}^k[\,\bar{\nn}\,+\,\m^k\,]\,
						-\,\big(\,[\,\bar{\nn}\,+\,\m^k\,]\,\cdot \mathbb{D}^k[\,\bar{\nn}\,+\,\m^k\,]\,\big)\,[\,\bar{\nn}\,
						+\,\m^k\,]\,\Big]\,-\\
						&\hspace{8cm}-\tilde {\mathrm{h}}^k\,+\,\tilde \beta^k(\bar{ \nn}\,+\,\m^k\,)\,+\,f^k,\\
	F^{k}_{3} \,&:=		(\,1\,+\w^k\,)\Big[\,\partial_t \frac{\partial {\rm W_{{\rm F}}}^k}{\partial\w}\,
							+\,\uu^k\cdot\nabla\,\,{\rm \frac{\partial W_{{\rm F}}^k}{\partial\w}}\,\Big]\,+
							\\&\hspace{2cm}+\,\Div\{\,\tilde \lambda^k_{1}\nabla \w^k\}
							\,+\,\Div\{\tilde \lambda^k_{2}(\,[\bar \nn\,+\m^k]\cdot \nabla \w^k\,)[\bar{\nn}\,
							+\,\m^k]\}\,\\
							&\hspace{8cm}+\,\sigma^{{\rm L},k}:\mathbb{D}^k\,+\,{\rm g}^k\cdot \mcN^k,
\end{alignedat}	
\end{equation*}
while the pressure $\pre^k$ satisfies
\begin{equation*}
	\pre^k = (-\Delta)^{-1}\left[\, -\uu^k\cdot \nabla \uu^{k+1}\,-\Div\,\Aa[\uu^k]\,+\,F^k_1\,\right].
\end{equation*}
The core of the proof of Theorem \ref{main-thm3} relies on the following proposition:
\begin{prop}\label{prop:the_core}
	Let $(\,\uu^k,\,\w^k,\,\m^k\,)$ be the unique and global-in-time classical solution for system \eqref{appx-system-nu}. Then
	such a solution belongs to the same class affinity defined in Theorem \ref{main-thm3} and the pressure $\pre^k$ belongs to $L^1_t\BB_{2,1}^{\dd/2}\cap\BB_{2,1}^{\dd/2+1}$, fulfilling
	\begin{equation}\label{prop:the_core_ineq1}
		\|\,\uu^k\,\|_{\X_1}\,+\,\| \,\nabla \pre^{k} \,\|_{L^1_t\BB_{2,1}^{\frac{\dd}{2}-1}\cap \BB_{2,1}^\frac{\dd}{2}}  \,\leq \,\ee^2,\quad\quad
		\|\,\w^k \, \|_{\X_2}\,+\,
		\|\,\m^k\, \|_{\X_3} \, \leq \,\ee^3,
	\end{equation}
	Moreover, the difference between two consecutive solutions satisfies
	\begin{equation}\label{prop:the_core_ineq2}
	\begin{alignedat}{8}
		\|\,\delta \uu^{k}\,\|_{\X_{1}}\,+\,\| \,\nabla \delta \pre^{k} \,\|_{L^1_t\BB_{2,1}^{\frac{\dd}{2}-1}\cap \BB_{2,1}^\frac{\dd}{2}} \leq \ee^{k},\quad\quad
		\,\|\,\delta \w^{k}  \|_{\X_{2}}\,+\,\|\,\delta \m^{k} \|_{\X_{3}} \leq 
		\ee^{k+1}.
	\end{alignedat}
	\end{equation}
\end{prop}

\smallskip
\noindent
If Proposition \ref{prop:the_core} holds, than the proof of Theorem \ref{main-thm3} leads thanks to the convergence of the following series:
\begin{equation*}
	\sum_{k\in\NN}\|\,\delta \uu^{k}\,\|_{\X_{1}} \,+\, 
		\|\,\delta \w^{k}  \|_{\X_{2}}\,+\,\|\,\delta \m^{k} \|_{\X_{3}}\,<\,+\infty.
\end{equation*}
The approximate solutions $(\,\uu^k,\,\w^k,\,\m^k\,)$ form a Cauchy sequence in $\X_1\times\X_2\times\X_3$ and we claim that the limit $(\,\uu,\,\w,\,\m\,)$ is a strong solution of system \eqref{main_system2}.

\smallskip\noindent
We remark that in system \eqref{appx-system-nu} we do not consider an unitary constraint on the direct field $\nn^k$. Nevertheless, Proposition \ref{prop:the_core} leads to an uniform bound for the $L^\infty_{t,x}$-norm:
\begin{equation*}
	\| \,\nn^k\,\|_{L^\infty_{t,x}}\,\lesssim \, 1+\|\,\m^k\,\|_{L^\infty_t\BB_{2,1}^\frac{\dd}{2}}\lesssim 1+\ee^3.
\end{equation*}
We will recover the unitary constraint on the director field $|\nn|=1$, in section \ref{sec:passagetothelimit}, when passing to the limit as $k$ goes to $\infty$.

\noindent\smallskip
We now deal with the proof of Proposition \ref{prop:the_core}. We separately analyze each equation of \eqref{appx-system-nu} in sections \ref{thebalanceoflinearmomentum}, \ref{thebalanceofangularmomentum} and \ref{thetemperatureequation}. 
We proceed by induction and the base case readily follows by the definition of $(\uu^0,\,\omega^0,\,\m^0)$, which is identically null.
%%%%%%%%%%%%%%%%%%%%%%%%%%%%%%%%%%%%%%%%%%%%%%%%%%%%%%%%%%%%%%%%%%%%%%%%%%%%%%%%%%%%%%%%%%%%%%%%%%%%%%%%%%%%%
%%%%%%%%%%%%%%%%%%%%%%%%%%%%%%%%%%%%%%%%%%%%%%%%%%%%%%%%%%%%%%%%%%%%%%%%%%%%%%%%%%%%%%%%%%%%%%%%%%%%%%%%%%%%%
%					 The balance of linear momentum
%%%%%%%%%%%%%%%%%%%%%%%%%%%%%%%%%%%%%%%%%%%%%%%%%%%%%%%%%%%%%%%%%%%%%%%%%%%%%%%%%%%%%%%%%%%%%%%%%%%%%%%%%%%%%
%%%%%%%%%%%%%%%%%%%%%%%%%%%%%%%%%%%%%%%%%%%%%%%%%%%%%%%%%%%%%%%%%%%%%%%%%%%%%%%%%%%%%%%%%%%%%%%%%%%%%%%%%%%%%
\subsection{The balance of linear momentum}\label{thebalanceoflinearmomentum}$\,$

\smallskip\noindent
We begin with considering the approximate velocity field $\uu^k$ and we aim to prove that
\begin{equation}\label{claim:momentum}
	\|\,\uu^{k+1}\,\|_{\X_1}\,+\,\|\,\nabla \pre^{k+1}\,\|_{L^1_t\BB_{2,1}^\frac{\dd}{2}\cap \BB_{2,1}^{\frac{\dd}{2}-1}}\,\leq \, \ee^2\quad\text{and}\quad \|\,\delta \uu^{k+1}\,\|_{\X_1}\,\leq\, \ee^{k+1},
\end{equation}
assuming that inequalities \eqref{prop:the_core_ineq1} and \eqref{prop:the_core_ineq2} are satisfied by $(\uu^k,\,\pre^k,\,\Tt^k,\,\m^k)$ and $(\delta \uu^k,\,\delta \pre^k,\, \delta \Tt^k,\, \delta \m^k)$ respectively, for a fixed positive integer $k$. We recall that $\uu^{k+1}$ is solution of
\begin{equation*}
	\begin{cases}
		\,\partial_t\, 	\uu^{k+1}\, 	\,-\,		\Aa	[\,\uu^{k+1}\,]	\,+\nabla \pre^{k+1}\,	=\,-\,\uu^{k}\cdot\nabla \uu^{k+1}\,+\,F_{1}^k,\\
		\,\Div\,\uu^{k+1}\,=\,0,
	\end{cases}
\end{equation*}
where $F^k_{1}:=	\,\Div\{\, \sigma^{{\rm E},\,k}\,+\,\tilde \sigma^{{\rm L},k}\,\}$.
Thanks to Theorem \ref{thm-superjuve-stokes}, we gather that the norms of $\uu^{k+1}$ and the pressure $\pre^{k+1}$ are bounded by
\begin{equation*}
\begin{aligned}
	\|\,\uu^{k+1}\,\|_{\X_1}\,&+\,\|\,\nabla \pre^{k+1}\,\|_{L^1_t\BB_{2,1}^{\frac{\dd}{2}-1}\cap \BB_{2,1}^\frac{\dd}{2}}\,
	\lesssim\,\|\,\uu_0\,\|_{\BB_{2,1}^{\frac{\dd}{2}-1}\cap\BB_{2,1}^\frac{\dd}{2}}\,+\,\|\,-\uu^{k}\cdot\nabla \uu^{k+1}\,+\,F^k_{1}\,\|_{L^1_t\,\BB_{2,1}^{\frac{\dd}{2}-1}\cap \BB_{2,1}^{\frac{\dd}{2}}}\\
	&\lesssim\,\|\,\uu_0\,\|_{\BB_{2,1}^{\frac{\dd}{2}-1}\cap\BB_{2,1}^\frac{\dd}{2}}\,+\,\|\,\,\uu^{k}\,\|_{L^\infty_t\BB_{2,1}^{\frac{\dd}{2}-1}\cap \BB_{2,1}^{\frac{\dd}{2}}}
	\|\, \nabla \uu^{k+1}\,
	\|_{L^1_t \BB_{2,1}^{\frac{\dd}{2}}}\,+\,
	\|\,F^k_{1}\,\|_{L^1_t\,\BB_{2,1}^{\frac{\dd}{2}-1}\cap \BB_{2,1}^{\frac{\dd}{2}}}\\
	&\lesssim \,\|\,\uu_0\,\|_{\BB_{2,1}^{\frac{\dd}{2}-1}\cap\BB_{2,1}^\frac{\dd}{2}}\,+\,\ee \,\|\,\uu^{k+1}\,\|_{\X_1}\,+\,\|\,F^k_{1}\,\|_{L^1_t\,\BB_{2,1}^{\frac{\dd}{2}-1}\cap \BB_{2,1}^{\frac{\dd}{2}}},
\end{aligned}
\end{equation*}
where we have used the continuity of the product from $\BB_{2,1}^{\dd/2-1}\times \BB_{2,1}^{\dd/2}$ to $\BB_{2,1}^{\dd/2-1}$ and the algebra structure of $\BB_{2,1}^{\dd/2}$.
Making use of to the small assumption \eqref{ineq:realsmallness-condition-main-thm} to the initial data, with $\ee$ small enough, we then obtain
\begin{equation}\label{bound00}
	\|\,\uu^{k+1}\,\|_{\X_1}\,+\,\|\,\nabla \pre^{k+1}\,\|_{L^1_t\BB_{2,1}^{\frac{\dd}{2}-1}\cap \BB_{2,1}^\frac{\dd}{2}}\,\lesssim\,\ee^4\,+ \,\|\,F^k_{1}\,\|_{L^1_t\,\BB_{2,1}^{\frac{\dd}{2}-1}\cap \BB_{2,1}^{\frac{\dd}{2}}}.
\end{equation}
We claim that the norm of $F^k_1$ satisfies the following inequality
\begin{equation}\label{ineq:lemma1ELtensors}
		\|\,F^k_{1}\,\|_{L^1_t\,\BB_{2,1}^{\frac{\dd}{2}-1}\cap \BB_{2,1}^{\frac{\dd}{2}}}
		\,
		\lesssim\,
		\|\,\Div\,\sigma^{\rm \,E}\,\|_{L^1_t\BB_{2,1}^{\frac{\dd}{2}-1}\cap \BB_{2,1}^\frac{\dd}{2}}
		\,+\,
		\|\,\Div\,\tilde \sigma^{\rm \,L}\,\|_{L^1_t\BB_{2,1}^{\frac{\dd}{2}-1}\cap \BB_{2,1}^\frac{\dd}{2}}
		\,\lesssim\,\ee^3,
\end{equation}
and we postpone the proof to the Appendix, in Section \ref{Appendix}.
We then replace inequality \eqref{ineq:lemma1ELtensors} into the estimate \eqref{bound00}, from which we deduce
\begin{equation*}
	\|\,\uu^{k+1}\,\|_{\X_1}+\,\|\,\nabla \pre^{k+1}\,\|_{L^1_t\BB_{2,1}^{\frac{\dd}{2}-1}\cap \BB_{2,1}^\frac{\dd}{2}}\lesssim \ee^3\quad\Rightarrow\quad \|\,\uu^{k+1}\,\|_{\X_1}
	+\,\|\,\nabla \pre^{k+1}\,\|_{L^1_t\BB_{2,1}^{\frac{\dd}{2}-1}\cap \BB_{2,1}^\frac{\dd}{2}}\leq \ee^2,
\end{equation*}
whenever $\ee$ is taken small enough. This concludes the proof of the first inequality in \eqref{claim:momentum}. 

\smallskip\noindent
We now take into account the difference between two consecutive velocity, $\delta \uu^{k+1}:=\uu^{k+2}-\uu^{k+1}$, 
which is solution of 
\begin{equation*}
	\partial_t \delta\uu^{k+1}\,+\,\uu^{k+1}\cdot \nabla \delta \uu^{k+1}\,+\,\delta \uu^{k}\cdot \nabla  \uu^{k+1}\,-\Aa[\,\delta \uu^{k+1}\,]\,+\,\nabla \delta \pre^{k+1}=\,\delta F^{k}_1.
\end{equation*}
with $\delta F^{k}_1\,=\,\Div\,\{\,\delta\sigma^{{\rm E},\,k}\,+\,\delta \sigma^{{\rm L},\,k}\}$. Applying Theorem \ref{thm-superjuve-stokes},  we gather
\begin{equation*}
	\begin{aligned}
		\|\,\delta \uu^{k+1}\,\|_{\X_1}	\,&+\,\|\,\nabla \delta \pre^{k+1}\,\|_{L^1_t\BB_{2,1}^{\frac{\dd}{2}-1}\cap \BB_{2,1}^{\frac{\dd}{2}}}\lesssim\,
		\|	\,-\,\uu^{k+1}\cdot \nabla \delta \uu^{k+1}\,-\,\delta \uu^{k}\cdot \nabla  \uu^{k+1}\,+\,\delta F^{k}_1	\, \|_{L^1_t \BB_{2,1}^{\frac{\dd}{2}-1}\cap \BB_{2,1}^{\frac{\dd}{2}}}\\
		&\lesssim \|\,(\uu^k,\,\uu^{k+1})\,\|_{L^2_t\BB_{2,1}^{\frac{\dd}{2}}\cap \BB_{2,1}^{\frac{\dd}{2}+1}}
				\|\,(\delta \uu^{k},\,\delta \uu^{k+1})\,\|_{L^2_t\BB_{2,1}^{\frac{\dd}{2}}\cap \BB_{2,1}^{\frac{\dd}{2}+1}}\,+\, \|\,\delta F^{k}_1	\, \|_{L^1_t \BB_{2,1}^{\frac{\dd}{2}-1}\cap \BB_{2,1}^{\frac{\dd}{2}}}\\
		&\lesssim		\ee^2\|\,\delta \uu^{k+1}\,\|_{L^2_t\BB_{2,1}^{\frac{\dd}{2}}\cap \BB_{2,1}^{\frac{\dd}{2}+1}}\,+\,\ee^2\|\,\delta \uu^{k}\,\|_{L^2_t\BB_{2,1}^{\frac{\dd}{2}}\cap \BB_{2,1}^{\frac{\dd}{2}+1}}
						+\, \|\,\delta F^{k}_1	\, \|_{L^1_t \BB_{2,1}^{\frac{\dd}{2}-1}\cap \BB_{2,1}^{\frac{\dd}{2}}}.
	\end{aligned}
\end{equation*}
Assuming $\ee$ small enough and thanks to the induction hypotheses we deduce that
\begin{equation}\label{Pjaca}
	\|\,\delta \uu^{k+1}\,\|_{\X_1}\,+\,\|\,\nabla \delta \pre^{k+1}\,\|_{L^1_t\BB_{2,1}^{\frac{\dd}{2}-1}\cap \BB_{2,1}^\frac{\dd}{2}}	\,\lesssim\, \ee^2\|\,\delta \uu^k\,\|_{\X_1}\,+\,\|\,\delta F^{k+1}\,\|_{\X_1}\,\leq \ee^{k+2}\,+\,\|\,\delta F^{k}\,\|_{\X_1}.
\end{equation}
We claim that the norm of $\delta F^{k}_1\,=\,\Div\,\delta \sigma^{\,{\rm E},\, k}\,+\,\Div\,\delta \sigma^{\,{\rm L},\,k}$ can be bounded as follows:
\begin{equation}\label{est:momeqL}
\begin{aligned}
	\|\,		\Div\,\delta \sigma^{\,{\rm L},\, k}\,\|_{L^1_t\BB_{2,1}^{\frac{\dd}{2}-1}\cap \BB_{2,1}^{\frac{\dd}{2}	}}\,&+\,
	\|\,		\Div\,\delta \sigma^{\,{\rm E},\, k}\,\|_{L^1_t\BB_{2,1}^{\frac{\dd}{2}-1}\cap \BB_{2,1}^{\frac{\dd}{2}	}}
	\,\lesssim\\&\lesssim
	\ee^2\Big[\,\|\,\delta \uu^k\,\|_{\X_1}\,+\,\|\,\delta \omega^k\,\|_{\X_2}\,+\,\|\,\delta \m^k\,\|_{\X_3}\,\Big]\,
	\,\leq\, \ee^{k+2},
\end{aligned}
\end{equation}
and we postpone the proof to the appendix, in section \ref{Appendix}.

\smallskip\noindent
We then replace inequalities \eqref{Pjaca} and \eqref{est:momeqL} into \eqref{bound00} and assuming $\ee$ small enough, we finally gather the following bound for $\delta \uu^{k+1}$ and $\delta \pre^{k+1}$: 
\begin{equation*}
	\|\,\delta \uu^{k+1}\,\|_{\X_1}	\,+\,\|\,\nabla \delta \pre^{k+1}\,\|_{L^1_t\BB_{2,1}^{\frac{\dd}{2}-1}\cap \BB_{2,1}^\frac{\dd}{2}}\,\lesssim\, \ee^{k+2} 
	\;\quad\Rightarrow\quad\; \|\,\delta \uu^{k+1}\,\|_{\X_1}\,+\,\|\,\nabla \delta \pre^{k+1}\,\|_{L^1_t\BB_{2,1}^{\frac{\dd}{2}-1}\cap \BB_{2,1}^\frac{\dd}{2}} \leq \ee^{k+1},	
\end{equation*}
which concludes the proof of the estimate in \eqref{prop:the_core_ineq2}.
%%%%%%%%%%%%%%%%%%%%%%%%%%%%%%%%%%%%%%%%%%%%%%%%%%%%%%%%%%%%%%%%%%%%%%%%%%%%%%%%%%%%%%%%%%%%%%%%%%%%%%%%%%%%%
%%%%%%%%%%%%%%%%%%%%%%%%%%%%%%%%%%%%%%%%%%%%%%%%%%%%%%%%%%%%%%%%%%%%%%%%%%%%%%%%%%%%%%%%%%%%%%%%%%%%%%%%%%%%%
%%%%%%%%%%%%%%%%%%%%%%%%%%%%%%%%%%%%%%%%%%%%%%%%%%%%%%%%%%%%%%%%%%%%%%%%%%%%%%%%%%%%%%%%%%%%%%%%%%%%%%%%%%%%%
%					The balance of angular momentum
%%%%%%%%%%%%%%%%%%%%%%%%%%%%%%%%%%%%%%%%%%%%%%%%%%%%%%%%%%%%%%%%%%%%%%%%%%%%%%%%%%%%%%%%%%%%%%%%%%%%%%%%%%%%%
%%%%%%%%%%%%%%%%%%%%%%%%%%%%%%%%%%%%%%%%%%%%%%%%%%%%%%%%%%%%%%%%%%%%%%%%%%%%%%%%%%%%%%%%%%%%%%%%%%%%%%%%%%%%%
\subsection{The balance of angular momentum}\label{thebalanceofangularmomentum}$\,$

\noindent
We now take into account the approximate director field $\nn^k\,=\,\bar{\nn}\,+\,\m^k$ and we aim to prove by induction 
the following inequalities
\begin{equation}\label{angular_mom_ind}
	\|\,\m^{k+1}\,\|_{\X_2}\,\leq \ee^{3}\quad\text{and}\quad \|\,\delta \m^{k+1}\,\|_{\X_2}\,\leq \ee^{k+2},
\end{equation}
assuming the hypothesis \eqref{prop:the_core_ineq1} and \eqref{prop:the_core_ineq2} for $(\uu^k,\,\Tt^k,\,\m^k)$ and $(\delta \uu^k,\, \delta \Tt^k,\, \delta \m^k)$, for a fixed integer $k$. 
The approximate solution $\m^{k+1}$ satisfies the following parabolic PDE:
\begin{equation*}
	\partial_t \m^{k+1}\,-\,\Bb[\,\m^{k+1}\,] \,=\,-\,\uu^k\cdot\nabla \m^{k+1} \,+\,F^{k}_2,
\end{equation*}
where the forcing term $F_2^k$ is defined by means of
\begin{equation*}
\begin{alignedat}{4}
	F^{k}_{2} \,&:=	\,\tilde \gamma_{1}(\,\omega^k\,)\Big[\,\partial_t\m^k\,
						+\,\uu^k \cdot \nabla \m^k\,\Big]\,
						-\,\gamma_{1}(\,\omega^k\,)\,\Omega^k(\bar{\nn}\,+\,\m^k)\,+
						\,\gamma_{2}(\,\omega^k\,)\Big[\,\mathbb{D}^k[\,\bar{\nn}\,+\,\m^k\,]\,-\\&\hspace{1cm}
						-\,\big(\,[\,\bar{\nn}\,+\,\m^k\,]\,\cdot \mathbb{D}^k[\,\bar{\nn}\,+\,\m^k\,]\,\big)\,[\,\bar{\nn}\,
						+\,\m^k\,]\,\Big]\,-
						\tilde {\mathrm{h}}^k\,+\,\tilde \beta^k(\bar{ \nn}\,+\,\m^k\,)+\,f^k,\\
	\tilde{\rm h}^k\,&:=\,K_3(\,\omega^k\,)(\,\nabla \m^k\odot \nabla \m^k\,)(\bar \nn\,+\,\m^k\,)\,-\,
								\Div
								\Big\{\,
									\tilde K_1(\,\omega^k\,) \nabla \m^k\,+\,
									\tilde K_2(\,\omega^k\,) \, \Div\,\m^k\,\Id\,+\\&\hspace{0.34cm}+\,
									\tilde K_3(\,\omega^k\,)\,[\,(\bar \nn\,+\,\m^k\,)\cdot \nabla \m^k\otimes(\bar \nn\,+\,\m^k)\,+\,
									\tilde K_4(\,\omega^k\,)\,\tr \nabla \m^k \,
								\Big\}\,+\\&\hspace{0.34cm}+\,
								\bar K_3\,
								\Div
								\Big\{\,
									[\,\m^k\cdot \nabla \m^k\,	]\otimes(\,\bar \nn\,+\,\m^k\,)	\,+\,
									[\,(\bar \nn\,+\,\m^k)\cdot\nabla \m^k\,	]\otimes\m^k	\,+\,
									[\m^k\cdot \nabla \m^k\,	]\otimes\m^k\,
								\Big\},\\
	f^k\,&:=\,\Div\,\Big\{\,(\,\tilde K_2(\,\omega^k\,)\,+\,\tilde K_4(\,\omega^k\,)\,)\,(\Div\,\m^k)\,(\bar \nn\,+\,\m^k\,)\,\Big\}(\bar \nn\,+\,\m^k\,)\,+\\&+\,
			(\,\bar K_2\,+\,\bar K_4\,)\bigg\{\Div\,\Big\{\,\,(\Div\,\m^k)\,\m^k\,\Big\}\m^k\,+\,
			\Div\,\Big\{\,\,(\Div\,\m^k)\,\bar \nn\,\Big\}\m^k\,+\,
			\Div\,\Big\{\,\,(\Div\,\m^k)\m^k\,\Big\}\bar \nn\bigg\},\\
	\beta^k\,&=\,K_1(\,\omega^k\,)|\,\nabla \m^k\,|^2\,+\,
	(\,K_2(\,\omega^k\,)\,+\,K_4(\,\omega^k\,))\,|\,\Div\,\m^k\,|^2+\, 
	K_3(\,\omega^k\,)\,|\,(\,\bar \nn\,+\,\m^k)\cdot\nabla \m^k\,|^2.
\end{alignedat}
\end{equation*}
Thanks to Theorem \eqref{thm-superjuve} we gather that the $\X_2$-norm of $\m^k$ is bounded by
\begin{equation*}
\begin{aligned}
	\|\,\m^k\,\|_{\X_2}\,
	&\lesssim \, 
	\|\,\m_0\,\|_{\BB_{2,1}^\frac{\dd}{2}\cap \BB_{2,1}^{\frac{\dd}{2}+1}}\,+\,
	\|\,\uu^k\cdot\nabla \m^{k+1}\,\|_{L^1_t \BB_{2,1}^\frac{\dd}{2}\cap \BB_{2,1}^{\frac{\dd}{2}+1}}\,+\,
	\|\,F_2^k\,\|_{L^1_t \BB_{2,1}^\frac{\dd}{2}\cap \BB_{2,1}^{\frac{\dd}{2}+1}}\\
	&\lesssim
	\|\,\m_0\,\|_{\BB_{2,1}^\frac{\dd}{2}\cap \BB_{2,1}^{\frac{\dd}{2}+1}}\,+\,
	\|\,\uu^k\,\|_{L^2_t \BB_{2,1}^\frac{\dd}{2}\cap \BB_{2,1}^{\frac{\dd}{2}+1}}
	\|\,\m^{k+1}\,\|_{L^2_t \BB_{2,1}^\frac{\dd}{2}\cap \BB_{2,1}^{\frac{\dd}{2}+1}}\,+\,
	\|\,F_2^k\,\|_{L^1_t \BB_{2,1}^\frac{\dd}{2}\cap \BB_{2,1}^{\frac{\dd}{2}+1}}\\
	&\lesssim\, \ee^3\,+\,
	\|\,\m_0\,\|_{\BB_{2,1}^\frac{\dd}{2}\cap \BB_{2,1}^{\frac{\dd}{2}+1}}\,+\,
	\ee^2\|\,\m^{k+1}\,\|_{\X_2}\,+\,
	\|\,F_2^k\,\|_{L^1_t \BB_{2,1}^\frac{\dd}{2}\cap \BB_{2,1}^{\frac{\dd}{2}+1}},
\end{aligned}
\end{equation*}
where we have used the algebra structure of $\BB_{2,1}^\frac{\dd}{2}\cap \BB_{2,1}^{\frac{\dd}{2}+1}$. Assuming $\ee$ small enough, we can absorb the $\m^{k+1}$-term on the right hand side by the left-hand side of the above estimate, thus
\begin{equation}
	\|\,\m^k\,\|_{\X_2}\,
	\lesssim \, 
	\ee^4\,+\,\|\,F_2^k\,\|_{L^1_t \BB_{2,1}^\frac{\dd}{2}\cap \BB_{2,1}^{\frac{\dd}{2}+1}}.
\end{equation}
We now aim to bound the norm of $F_2^k$. We first observe that
\begin{align*}
	\Big\|\,\tilde \gamma_{1}(\,\omega^k\,)\Big[\,\partial_t\m^k\,+\,\uu^k \cdot \nabla \m^k\,\Big]\,\Big\|_{L^1_t \BB_{2,1}^\frac{\dd}{2}\cap \BB_{2,1}^{\frac{\dd}{2}+1}}
	\,\lesssim\,
	\|\,\omega^k\,\|_{L^\infty_t\BB_{2,1}^\frac{\dd}{2}}\Big[\,\|\,\partial_t\m^k\,\|_{L^1_t \BB_{2,1}^\frac{\dd}{2}}+
	\|\,	\uu^k	\,\|_{L^2_t\BB_{2,1}^\frac{\dd}{2}}{\scriptstyle\times}\\{\scriptstyle\times}
	\|\,	\nabla \m^k	\,\|_{L^2_t\BB_{2,1}^\frac{\dd}{2}}
	\Big]\,+\,
	\|\,\nabla \omega^k\,\|_{L^2_t\BB_{2,1}^\frac{\dd}{2}}
	\Big[\,\|\,\partial_t\m^k\,\|_{L^2_t \BB_{2,1}^\frac{\dd}{2}}+
	\|\,	\uu^k	\,\|_{L^\infty_t\BB_{2,1}^\frac{\dd}{2}}
	\|\,	\nabla \m^k	\,\|_{L^2_t\BB_{2,1}^\frac{\dd}{2}}
	\Big]\,+\\+\,
	\|\,\omega^k\,\|_{L^\infty_t\BB_{2,1}^\frac{\dd}{2}}\Big[\,\|\,\partial_t\nabla \m^k\,\|_{L^1_t \BB_{2,1}^\frac{\dd}{2}}+
	\|\,\nabla 	\uu^k	\,\|_{L^2_t\BB_{2,1}^\frac{\dd}{2}}
	\|\,	\nabla \m^k	\,\|_{L^2_t\BB_{2,1}^\frac{\dd}{2}}
	\,+\\+\,
	\|\,	\uu^k	\,\|_{L^2_t\BB_{2,1}^\frac{\dd}{2}}
	\|\,	\Delta \m^k	\,\|_{L^2_t\BB_{2,1}^\frac{\dd}{2}}
	\Big]\,
	\lesssim 
	\ee^4,
\end{align*}
while
\begin{align*}
	\|\,\gamma_{1}(\,\omega^k\,)\,\Omega^k(\bar{\nn}\,+\,\m^k)\,\|_{L^1_t \BB_{2,1}^\frac{\dd}{2}\cap \BB_{2,1}^{\frac{\dd}{2}+1}}
	\lesssim
	(1+\|\,\omega^k\,\|_{L^\infty_t\BB_{2,1}^\frac{\dd}{2}})
	\|\,	\nabla \uu^k	\,\|_{L^1_t\BB_{2,1}^\frac{\dd}{2}}
	\Big(
		1+\|\,\m^k\,\|_{L^\infty_t\BB_{2,1}^\frac{\dd}{2}}
	\Big)+\\
	\|\,\nabla \omega^k\,\|_{L^\infty_t\BB_{2,1}^\frac{\dd}{2}}
	\|\,	\nabla \uu^k	\,\|_{L^1_t\BB_{2,1}^\frac{\dd}{2}}
	\Big(
		1+\|\,\m^k\,\|_{L^\infty_t\BB_{2,1}^\frac{\dd}{2}}
	\Big)+
	(1+\|\,\omega^k\,\|_{L^\infty_t\BB_{2,1}^\frac{\dd}{2}})
	\|\,	\Delta \uu^k	\,\|_{L^1_t\BB_{2,1}^\frac{\dd}{2}}{\scriptstyle\times}\\{\scriptstyle\times}
	\Big(
		1+\|\,\m^k\,\|_{L^\infty_t\BB_{2,1}^\frac{\dd}{2}}
	\Big)+
	(1+\|\,\omega^k\,\|_{L^\infty_t\BB_{2,1}^\frac{\dd}{2}})
	\|\,	\nabla \uu^k	\,\|_{L^1_t\BB_{2,1}^\frac{\dd}{2}}
	\|\,\nabla \m^k\,\|_{L^\infty_t\BB_{2,1}^\frac{\dd}{2}}\lesssim \ee^4,
\end{align*}
where we have assumed $\bar \alpha_4> 1/\ee^3$ large enough, so that
\begin{equation*}
	\| \nabla \uu^k \|_{L^1_t\BB_{2,1}^\frac{\dd}{2}\cap \BB_{2,1}^{\frac{\dd}{2}+1}}
	\leq
	\frac{1}{\bar \alpha_4}
	\|  \uu^k	\|_{\X_1}
	\leq 
	\frac{1}{\bar \alpha_4}\ee^2 \leq \ee^4.
\end{equation*}
As third term, we take into account
\begin{align*}
	\Big\|\gamma_{2}(\omega^k)\Big[\mathbb{D}^k[\bar{\nn}+\m^k]-\big([\bar{\nn}+\m^k]\cdot \mathbb{D}^k[\bar{\nn}+\m^k]\big)[\bar{\nn}
						+\m^k]\Big]
	\Big\|_{L^1_t\BB_{2,1}^\frac{\dd}{2}\cap \BB_{2,1}^{\frac{\dd}{2}+1}}
	\lesssim
	(1+\|\omega^k\|_{L^\infty_t\BB_{2,1}^\frac{\dd}{2}}){\scriptstyle\times}\\
	{\scriptstyle\times}
	\|\nabla \uu^k\|_{L^1_t\BB_{2,1}^\frac{\dd}{2}}
	\Big(1+\|\m^k\|_{L^\infty_t\BB_{2,1}^\frac{\dd}{2}}^3\Big)
	+
	\|\nabla \omega^k\|_{L^\infty_t\BB_{2,1}^\frac{\dd}{2}}
	\|\nabla \uu^k\|_{L^1_t\BB_{2,1}^\frac{\dd}{2}}
	\Big(1+\|\m^k\|_{L^\infty_t\BB_{2,1}^\frac{\dd}{2}}^3\Big)+\\+
	(1+\|\omega^k\|_{L^\infty_t\BB_{2,1}^\frac{\dd}{2}})
	\|\Delta \uu^k\|_{L^1_t\BB_{2,1}^\frac{\dd}{2}}\Big(1+\|\m^k\|_{L^\infty_t\BB_{2,1}^\frac{\dd}{2}}^3\Big)
	+(1+\|\omega^k\|_{L^\infty_t\BB_{2,1}^\frac{\dd}{2}})
	\|\nabla \uu^k\|_{L^1_t\BB_{2,1}^\frac{\dd}{2}}{\scriptstyle\times}\\
	{\scriptstyle\times}
	\|\nabla \m^k\|_{L^\infty_t\BB_{2,1}^\frac{\dd}{2}}\Big(1+\|\m^k\|_{L^\infty_t\BB_{2,1}^\frac{\dd}{2}}^2\Big)
	\lesssim \ee^4,
\end{align*}
from which we deduce that $F_2^k$ fulfills the following inequality
\begin{equation}\label{ineq:angularmom_F}
\begin{aligned}
	\|\,F^k_2\,\|_{L^1_t\BB_{2,1}^\frac{\dd}{2}\cap \BB_{2,1}^{\frac{\dd}{2}+1}}
	\lesssim
	\ee^4\,+\,
	\|\,{\rm h}^k\,\|_{L^1_t\BB_{2,1}^\frac{\dd}{2}\cap \BB_{2,1}^{\frac{\dd}{2}+1}}\,+\,
	\|\,\tilde \beta^k\,\|_{L^1_t\BB_{2,1}^\frac{\dd}{2}\cap \BB_{2,1}^{\frac{\dd}{2}+1}}\,+\,
	\|\,f^k\,\|_{L^1_t\BB_{2,1}^\frac{\dd}{2}\cap \BB_{2,1}^{\frac{\dd}{2}+1}}.
\end{aligned}
\end{equation}
We proceed controlling $\rm h^k$. The first term we take into account is
\begin{align*}
	\|K_3(\omega^k)(\nabla \m^k\odot \nabla \m^k)(\bar \nn+\m^k)\|_{L^1_t\BB_{2,1}^\frac{\dd}{2}\cap \BB_{2,1}^{\frac{\dd}{2}+1}}
	\lesssim
	( 1+ \|\omega^k\|_{L^\infty_t\BB_{2,1}^\frac{\dd}{2}})
	\|\nabla \m^k\|_{L^2_t\BB_{2,1}^\frac{\dd}{2}}^2{\scriptstyle\times}\\
	{\scriptstyle\times}
	( 1+ \|\m^k\|_{L^\infty_t\BB_{2,1}^\frac{\dd}{2}})+
	\|\nabla \omega^k\|_{L^\infty_t\BB_{2,1}^\frac{\dd}{2}}\|\nabla \m^k\|_{L^2_t\BB_{2,1}^\frac{\dd}{2}}^2
	( 1+ \|\m^k\|_{L^\infty_t\BB_{2,1}^\frac{\dd}{2}})+\\+
	\| \omega^k\|_{L^\infty_t\BB_{2,1}^\frac{\dd}{2}}
	\|\nabla \m^k\|_{L^2_t\BB_{2,1}^\frac{\dd}{2}}
	\|\Delta \m^k\|_{L^2_t\BB_{2,1}^\frac{\dd}{2}}
	\lesssim
	\ee^4.
\end{align*}
Furthermore, the following inequality is satisfied
\begin{align*}
	\Big\|\Div\Big( \tilde  K_1(\omega^k)\nabla \m^k\Big)\Big\|_{L^1_t\BB_{2,1}^\frac{\dd}{2}\cap \BB_{2,1}^{\frac{\dd}{2}+1}}
	\lesssim
	\| \omega^k\|_{L^\infty_t\BB_{2,1}^\frac{\dd}{2}}
	\| \Delta \m^k \|_{L^1_t\BB_{2,1}^\frac{\dd}{2}}+
	\| \nabla \omega^k\|_{L^2_t\BB_{2,1}^\frac{\dd}{2}}
	\| \nabla \m^k \|_{L^2_t\BB_{2,1}^\frac{\dd}{2}}+\\
	+
	\| \nabla \omega^k\|_{L^\infty_t\BB_{2,1}^\frac{\dd}{2}}
	\| \Delta \m^k \|_{L^1_t\BB_{2,1}^\frac{\dd}{2}}+
	\| \omega^k\|_{L^\infty_t\BB_{2,1}^\frac{\dd}{2}}
	\| \Delta \nabla \m^k \|_{L^1_t\BB_{2,1}^\frac{\dd}{2}}+
	\| \Delta \omega^k\|_{L^2_t\BB_{2,1}^\frac{\dd}{2}}
	\| \Delta \m^k \|_{L^2_t\BB_{2,1}^\frac{\dd}{2}}+\\+
	\| \nabla \omega^k\|_{L^\infty_t\BB_{2,1}^\frac{\dd}{2}}
	\| \Delta \nabla \m^k \|_{L^1_t\BB_{2,1}^\frac{\dd}{2}}
	\lesssim
	\ee^4.
\end{align*}
With a similar approach we can bound the following term
\begin{align*}
	\Big\|\, \Div 	\Big\{\tilde K_2(\omega^k) \Div\m^k\Id+	\tilde K_4(\omega^k)\tr \nabla \m^k 	\Big\}\,
	\Big\|_{L^1_t\BB_{2,1}^\frac{\dd}{2}\cap \BB_{2,1}^{\frac{\dd}{2}+1}}
	\lesssim \ee^4.
\end{align*}
We need then to control
\begin{align*}
	\Big\|\,	\Div\,\Big\{\tilde K_3(\omega^k)[(\bar \nn+\m^k)\cdot \nabla \m^k\otimes(\bar \nn+\m^k)\Big\}
	\,\Big\|_{L^1_t\BB_{2,1}^\frac{\dd}{2}\cap \BB_{2,1}^{\frac{\dd}{2}+1}}
	\lesssim
	\| \nabla \omega^k\|_{L^2_t\BB_{2,1}^\frac{\dd}{2}}\| \nabla \m^k \|_{L^2_t\BB_{2,1}^\frac{\dd}{2}}{\scriptstyle\times}\\
	{\scriptstyle\times}
	\Big( 1+ \|  \m^k \|_{L^\infty_t\BB_{2,1}^\frac{\dd}{2}}^2\Big)+	
	\|\omega^k\|_{L^\infty_t\BB_{2,1}^\frac{\dd}{2}}\| \Delta \m^k \|_{L^1_t\BB_{2,1}^\frac{\dd}{2}}
	\Big( 1+ \|  \m^k \|_{L^\infty_t\BB_{2,1}^\frac{\dd}{2}}^2\Big)+	
	\|\omega^k\|_{L^\infty_t\BB_{2,1}^\frac{\dd}{2}}{\scriptstyle\times}\\
	{\scriptstyle\times}\| \nabla \m^k \|_{L^2_t\BB_{2,1}^\frac{\dd}{2}}^2
	\Big( 1+ \|  \m^k \|_{L^\infty_t\BB_{2,1}^\frac{\dd}{2}}\Big)+
	%------------------------
	\| \Delta \omega^k\|_{L^2_t\BB_{2,1}^\frac{\dd}{2}}\| \nabla \m^k \|_{L^2_t\BB_{2,1}^\frac{\dd}{2}}
	\Big( 1+ \|  \m^k \|_{L^\infty_t\BB_{2,1}^\frac{\dd}{2}}^2\Big)+\\+
	\| \nabla \omega^k\|_{L^2_t\BB_{2,1}^\frac{\dd}{2}}\| \Delta \m^k \|_{L^2_t\BB_{2,1}^\frac{\dd}{2}}
	\Big( 1+ \|  \m^k \|_{L^\infty_t\BB_{2,1}^\frac{\dd}{2}}^2\Big)+
	\| \nabla \omega^k\|_{L^\infty_t\BB_{2,1}^\frac{\dd}{2}}\| \nabla \m^k \|_{L^2_t\BB_{2,1}^\frac{\dd}{2}}^2{\scriptstyle\times}\\
	{\scriptstyle\times}
	\Big( 1+ \|  \m^k \|_{L^\infty_t\BB_{2,1}^\frac{\dd}{2}}\Big)+
	%------------------------
	\|\nabla \omega^k\|_{L^\infty_t\BB_{2,1}^\frac{\dd}{2}}\| \Delta \m^k \|_{L^1_t\BB_{2,1}^\frac{\dd}{2}}
	\Big( 1+ \|  \m^k \|_{L^\infty_t\BB_{2,1}^\frac{\dd}{2}}^2\Big)+
	\|\omega^k\|_{L^\infty_t\BB_{2,1}^\frac{\dd}{2}}{\scriptstyle\times}\\
	{\scriptstyle\times}\| \nabla \Delta \m^k \|_{L^1_t\BB_{2,1}^\frac{\dd}{2}}
	\Big( 1+ \|  \m^k \|_{L^\infty_t\BB_{2,1}^\frac{\dd}{2}}^2\Big)+
	\| \omega^k\|_{L^\infty_t\BB_{2,1}^\frac{\dd}{2}}\| \Delta \m^k \|_{L^1_t\BB_{2,1}^\frac{\dd}{2}}
	 \| \nabla  \m^k \|_{L^\infty_t\BB_{2,1}^\frac{\dd}{2}}{\scriptstyle\times}\\
	{\scriptstyle\times}
	\Big( 1+ \|  \m^k \|_{L^\infty_t\BB_{2,1}^\frac{\dd}{2}}\Big)
	%------------------------
	\| \nabla\omega^k\|_{L^\infty_t\BB_{2,1}^\frac{\dd}{2}}
	\| \nabla \m^k \|_{L^2_t\BB_{2,1}^\frac{\dd}{2}}^2
	\Big( 1+ \|  \m^k \|_{L^\infty_t\BB_{2,1}^\frac{\dd}{2}}\Big)+
	\| \omega^k\|_{L^\infty_t\BB_{2,1}^\frac{\dd}{2}}
	\| \nabla \m^k \|_{L^2_t\BB_{2,1}^\frac{\dd}{2}}{\scriptstyle\times}\\
	{\scriptstyle\times}
	\| \Delta \m^k \|_{L^2_t\BB_{2,1}^\frac{\dd}{2}}
	\Big( 1+ \|  \m^k \|_{L^\infty_t\BB_{2,1}^\frac{\dd}{2}}\Big)+
	\| \omega^k\|_{L^\infty_t\BB_{2,1}^\frac{\dd}{2}}
	\| \nabla \m^k \|_{L^2_t\BB_{2,1}^\frac{\dd}{2}}^2
	\| \nabla \m^k \|_{L^\infty_t\BB_{2,1}^\frac{\dd}{2}}
	\lesssim \ee^4,
\end{align*}
and, similarly, also the following inequality holds
\begin{align*}
	\Big\|\, \Div
								\Big\{
									[\m^k\cdot \nabla \m^k	]\otimes(\bar \nn+\m^k)	&+
									[\,(\bar \nn\,+\,\m^k)\cdot\nabla \m^k\,	]\otimes\m^k+\\&+
									[\m^k\cdot \nabla \m^k]\otimes\m^k
								\Big\}
								\,\Big\|_{L^1_t\BB_{2,1}^\frac{\dd}{2}\cap \BB_{2,1}^{\frac{\dd}{2}+1}}
							\lesssim \ee^4,
\end{align*}
which finally leads to
\begin{equation}\label{ineq:angularmom_h}
	\|\,{\rm h}^k\,\|_{L^1_t\BB_{2,1}^\frac{\dd}{2}\cap \BB_{2,1}^{\frac{\dd}{2}+1}} \lesssim \ee^4.
\end{equation}
We then take into consideration $f^k$ and we begin with estimating
\begin{align*}
	\Big\|\Div\Big\{(\tilde K_2(\omega^k)+\tilde K_4(\omega^k))(\Div\m^k)(\bar \nn+\m^k)\Big\}(\bar \nn+\m^k)
	\Big\|_{L^1_t\BB_{2,1}^\frac{\dd}{2}\cap \BB_{2,1}^{\frac{\dd}{2}+1}}\lesssim
	\|\nabla \omega^k\|_{L^2_t\BB_{2,1}^\frac{\dd}{2}}{\scriptstyle\times}\\
	{\scriptstyle\times}
	\| \nabla \m^k\|_{L^2_t\BB_{2,1}^\frac{\dd}{2}}
	\Big( 1+ \| \m^k \|_{L^\infty_t\BB_{2,1}^\frac{\dd}{2}}^2\Big)+
	\| \omega^k\|_{L^2_t\BB_{2,1}^\frac{\dd}{2}}
	\| \Delta \m^k\|_{L^2_t\BB_{2,1}^\frac{\dd}{2}}
	\Big( 1+ \| \m^k \|_{L^\infty_t\BB_{2,1}^\frac{\dd}{2}}^2\Big)+
	\\
	+
	\| \omega^k\|_{L^2_t\BB_{2,1}^\frac{\dd}{2}}
	\| \nabla \m^k\|_{L^2_t\BB_{2,1}^\frac{\dd}{2}}^2
	\Big( 1+ \| \m^k \|_{L^\infty_t\BB_{2,1}^\frac{\dd}{2}}\Big)+
	%-----------------------------------------------------------
	\|\Delta \omega^k\|_{L^2_t\BB_{2,1}^\frac{\dd}{2}}
	\| \nabla \m^k\|_{L^2_t\BB_{2,1}^\frac{\dd}{2}}
	\Big( 1+ \| \m^k \|_{L^\infty_t\BB_{2,1}^\frac{\dd}{2}}^2\Big)+\\+
	\|\nabla \omega^k\|_{L^2_t\BB_{2,1}^\frac{\dd}{2}}
	\| \Delta \m^k\|_{L^2_t\BB_{2,1}^\frac{\dd}{2}}
	\Big( 1+ \| \m^k \|_{L^\infty_t\BB_{2,1}^\frac{\dd}{2}}^2\Big)+
	\|\nabla \omega^k\|_{L^2_t\BB_{2,1}^\frac{\dd}{2}}
	\| \nabla \m^k\|_{L^2_t\BB_{2,1}^\frac{\dd}{2}}
	\| \nabla \m^k\|_{L^\infty_t\BB_{2,1}^\frac{\dd}{2}}{\scriptstyle\times}\\
	{\scriptstyle\times}
	\Big( 1+ \| \m^k \|_{L^\infty_t\BB_{2,1}^\frac{\dd}{2}}\Big)+
	%-----------------------------------------------------------
	\| \nabla \omega^k\|_{L^2_t\BB_{2,1}^\frac{\dd}{2}}
	\| \Delta \m^k\|_{L^2_t\BB_{2,1}^\frac{\dd}{2}}
	\Big( 1+ \| \m^k \|_{L^\infty_t\BB_{2,1}^\frac{\dd}{2}}^2\Big)+
	\| \omega^k\|_{L^2_t\BB_{2,1}^\frac{\dd}{2}}{\scriptstyle\times}\\
	{\scriptstyle\times}
	\| \nabla \Delta \m^k\|_{L^2_t\BB_{2,1}^\frac{\dd}{2}}
	\Big( 1+ \| \m^k \|_{L^\infty_t\BB_{2,1}^\frac{\dd}{2}}^2\Big)+
	\| \omega^k\|_{L^2_t\BB_{2,1}^\frac{\dd}{2}}
	\| \Delta \m^k\|_{L^2_t\BB_{2,1}^\frac{\dd}{2}}
	\| \nabla \m^k \|_{L^\infty_t\BB_{2,1}^\frac{\dd}{2}}{\scriptstyle\times}\\
	{\scriptstyle\times}
	\Big( 1+ \| \m^k \|_{L^\infty_t\BB_{2,1}^\frac{\dd}{2}}\Big)+
	%-----------------------------------------------------------
	\| \nabla \omega^k\|_{L^2_t\BB_{2,1}^\frac{\dd}{2}}
	\| \nabla \m^k\|_{L^2_t\BB_{2,1}^\frac{\dd}{2}}^2
	\Big( 1+ \| \m^k \|_{L^\infty_t\BB_{2,1}^\frac{\dd}{2}}\Big)+
	\| \omega^k\|_{L^2_t\BB_{2,1}^\frac{\dd}{2}}{\scriptstyle\times}\\
	{\scriptstyle\times}
	\| \nabla \m^k\|_{L^2_t\BB_{2,1}^\frac{\dd}{2}}
	\| \Delta \m^k\|_{L^2_t\BB_{2,1}^\frac{\dd}{2}}
	\Big( 1+ \| \m^k \|_{L^\infty_t\BB_{2,1}^\frac{\dd}{2}}\Big)+
	\| \omega^k\|_{L^2_t\BB_{2,1}^\frac{\dd}{2}}
	\| \nabla \m^k\|_{L^2_t\BB_{2,1}^\frac{\dd}{2}}^2
	\| \nabla \m^k \|_{L^\infty_t\BB_{2,1}^\frac{\dd}{2}}
	\lesssim \ee^4.
\end{align*}
Furthermore
\begin{align*}
	\Big\|\Div\Big\{(\Div\m^k) \m^k\Big\}\m^k+\Div\Big\{(\Div\m^k)\bar \nn\Big\}\m^k+
			\Div\Big\{(\Div\m^k)\m^k\Big\}\bar \nn\bigg\}\Big\|_{L^1_t\BB_{2,1}^\frac{\dd}{2}\cap \BB_{2,1}^{\frac{\dd}{2}+1}}
	\lesssim \\
	\lesssim
	\| \Delta	\m^k \|_{L^1_t\BB_{2,1}^\frac{\dd}{2}}
	\| 			\m^k \|_{L^\infty_t\BB_{2,1}^\frac{\dd}{2}}
	\Big(1+\| 			\m^k \|_{L^\infty_t\BB_{2,1}^\frac{\dd}{2}}\Big)+
	\| \nabla	\m^k \|_{L^2_t\BB_{2,1}^\frac{\dd}{2}}^2
	\Big(1+\| 			\m^k \|_{L^\infty_t\BB_{2,1}^\frac{\dd}{2}}\Big)+\\+
	\| \Delta	\m^k \|_{L^1_t\BB_{2,1}^\frac{\dd}{2}}
	\| 			\m^k \|_{L^\infty_t\BB_{2,1}^\frac{\dd}{2}}+
	%-----------------------------------------------------------
	\| \nabla \Delta	\m^k \|_{L^1_t\BB_{2,1}^\frac{\dd}{2}}
	\| 			\m^k \|_{L^\infty_t\BB_{2,1}^\frac{\dd}{2}}
	\Big(1+\| 			\m^k \|_{L^\infty_t\BB_{2,1}^\frac{\dd}{2}}\Big)+\\+
	\| \Delta	\m^k \|_{L^1_t\BB_{2,1}^\frac{\dd}{2}}
	\| 	\nabla	\m^k \|_{L^\infty_t\BB_{2,1}^\frac{\dd}{2}}
	\Big(1+\| 			\m^k \|_{L^\infty_t\BB_{2,1}^\frac{\dd}{2}}\Big)+
	\|\nabla \Delta	\m^k \|_{L^1_t\BB_{2,1}^\frac{\dd}{2}}{\scriptstyle\times}\\
	{\scriptstyle\times}
	\| 	\nabla	\m^k \|_{L^\infty_t\BB_{2,1}^\frac{\dd}{2}}+
	\| \Delta	\m^k \|_{L^2_t\BB_{2,1}^\frac{\dd}{2}}^2\lesssim \ee^4,
\end{align*}
which finally leads to the following inequality
\begin{equation}\label{ineq:angularmom_f}
	\|\,f^k\,\|_{L^1_t\BB_{2,1}^\frac{\dd}{2}\cap \BB_{2,1}^{\frac{\dd}{2}+1}}\, \lesssim\, \ee^4.
\end{equation}
Thus it remains to control the $\beta^k$-term. We observe that
\begin{align*}
	\Big\|\,K_1(\omega^k)|\nabla \m^k|^2+
	(K_2(\omega^k)+K_4(\omega^k))|\Div\m^k|^2\,\Big\|_{L^1_t\BB_{2,1}^\frac{\dd}{2}\cap \BB_{2,1}^{\frac{\dd}{2}+1}}
	\lesssim
	\Big(1+	\| 			\omega^k \|_{L^\infty_t\BB_{2,1}^\frac{\dd}{2}}\Big){\scriptstyle\times}\\
	{\scriptstyle\times}
	\|\nabla \m^k\|_{L^2_t\BB_{2,1}^\frac{\dd}{2}}^2+
	\| \nabla \omega^k \|_{L^\infty_t\BB_{2,1}^\frac{\dd}{2}}
	\|\nabla \m^k\|_{L^2_t\BB_{2,1}^\frac{\dd}{2}}^2+
	\Big(1+	\| 			\omega^k \|_{L^\infty_t\BB_{2,1}^\frac{\dd}{2}}\Big){\scriptstyle\times}\\
	{\scriptstyle\times}
	\|\nabla \m^k\|_{L^2_t\BB_{2,1}^\frac{\dd}{2}}
	\|\Delta \m^k\|_{L^2_t\BB_{2,1}^\frac{\dd}{2}}\lesssim \ee^4,
\end{align*}
together with
\begin{align*}
	\Big\|
	K_3(\,\omega^k\,)\,|\,(\,\bar \nn\,+\,\m^k)\cdot\nabla \m^k\,|^2
	\Big\|_{L^1_t\BB_{2,1}^\frac{\dd}{2}\cap \BB_{2,1}^{\frac{\dd}{2}+1}}
	\lesssim
	\Big(1+	\| 			\omega^k \|_{L^\infty_t\BB_{2,1}^\frac{\dd}{2}}\Big)
	\Big(1+	\| 			\m^k \|_{L^\infty_t\BB_{2,1}^\frac{\dd}{2}}^2\Big){\scriptstyle\times}\\
	{\scriptstyle\times}
	\|\nabla \m^k\|_{L^2_t\BB_{2,1}^\frac{\dd}{2}}^2+
	\| \nabla \omega^k \|_{L^\infty_t\BB_{2,1}^\frac{\dd}{2}}
	\Big(1+	\| 			\m^k \|_{L^\infty_t\BB_{2,1}^\frac{\dd}{2}}^2\Big)
	\|\nabla \m^k\|_{L^2_t\BB_{2,1}^\frac{\dd}{2}}^2+
	\Big(1+	\| 	(\m^k,\,	\omega^k) \|_{L^\infty_t\BB_{2,1}^\frac{\dd}{2}}\Big){\scriptstyle\times}\\
	{\scriptstyle\times}
	\| \nabla \m^k \|_{L^\infty_t\BB_{2,1}^\frac{\dd}{2}}
	\|\nabla \m^k\|_{L^2_t\BB_{2,1}^\frac{\dd}{2}}^2+
	\Big(1+	\| 			\omega^k \|_{L^\infty_t\BB_{2,1}^\frac{\dd}{2}}\Big)
	\Big(1+	\| 			\m^k \|_{L^\infty_t\BB_{2,1}^\frac{\dd}{2}}^2\Big){\scriptstyle\times}\\
	{\scriptstyle\times}
	\|\nabla \m^k\|_{L^2_t\BB_{2,1}^\frac{\dd}{2}}
	\|\Delta \m^k\|_{L^2_t\BB_{2,1}^\frac{\dd}{2}}\lesssim \ee^4,
\end{align*}
from which we deduce that
\begin{equation}\label{ineq:angularmom_beta}
	\|\,\beta^k\,\|_{L^1_t\BB_{2,1}^\frac{\dd}{2}\cap \BB_{2,1}^{\frac{\dd}{2}+1}} \lesssim \ee^4.
\end{equation}
Replacing inequalities \eqref{ineq:angularmom_h}, \eqref{ineq:angularmom_f} and \eqref{ineq:angularmom_beta} into \eqref{ineq:angularmom_F}, we finally achieve that
\begin{equation*}
	\|\,\m^{k+1}\,\|_{\X_2}\,
	\lesssim \, 
	\ee^4\,+\,\|\,F_2^k\,\|_{L^1_t \BB_{2,1}^\frac{\dd}{2}\cap \BB_{2,1}^{\frac{\dd}{2}+1}}\,\lesssim\,\ee^4
	\quad\quad\Rightarrow\quad\quad
	\|\,\m^{k+1}\,\|_{\X_2}\,\leq \ee^3,
\end{equation*}
assuming a constant $\ee$ small enough. This concludes the proof of the first inequality in \eqref{angular_mom_ind}.
We now take into account the difference between two consecutive director fields, $\delta \m^{k+1}:=\m^{k+2}-\m^{k+1}$, 
which fulfills the following differential equation
\begin{equation*}
	\partial_t \delta\m^{k+1}\,+\,\uu^{k+1}\cdot \nabla \delta \m^{k+1}\,+\,\delta \uu^{k}\cdot \nabla  \m^{k+1}\,-\Bb[\,\delta \m^{k+1}\,]\,=\,\delta F^{k}_2.
\end{equation*}
Thanks to Theorem \ref{thm-superjuve} we first get
\begin{equation*}
	\|\,\delta \m^{k+1}\,\|_{\X_2} \,\lesssim\, 
	\|\,(\,\uu^{k+1}\cdot \nabla \delta \m^{k+1},\,\,\delta \uu^{k+1}\cdot \nabla \m^{k+1},\,
	\,\delta F_2^k) \,\|_{L^1_t\BB_{2,1}^\frac{\dd}{2}\cap \BB_{2,1}^{\frac{\dd}{2}+1}}.
\end{equation*}
We then investigate each term on the right-hand side. We begin with
\begin{align*}
	\|(\,\uu^{k+1}\cdot \nabla \delta \m^{k+1},\,&\delta \uu^{k+1}\cdot \nabla \m^{k+1})\, \|_{L^1_t\BB_{2,1}^\frac{\dd}{2}\cap \BB_{2,1}^{\frac{\dd}{2}+1}}\lesssim\\
	&\lesssim
	\|\uu^{k+1}						\|_{L^2_t\BB_{2,1}^\frac{\dd}{2}}
	\|\nabla \delta \m^{k+1}			\|_{L^2_t\BB_{2,1}^\frac{\dd}{2}}
	+
	\|\delta \uu^{k+1}				\|_{L^2_t\BB_{2,1}^\frac{\dd}{2}}
	\|\nabla  \m^{k+1}			\|_{L^2_t\BB_{2,1}^\frac{\dd}{2}}
	\\
	&\lesssim
	\ee^2 \|\,\delta \m^{k+1}\,\|_{\X_2}\,+\,\ee^3\|\,\delta \uu^{k+1}\,\|_{\X_1} \lesssim 
	\ee^2 \|\, \delta \m^{k+1}\,\|_{\X_2}\,+\,\ee^{k+4},
\end{align*}
which yields
\begin{equation}\label{ineq:ang_mom_delta_m}	
	\|\,\delta \m^{k+1}\,\|_{\X_2} \,\lesssim\, \ee^{k+4}\,+\,\|\,\delta F_2^k \,\|_{L^1_t\BB_{2,1}^\frac{\dd}{2}\cap \BB_{2,1}^{\frac{\dd}{2}+1}}.
\end{equation}
We then analyze $\delta F_2^k$, which can be formulated as follows:
\begin{equation}\label{def:deltaF2k}
\begin{aligned}
	\delta F_2^k=
	 \delta \tilde \gamma_{1}^k\Big[\partial_t\m^{k+1}	+\uu^{k+1} \cdot \nabla \m^{k+1}\Big]
	+
	\tilde \gamma_{1}(\omega^k)\Big[\partial_t \delta \m^k	+ \delta \uu^k \cdot \nabla \m^{k+1}
	+\uu^k \cdot \nabla  \delta \m^k\Big]-\\
	%----------------------------------------------------------------------------
	-\delta\gamma_{1}^k\Omega^{k+1}(\bar{\nn}+\m^{k+1})
	-\gamma_{1}(\omega^k)\delta\Omega^k(\bar{\nn}+\m^{k+1})
	-\gamma_{1}(\omega^k)\Omega^k\delta \m^k
	%----------------------------------------------------------------------------
	+		\delta \gamma_{2}^k\Big[\mathbb{D}^{k+1}[\bar{\nn}+\m^{k+1}]-\\-
						\big([\bar{\nn}+\m^{k+1}]\cdot \mathbb{D}^{k+1}[\bar{\nn}+\m^{k+1}]\big)[\bar{\nn}
						+\m^{k+1}]\Big]																	
	+		\gamma_{2}(\omega^k)\Big[
						\delta \mathbb{D}^{k}[\bar{\nn}+\m^{k+1}]+\mathbb{D}^k\delta \m^k-\\-
						\big([\delta \m^k\cdot \mathbb{D}^{k+1}[\bar{\nn}+\m^{k+1}]\big)[\bar{\nn}
						+\m^{k+1}]-
						\big([\bar{\nn}+\m^k]\cdot \delta \mathbb{D}^k[\bar{\nn}+\m^{k+1}]\big)[\bar{\nn}
						+\m^{k+1}]-\\+
						\big([\bar{\nn}+\m^k]\cdot \mathbb{D}^k\delta \m^k\big)[\bar{\nn}
						+\m^{k+1}]\Big]	-
						\big([\bar{\nn}+\m^k]\cdot \mathbb{D}^k[\bar{\nn}+\m^k]\big)\delta \m^k
						\Big]	-\\-\delta {\rm h}^k + \delta \tilde{\beta}^k(\bar{\nn}+\m^{k+1})+	
						 \tilde{\beta}^k\delta\m^{k}+ \delta f^k.															\\
\end{aligned}
\end{equation}
We first observe that
\begin{align*}
	\| 	\delta \tilde \gamma_{1}^k\Big[\partial_t\m^{k+1}	+\uu^{k+1} \cdot \nabla \m^{k+1}\Big]\|_{L^1_t\BB_{2,1}^\frac{\dd}{2}\cap \BB_{2,1}^{\frac{\dd}{2}+1}}
	\lesssim
	\| \delta \omega^k \|_{L^\infty_t\BB_{2,1}^\frac{\dd}{2}}
	\Big( \| \partial_t\m^k\|_{L^1_t\BB_{2,1}^\frac{\dd}{2}}+\|\uu^{k+1} \|_{L^2_t\BB_{2,1}^\frac{\dd}{2}}{\scriptstyle\times}\\
	{\scriptstyle\times}
	\|\nabla \m^{k+1} \|_{L^2_t\BB_{2,1}^\frac{\dd}{2}}\Big)+
	\| \nabla \delta \omega^k \|_{L^\infty_t\BB_{2,1}^\frac{\dd}{2}}\Big( \| \partial_t\m^k\|_{L^1_t\BB_{2,1}^\frac{\dd}{2}}+\|\uu^{k+1} \|_{L^2_t\BB_{2,1}^\frac{\dd}{2}}
	\|\nabla \m^{k+1} \|_{L^2_t\BB_{2,1}^\frac{\dd}{2}}\Big)+\\+
	\| \delta \omega^k \|_{L^\infty_t\BB_{2,1}^\frac{\dd}{2}}
	\Big( \| \partial_t\nabla \m^k\|_{L^1_t\BB_{2,1}^\frac{\dd}{2}}+
	\|\nabla \uu^{k+1} \|_{L^2_t\BB_{2,1}^\frac{\dd}{2}}
	\|\nabla \m^{k+1} \|_{L^2_t\BB_{2,1}^\frac{\dd}{2}}+
	\|\uu^{k+1} \|_{L^2_t\BB_{2,1}^\frac{\dd}{2}}{\scriptstyle\times}\\
	{\scriptstyle\times}
	\|\Delta \m^{k+1} \|_{L^2_t\BB_{2,1}^\frac{\dd}{2}}\Big)
	\lesssim \ee^{3}\| \delta \omega^k \|_{\X_3} \lesssim
	\ee^{k+4},
\end{align*}
and moreover
\begin{align*}
	\| \tilde \gamma_{1}(\omega^k)\Big[\partial_t \delta \m^k	+ \delta \uu^k \cdot \nabla \m^{k+1}
	+\uu^k \cdot \nabla  \delta \m^k\Big]\|_{L^1_t\BB_{2,1}^\frac{\dd}{2}\cap \BB_{2,1}^{\frac{\dd}{2}+1}}
	\lesssim
	\| \omega^k \|_{L^\infty_t\BB_{2,1}^\frac{\dd}{2}\cap \BB_{2,1}^{\frac{\dd}{2}+1}}{\scriptstyle\times}\\
	{\scriptstyle\times}
	\Big(
	\| \partial_t\delta \m^k \|_{L^1_t\BB_{2,1}^\frac{\dd}{2}\cap \BB_{2,1}^{\frac{\dd}{2}+1}}+ 
	\| \delta \uu^{k} \|_{L^2_t\BB_{2,1}^\frac{\dd}{2}\cap \BB_{2,1}^{\frac{\dd}{2}+1}}
	\|\nabla \m^{k+1} \|_{L^2_t\BB_{2,1}^\frac{\dd}{2}\cap \BB_{2,1}^{\frac{\dd}{2}+1}}+
	\|\uu^{k} \|_{L^2_t\BB_{2,1}^\frac{\dd}{2}\cap \BB_{2,1}^{\frac{\dd}{2}+1}}{\scriptstyle\times}\\
	{\scriptstyle\times}
	\|\nabla \delta \m^{k} \|_{L^2_t\BB_{2,1}^\frac{\dd}{2}\cap \BB_{2,1}^{\frac{\dd}{2}+1}}\Big)\lesssim
	\ee^3\Big( \| \delta \m^k \|_{\X_2}+ \| \delta \uu^k \|_{\X_1}\big)\lesssim \ee^{k+3}.
	%--------------------------------------------------------------
\end{align*}
Then, we gather
\begin{align*}
	\|\delta\gamma_{1}^k\Omega^{k+1}(\bar{\nn}+\m^{k+1})
	+\gamma_{1}(\omega^k)\delta\Omega^k(\bar{\nn}+\m^{k+1})
	+\gamma_{1}(\omega^k)\Omega^k\delta \m^k\|_{L^1_t\BB_{2,1}^\frac{\dd}{2}\cap \BB_{2,1}^{\frac{\dd}{2}+1}}
	\lesssim
	\| \delta \omega^k \|_{L^\infty_t\BB_{2,1}^\frac{\dd}{2}}{\scriptstyle\times}\\
	{\scriptstyle\times}
	\| \nabla \uu^{k+1} \|_{L^1_t\BB_{2,1}^\frac{\dd}{2}}
	(1+\|\m^{k+1}\|_{L^\infty_t\BB_{2,1}^\frac{\dd}{2}} )+
	\| \nabla \delta \omega^k \|_{L^\infty_t\BB_{2,1}^\frac{\dd}{2}}
	\| \nabla \uu^{k+1} \|_{L^1_t\BB_{2,1}^\frac{\dd}{2}}
	(1+\|\m^{k+1}\|_{L^\infty_t\BB_{2,1}^\frac{\dd}{2}} )+\\+
	\| \delta \omega^k \|_{L^\infty_t\BB_{2,1}^\frac{\dd}{2}}
	\| \Delta \uu^{k+1} \|_{L^1_t\BB_{2,1}^\frac{\dd}{2}}
	(1+\|\m^{k+1}\|_{L^\infty_t\BB_{2,1}^\frac{\dd}{2}} )+
	\| \delta \omega^k \|_{L^\infty_t\BB_{2,1}^\frac{\dd}{2}}
	\| \nabla \uu^{k+1} \|_{L^1_t\BB_{2,1}^\frac{\dd}{2}}
	\|\nabla \m^{k+1}\|_{L^\infty_t\BB_{2,1}^\frac{\dd}{2}} +\\+
	%----------------------------------------------------------
	(1+\|\omega^{k+1}\|_{L^\infty_t\BB_{2,1}^\frac{\dd}{2}}+\|\m^{k+1}\|_{L^\infty_t\BB_{2,1}^\frac{\dd}{2}} )
	\| \nabla \delta \uu^{k+1} \|_{L^1_t\BB_{2,1}^\frac{\dd}{2}}+
	\|\nabla \omega^{k+1}\|_{L^\infty_t\BB_{2,1}^\frac{\dd}{2}}(1+\|\m^{k+1}\|_{L^\infty_t\BB_{2,1}^\frac{\dd}{2}} ){\scriptstyle\times}\\
	{\scriptstyle\times}
	\| \nabla \delta \uu^{k+1} \|_{L^1_t\BB_{2,1}^\frac{\dd}{2}}
	(1+\|\omega^{k+1}\|_{L^\infty_t\BB_{2,1}^\frac{\dd}{2}})\|\nabla \m^{k+1}\|_{L^\infty_t\BB_{2,1}^\frac{\dd}{2}} 
	\| \nabla \delta \uu^{k+1} \|_{L^1_t\BB_{2,1}^\frac{\dd}{2}}+
	(1+\|\omega^{k+1}\|_{L^\infty_t\BB_{2,1}^\frac{\dd}{2}}+\\+\|\m^{k+1}\|_{L^\infty_t\BB_{2,1}^\frac{\dd}{2}} )
	\| \Delta \delta \uu^{k+1} \|_{L^1_t\BB_{2,1}^\frac{\dd}{2}}+
	%----------------------------------------------------------
	(1+\|\omega^{k}\|_{L^\infty_t\BB_{2,1}^\frac{\dd}{2}})\|\delta \m^{k}\|_{L^\infty_t\BB_{2,1}^\frac{\dd}{2}} )
	\| \nabla \uu^{k} \|_{L^1_t\BB_{2,1}^\frac{\dd}{2}}+
	\|\nabla \omega^{k}\|_{L^\infty_t\BB_{2,1}^\frac{\dd}{2}}{\scriptstyle\times}\\
	{\scriptstyle\times}\|\delta \m^{k}\|_{L^\infty_t\BB_{2,1}^\frac{\dd}{2}} 
	\| \nabla \uu^{k} \|_{L^1_t\BB_{2,1}^\frac{\dd}{2}}+
	(1+\|\omega^{k}\|_{L^\infty_t\BB_{2,1}^\frac{\dd}{2}})\|\nabla \delta \m^{k}\|_{L^\infty_t\BB_{2,1}^\frac{\dd}{2}}\| \nabla \uu^{k} \|_{L^1_t\BB_{2,1}^\frac{\dd}{2}} +
	(1+\|\omega^{k}\|_{L^\infty_t\BB_{2,1}^\frac{\dd}{2}}){\scriptstyle\times}\\
	{\scriptstyle\times}\|\delta \m^{k}\|_{L^\infty_t\BB_{2,1}^\frac{\dd}{2}}\| \Delta \uu^{k} \|_{L^1_t\BB_{2,1}^\frac{\dd}{2}}
	\lesssim \ee^2\Big(	\| \delta \m^k \|_{\X_2} +\| \delta \omega^k \|_{\X_3} \Big) + \| \nabla \delta \uu^k \|_{L^1_t\BB_{2,1}^\frac{\dd}{2}\cap L^2\BB_{2,1}^{\frac{\dd}{2}+1}} \leq 
	\ee^{k+3},
\end{align*}
where we have assumed $\bar \alpha_4>1/\ee^3$ large enough in order to have
\begin{equation*}
	\| \nabla \delta \uu^k \|_{L^1_t\BB_{2,1}^\frac{\dd}{2}\cap L^2\BB_{2,1}^{\frac{\dd}{2}+1}}
	\leq \frac{1}{\bar \alpha_4 }\| \delta \uu^k \|_{\X_1} \leq \frac{1}{\bar \alpha_4 }\ee^{k}\leq \ee^{k+3}.
\end{equation*}
Similarly, we handle the next term through
\begin{align*}
	\Big\| \delta \gamma_{2}^k\Big[\mathbb{D}^{k+1}[\bar{\nn}+\m^{k+1}]-
						\big([\bar{\nn}+\m^{k+1}]\cdot \mathbb{D}^{k+1}[\bar{\nn}+\m^{k+1}]\big)[\bar{\nn}
						+\m^{k+1}]\Big]																	
	\Big\|_{L^1_t\BB_{2,1}^\frac{\dd}{2}\cap L^2\BB_{2,1}^{\frac{\dd}{2}+1}}\lesssim
	\\
	\lesssim
	\| \delta \omega^k \|_{L^\infty_t\BB_{2,1}^\frac{\dd}{2}}
	\| \nabla \uu^{k+1} \|_{L^1_t\BB_{2,1}^\frac{\dd}{2}}
	\Big( 1+ \|\m^{k+1}\|_{L^\infty_t\BB_{2,1}^\frac{\dd}{2}}^3
	\Big) +
	\| \nabla \delta \omega^k \|_{L^\infty_t\BB_{2,1}^\frac{\dd}{2}}
	\| \nabla \uu^{k+1} \|_{L^1_t\BB_{2,1}^\frac{\dd}{2}}{\scriptstyle\times}\\
	{\scriptstyle\times}
	\Big( 1+ \|\m^{k+1}\|_{L^\infty_t\BB_{2,1}^\frac{\dd}{2}}^3
	\Big)+
	\| \delta \omega^k \|_{L^\infty_t\BB_{2,1}^\frac{\dd}{2}}
	\| \Delta \uu^{k+1} \|_{L^1_t\BB_{2,1}^\frac{\dd}{2}}
	\Big( 1+ \|\m^{k+1}\|_{L^\infty_t\BB_{2,1}^\frac{\dd}{2}}^3
	\Big)+\\
	\| \delta \omega^k \|_{L^\infty_t\BB_{2,1}^\frac{\dd}{2}}
	\| \uu^{k+1} \|_{L^1_t\BB_{2,1}^\frac{\dd}{2}}
	\|\nabla \m^{k+1}\|_{L^\infty_t\BB_{2,1}^\frac{\dd}{2}}
	\Big( 1+ \|\m^{k+1}\|_{L^\infty_t\BB_{2,1}^\frac{\dd}{2}}^2
	\Big)
	\lesssim
	\ee^{k+3},
\end{align*}
together with
\begin{align*}
	\Big\| \gamma_{2}(\omega^k)\Big[
						\delta \mathbb{D}^{k}[\bar{\nn}+\m^{k+1}]+\mathbb{D}^k\delta \m^k\Big]
	\Big\|
	\lesssim
	\Big(1+\|\omega^k\|_{L^\infty_t\BB_{2,1}^\frac{\dd}{2}}\Big)
	\Big(1+\|\m^{k+1}\|_{L^\infty_t\BB_{2,1}^\frac{\dd}{2}}\Big)
	\| \nabla \delta \uu^k \|_{L^1_t\BB_{2,1}^\frac{\dd}{2}} +\\+
	\|\nabla \omega^k\|_{L^\infty_t\BB_{2,1}^\frac{\dd}{2}}
	\Big(1+\|\m^{k+1}\|_{L^\infty_t\BB_{2,1}^\frac{\dd}{2}}\Big)
	\| \nabla \delta \uu^k \|_{L^1_t\BB_{2,1}^\frac{\dd}{2}} + 
	\|\nabla \m^{k+1}\|_{L^\infty_t\BB_{2,1}^\frac{\dd}{2}}
	\Big(1+\|\omega^k\|_{L^\infty_t\BB_{2,1}^\frac{\dd}{2}}\Big){\scriptstyle\times}\\
	{\scriptstyle\times}
	\| \nabla \delta \uu^k \|_{L^1_t\BB_{2,1}^\frac{\dd}{2}}+
	\Big(1+\|\omega^k\|_{L^\infty_t\BB_{2,1}^\frac{\dd}{2}}\Big)
	\Big(1+\|\m^{k+1}\|_{L^\infty_t\BB_{2,1}^\frac{\dd}{2}}\Big)
	\| \Delta \delta \uu^k \|_{L^1_t\BB_{2,1}^\frac{\dd}{2}}+
	%-----------------------------------------------------------
	\Big(1+\|\omega^k\|_{L^\infty_t\BB_{2,1}^\frac{\dd}{2}}\Big){\scriptstyle\times}\\
	{\scriptstyle\times}
	\|\delta \m^k\|_{L^\infty_t\BB_{2,1}^\frac{\dd}{2}}
	\| \nabla \uu^k \|_{L^1_t\BB_{2,1}^\frac{\dd}{2}} +
	\|\nabla \omega^k\|_{L^\infty_t\BB_{2,1}^\frac{\dd}{2}}
	\|\delta \m^k\|_{L^\infty_t\BB_{2,1}^\frac{\dd}{2}}
	\| \nabla \uu^k \|_{L^1_t\BB_{2,1}^\frac{\dd}{2}} + 
	\|\nabla\delta  \m^k\|_{L^\infty_t\BB_{2,1}^\frac{\dd}{2}}{\scriptstyle\times}\\
	{\scriptstyle\times}
	\Big(1+\|\omega^k\|_{L^\infty_t\BB_{2,1}^\frac{\dd}{2}}\Big)
	\| \nabla  \uu^k \|_{L^1_t\BB_{2,1}^\frac{\dd}{2}}+
	\Big(1+\|\omega^k\|_{L^\infty_t\BB_{2,1}^\frac{\dd}{2}}\Big)
	\|\delta\m^k\|_{L^\infty_t\BB_{2,1}^\frac{\dd}{2}}
	\| \Delta \uu^k \|_{L^1_t\BB_{2,1}^\frac{\dd}{2}} \lesssim \ee^{k+3}.
\end{align*}
Similarly, the following inequality is satisfied:
\begin{align*}
	\Big\|\gamma_{2}(\omega^k)\Big[
						\big([\delta \m^k\cdot \mathbb{D}^{k+1}[\bar{\nn}+\m^{k+1}]\big)[\bar{\nn}
						+\m^{k+1}]-
						\big([\bar{\nn}+\m^k]\cdot \delta \mathbb{D}^k[\bar{\nn}+\m^{k+1}]\big)[\bar{\nn}
						+\m^{k+1}]-\\+
						\big([\bar{\nn}+\m^k]\cdot \mathbb{D}^k\delta \m^k\big)[\bar{\nn}
						+\m^{k+1}]\Big]	-
						\big([\bar{\nn}+\m^k]\cdot \mathbb{D}^k[\bar{\nn}+\m^k]\big)\delta \m^k
						\Big]\Big\|_{L^1_t\BB_{2,1}^\frac{\dd}{2}\cap \BB_{2,1}^{\frac{\dd}{2}+1}}\lesssim
						\ee^{k+3}.
\end{align*}
Thus, we finally deduce that the norm of $\delta \m^{k+1}$ fulfills
\begin{equation}\label{ineq:ang_mom_delta_m2}	
	\|\,\delta \m^{k+1}\,\|_{\X_2} \,\lesssim\, \ee^{k+3}\,+\,\|\,(\delta {\rm h}^k,\,\delta\tilde{\beta}^k(\bar{\nn}+\m^{k+1}),\,
	 \,\tilde{\beta}^k \delta \m^k,\,\delta f^k)\|_{L^1_t\BB_{2,1}^\frac{\dd}{2}\cap \BB_{2,1}^{\frac{\dd}{2}+1}}.
\end{equation}
We begin with $\delta {\rm h}^k={\rm h}^{k+1}-{\rm h}^k$. First we observe that 
\begin{align*}
	\Big\|
		\delta K_3^k(\nabla \m^{k+1}\odot \nabla \m^{k+1})(\bar \nn+\m^{k+1})	
	\Big\|_{L^1_t\BB_{2,1}^\frac{\dd}{2}\cap \BB_{2,1}^{\frac{\dd}{2}+1}}
	\lesssim
	\| \delta \omega^k \|_{L^\infty_t\BB_{2,1}^\frac{\dd}{2}}
	\| \nabla \m^{k+1}\|_{L^2_t\BB_{2,1}^\frac{\dd}{2}}^2
	(1+\\+ \|\m^{k+1}\|_{L^\infty_t\BB_{2,1}^\frac{\dd}{2}})+
	\| \nabla \delta \omega^k \|_{L^\infty_t\BB_{2,1}^\frac{\dd}{2}}
	\| \nabla \m^{k+1}\|_{L^2_t\BB_{2,1}^\frac{\dd}{2}}^2
	(1+ \|\m^{k+1}\|_{L^\infty_t\BB_{2,1}^\frac{\dd}{2}})	
	\| \delta \omega^k \|_{L^\infty_t\BB_{2,1}^\frac{\dd}{2}}{\scriptstyle\times}\\
	{\scriptstyle\times}
	\| \nabla \m^{k+1}\|_{L^2_t\BB_{2,1}^\frac{\dd}{2}}
	\| \Delta \m^{k+1}\|_{L^2_t\BB_{2,1}^\frac{\dd}{2}}
	(1+ \|\m^{k+1}\|_{L^\infty_t\BB_{2,1}^\frac{\dd}{2}})+
	\| \delta \omega^k \|_{L^\infty_t\BB_{2,1}^\frac{\dd}{2}}
	\| \nabla \m^{k+1}\|_{L^2_t\BB_{2,1}^\frac{\dd}{2}}^2{\scriptstyle\times}\\
	{\scriptstyle\times}
	\|\nabla \m^{k+1}\|_{L^\infty_t\BB_{2,1}^\frac{\dd}{2}}
	\lesssim \ee^3 \Big( \|\delta \m^k \|_{\X_2}+\|\delta \omega^k \|_{\X_3}\Big)
	\lesssim \ee^{k+3}.
\end{align*}
Moreover
\begin{align*}
	\Big\|
		K_3^k(\nabla \delta  \m^k\odot \nabla \m^{k+1})(\bar \nn+\m^{k+1})+
		K_3^k(\nabla \m^k\odot \nabla \delta \m^k)(\bar \nn+\m^{k+1})+
	\Big\|_{L^1_t\BB_{2,1}^\frac{\dd}{2}\cap \BB_{2,1}^{\frac{\dd}{2}+1}}
	\lesssim\\\lesssim
	(1+\|\omega^{k}\|_{L^\infty_t\BB_{2,1}^\frac{\dd}{2}})
	(1+\|\m^{k+1}\|_{L^\infty_t\BB_{2,1}^\frac{\dd}{2}})
	\| \nabla \delta \m^k \|_{L^2_t\BB_{2,1}^\frac{\dd}{2}}
	\| (\nabla \m^{k+1},\,\nabla \m^{k+1}) \|_{L^2_t\BB_{2,1}^\frac{\dd}{2}}+\\+
	%--------------------------------------------------
	\|\nabla \omega^{k}\|_{L^\infty_t\BB_{2,1}^\frac{\dd}{2}}
	(1+\|\m^{k+1}\|_{L^\infty_t\BB_{2,1}^\frac{\dd}{2}})
	\| \nabla \delta \m^k \|_{L^2_t\BB_{2,1}^\frac{\dd}{2}}
	\| (\nabla \m^{k+1},\,\nabla \m^{k+1}) \|_{L^2_t\BB_{2,1}^\frac{\dd}{2}}+\\+
	%--------------------------------------------------
	(1+\|\omega^{k}\|_{L^\infty_t\BB_{2,1}^\frac{\dd}{2}})
	\|\nabla \m^{k+1}\|_{L^\infty_t\BB_{2,1}^\frac{\dd}{2}})
	\| \nabla \delta \m^k \|_{L^2_t\BB_{2,1}^\frac{\dd}{2}}
	\| (\nabla \m^{k+1},\,\nabla \m^{k+1}) \|_{L^2_t\BB_{2,1}^\frac{\dd}{2}}+\\+
	%--------------------------------------------------
	(1+\|\omega^{k}\|_{L^\infty_t\BB_{2,1}^\frac{\dd}{2}})
	(1+\|\m^{k+1}\|_{L^\infty_t\BB_{2,1}^\frac{\dd}{2}})
	\| \Delta \delta \m^k \|_{L^2_t\BB_{2,1}^\frac{\dd}{2}}
	\| (\nabla \m^{k+1},\,\nabla \m^{k+1}) \|_{L^2_t\BB_{2,1}^\frac{\dd}{2}}+\\+
	%--------------------------------------------------
	(1+\|\omega^{k}\|_{L^\infty_t\BB_{2,1}^\frac{\dd}{2}})
	(1+\|\m^{k+1}\|_{L^\infty_t\BB_{2,1}^\frac{\dd}{2}})
	\| \nabla \delta \m^k \|_{L^2_t\BB_{2,1}^\frac{\dd}{2}}
	\| (\Delta \m^{k+1},\,\Delta \m^{k+1}) \|_{L^2_t\BB_{2,1}^\frac{\dd}{2}}\lesssim
	 \ee^{k+3}
	%--------------------------------------------------
\end{align*}
and
\begin{align*}
	\Big\|
			K_3^k(\nabla \m^k\odot \nabla \delta \m^k)\delta\m^{k}
	\Big\|_{L^1_t\BB_{2,1}^\frac{\dd}{2}\cap \BB_{2,1}^{\frac{\dd}{2}+1}}	
	\lesssim	
	(1+\|\omega^{k}\|_{L^\infty_t\BB_{2,1}^\frac{\dd}{2}})
	\| \nabla \delta \m^k \|_{L^2_t\BB_{2,1}^\frac{\dd}{2}}
	\| \nabla \m^{k} \|_{L^2_t\BB_{2,1}^\frac{\dd}{2}}{\scriptstyle\times}\\
	{\scriptstyle\times}
	\|\delta \m^{k}\|_{L^\infty_t\BB_{2,1}^\frac{\dd}{2}}+
	%--------------------------------------------------
	\|\nabla \omega^{k}\|_{L^\infty_t\BB_{2,1}^\frac{\dd}{2}}
	\| \nabla \delta \m^k \|_{L^2_t\BB_{2,1}^\frac{\dd}{2}}
	\| \nabla \m^{k} \|_{L^2_t\BB_{2,1}^\frac{\dd}{2}}
	\|\delta \m^{k}\|_{L^\infty_t\BB_{2,1}^\frac{\dd}{2}}+
	%--------------------------------------------------
	(1+\|\omega^{k}\|_{L^\infty_t\BB_{2,1}^\frac{\dd}{2}}){\scriptstyle\times}\\
	{\scriptstyle\times}
	\| \Delta \delta \m^k \|_{L^2_t\BB_{2,1}^\frac{\dd}{2}}
	\| \nabla \m^{k} \|_{L^2_t\BB_{2,1}^\frac{\dd}{2}}
	\|\delta \m^{k}\|_{L^\infty_t\BB_{2,1}^\frac{\dd}{2}}+
	%--------------------------------------------------
	(1+\|\omega^{k}\|_{L^\infty_t\BB_{2,1}^\frac{\dd}{2}}
	\| \nabla \delta \m^k \|_{L^2_t\BB_{2,1}^\frac{\dd}{2}}
	\| \Delta \m^{k} \|_{L^2_t\BB_{2,1}^\frac{\dd}{2}}{\scriptstyle\times}\\
	{\scriptstyle\times}
	\|\delta \m^{k}\|_{L^\infty_t\BB_{2,1}^\frac{\dd}{2}}+
	%--------------------------------------------------
	(1+\|\omega^{k}\|_{L^\infty_t\BB_{2,1}^\frac{\dd}{2}}
	\| \nabla \delta \m^k \|_{L^2_t\BB_{2,1}^\frac{\dd}{2}}
	\| \nabla \m^{k} \|_{L^2_t\BB_{2,1}^\frac{\dd}{2}}
	\| \nabla \delta \m^{k}\|_{L^\infty_t\BB_{2,1}^\frac{\dd}{2}}\lesssim
	\ee^{k+3}.
\end{align*}
Furthermore
\begin{align*}
	\Big\|
			\Div
								\Big\{
									\delta \tilde K_1^k \nabla \m^{k+1}+
									\tilde K_1^k \nabla  \delta \m^k+
									\delta \tilde K_4^k\tr \nabla \m^{k+1}+
									 \tilde K_4^k\tr \nabla  \delta\m^k 
	\Big\|_{L^1_t\BB_{2,1}^\frac{\dd}{2}\cap \BB_{2,1}^{\frac{\dd}{2}+1}}	
	\lesssim
	\| \nabla \delta \omega^k \|_{L^2_t\BB_{2,1}^\frac{\dd}{2}}{\scriptstyle\times}\\
	{\scriptstyle\times}
	\| \nabla  \m^{k+1} \|_{L^2_t\BB_{2,1}^\frac{\dd}{2}}+
	\|  \delta \omega^k \|_{L^2_t\BB_{2,1}^\frac{\dd}{2}}
	\| \Delta  \m^{k+1} \|_{L^2_t\BB_{2,1}^\frac{\dd}{2}}+
	\| \nabla \omega^k \|_{L^2_t\BB_{2,1}^\frac{\dd}{2}}
	\|  \nabla \delta \m^k \|_{L^2_t\BB_{2,1}^\frac{\dd}{2}}+
	\|   \omega^k \|_{L^\infty_t\BB_{2,1}^\frac{\dd}{2}}{\scriptstyle\times}\\
	{\scriptstyle\times}
	\| \Delta  \delta \m^{k} \|_{L^1_t\BB_{2,1}^\frac{\dd}{2}}+ 
	%----------------------------------------------------------------------
	\| \Delta \delta \omega^k \|_{L^1_t\BB_{2,1}^\frac{\dd}{2}}
	\| \nabla  \m^{k+1} \|_{L^\infty_t\BB_{2,1}^\frac{\dd}{2}}+
	\| \nabla \delta \omega^k \|_{L^2_t\BB_{2,1}^\frac{\dd}{2}}
	\| \Delta  \m^{k+1} \|_{L^2_t\BB_{2,1}^\frac{\dd}{2}}+\\+
	%----------------------------------------------------------------------
	\|  \nabla \delta \omega^k \|_{L^2_t\BB_{2,1}^\frac{\dd}{2}}
	\| \Delta  \m^{k+1} \|_{L^2_t\BB_{2,1}^\frac{\dd}{2}}+
	\|   \delta \omega^k \|_{L^\infty_t\BB_{2,1}^\frac{\dd}{2}}
	\| \nabla \Delta  \m^{k+1} \|_{L^1_t\BB_{2,1}^\frac{\dd}{2}}+
	%----------------------------------------------------------------------
	\|  \nabla \omega^k \|_{L^2_t\BB_{2,1}^\frac{\dd}{2}}{\scriptstyle\times}\\
	{\scriptstyle\times}
	\| \Delta  \delta \m^{k} \|_{L^2_t\BB_{2,1}^\frac{\dd}{2}}+ 
	\|   \omega^k \|_{L^\infty_t\BB_{2,1}^\frac{\dd}{2}}
	\| \nabla\Delta  \delta \m^{k} \|_{L^1_t\BB_{2,1}^\frac{\dd}{2}}+ 
	%----------------------------------------------------------------------
	\|  \nabla \omega^k \|_{L^2_t\BB_{2,1}^\frac{\dd}{2}}{\scriptstyle\times}\\
	{\scriptstyle\times}
	\| \Delta  \delta \m^{k} \|_{L^2_t\BB_{2,1}^\frac{\dd}{2}}+ 
	\|   \omega^k \|_{L^\infty_t\BB_{2,1}^\frac{\dd}{2}}
	\| \nabla \Delta  \delta \m^{k} \|_{L^1_t\BB_{2,1}^\frac{\dd}{2}}
	\lesssim
	\ee^{k+3}.
\end{align*}
Finally, the following inequality holds
\begin{align*}
	\Big\|
		\Div\,
		\Big\{
			\delta\tilde K_3^k[(\bar \nn+\m^{k+1})\cdot \nabla \m^{k+1}\otimes(\bar \nn+\m^{k+1})+
		\Big\}
	\Big\|_{L^1_t\BB_{2,1}^\frac{\dd}{2}\cap \BB_{2,1}^{\frac{\dd}{2}+1}}
	\lesssim
	%----------------------------------------------------------------------------------
	\| \nabla \omega^k 	\|_{L^2_t\BB_{2,1}^\frac{\dd}{2}}
	\| \nabla \m^{k+1} 	\|_{L^2_t\BB_{2,1}^{\frac{\dd}{2}}}{\scriptstyle\times}\\
	{\scriptstyle\times}
	(1+\|\m^k\|_{L^\infty_t\BB_{2,1}^\frac{\dd}{2}}^2)+
	\| \delta \omega^k 	\|_{L^\infty_t\BB_{2,1}^\frac{\dd}{2}}
	\| \Delta \m^{k+1} 	\|_{L^1_t\BB_{2,1}^{\frac{\dd}{2}}}
	(1+\|\m^k\|_{L^\infty_t\BB_{2,1}^\frac{\dd}{2}}^2)+ 	
	\| \delta \omega^k 	\|_{L^\infty_t\BB_{2,1}^\frac{\dd}{2}}{\scriptstyle\times}\\
	{\scriptstyle\times}
	\| \nabla \m^{k+1} 	\|_{L^2_t\BB_{2,1}^{\frac{\dd}{2}}}
	\|\nabla \m^k\|_{L^2_t\BB_{2,1}^\frac{\dd}{2}}
	(1+\|\m^k\|_{L^\infty_t\BB_{2,1}^\frac{\dd}{2}})+
	%----------------------------------------------------------------------------------
	%----------------------------------------------------------------------------------
	%----------------------------------------------------------------------------------
	\| \nabla \delta \omega^k 	\|_{L^2_t\BB_{2,1}^\frac{\dd}{2}}
	\| \nabla \m^{k+1} 	\|_{L^2_t\BB_{2,1}^{\frac{\dd}{2}}}{\scriptstyle\times}\\
	{\scriptstyle\times}
	(1+\|\m^k\|_{L^\infty_t\BB_{2,1}^\frac{\dd}{2}}^2)+
	\| \nabla \omega^k 	\|_{L^2_t\BB_{2,1}^\frac{\dd}{2}}
	\| \Delta \m^{k+1} 	\|_{L^2_t\BB_{2,1}^{\frac{\dd}{2}}}
	(1+\|\m^k\|_{L^\infty_t\BB_{2,1}^\frac{\dd}{2}}^2)+
	\| \nabla \omega^k 	\|_{L^2_t\BB_{2,1}^\frac{\dd}{2}}{\scriptstyle\times}\\
	{\scriptstyle\times}
	\| \nabla \m^{k+1} 	\|_{L^2_t\BB_{2,1}^{\frac{\dd}{2}}}
	\| \nabla \m^k 	\|_{L^\infty_t\BB_{2,1}^\frac{\dd}{2}}
	(1+\|\m^k\|_{L^\infty_t\BB_{2,1}^\frac{\dd}{2}})+
	%----------------------------------------------------------------------------------
	\| \nabla \delta \omega^k 	\|_{L^\infty_t\BB_{2,1}^\frac{\dd}{2}}
	\| \Delta \m^{k+1} 	\|_{L^1_t\BB_{2,1}^{\frac{\dd}{2}}}{\scriptstyle\times}\\
	{\scriptstyle\times}
	(1+\|\m^k\|_{L^\infty_t\BB_{2,1}^\frac{\dd}{2}}^2)+ 	
	\| \nabla \delta \omega^k 	\|_{L^\infty_t\BB_{2,1}^\frac{\dd}{2}}
	\| \nabla \Delta \m^{k+1} 	\|_{L^1_t\BB_{2,1}^{\frac{\dd}{2}}}
	(1+\|\m^k\|_{L^\infty_t\BB_{2,1}^\frac{\dd}{2}}^2)+
	\| \nabla \delta \omega^k 	\|_{L^\infty_t\BB_{2,1}^\frac{\dd}{2}}{\scriptstyle\times}\\
	{\scriptstyle\times}
	\| \Delta \m^{k+1} 	\|_{L^1_t\BB_{2,1}^{\frac{\dd}{2}}}
	\|\nabla  \m^k\|_{L^\infty_t\BB_{2,1}^\frac{\dd}{2}}
	(1+\|\m^k\|_{L^\infty_t\BB_{2,1}^\frac{\dd}{2}})+
	%----------------------------------------------------------------------------------
	\|\nabla  \delta \omega^k 	\|_{L^\infty_t\BB_{2,1}^\frac{\dd}{2}}
	\| \nabla \m^{k+1} 	\|_{L^2_t\BB_{2,1}^{\frac{\dd}{2}}}
	\|\nabla \m^k\|_{L^2_t\BB_{2,1}^\frac{\dd}{2}}{\scriptstyle\times}\\
	{\scriptstyle\times}
	(1+\|\m^k\|_{L^\infty_t\BB_{2,1}^\frac{\dd}{2}})+
	\| \delta \omega^k 	\|_{L^\infty_t\BB_{2,1}^\frac{\dd}{2}}
	\| \Delta \m^{k+1} 	\|_{L^2_t\BB_{2,1}^{\frac{\dd}{2}}}
	\|\nabla \m^k\|_{L^2_t\BB_{2,1}^\frac{\dd}{2}}
	(1+\|\m^k\|_{L^\infty_t\BB_{2,1}^\frac{\dd}{2}})+\\+
	\|\nabla  \delta \omega^k 	\|_{L^\infty_t\BB_{2,1}^\frac{\dd}{2}}
	\| \nabla \m^{k+1} 	\|_{L^2_t\BB_{2,1}^{\frac{\dd}{2}}}
	\|\Delta \m^k\|_{L^2_t\BB_{2,1}^\frac{\dd}{2}}{\scriptstyle\times}\
	(1+\|\m^k\|_{L^\infty_t\BB_{2,1}^\frac{\dd}{2}})+\\+
	\|\nabla  \delta \omega^k 	\|_{L^\infty_t\BB_{2,1}^\frac{\dd}{2}}
	\| \nabla \m^{k+1} 	\|_{L^2_t\BB_{2,1}^{\frac{\dd}{2}}}
	\|\Delta \m^k\|_{L^2_t\BB_{2,1}^\frac{\dd}{2}}
	\|\nabla \m^k\|_{L^\infty_t\BB_{2,1}^\frac{\dd}{2}}
	\lesssim \ee^{k+3},
\end{align*}
and with a similar approach also the following estimate is fulfilled
\begin{align*}
	\Big\|&
		\Div\,
		\Big\{
			\tilde K_3^k[\delta\m^k\cdot \nabla \m^{k+1}\otimes(\bar \nn+\m^{k+1})
			\tilde K_3^k[\delta\m^k\cdot \nabla \m^{k+1}\otimes(\bar \nn+\m^{k+1})+\\&+
			\tilde K_3^k[(\bar \nn+\m^k)\cdot \nabla\delta \m^k\otimes(\bar \nn+\m^{k+1})+
			\tilde K_3^k[(\bar \nn+\m^k)\cdot \nabla \m^k\otimes\delta\m^k
		\Big\}
	\Big\|_{L^1_t\BB_{2,1}^\frac{\dd}{2}\cap \BB_{2,1}^{\frac{\dd}{2}+1}}\lesssim \ee^{k+3}.
\end{align*}
Summarizing all the previous considerations leads to a bound for the molecular field ${\rm h^k}$:
\begin{equation} \label{ineq:ang_mom_delta_m2-h}
	\|\,\delta {\rm h}^k\|_{L^1_t\BB_{2,1}^\frac{\dd}{2}\cap \BB_{2,1}^{\frac{\dd}{2}+1}} \lesssim \ee^{k+3}.
\end{equation}
Now, we take into consideration $\delta \tilde \beta^k (\nn +\m^{k+1}) =(\tilde \beta^{k+1}-\tilde \beta^k) (\nn +\m^{k+1}) $ in \eqref{ineq:ang_mom_delta_m2}. We begin with analyzing the term $\delta \tilde \beta^k$, first by
\begin{align*}
	\big\|\,\Big(\delta 	K_1^k	|\nabla 			\m^{k+1}				|^2+
				K_1^k	\nabla \delta 	\m^k :\nabla 		\m^{k+1}+
				K_1^k	\nabla 			\m^k :\nabla \delta \m^k+
				(\delta  K_2^k+\delta K_4^k)|\Div\m^{k+1}|^2+ \\
				(  K_2^k+ K_4^k)\Div \delta \m^{k}\Div \m^{k+1}+ 
				(  K_2^k+ K_4^k)\Div  \m^{k}\Div \delta\m^k
			\big\|_{L^1_t\BB_{2,1}^\frac{\dd}{2}\cap \BB_{2,1}^{\frac{\dd}{2}+1}}+				
			\lesssim
	%--------------------------------------------------------------------------------------
	%--------------------------------------------------------------------------------------			
		\|		\delta 	\omega^k		\|_{L^\infty_t\BB_{2,1}^\frac{\dd}{2}}{\scriptstyle\times}\\
		{\scriptstyle\times}
		\|	\nabla \m^{k+1}\|_{L^2_t\BB_{2,1}^\frac{\dd}{2}}^2+
	%--------------------------------------------------------------------------------------		
		\|	\nabla 	\delta 	\omega^k		\|_{L^\infty_t\BB_{2,1}^\frac{\dd}{2}}
		\|	\nabla \m^{k+1}\|_{L^2_t\BB_{2,1}^\frac{\dd}{2}}^2+
		\|	\nabla 	\delta 	\omega^k		\|_{L^\infty_t\BB_{2,1}^\frac{\dd}{2}}
		\|	\Delta \m^{k+1}\|_{L^2_t\BB_{2,1}^\frac{\dd}{2}}{\scriptstyle\times}\\
		{\scriptstyle\times}
		\|	\nabla \m^{k+1}\|_{L^2_t\BB_{2,1}^\frac{\dd}{2}}+
	%--------------------------------------------------------------------------------------
	%--------------------------------------------------------------------------------------			
	(1+\|\omega^k\|_{L^\infty_t\BB_{2,1}^\frac{\dd}{2}})
	\|	\nabla 	\delta 	\m^k		\|_{L^2_t\BB_{2,1}^\frac{\dd}{2}}
	\|	(\nabla 	\m^k,\,\nabla \m^{k+1})		\|_{L^\infty_t\BB_{2,1}^\frac{\dd}{2}}+\\+
	\|\nabla \omega^k\|_{L^\infty_t\BB_{2,1}^\frac{\dd}{2}}
	\|	\nabla 	\delta 	\m^k		\|_{L^2_t\BB_{2,1}^\frac{\dd}{2}}
	\|	(\nabla 	\m^k,\,\nabla \m^{k+1})		\|_{L^\infty_t\BB_{2,1}^\frac{\dd}{2}} + 
	\| 	 \omega^k\|_{L^\infty_t\BB_{2,1}^\frac{\dd}{2}}
	\|	\Delta	\delta 	\m^k		\|_{L^2_t\BB_{2,1}^\frac{\dd}{2}}{\scriptstyle\times}\\
		{\scriptstyle\times}
	\|	(\nabla 	\m^k,\,\nabla \m^{k+1})		\|_{L^\infty_t\BB_{2,1}^\frac{\dd}{2}}+
	\| 	  \omega^k\|_{L^\infty_t\BB_{2,1}^\frac{\dd}{2}}
	\|	\nabla	\delta 	\m^k		\|_{L^2_t\BB_{2,1}^\frac{\dd}{2}}
	\|	(\Delta 	\m^k,\,\Delta \m^{k+1})		\|_{L^\infty_t\BB_{2,1}^\frac{\dd}{2}}
	\lesssim \ee^{k+3},
\end{align*}
then, denoting by
\begin{align*}
	\I&:=
		\delta 	K_3^k|[\bar \nn+	\m^{k+1}]\cdot\nabla \m^{k+1}|^2+													%	A	
				K_3^k [	\delta	\m^k \cdot\nabla \m^{k+1}]:[(\bar \nn+	\m^{k+1})\cdot\nabla \m^{k+1}]+\\&+			%	B
				K_3^k [	(\bar \nn+		\m^k )\cdot\nabla \delta \m^k]:[(\bar \nn+	\m^{k+1})\cdot\nabla \m^{k+1}]+	%	C
				K_3^k [	(\bar \nn+		\m^k )\cdot\nabla \m^k]: [\delta	\m^k\cdot\nabla \m^{k+1}]+\\&\hspace{7cm}+	%	D
				K_3^k [	(\bar \nn+		\m^k )\cdot\nabla \delta \m^k):[\bar \nn+	\m^k)\cdot\nabla \delta \m^k],	%	E
\end{align*}
we gather that
\begin{align*}
	\|\,\I\,\|_{L^1_t\BB_{2,1}^\frac{\dd}{2}\cap \BB_{2,1}^{\frac{\dd}{2}+1}} \lesssim
	%------------------------------------------------------------------------ A
	\| \delta \omega^k\|_{L^\infty_t\BB_{2,1}^\frac{\dd}{2}}
	(1+\|\m^{k+1}\|_{L^\infty_t\BB_{2,1}^\frac{\dd}{2}}^2)
	\|\nabla \m^{k+1}\|_{L^2_t\BB_{2,1}^\frac{\dd}{2}}^2+
	%------------------------------------------------------------------------ B & D
	(1+\| \omega^k\|_{L^\infty_t\BB_{2,1}^\frac{\dd}{2}}){\scriptstyle\times}\\
	{\scriptstyle\times}
	\|\delta \m^{k}\|_{L^\infty_t\BB_{2,1}^\frac{\dd}{2}}
	(1+\|\m^{k+1}\|_{L^\infty_t\BB_{2,1}^\frac{\dd}{2}})
	\|(\nabla \m^{k},\,\nabla \m^{k+1})\|_{L^2_t\BB_{2,1}^\frac{\dd}{2}}^2+	
	%------------------------------------------------------------------------ C & E
	(1+\| \omega^k\|_{L^\infty_t\BB_{2,1}^\frac{\dd}{2}}){\scriptstyle\times}\\
	{\scriptstyle\times}
	(1+\| (\m^k,\,\m^{k+1})\|_{L^\infty_t\BB_{2,1}^\frac{\dd}{2}})
	\|\nabla \delta \m^{k}	\|_{L^2_t\BB_{2,1}^\frac{\dd}{2}}
	\|(\nabla \m^{k},\,\nabla \m^{k+1})\|_{L^2_t\BB_{2,1}^\frac{\dd}{2}}+\\
	%------------------------------------------------------------------------
	%------------------------------------------------------------------------
	%--------------------------A			------------------------------------
	%------------------------------------------------------------------------
	%------------------------------------------------------------------------
	+\|\nabla \delta \omega^k\|_{L^\infty_t\BB_{2,1}^\frac{\dd}{2}}
	(1+\|\m^{k+1}\|_{L^\infty_t\BB_{2,1}^\frac{\dd}{2}}^2)
	\|\nabla \m^{k+1}\|_{L^2_t\BB_{2,1}^\frac{\dd}{2}}^2+
	+\| \delta \omega^k\|_{L^\infty_t\BB_{2,1}^\frac{\dd}{2}}
	\|\nabla \m^{k+1}\|_{L^\infty_t\BB_{2,1}^\frac{\dd}{2}}{\scriptstyle\times}\\
	{\scriptstyle\times}
	(1+\|\m^{k+1}\|_{L^\infty_t\BB_{2,1}^\frac{\dd}{2}})
	\|\nabla \m^{k+1}\|_{L^2_t\BB_{2,1}^\frac{\dd}{2}}^2+
	\|\nabla \delta \omega^k\|_{L^\infty_t\BB_{2,1}^\frac{\dd}{2}}
	(1+\|\m^{k+1}\|_{L^\infty_t\BB_{2,1}^\frac{\dd}{2}}^2){\scriptstyle\times}\\
	{\scriptstyle\times}
	\|\nabla \m^{k+1}\|_{L^2_t\BB_{2,1}^\frac{\dd}{2}}
	\|\Delta \m^{k+1}\|_{L^2_t\BB_{2,1}^\frac{\dd}{2}}+
	%------------------------------------------------------------------------
	%------------------------------------------------------------------------
	%--------------------------B&D		------------------------------------
	%------------------------------------------------------------------------
	%------------------------------------------------------------------------
	\|\nabla \omega^k\|_{L^\infty_t\BB_{2,1}^\frac{\dd}{2}}
	\|\delta \m^{k}\|_{L^\infty_t\BB_{2,1}^\frac{\dd}{2}}
	(1+\|\m^{k+1}\|_{L^\infty_t\BB_{2,1}^\frac{\dd}{2}}){\scriptstyle\times}\\
	{\scriptstyle\times}
	\|(\nabla \m^{k},\,\nabla \m^{k+1})\|_{L^2_t\BB_{2,1}^\frac{\dd}{2}}^2+	
	%----------------------------------------------------------------------
	(1+\| \omega^k\|_{L^\infty_t\BB_{2,1}^\frac{\dd}{2}})
	\|\nabla \delta \m^{k}\|_{L^\infty_t\BB_{2,1}^\frac{\dd}{2}}
	(1+\|\m^{k+1}\|_{L^\infty_t\BB_{2,1}^\frac{\dd}{2}}){\scriptstyle\times}\\
	{\scriptstyle\times}
	\|(\nabla \m^{k},\,\nabla \m^{k+1})\|_{L^2_t\BB_{2,1}^\frac{\dd}{2}}^2+
	%----------------------------------------------------------------------
	(1+\| \omega^k\|_{L^\infty_t\BB_{2,1}^\frac{\dd}{2}})
	\|\delta \m^{k}\|_{L^\infty_t\BB_{2,1}^\frac{\dd}{2}}
	\|\nabla \m^{k+1}\|_{L^\infty_t\BB_{2,1}^\frac{\dd}{2}}{\scriptstyle\times}\\
	{\scriptstyle\times}
	\|(\nabla \m^{k},\,\nabla \m^{k+1})\|_{L^2_t\BB_{2,1}^\frac{\dd}{2}}^2+	
	%----------------------------------------------------------------------
	(1+\| \omega^k\|_{L^\infty_t\BB_{2,1}^\frac{\dd}{2}})
	\|\delta \m^{k}\|_{L^\infty_t\BB_{2,1}^\frac{\dd}{2}}
	(1+\|\m^{k+1}\|_{L^\infty_t\BB_{2,1}^\frac{\dd}{2}}){\scriptstyle\times}\\
	{\scriptstyle\times}
	\|(\Delta \m^{k},\,\Delta \m^{k+1})\|_{L^2_t\BB_{2,1}^\frac{\dd}{2}}
	\|(\nabla \m^{k},\,\nabla \m^{k+1})\|_{L^2_t\BB_{2,1}^\frac{\dd}{2}}+
	%------------------------------------------------------------------------
	%------------------------------------------------------------------------
	%--------------------------C&E		------------------------------------
	%------------------------------------------------------------------------
	%------------------------------------------------------------------------
	\|\nabla \omega^k\|_{L^\infty_t\BB_{2,1}^\frac{\dd}{2}}{\scriptstyle\times}\\
	{\scriptstyle\times}
	(1+\| (\m^k,\,\m^{k+1})\|_{L^\infty_t\BB_{2,1}^\frac{\dd}{2}})
	\|\nabla \delta \m^{k}	\|_{L^2_t\BB_{2,1}^\frac{\dd}{2}}
	\|(\nabla \m^{k},\,\nabla \m^{k+1})\|_{L^2_t\BB_{2,1}^\frac{\dd}{2}}+
	%------------------------------------------------------------------------
	(1+\| \omega^k\|_{L^\infty_t\BB_{2,1}^\frac{\dd}{2}}){\scriptstyle\times}\\
	{\scriptstyle\times}
	\| (\nabla \m^k,\,\nabla \m^{k+1})\|_{L^\infty_t\BB_{2,1}^\frac{\dd}{2}}
	\|\nabla \delta \m^{k}	\|_{L^2_t\BB_{2,1}^\frac{\dd}{2}}
	\|(\nabla \m^{k},\,\nabla \m^{k+1})\|_{L^2_t\BB_{2,1}^\frac{\dd}{2}}+
	%------------------------------------------------------------------------
	(1+\| \omega^k\|_{L^\infty_t\BB_{2,1}^\frac{\dd}{2}}){\scriptstyle\times}\\
	{\scriptstyle\times}
	(1+\| (\m^k,\,\m^{k+1})\|_{L^\infty_t\BB_{2,1}^\frac{\dd}{2}})
	\|\Delta \delta \m^{k}	\|_{L^2_t\BB_{2,1}^\frac{\dd}{2}}
	\|(\nabla \m^{k},\,\nabla \m^{k+1})\|_{L^2_t\BB_{2,1}^\frac{\dd}{2}}+
	%------------------------------------------------------------------------
	(1+\| \omega^k\|_{L^\infty_t\BB_{2,1}^\frac{\dd}{2}}){\scriptstyle\times}\\
	{\scriptstyle\times}
	(1+\| (\m^k,\,\m^{k+1})\|_{L^\infty_t\BB_{2,1}^\frac{\dd}{2}})
	\|\nabla \delta \m^{k}	\|_{L^2_t\BB_{2,1}^\frac{\dd}{2}}
	\|(\Delta \m^{k},\,\Delta \m^{k+1})\|_{L^2_t\BB_{2,1}^\frac{\dd}{2}}
	\lesssim \ee^{k+3}.
\end{align*}
This finally leads to the following bound for $\delta \tilde \beta^k(\bar\nn + \m^{k+1})$:
\begin{equation}\label{ineq:ang_mom_delta_m2-beta1}
	\| \delta \tilde \beta^k(\bar\nn + \m^{k+1}) \|_{L^1_t\BB_{2,1}^\frac{\dd}{2}\cap \BB_{2,1}^{\frac{\dd}{2}+1}}\lesssim
	\|	\delta \tilde \beta^k \|_{L^1_t\BB_{2,1}^\frac{\dd}{2}\cap\BB_{2,1}^{\frac{\dd}{2}+1}}
	(1+\|\m^{k+1}\|_{L^\infty_t\BB_{2,1}^\frac{\dd}{2}\cap \BB_{2,1}^{\frac{\dd}{2}+1}})
	\lesssim \ee^{k+3}.
\end{equation}
Similarly, we observe that
\begin{equation}\label{ineq:ang_mom_delta_m2-beta2}
\begin{aligned}
	\| \tilde \beta^k \delta\m^k \|_{L^1_t\BB_{2,1}^\frac{\dd}{2}\cap \BB_{2,1}^{\frac{\dd}{2}+1}}
	\lesssim
	\| \tilde \beta^k  \|_{L^1_t\BB_{2,1}^\frac{\dd}{2}}
	\| \delta\m^k 	\|_{L^\infty_t\BB_{2,1}^\frac{\dd}{2}}+
	\| \nabla \tilde \beta^k  \|_{L^1_t\BB_{2,1}^\frac{\dd}{2}}
	\| \delta\m^k 	\|_{L^\infty_t\BB_{2,1}^\frac{\dd}{2}}+\\+
	\| \tilde \beta^k  \|_{L^1_t\BB_{2,1}^\frac{\dd}{2}}
	\| \nabla \delta\m^k 	\|_{L^\infty_t\BB_{2,1}^\frac{\dd}{2}}
	\lesssim 
	\ee^2 \| \delta \m^k \|_{\X_2}\lesssim \ee^{k+3}.
\end{aligned}
\end{equation}
It then remains to control the term $\delta f^k = f^{k+1}-f^k$. We begin with
\begin{align*}
	\I\I:=
	\Div\Big\{(\delta \tilde K_2^k	+\delta 	\tilde K_4^k)(\Div\m^{k+1})(\bar \nn+\m^{k+1})\Big\}(\bar \nn+\m^{k+1})+
	\Div\Big\{(	\tilde K_2^k	+	\tilde K_4^k)(\Div\delta \m^k)
	{\scriptstyle\times}\\{\scriptstyle\times}(\bar \nn+\m^{k+1})\Big\}(\bar \nn+\m^{k+1})+
	\Div\Big\{(	\tilde K_2^k	+	\tilde K_4^k)(\Div\m^k)\delta \m^k\Big\}(\bar \nn+\m^{k+1})+\\+
	\Div\Big\{(	\tilde K_2^k	+	\tilde K_4^k)(\Div\m^k)(\bar \nn+\m^k)\Big\}\delta \m^k.
\end{align*}
We observe that
\begin{align*}
	\| \I\I \|_{L^1_t\BB_{2,1}^\frac{\dd}{2}\cap \BB_{2,1}^{\frac{\dd}{2}+1}}
	\lesssim
	%----------------------------------------------------	
	%---------------------------------------------------- A
	\| \nabla \delta \omega^k 	\|_{L^2_t\BB_{2,1}^\frac{\dd}{2}}
	\| \nabla \m^{k+1} 			\|_{L^2_t\BB_{2,1}^\frac{\dd}{2}}
	(1+\|\m^k					\|_{L^\infty_t\BB_{2,1}^\frac{\dd}{2}}^2)+
	%----------------------------------------------------B
	\|  \delta \omega^k 			\|_{L^2_t\BB_{2,1}^\frac{\dd}{2}}{\scriptstyle\times}\\
	{\scriptstyle\times}
	\| \Delta \m^{k+1} 			\|_{L^2_t\BB_{2,1}^\frac{\dd}{2}}
	(1+\|\m^k					\|_{L^\infty_t\BB_{2,1}^\frac{\dd}{2}}^2)+
	%----------------------------------------------------C
	\|  \delta \omega^k 			\|_{L^2_t\BB_{2,1}^\frac{\dd}{2}}
	\| \nabla  \m^{k+1} 			\|_{L^2_t\BB_{2,1}^\frac{\dd}{2}}
	\| \nabla \m^k 				\|_{L^\infty_t\BB_{2,1}^\frac{\dd}{2}}{\scriptstyle\times}\\
	{\scriptstyle\times}
	(1+\|\m^k\|_{L^\infty_t\BB_{2,1}^\frac{\dd}{2}})+
	%----------------------------------------------------
	%----------------------------------------------------
	%---------------------------------------------------- A+1
	\| \Delta \delta \omega^k 	\|_{L^2_t\BB_{2,1}^\frac{\dd}{2}}
	\| \nabla \m^{k+1} 			\|_{L^2_t\BB_{2,1}^\frac{\dd}{2}}
	(1+\|\m^k					\|_{L^\infty_t\BB_{2,1}^\frac{\dd}{2}}^2)+
	%----------------------------------------------------
	\| \nabla \delta \omega^k 	\|_{L^2_t\BB_{2,1}^\frac{\dd}{2}}{\scriptstyle\times}\\
	{\scriptstyle\times}
	\| \Delta \m^{k+1} 			\|_{L^2_t\BB_{2,1}^\frac{\dd}{2}}
	(1+\|\m^k					\|_{L^\infty_t\BB_{2,1}^\frac{\dd}{2}}^2)+
	%----------------------------------------------------
	\|  \nabla \delta \omega^k 	\|_{L^2_t\BB_{2,1}^\frac{\dd}{2}}
	\|  \nabla \m^{k+1} 			\|_{L^2_t\BB_{2,1}^\frac{\dd}{2}}
	\|  \nabla \m^k				\|_{L^\infty_t\BB_{2,1}^\frac{\dd}{2}}{\scriptstyle\times}\\
	{\scriptstyle\times}
	(1+\|\m^k					\|_{L^\infty_t\BB_{2,1}^\frac{\dd}{2}})+
	%----------------------------------------------------
	%----------------------------------------------------
	%---------------------------------------------------- B
	\| \nabla \omega^k 			\|_{L^2_t\BB_{2,1}^\frac{\dd}{2}}
	\| \nabla \delta \m^k 		\|_{L^2_t\BB_{2,1}^\frac{\dd}{2}}
	(1+\|\m^{k+1} \|_{L^\infty_t\BB_{2,1}^\frac{\dd}{2}}^2)+
	%---------------------------------------------------- 
	\| 		 \omega^k 			\|_{L^\infty_t\BB_{2,1}^\frac{\dd}{2}}{\scriptstyle\times}\\
	{\scriptstyle\times}
	\| \Delta \delta \m^k 		\|_{L^1_t\BB_{2,1}^\frac{\dd}{2}}
	(1+\|\m^{k+1} \|_{L^\infty_t\BB_{2,1}^\frac{\dd}{2}}^2)+
	%---------------------------------------------------- 
	\| 		 \omega^k 			\|_{L^\infty_t\BB_{2,1}^\frac{\dd}{2}}
	\| \nabla \delta \m^k 		\|_{L^2_t\BB_{2,1}^\frac{\dd}{2}}
	\| \nabla 	\m^{k+1} 		\|_{L^2_t\BB_{2,1}^\frac{\dd}{2}}{\scriptstyle\times}\\
	{\scriptstyle\times}
	(1+\|\m^{k+1} 				\|_{L^\infty_t\BB_{2,1}^\frac{\dd}{2}})+
	%----------------------------------------------------
	%----------------------------------------------------
	%---------------------------------------------------- B+1
	\| \nabla \omega^k 			\|_{L^2_t\BB_{2,1}^\frac{\dd}{2}}
	\| \nabla \delta \m^k 		\|_{L^2_t\BB_{2,1}^\frac{\dd}{2}}
	(1+\|\m^{k+1} \|_{L^\infty_t\BB_{2,1}^\frac{\dd}{2}}^2)+
	%---------------------------------------------------- 
	\| 		 \omega^k 			\|_{L^\infty_t\BB_{2,1}^\frac{\dd}{2}}{\scriptstyle\times}\\
	{\scriptstyle\times}
	\| \Delta \delta \m^k 		\|_{L^1_t\BB_{2,1}^\frac{\dd}{2}}
	(1+\|\m^{k+1} \|_{L^\infty_t\BB_{2,1}^\frac{\dd}{2}}^2)+
	%---------------------------------------------------- 
	\| 		 \omega^k 			\|_{L^\infty_t\BB_{2,1}^\frac{\dd}{2}}
	\| \nabla \delta \m^k 		\|_{L^2_t\BB_{2,1}^\frac{\dd}{2}}
	\| \nabla 	\m^{k+1} 		\|_{L^2_t\BB_{2,1}^\frac{\dd}{2}}{\scriptstyle\times}\\
	{\scriptstyle\times}
	(1+\|\m^{k+1} 				\|_{L^\infty_t\BB_{2,1}^\frac{\dd}{2}})+
	%----------------------------------------------------
	%----------------------------------------------------
	%----------------------------------------------------C
	\| \nabla \delta \omega^k 			\|_{L^2_t\BB_{2,1}^\frac{\dd}{2}}
	\| \nabla  \m^{k+1} 			\|_{L^2_t\BB_{2,1}^\frac{\dd}{2}}
	\| \nabla \m^k 				\|_{L^\infty_t\BB_{2,1}^\frac{\dd}{2}}
	(1+\|\m^k\|_{L^\infty_t\BB_{2,1}^\frac{\dd}{2}})+	\\+
	%----------------------------------------------------
	\|  \delta \omega^k 			\|_{L^2_t\BB_{2,1}^\frac{\dd}{2}}
	\| \Delta  \m^{k+1} 			\|_{L^2_t\BB_{2,1}^\frac{\dd}{2}}
	\| \nabla \m^k 				\|_{L^\infty_t\BB_{2,1}^\frac{\dd}{2}}
	(1+\|\m^k\|_{L^\infty_t\BB_{2,1}^\frac{\dd}{2}})+
	%----------------------------------------------------
	\|  \delta \omega^k 			\|_{L^2_t\BB_{2,1}^\frac{\dd}{2}}{\scriptstyle\times}\\
	{\scriptstyle\times}
	\| \nabla  \m^{k+1} 			\|_{L^2_t\BB_{2,1}^\frac{\dd}{2}}
	\| \Delta \m^k 				\|_{L^\infty_t\BB_{2,1}^\frac{\dd}{2}}
	(1+\|\m^k\|_{L^\infty_t\BB_{2,1}^\frac{\dd}{2}})\lesssim \ee^{k+3},
\end{align*}
Furthermore, defining
\begin{align*}
	\J = 	\Div\Big\{(\Div\,\delta \m^{k})\m^{k+1}\Big\}\m^{k+1}+
			\Div\Big\{(\Div\, \m^{k})\delta\m^{k}\Big\}\m^{k+1}+\\+
			\Div\Big\{(\Div\, \m^{k})\m^{k}\Big\}\delta\m^{k}	
	%-----------------------------------------------------------------		
			+\Div\Big\{(\Div\delta \m^{k})\bar \nn\Big\}\m^{k+1}
			+\Div\Big\{(\Div\m^{k})\bar \nn\Big\}\delta \m^{k}+\\
	%-----------------------------------------------------------------		
			+\Div\Big\{(\Div\delta \m^{k})\m^{k+1}\Big\}\bar \nn
			+\Div\Big\{(\Div\m^{k})\delta \m^{k}\Big\}\bar \nn,
\end{align*}
we achieve the following inequality
\begin{align*}
	\| \J \|_{L^1_t\BB_{2,1}^\frac{\dd}{2}\cap \BB_{2,1}^{\frac{\dd}{2}+1}} \lesssim
	%-----------------------------------------------------------------------
	%-----------------------------------------------------------------------
	%-----------------------------------------------------------------------
	\| \Delta \delta	\m^k						\|_{L^1_t		\BB_{2,1}^\frac{\dd}{2}}
	\| 			\m^{k+1}						\|_{L^\infty_t	\BB_{2,1}^\frac{\dd}{2}}^2+
	\| \nabla \delta	\m^k						\|_{L^1_t		\BB_{2,1}^\frac{\dd}{2}}{\scriptstyle\times}\\
	{\scriptstyle\times}
	\| 		\nabla\m^{k+1}					\|_{L^\infty_t	\BB_{2,1}^\frac{\dd}{2}}
	\| 		\m^{k+1}							\|_{L^\infty_t	\BB_{2,1}^\frac{\dd}{2}}+
	%-----------------------------------------------------------------------
	\| \nabla \Delta \delta	\m^k						\|_{L^1_t		\BB_{2,1}^\frac{\dd}{2}}
	\| 		\m^{k+1}									\|_{L^\infty_t	\BB_{2,1}^\frac{\dd}{2}}^2+\\+
	\| \Delta \delta	\m^k								\|_{L^1_t		\BB_{2,1}^\frac{\dd}{2}}
	\| 					\nabla \m^{k+1}				\|_{L^\infty_t	\BB_{2,1}^\frac{\dd}{2}}
	\| 					\m^{k+1}						\|_{L^\infty_t	\BB_{2,1}^\frac{\dd}{2}}+
	%-----------------------------------------------------------------------
	%-----------------------------------------------------------------------
	%-----------------------------------------------------------------------
	\| \Delta 	\m^k						\|_{L^1_t		\BB_{2,1}^\frac{\dd}{2}}
	\| 			\delta\m^{k}						\|_{L^\infty_t	\BB_{2,1}^\frac{\dd}{2}}^2+\\+
	\| \nabla \m^k						\|_{L^1_t		\BB_{2,1}^\frac{\dd}{2}}
	\| 		\nabla\delta	\m^{k}					\|_{L^\infty_t	\BB_{2,1}^\frac{\dd}{2}}
	\| 		\m^{k+1}							\|_{L^\infty_t	\BB_{2,1}^\frac{\dd}{2}}+
	%-----------------------------------------------------------------------
	\| \nabla \Delta 	\m^k						\|_{L^1_t		\BB_{2,1}^\frac{\dd}{2}}
	\| 		\delta	\m^{k}									\|_{L^\infty_t	\BB_{2,1}^\frac{\dd}{2}}
	\|		\m^{k+1}									\|_{L^\infty_t	\BB_{2,1}^\frac{\dd}{2}}+\\+
	\| \Delta \m^k								\|_{L^1_t		\BB_{2,1}^\frac{\dd}{2}}
	\| 				\nabla \delta\m^{k}				\|_{L^\infty_t	\BB_{2,1}^\frac{\dd}{2}}
	\| 					\m^{k+1}						\|_{L^\infty_t	\BB_{2,1}^\frac{\dd}{2}}+
	%-----------------------------------------------------------------------
	%-----------------------------------------------------------------------
	%-----------------------------------------------------------------------
	\| \Delta 	\m^k						\|_{L^1_t		\BB_{2,1}^\frac{\dd}{2}}
	\|		\m^{k}						\|_{L^\infty_t	\BB_{2,1}^\frac{\dd}{2}}
	\| 			\delta\m^{k}						\|_{L^\infty_t	\BB_{2,1}^\frac{\dd}{2}}+\\+
	\| \nabla \m^k						\|_{L^1_t		\BB_{2,1}^\frac{\dd}{2}}
	\| 		\nabla\m^{k}					\|_{L^\infty_t	\BB_{2,1}^\frac{\dd}{2}}
	\| 		\delta	\m^{k}				\|_{L^\infty_t	\BB_{2,1}^\frac{\dd}{2}}+
	%-----------------------------------------------------------------------
	\| \nabla \Delta 	\m^k						\|_{L^1_t		\BB_{2,1}^\frac{\dd}{2}}
	\| 			\m^{k}									\|_{L^\infty_t	\BB_{2,1}^\frac{\dd}{2}}
	\|		\delta \m^{k}									\|_{L^\infty_t	\BB_{2,1}^\frac{\dd}{2}}+\\+
	\| \Delta \m^k								\|_{L^1_t		\BB_{2,1}^\frac{\dd}{2}}
	\| 				\nabla \m^{k}				\|_{L^\infty_t	\BB_{2,1}^\frac{\dd}{2}}
	\| 		\delta			\m^{k}						\|_{L^\infty_t	\BB_{2,1}^\frac{\dd}{2}}+
	%-----------------------------------------------------------------------
	%-----------------------------------------------------------------------
	%-----------------------------------------------------------------------
	\| \Delta \delta \m^k 		 \|_{L^1_t\BB_{2,1}^\frac{\dd}{2}}
	\| 		 (\m^k,\,\m^{k+1})	 \|_{L^\infty_t\BB_{2,1}^\frac{\dd}{2}}+\\+
	\| \nabla \delta \m^k 		 \|_{L^2_t\BB_{2,1}^\frac{\dd}{2}}
	\| 		 (\nabla \m^k,\,\nabla \m^{k+1})	 \|_{L^2_t\BB_{2,1}^\frac{\dd}{2}}+
	%-----------------------------------------------------------------------
	\| \nabla \Delta \delta \m^k 		 \|_{L^1_t\BB_{2,1}^\frac{\dd}{2}}
	\| 		 (\m^k,\,\m^{k+1})	 \|_{L^\infty_t\BB_{2,1}^\frac{\dd}{2}}+\\+
	\| \Delta \delta \m^k 		 \|_{L^2_t\BB_{2,1}^\frac{\dd}{2}}
	\| 		 (\nabla \m^k,\,\nabla \m^{k+1})	 \|_{L^2_t\BB_{2,1}^\frac{\dd}{2}}+
	%-----------------------------------------------------------------------
	\| \Delta \delta \m^k 		 \|_{L^1_t\BB_{2,1}^\frac{\dd}{2}}
	\| 		 (\nabla \m^k,\,\nabla \m^{k+1})	 \|_{L^\infty_t\BB_{2,1}^\frac{\dd}{2}}+\\+
	\| \nabla \delta \m^k 		 \|_{L^2_t\BB_{2,1}^\frac{\dd}{2}}
	\| 		 (\Delta \m^k,\,\Delta \m^{k+1})	 \|_{L^2_t\BB_{2,1}^\frac{\dd}{2}} 
	\lesssim \ee^{k+3}.
\end{align*}
Summarizing the previous considerations, we finally deduce that
\begin{equation}\label{ineq:ang_mom_delta_m2-f}
	\|\,\delta f^k\,\|_{L^1_t\BB_{2,1}^\frac{\dd}{2}\cap \BB_{2,1}^{\frac{\dd}{2}+1}}
	\lesssim 
	\ee^{k+3}.
\end{equation}
We then plug inequalities \eqref{ineq:ang_mom_delta_m2-h}, \eqref{ineq:ang_mom_delta_m2-beta1}, \eqref{ineq:ang_mom_delta_m2-beta2} and \eqref{ineq:ang_mom_delta_m2-f} into \eqref{ineq:ang_mom_delta_m2} to finally obtain
\begin{equation*}
	\|\,\delta \m^{k+1}\,\|_{\X_2}\,\lesssim\,\ee^{k+3} \quad\quad\Rightarrow\quad\quad \|\,\delta \m^{k+1} \,\|_{\X_2} \leq \ee^{k+2},
\end{equation*}
which concludes the proof of inequalities \eqref{angular_mom_ind}.

\subsection{The temperature equation}\label{thetemperatureequation} $\,$

\noindent We now deal with the approximate temperature given by $\Tt^k = \omega^k + \bar{\Tt}$ and we claim that the following inequalities hold by induction:
\begin{equation}\label{ineq:main-temp-ineq}
	\|\, 		\omega^{k+1} \,\|_{\X_3}\,\leq \, \ee^{3}\quad\quad\text{and}\quad\quad
	\|\, \delta \omega^{k+1} \,\|_{\X_3}\,\leq \, \ee^{k+2}.
\end{equation}
We assume that inequalities \eqref{prop:the_core_ineq1} and \eqref{prop:the_core_ineq2} are satisfied for a positive integer $k$. 
We then recall that $\omega^{k+1}$ is a classical solution of the following parabolic equation:
\begin{equation*}
	\partial_t \omega^{k+1}  \,-\,\Cc[\,\omega^{k+1}\,]\,=\,-\uu^k\cdot \nabla \omega^{k+1}\,+\,F^k_3,
\end{equation*}
where the forcing term $F^k_3$ is defined by means of
\begin{align*}
	F^k_3 = (1+\w^k)\Big[\partial_t  \frac{\partial {\rm W}^k_{{\rm F}}}{\partial\w}
							&+\uu^k\cdot\nabla{\rm \frac{\partial W_{{\rm F}}^k}{\partial\w}}\Big]
							+\Div\{\tilde \lambda^k_{1}\nabla \w^k\}
							+\\&+\Div\{\tilde \lambda^k_{2}([\bar \nn+\m^k]\cdot \nabla \w^k)[\bar{\nn}
							+\m^k]\}
							+\sigma^{{\rm L},k}:\mathbb{D}^k+{\rm g}^k\cdot \mcN^k.
\end{align*}
Thanks to Theorem \ref{thm-superjuve}, the following bound for the Besov-norm of the solution holds
\begin{align*}
	\| \omega^{k+1}\|_{\X_3}
	\,&\lesssim\,
	\| \,\omega_0\,\|_{\BB_{2,1}^{\frac{\dd}{2}-2}\cap \BB_{2,1}^\frac{\dd}{2}}\,+\,
	\| \,\uu^k\cdot \nabla\omega^{k+1}\,\|_{L^1_t\BB_{2,1}^{\frac{\dd}{2}-2}\cap \BB_{2,1}^\frac{\dd}{2}}\,+\,
	\| \,F^k_3\,\|_{L^1_t\BB_{2,1}^{\frac{\dd}{2}-2}\cap \BB_{2,1}^\frac{\dd}{2}}
	\\
	&\lesssim
	\ee^4\,+\,
	\|\,\uu^k				\,\|_{L^2_t\BB_{2,1}^{\frac{\dd}{2}}}
	\|\,\nabla \omega^{k+1}	\,\|_{L^2_t\BB_{2,1}^{\frac{\dd}{2}-2}\cap\BB_{2,1}^\frac{\dd}{2}}
	\,+\,
	\| \,F^k_3\,\|_{L^1_t\BB_{2,1}^{\frac{\dd}{2}-2}\cap \BB_{2,1}^\frac{\dd}{2}}\\
	&\lesssim\,
	\ee^4\,+\, \ee^2 \|\,\omega^{k+1}\,\|_{\X_3}\,+\,\| \,F^k_3\,\|_{L^1_t\BB_{2,1}^{\frac{\dd}{2}-2}\cap \BB_{2,1}^\frac{\dd}{2}},
\end{align*}
hence, assuming $\ee$ small enough, we deduce the following inequality:
\begin{equation}\label{ineq:temp-eq-main}
	\| \omega^{k+1}\|_{\X_3}\,\lesssim\, \ee^4\,+\,\| \,F^k_3\,\|_{L^1_t\BB_{2,1}^{\frac{\dd}{2}-2}\cap \BB_{2,1}^\frac{\dd}{2}}.
\end{equation}
We then control the norm of $F^k_3$ developing any term. We first observe that the explicit formula of $\partial_t \partial {\rm W}^k_{{\rm F}}/\partial\w$ is given by
\begin{align*}
	\partial_t \omega^k \bigg[\frac{K_1'(\omega^k)}{2}|\nabla \m^k|^2+\frac{K_2'(\omega^k)}{2}|\Div \m^k|^2+
								\frac{K_3'(\omega^k)}{2}|(\bar{\nn}+\m^k)\cdot \nabla \m^k|^2	 + \frac{K_4'(\omega^k)}{2}\trc\{(\nabla \m^k)^2\}
						\bigg] +\\+ K_1(\omega^k)\nabla \m^k:\nabla \partial_t \m^k + K_2(\omega^k)\Div\, \m^k\Div \partial_t \m^k +K_4(\omega^k)\trc\{\nabla \m^k\nabla \partial_t\m^k\}+ \\+
						K_3(\omega^k) \big[ \partial_t \m^k\cdot \nabla \m^k \big]:\big[ (\bar{\nn}+\m^k)\cdot\nabla \m^k\big]+
						K_3(\omega^k) \big[  (\bar{\nn}+\m^k)\cdot \nabla \partial_t\m^k \big]:\big[ (\bar{\nn}+\m^k)\cdot\nabla \m^k\big].
\end{align*}
hence, we gather
\begin{equation}\label{ineq:temp-est1}
\begin{aligned}
	\Big\| (1+\omega^k) \partial_t  \frac{\partial {\rm W}^k_{{\rm F}}}{\partial\w} \Big\|_{L^1_t\BB_{2,1}^{\frac{\dd}{2}-2}\cap \BB_{2,1}^\frac{\dd}{2}}
	\lesssim
	\Big( 1+ \|\omega^k\|_{L^\infty_t\BB_{2,1}^\frac{\dd}{2}}\Big)
	\Big\| \,\partial_t  \frac{\partial {\rm W}^k_{{\rm F}}}{\partial\w}\,\Big\|_{L^1_t\BB_{2,1}^{\frac{\dd}{2}-2}\cap \BB_{2,1}^\frac{\dd}{2}}
	\lesssim \Big( 1+ \|\omega^k\|_{L^\infty_t\BB_{2,1}^\frac{\dd}{2}}^2\Big){\scriptstyle\times}\\
	{\scriptstyle\times}
	\Big( 1+ \|\m^k\|_{L^\infty_t\BB_{2,1}^\frac{\dd}{2}}^2\Big)\|\partial_t\omega^k\|_{L^1_t\BB_{2,1}^{\frac{\dd}{2}-2}\cap \BB_{2,1}^\frac{\dd}{2}}\| \nabla \m^k \|_{L^\infty_t\BB_{2,1}^\frac{\dd}{2}}^2
	+ 
	\Big( 1+ \|\omega^k\|_{L^\infty_t\BB_{2,1}^\frac{\dd}{2}}^2\Big) \|\partial_t\m^k\|_{L^1_t\BB_{2,1}^\frac{\dd}{2}}{\scriptstyle\times}\\
	{\scriptstyle\times}\| \nabla \m^k \|_{L^\infty_t\BB_{2,1}^{\frac{\dd}{2}-1}\cap \BB_{2,1}^\frac{\dd}{2}}^2+
	\Big( 1+ \|\omega^k\|_{L^\infty_t\BB_{2,1}^\frac{\dd}{2}}^2\Big)\Big( 1+ \|\m^k\|_{L^\infty_t\BB_{2,1}^\frac{\dd}{2}}\Big){\scriptstyle\times}\\
	{\scriptstyle\times}\| \nabla \m^k \|_{L^\infty_t\BB_{2,1}^{\frac{\dd}{2}-1}\cap \BB_{2,1}^\frac{\dd}{2}}\|\partial_t \nabla \m^k \|_{L^1_t\BB_{2,1}^{\frac{\dd}{2}-1}\cap \BB_{2,1}^\frac{\dd}{2}}
	\,\lesssim\,\ee^4.
\end{aligned}
\end{equation}
On the other hand, a similar technique leads to
\begin{equation}\label{ineq:temp-est2}
\begin{aligned}
	\Big\|(1+\omega^k)\uu^k\cdot\nabla\frac{\partial {\rm W}^k_{{\rm F}}}{\partial\w} \Big\|_{L^1_t\BB_{2,1}^{\frac{\dd}{2}-2}\cap \BB_{2,1}^\frac{\dd}{2}}
	\lesssim
	\Big( 1+ \|\omega^k\|_{L^\infty_t\BB_{2,1}^\frac{\dd}{2}}\Big)\| \uu^k \|_{L^2_t\BB_{2,1}^\frac{\dd}{2}}
	\Big\| \nabla\frac{\partial {\rm W}^k_{{\rm F}}}{\partial\w} \Big\|_{L^2_t\BB_{2,1}^{\frac{\dd}{2}-2}\cap \BB_{2,1}^\frac{\dd}{2}}\lesssim\\
	\lesssim
	\Big( 1+ \|\omega^k\|_{L^\infty_t\BB_{2,1}^\frac{\dd}{2}}^2\Big)\| \uu^k \|_{L^2_t\BB_{2,1}^\frac{\dd}{2}}
	\Big( 1+ \|\m^k\|_{L^\infty_t\BB_{2,1}^\frac{\dd}{2}}^2\Big)
	\|\nabla \omega^k\|_{L^2_t\BB_{2,1}^{\frac{\dd}{2}-2}\cap \BB_{2,1}^\frac{\dd}{2}}\| \nabla \m^k \|_{L^\infty_t\BB_{2,1}^\frac{\dd}{2}}^2 +	\\
	+\Big( 1+ \|\omega^k\|_{L^\infty_t\BB_{2,1}^\frac{\dd}{2}}^2\Big)\Big( 1+ \|\m^k\|_{L^\infty_t\BB_{2,1}^\frac{\dd}{2}}\Big)
	\| \uu^k \|_{L^2_t\BB_{2,1}^\frac{\dd}{2}}
	\| \nabla \m^k \|_{L^\infty_t\BB_{2,1}^{\frac{\dd}{2}-1}\cap \BB_{2,1}^\frac{\dd}{2}}\|\Delta \m^k \|_{L^2_t\BB_{2,1}^{\frac{\dd}{2}-1}\cap \BB_{2,1}^\frac{\dd}{2}}+\\
	+\Big( 1+ \|\omega^k\|_{L^\infty_t\BB_{2,1}^\frac{\dd}{2}}^2\Big)\| \nabla \m^k\|_{L^\infty_t\BB_{2,1}^\frac{\dd}{2}}
	\| \uu^k \|_{L^2_t\BB_{2,1}^\frac{\dd}{2}}
	\| \nabla \m^k \|_{L^\infty_t\BB_{2,1}^{\frac{\dd}{2}-1}\cap \BB_{2,1}^\frac{\dd}{2}}\|\nabla \m^k \|_{L^2_t\BB_{2,1}^{\frac{\dd}{2}-1}\cap \BB_{2,1}^\frac{\dd}{2}}
	\lesssim \ee^4.
\end{aligned}
\end{equation}
We keep analyzing any term in $F^k_3$, and we proceed estimating 
\begin{align*}
	\| \Div\,\{\tilde \lambda_1^k \nabla \omega^k \}\|_{L^1_t\BB_{2,1}^{\frac{\dd}{2}-2}\cap \BB_{2,1}^\frac{\dd}{2}}
	\lesssim
	(1+\|\omega^k \|_{L^\infty_t\BB_{2,1}^\frac{\dd}{2}})
	\| \nabla 	\omega^k \|_{L^2_t\BB_{2,1}^{\frac{\dd}{2}}}
	\| \nabla 	\omega^k \|_{L^2_t\BB_{2,1}^{\frac{\dd}{2}-2}\cap \BB_{2,1}^\frac{\dd}{2}}
	+\\+
	\| 		 	\omega^k \|_{L^\infty_t\BB_{2,1}^\frac{\dd}{2}}
	\| \Delta 	\omega^k \|_{L^1_t\BB_{2,1}^{\frac{\dd}{2}-2}\cap \BB_{2,1}^\frac{\dd}{2}}
	\lesssim \ee^4,
\end{align*}
together with
\begin{align*}
	\| \Div\,\{\tilde \lambda^k_{2}([\bar \nn+\m^k]\cdot \nabla \w^k)[\bar{\nn}
							+\m^k]\} \|_{L^1_t\BB_{2,1}^{\frac{\dd}{2}-2}\cap \BB_{2,1}^\frac{\dd}{2}}
	\lesssim
	(1+\|\omega^k \|_{L^\infty_t\BB_{2,1}^\frac{\dd}{2}})
	(1+\|\m^k \|_{L^\infty_t\BB_{2,1}^\frac{\dd}{2}}^2){\scriptstyle\times}\\
	{\scriptstyle\times}
	\| \nabla 	\omega^k \|_{L^2_t \BB_{2,1}^\frac{\dd}{2}}
	\| \nabla 	\omega^k \|_{L^2_t\BB_{2,1}^{\frac{\dd}{2}-2}\cap \BB_{2,1}^\frac{\dd}{2}}+
	(1+\|\omega^k \|_{L^\infty_t\BB_{2,1}^\frac{\dd}{2}})
	(1+\|\m^k \|_{L^\infty_t\BB_{2,1}^\frac{\dd}{2}}^2)
	\| \Delta 	\omega^k \|_{L^1_t \BB_{2,1}^{\frac{\dd}{2}-2}\cap \BB_{2,1}^\frac{\dd}{2}}+\\
	+(1+\|\omega^k \|_{L^\infty_t\BB_{2,1}^\frac{\dd}{2}})
	(1+\|\m^k \|_{L^\infty_t\BB_{2,1}^\frac{\dd}{2}})
	\|\nabla \m^k \|_{L^2_t\BB_{2,1}^\frac{\dd}{2}}
	\| \nabla 	\omega^k \|_{L^2_t \BB_{2,1}^{\frac{\dd}{2}-2}\cap \BB_{2,1}^\frac{\dd}{2}}
	\lesssim \ee^4.
\end{align*}
It remains to control the viscous dissipation $\sigma^{{\rm\,L},\,k}:\mathbb{D}^k\,+\,{\rm g}^k\cdot \mcN^k\,$ whose explicit formula is 
\begin{align*}
	\alpha_1^k |(\bar{\nn}+\m^k)\cdot { \mathbb{D}^k}(\bar{\nn}+\m^k)|^2&+(\alpha_2^k+\alpha_3^k+\alpha_5^k-\alpha_6^k)\mcN^k\cdot { \mathbb{D}}^k(\bar{\nn}+\m^k)+\\&+
	\alpha_4^k|{ \mathbb{D}}^k|^2 +(\alpha_5^k+\alpha_6^k) |{ \mathbb{D}}^k(\bar{\nn}+\m^k)|^2+
	(\alpha_3^k -\alpha_2^k)|\mcN^k|^2.
\end{align*}
We then deduce that
\begin{align*}
	\|\sigma^{{\rm\,L},\,k}:\mathbb{D}^k\,+\,{\rm g}^k\cdot \mcN^k\,\|_{L^1_t\BB_{2,1}^{\frac{\dd}{2}-2}\cap \BB_{2,1}^\frac{\dd}{2}}
	\lesssim
	(1+\|\omega^k \|_{L^\infty_t\BB_{2,1}^\frac{\dd}{2}})
	\Big[
		(1+\|\m^k \|_{L^\infty_t\BB_{2,1}^\frac{\dd}{2}}^4)
		\|\nabla \uu^k \|_{L^2_t\BB_{2,1}^{\frac{\dd}{2}-1}\cap \BB_{2,1}^\frac{\dd}{2}}^2+\\
		+(1+\|\m^k \|_{L^\infty_t\BB_{2,1}^\frac{\dd}{2}})
		\| \mcN^k \|_{L^2_t\BB_{2,1}^{\frac{\dd}{2}-1}\cap \BB_{2,1}^\frac{\dd}{2}}
		\| \nabla \uu^k \|_{L^2_t\BB_{2,1}^{\frac{\dd}{2}-1}\cap \BB_{2,1}^\frac{\dd}{2}}+
		\| \mcN^k \|_{L^2_t\BB_{2,1}^{\frac{\dd}{2}-1}\cap \BB_{2,1}^\frac{\dd}{2}}^2
	\Big]
	\lesssim \ee^4,
\end{align*}
where we have used the following bound for the co-rotational time flux $\mcN^k$
\begin{align*}
 \| \mcN^k \|_{L^2_t\BB_{2,1}^{\frac{\dd}{2}-1}\cap \BB_{2,1}^\frac{\dd}{2}} =
 \| \partial_t \m^k + \uu^k\cdot \nabla \m^k - \Omega^k \bar \nn - \Omega^k \m^k \|_{L^2_t\BB_{2,1}^{\frac{\dd}{2}-1}\cap \BB_{2,1}^\frac{\dd}{2}}
 \lesssim
 \| \partial_t \m^k \|_{L^2_t\BB_{2,1}^{\frac{\dd}{2}-1}\cap \BB_{2,1}^\frac{\dd}{2}}+\\+
 \| \uu^k \|_{L^2_t\BB_{2,1}^\frac{\dd}{2}} 
 \| \nabla \m^k \|_{L^\infty_t\BB_{2,1}^{\frac{\dd}{2}-1}\cap \BB_{2,1}^\frac{\dd}{2}}
 + \| \nabla \uu^k \|_{L^2_t\BB_{2,1}^{\frac{\dd}{2}-1}\cap \BB_{2,1}^\frac{\dd}{2}}
 (1+\|\m^k\|_{L^\infty_t\BB_{2,1}^\frac{\dd}{2}}) \lesssim \ee^3.
\end{align*}
This concludes the proof of the first inequality of \eqref{ineq:main-temp-ineq}.

\smallskip
\noindent We now take into account the difference between two consecutive approximate temperatures, more precisely $\delta \omega^k=\omega^{k+1}-\omega^k$ and we aim to prove by induction the second inequality 
in \eqref{ineq:main-temp-ineq}. We observe that $\delta \omega^{k+1}$ is solution of the following parabolic equation:
\begin{equation*}
	\partial_t \delta \omega^{k+1}  \,-\,\Cc[\, \delta \omega^{k+1}\,]\,=\,-\uu^{k+1}\cdot \nabla \delta \omega^{k+1}\,-\delta \uu^k\cdot \nabla \omega^{k+1}\,+\,\delta F^k_3,
\end{equation*}
with $\delta F^k_3 =F^{k+1}_3 - F^k_3$. Thanks to Theorem \ref{thm-superjuve}, we gather a first estimate of the $\X_3$-norm for $\delta \omega^k$: 
\begin{align*}
	\| \delta \omega^{k+1} \|_{\X_3}
	&\lesssim
	\| -\uu^{k+1}\cdot \nabla \delta \omega^{k+1}-\delta \uu^k\cdot \nabla \omega^{k+1}+\delta F^k_3 \|_{L^1_t\BB_{2,1}^{\frac{\dd}{2}-2}\cap \BB_{2,1}^\frac{\dd}{2}} \\
	&\lesssim
	\| \uu^{k+1} 					\|_{L^\infty_t\BB_{2,1}^{\frac{\dd}{2}-1}}
	\| \nabla \delta \omega^{k+1} 	\|_{L^1_t\BB_{2,1}^{\frac{\dd}{2}-1}		}+
	\| \uu^{k+1} 					\|_{L^2_t\BB_{2,1}^{\frac{\dd}{2}}		}
	\| \nabla \delta \omega^{k+1} 	\|_{L^2_t\BB_{2,1}^{\frac{\dd}{2}}		}+ \\ 
	&\hspace{0.5cm}+
	\| \delta \uu^{k} 				\|_{L^\infty_t\BB_{2,1}^{\frac{\dd}{2}-1}}
	\| \nabla \omega^{k+1} 			\|_{L^1_t\BB_{2,1}^{\frac{\dd}{2}-1}		}+
	\| \delta \uu^{k}				\|_{L^2_t\BB_{2,1}^{\frac{\dd}{2}}		}
	\| \nabla \omega^{k+1}		 	\|_{L^2_t\BB_{2,1}^{\frac{\dd}{2}}		}+
	\| \delta F^k_3 \|_{L^1_t\BB_{2,1}^{\frac{\dd}{2}-2}\cap \BB_{2,1}^\frac{\dd}{2}} \\
	&\lesssim 
	\ee^2 \| \delta \omega^{k+1} \|_{\X_3} + \ee^{k+3} + \| \delta F^k_3 \|_{L^1_t\BB_{2,1}^{\frac{\dd}{2}-2}\cap \BB_{2,1}^\frac{\dd}{2}}.
\end{align*}
We then assume $\ee$ small enough in order to absorb the first term on the right-hand side of the above inequality. Hence, we achieve
\begin{equation}\label{ineq:the-end}
	\|\, \delta \omega^{k+1}\, \|_{\X_3}\lesssim \ee^{k+3} + \| \,\delta F^k_3\, \|_{L^1_t\BB_{2,1}^{\frac{\dd}{2}-2}\cap \BB_{2,1}^\frac{\dd}{2}}.
\end{equation}
We then proceed analyzing each term in $\delta F^3_k$. Denoting by  $\h^k:= {\partial {\rm W}^k_{{\rm F}}}/{\partial\w} $, the explicit formula of $\delta F^3_k$ is given by
\begin{equation}\label{def:deltaF3k}
\begin{aligned}
	\delta F^3_k = \delta \omega^k\Big[\partial_t  \h^{k+1}+\uu^{k+1}\cdot\nabla \h^{k+1}\Big] +
	(1+\w^k)\Big[\partial_t \delta \mathbb{H}^k	+ \delta \uu^{k}\cdot\nabla \h^{k+1}+ \uu^{k}\cdot\nabla \delta \h^{k}\Big] + \\+
	\Div \{\tilde \delta \lambda^k_{1}\nabla \w^{k+1}\} + \Div\{\tilde \lambda^k_{1}\nabla \delta \w^k\} + 
	\Div\{ \delta \tilde \lambda^k_{2}([\bar \nn+\m^{k+1}]\cdot \nabla \w^{k+1})[\bar{\nn} +\m^{k+1}]\} +\\+
	\Div\{  \tilde \lambda^k_{2}(\delta\m^{k}\cdot \nabla \w^{k+1})[\bar{\nn} +\m^{k+1}]\} +
	\Div\{ \delta \tilde \lambda^k_{2}([\bar \nn+\m^{k}]\cdot \nabla \delta \w^{k})[\bar{\nn} +\m^{k+1}]\} + \\ +
	\Div\{ \delta \tilde \lambda^k_{2}([\bar \nn+\m^{k}]\cdot \nabla  \w^{k})\delta\m^{k}\}  + 
	\delta \sigma^{{\rm L},k}:\mathbb{D}^{k+1}+\sigma^{{\rm L},k}:\delta \mathbb{D}^k +
	\delta {\rm g}^k\cdot \mcN^{k+1}+{\rm g}^k\cdot \delta \mcN^k.
\end{aligned}
\end{equation}
Proceeding as for proving \eqref{ineq:temp-est1}, we first gather that
\begin{align*}
	\| \delta \omega^k\,\partial_t  \h^{k+1} \|_{L^1_t\BB_{2,1}^{\frac{\dd}{2}-2}\cap \BB_{2,1}^\frac{\dd}{2}}
	\lesssim
	\|\delta \omega^k \|_{L^1_t\BB_{2,1}^\frac{\dd}{2}}\Big\|\partial_t  \frac{\partial {\rm W}^{k+1}_{{\rm F}}}{\partial\w}\,\Big\|_{L^1_t\BB_{2,1}^{\frac{\dd}{2}-2}\cap \BB_{2,1}^\frac{\dd}{2}}
	\lesssim
	\ee^4\| \delta \omega^k \|_{\X_3}\lesssim
	\ee^{k+3},
\end{align*}
while a similar approach to the one used in \eqref{ineq:temp-est2} leads to
\begin{align*}
	\| \delta \omega^k\,\uu^{k+1}\cdot\nabla \h^{k+1} \|_{L^1_t\BB_{2,1}^{\frac{\dd}{2}-2}\cap \BB_{2,1}^\frac{\dd}{2}} \lesssim
	\ee^4\| \delta \omega^k \|_{\X_3}\lesssim
	\ee^{k+3}.
\end{align*}
Recalling that $\nn^k = \bar{\nn} + \m^k$, we now observe that
\begin{align*}
	\partial_t \delta \mathbb{H}^k = 
	\partial_t \delta \omega^k 
	\bigg[
		\frac{{K_1'}^{k+1}}{2}|\nabla \m^{k+1}|^2+\frac{{K_2'}^{k+1}}{2}|\Div \m^{k+1}|^2+\frac{{K_3'}^{k+1}}{2}|\nn^{k+1}\cdot \nabla \m^{k+1}|^2+
	\\+
		\frac{{K_4'}^{k+1}}{2}\trc\{(\nabla \m^{k+1})^2\}
	\bigg] + 
	\partial_t \omega^k
	\bigg\{
		\frac{\delta {K_1'}^{k}}{2}|\nabla \m^{k+1}|^2+\frac{\delta{K_2'}^{k}}{2}|\Div \m^{k+1}|^2+
		\frac{\delta{K_3'}^{k}}{2}|\nn^{k+1}\cdot \nabla \m^{k+1}|^2+
	\\+
		\frac{\delta{K_4'}^{k}}{2}\trc\{(\nabla \m^{k+1})^2\}+
		\frac{{K_1'}^{k}}{2}\nabla\delta  \m^k:\nabla \m^{k+1}+
		\frac{{K_2'}^{k}}{2}\Div\delta \m^k\Div \m^{k+1}+
		\frac{{K_4'}^{k}}{2}\trc\{\nabla \delta \m^k\nabla \m^{k+1}\}+	
	\\+
		\frac{{K_3'}^{k}}{2}	 \big[ \delta \m^k\cdot \nabla \m^{k+1} \big]:\big[ \nn^{k+1}\cdot\nabla \m^{k+1}\big]+
		\frac{{K_1'}^{k}}{2}\nabla  \m^k:\nabla\delta \m^{k}+
		\frac{{K_2'}^{k}}{2}\Div \m^k\Div \delta\m^{k}+\\
		\frac{{K_4'}^{k}}{2}\trc\{\nabla  \m^k\nabla \delta\m^{k}\}+
		\frac{{K_3'}^{k}}{2}	 \big[ \nn^k\cdot \nabla \delta \m^{k} \big]:\big[ \nn^{k+1}\cdot\nabla \m^{k+1}\big]+
		\frac{{K_3'}^{k}}{2}	 \big[ \nn^k\cdot \nabla \m^{k} \big]:\big[ \delta \m^{k}\cdot\nabla \m^{k+1}\big]+
	\\+
	   	\frac{{K_3'}^{k}}{2}	 \big[ \nn^k\cdot \nabla \m^{k} \big]:\big[  \nn^{k}\cdot\nabla \delta \m^{k}\big]
	\bigg\} + 
	%-------------------------------
	%-------------------------------
	%-------------------------------
		\delta K_1^k\nabla \m^{k+1}:\partial_t\nabla \m^{k+1} +
		\delta K_2^k\Div  \m^{k+1}\Div \partial_t \m^{k+1}+
	\\+
		\delta K_3^k\big[ \partial_t \m^{k+1}\cdot \nabla \m^{k+1} \big]:\big[ \nn^{k+1}\cdot\nabla \m^{k+1}\big]+
		\delta K_4^k\trc\{\nabla \m^{k+1}\nabla \partial_t\m^{k+1}\}+
	\\+
	%-------------------------------
		 K_1^k\nabla \delta\m^{k}:\partial_t\nabla \m^{k+1} +
		 K_2^k\Div  \delta\m^{k}\Div \partial_t \m^{k+1}+
		 K_3^k\big[ \partial_t \delta\m^{k}\cdot \nabla \m^{k+1} \big]:\\:\big[ \nn^{k+1}\cdot\nabla \m^{k+1}\big]+
		 K_4^k\trc\{\nabla \delta\m^{k}\nabla \partial_t\m^{k+1}\}+
	%-------------------------------	 
		 K_1^k\nabla \delta\m^{k}:\partial_t\nabla \m^{k+1} +
		 K_2^k\Div  \delta\m^{k}\Div \partial_t \m^{k+1}+
	\\+
		 K_3^k\big[ \partial_t \delta\m^{k}\cdot \nabla \m^{k+1} \big]:\big[ \nn^{k+1}\cdot\nabla \m^{k+1}\big]+
		 K_4^k\trc\{\nabla \delta\m^{k}\nabla \partial_t\m^{k+1}\}+
	%-------------------------------	 
		 K_1^k\nabla \m^{k}:\partial_t\nabla \delta\m^{k} +
	\\+
		 K_2^k\Div  \m^{k}\Div \partial_t \delta\m^{k}+
		 K_3^k\big[ \partial_t \m^{k}\cdot \nabla \delta\m^{k} \big]:\big[ \nn^{k+1}\cdot\nabla \m^{k+1}\big]+
		 K_4^k\trc\{\nabla \m^{k}\nabla \partial_t\delta\m^{k}\}+	
	\\+	  
		 K_3^k\big[ \partial_t \m^{k}\cdot \nabla \m^{k} \big]:\big[ \delta\m^{k}\cdot\nabla \m^{k+1}\big]+
		 K_3^k\big[ \partial_t \m^{k}\cdot \nabla \m^{k} \big]:\big[ \nn^{k}\cdot\nabla  \delta\m^{k}\big]+
	\\+
	%-------------------------------
		 \delta K_3^k \big[  \nn^{k+1}\cdot \nabla \partial_t\m^{k+1} \big]:\big[ \nn^{k+1}\cdot\nabla \m^{k+1}\big]+	 	
		 K_3^k \big[\delta  \m^{k}\cdot \nabla \partial_t\m^{k+1} \big]:\big[ \nn^{k+1}\cdot\nabla \m^{k+1}\big]+
	\\+
		 K_3^k \big[	 \m^{k}\cdot \nabla \partial_t\delta \m^{k} \big]:\big[ \nn^{k+1}\cdot\nabla \m^{k+1}\big]+
		 K_3^k \big[	 \m^{k}\cdot \nabla \partial_t		\m^{k} \big]:\big[ \delta \m^{k}\cdot\nabla \m^{k+1}\big]+
	\\+
		 K_3^k \big[	 \m^{k}\cdot \nabla \partial_t		\m^{k} \big]:\big[  \nn^{k}\cdot\nabla\delta \m^{k}\big].		
\end{align*}
Thus, we deduce that
\begin{align*}
	\| (1+\omega^k) \partial_t \delta \h^k\|_{L^1_t\BB_{2,1}^{\frac{\dd}{2}-2}\cap \BB_{2,1}^\frac{\dd}{2}}
	\lesssim
	(1+\| \omega^k \|_{L^\infty_t\BB_{2,1}^\frac{\dd}{2}})
	\| \partial_t \delta \h^k\|_{L^1_t\BB_{2,1}^{\frac{\dd}{2}-2}\cap \BB_{2,1}^\frac{\dd}{2}}
	\lesssim
	(1+\\+\|( \omega^k,\,\omega^{k+1})\|_{L^\infty_t\BB_{2,1}^\frac{\dd}{2}}^2)
	\| \partial_t \delta \omega^k \|_{L^1_t\BB_{2,1}^{\frac{\dd}{2}-2}\cap \BB_{2,1}^\frac{\dd}{2}}
	\| (\nabla\m^k,\,\nabla \m^{k+1}) \|_{L^\infty_t\BB_{2,1}^\frac{\dd}{2}}^2
	(1+\|( \m^k,\,\m^{k+1})\|_{L^\infty_t\BB_{2,1}^\frac{\dd}{2}}^2)+\\
	+
	(1+\|( \omega^k,\,\omega^{k+1})\|_{L^\infty_t\BB_{2,1}^\frac{\dd}{2}})
	\| \partial_t \omega^k \|_{L^1_t\BB_{2,1}^{\frac{\dd}{2}-2}\cap \BB_{2,1}^\frac{\dd}{2}}
	\| \delta \omega^k \|_{L^\infty_t\BB_{2,1}^\frac{\dd}{2}}\| (\nabla\m^k,\,\nabla \m^{k+1}) \|_{L^\infty_t\BB_{2,1}^\frac{\dd}{2}}^2{\scriptstyle\times}\\
	{\scriptstyle\times}
	(1+\|( \m^k,\,\m^{k+1})\|_{L^\infty_t\BB_{2,1}^\frac{\dd}{2}}^2) + 
	(1+\|( \omega^k,\,\omega^{k+1})\|_{L^\infty_t\BB_{2,1}^\frac{\dd}{2}}^2)
	\| \partial_t \omega^k \|_{L^1_t\BB_{2,1}^{\frac{\dd}{2}-2}\cap \BB_{2,1}^\frac{\dd}{2}}
	\|\delta\m^k\|_{L^\infty_t\BB_{2,1}^\frac{\dd}{2}}{\scriptstyle\times}\\
	{\scriptstyle\times}
	(1+\|( \m^k,\,\m^{k+1})\|_{L^\infty_t\BB_{2,1}^\frac{\dd}{2}}) +
	(1+\|( \omega^k,\,\omega^{k+1})\|_{L^\infty_t\BB_{2,1}^\frac{\dd}{2}})
	\| \delta \omega^k \|_{L^\infty_t\BB_{2,1}^\frac{\dd}{2}}
	\| \nabla \m^k \|_{L^\infty_t\BB_{2,1}^{\frac{\dd}{2}-1}\cap \BB_{2,1}^\frac{\dd}{2}}{\scriptstyle\times}\\
	{\scriptstyle\times}
	\| \partial_t \nabla \m^k \|_{L^1_t\BB_{2,1}^{\frac{\dd}{2}-1}\cap \BB_{2,1}^\frac{\dd}{2}}
	(1+\|( \m^k,\,\m^{k+1})\|_{L^\infty_t\BB_{2,1}^\frac{\dd}{2}}^2)+
	(1+\|( \omega^k,\,\omega^{k+1})\|_{L^\infty_t\BB_{2,1}^\frac{\dd}{2}})
	\| \delta \omega^k \|_{L^\infty_t\BB_{2,1}^\frac{\dd}{2}}{\scriptstyle\times}\\
	{\scriptstyle\times}
	\| \nabla ( \m^k,\,\m^{k+1}) \|_{L^\infty_t\BB_{2,1}^{\frac{\dd}{2}-1}\cap \BB_{2,1}^\frac{\dd}{2}}^2
	\| \partial_t  \m^k \|_{L^1_t \BB_{2,1}^\frac{\dd}{2}}
	(1+\|( \m^k,\,\m^{k+1})\|_{L^\infty_t\BB_{2,1}^\frac{\dd}{2}})				+\\+
	(1+\|( \omega^k,\,\omega^{k+1})\|_{L^\infty_t\BB_{2,1}^\frac{\dd}{2}}^2)
	\| \omega^k \|_{L^\infty_t\BB_{2,1}^\frac{\dd}{2}}
	\| \nabla  \delta \m^k \|_{L^\infty_t\BB_{2,1}^{\frac{\dd}{2}-1}\cap \BB_{2,1}^\frac{\dd}{2}}
	\| \partial_t\nabla  \m^{k+1} \|_{L^1_t\BB_{2,1}^{\frac{\dd}{2}-1}\cap \BB_{2,1}^\frac{\dd}{2}}{\scriptstyle\times}\\
	{\scriptstyle\times}
	(1+\|( \m^k,\,\m^{k+1})\|_{L^\infty_t\BB_{2,1}^\frac{\dd}{2}}^2)+
	(1+\|( \omega^k,\,\omega^{k+1})\|_{L^\infty_t\BB_{2,1}^\frac{\dd}{2}}^2)
	\| \nabla ( \m^k,\,\m^{k+1}) \|_{L^\infty_t\BB_{2,1}^{\frac{\dd}{2}-1}\cap \BB_{2,1}^\frac{\dd}{2}}{\scriptstyle\times}\\
	{\scriptstyle\times}
	\| \nabla \delta \m^k \|_{L^\infty_t\BB_{2,1}^{\frac{\dd}{2}-1}\cap \BB_{2,1}^\frac{\dd}{2}}
	\| \partial_t  \m^k \|_{L^1_t \BB_{2,1}^\frac{\dd}{2}}
	(1+\|( \m^k,\,\m^{k+1})\|_{L^\infty_t\BB_{2,1}^\frac{\dd}{2}})				+\\
	%----------------------
	+
	(1+\|( \omega^k,\,\omega^{k+1})\|_{L^\infty_t\BB_{2,1}^\frac{\dd}{2}}^2)
	\| \omega^k \|_{L^\infty_t\BB_{2,1}^\frac{\dd}{2}}
	\| \nabla   \m^k \|_{L^\infty_t\BB_{2,1}^{\frac{\dd}{2}-1}\cap \BB_{2,1}^\frac{\dd}{2}}
	\| \partial_t\nabla \delta \m^{k} \|_{L^1_t\BB_{2,1}^{\frac{\dd}{2}-1}\cap \BB_{2,1}^\frac{\dd}{2}}{\scriptstyle\times}\\
	{\scriptstyle\times}
	(1+\|( \m^k,\,\m^{k+1})\|_{L^\infty_t\BB_{2,1}^\frac{\dd}{2}}^2)+
	(1+\|( \omega^k,\,\omega^{k+1})\|_{L^\infty_t\BB_{2,1}^\frac{\dd}{2}}^2)
	\| \nabla ( \m^k,\,\m^{k+1}) \|_{L^\infty_t\BB_{2,1}^{\frac{\dd}{2}-1}\cap \BB_{2,1}^\frac{\dd}{2}}{\scriptstyle\times}\\
	{\scriptstyle\times}
	\| \nabla  \m^k \|_{L^\infty_t\BB_{2,1}^{\frac{\dd}{2}-1}\cap \BB_{2,1}^\frac{\dd}{2}}
	\| \partial_t \delta \m^k \|_{L^1_t \BB_{2,1}^\frac{\dd}{2}}
	(1+\|( \m^k,\,\m^{k+1})\|_{L^\infty_t\BB_{2,1}^\frac{\dd}{2}})	
	\lesssim \ee^{k+3}.
\end{align*}
A similar approach leads the following inequalities to be satisfied:
\begin{align*}
	\| (1+\omega^k) \delta \uu^k\cdot \nabla \h^{k+1}\|_{L^1_t\BB_{2,1}^{\frac{\dd}{2}-2}\cap \BB_{2,1}^\frac{\dd}{2}}
	\lesssim
	(1+\|( \omega^k,\,\omega^{k+1})\|_{L^\infty_t\BB_{2,1}^\frac{\dd}{2}}^2)
	\| \delta \uu^k \|_{L^\infty_t\BB_{2,1}^{\frac{\dd}{2}-1}\cap\BB_{2,1}^{\frac{\dd}{2}}}{\scriptstyle\times}\\
	{\scriptstyle\times}
	\| \nabla \h^{k+1} \|_{L^1_t\BB_{2,1}^{\frac{\dd}{2}-1}\cap\BB_{2,1}^{\frac{\dd}{2}} }
	\lesssim 
	\ee^3 \| \delta \uu^k \|_{\X_1}
	\lesssim \ee^{k+3},
\end{align*}
together with
\begin{equation*}
\begin{aligned}
	\| (1+\omega^k) \uu^k\cdot \nabla \delta \h^{k}\|_{L^1_t\BB_{2,1}^{\frac{\dd}{2}-2}\cap \BB_{2,1}^\frac{\dd}{2}}
	\lesssim
	(1+\|( \omega^k,\,\omega^{k+1})\|_{L^\infty_t\BB_{2,1}^\frac{\dd}{2}}^2)
	\|  \uu^k \|_{L^\infty_t\BB_{2,1}^{\frac{\dd}{2}-1}\cap\BB_{2,1}^{\frac{\dd}{2}}}{\scriptstyle\times}\\
	{\scriptstyle\times}
	\| \nabla\delta \h^{k} \|_{L^1_t\BB_{2,1}^{\frac{\dd}{2}-1}\cap\BB_{2,1}^{\frac{\dd}{2}} }
	\lesssim 
	\ee^{k+3}.
\end{aligned}
\end{equation*}
Furthermore, we observe that
\begin{align*}
	\| \Div \{\tilde \delta \lambda^k_{1}\nabla \w^{k+1}\}+\Div\{\tilde \lambda^k_{1}\nabla \delta \w^k\}\|_{L^1_t\BB_{2,1}^{\frac{\dd}{2}-2}\cap\BB_{2,1}^{\frac{\dd}{2}}}
	\lesssim 
	\| \delta \omega^k \|_{L^\infty_t\BB_{2,1}^\frac{\dd}{2}}
	\| \Delta \omega^{k+1} \|_{L^1_t\BB_{2,1}^{\frac{\dd}{2}-2}\cap\BB_{2,1}^{\frac{\dd}{2}}}+\\+
	\| 		 \omega^k \|_{L^\infty_t\BB_{2,1}^\frac{\dd}{2}}
	\| \Delta \delta\omega^{k} \|_{L^1_t\BB_{2,1}^{\frac{\dd}{2}-2}\cap\BB_{2,1}^{\frac{\dd}{2}}} +
	\| \nabla \delta \omega^k \|_{L^2_t\BB_{2,1}^\frac{\dd}{2}}
	\| \nabla \omega^{k+1} \|_{L^2_t\BB_{2,1}^{\frac{\dd}{2}-2}\cap\BB_{2,1}^{\frac{\dd}{2}}}+
	\| 	\nabla	 \omega^k \|_{L^2_t\BB_{2,1}^\frac{\dd}{2}}{\scriptstyle\times}\\{\scriptstyle\times}
	\| \nabla \delta\omega^{k} \|_{L^2_t\BB_{2,1}^{\frac{\dd}{2}-2}\cap\BB_{2,1}^{\frac{\dd}{2}}} 
	\lesssim 
	\ee^3 \|\,\delta \omega^k \,\|_{\X_3}\lesssim
	\ee^{k+3},
\end{align*}
and also
\begin{align*}
	\| \Div\{ \delta \tilde \lambda^k_{2}([\bar \nn+\m^{k+1}]\cdot \nabla \w^{k+1})[\bar{\nn} +\m^{k+1}]\}+
	\|_{L^1_t\BB_{2,1}^{\frac{\dd}{2}-2}\cap\BB_{2,1}^{\frac{\dd}{2}}}
	\lesssim
	\| \delta \omega^k \|_{L^\infty_t\BB_{2,1}^\frac{\dd}{2}}
	(1+\| \m^{k+1}\|_{L^\infty_t\BB_{2,1}^\frac{\dd}{2}}^2){\scriptstyle\times}\\
	{\scriptstyle\times}\| \Delta \omega^{k+1} \|_{L^1_t\BB_{2,1}^{\frac{\dd}{2}-2}\cap \BB_{2,1}^\frac{\dd}{2}}
	+
	\| \nabla \delta \omega^k \|_{L^2_t\BB_{2,1}^\frac{\dd}{2}}
	(1+\| \m^{k+1}\|_{L^\infty_t\BB_{2,1}^\frac{\dd}{2}}^2)
	\| \nabla \omega^{k+1} \|_{L^2_t\BB_{2,1}^{\frac{\dd}{2}-2}\cap \BB_{2,1}^\frac{\dd}{2}}+
	\| \delta \omega^k \|_{L^\infty_t\BB_{2,1}^\frac{\dd}{2}}{\scriptstyle\times}\\
	{\scriptstyle\times}
	(1+\| \m^{k+1}\|_{L^\infty_t\BB_{2,1}^\frac{\dd}{2}})
	\| \nabla \m^{k+1} \|_{L^2_t\BB_{2,1}^\frac{\dd}{2}}
	\| \nabla \omega^{k+1} \|_{L^2_t\BB_{2,1}^{\frac{\dd}{2}-2}\cap \BB_{2,1}^\frac{\dd}{2}}
	\lesssim 
	\ee^3 \|\,\delta \omega^k\,\|_{\X_3}\,\lesssim \ee^{k+3}.
\end{align*}
A similar estimate holds also for
\begin{align*}
	\| \Div\{  \tilde \lambda^k_{2}(\delta\m^{k}\cdot \nabla \w^{k+1})[\bar{\nn} +\m^{k+1}]\}\|_{L^1_t\BB_{2,1}^{\frac{\dd}{2}-2}\cap\BB_{2,1}^{\frac{\dd}{2}}}
	\lesssim
	\|  \omega^k \|_{L^\infty_t\BB_{2,1}^\frac{\dd}{2}}
	(1+\| \m^{k+1}\|_{L^\infty_t\BB_{2,1}^\frac{\dd}{2}}){\scriptstyle\times}\\
	{\scriptstyle\times}
	\|\delta  \m^{k}\|_{L^\infty_t\BB_{2,1}^\frac{\dd}{2}}
	\| \Delta \omega^{k+1} \|_{L^1_t\BB_{2,1}^{\frac{\dd}{2}-2}\cap \BB_{2,1}^\frac{\dd}{2}}
	+
	\| \nabla  \omega^k \|_{L^2_t\BB_{2,1}^\frac{\dd}{2}}
	(1+\| \m^{k+1}\|_{L^\infty_t\BB_{2,1}^\frac{\dd}{2}})
	\| \delta \m^{k}\|_{L^\infty_t\BB_{2,1}^\frac{\dd}{2}}{\scriptstyle\times}\\
	{\scriptstyle\times}
	\| \nabla \omega^{k+1} \|_{L^2_t\BB_{2,1}^{\frac{\dd}{2}-2}\cap \BB_{2,1}^\frac{\dd}{2}}+
	\|  \omega^k \|_{L^\infty_t\BB_{2,1}^\frac{\dd}{2}}
	(1+\| \m^{k+1}\|_{L^\infty_t\BB_{2,1}^\frac{\dd}{2}})
	\| \nabla \delta \m^{k} \|_{L^2_t\BB_{2,1}^{\frac{\dd}{2}-2}\cap \BB_{2,1}^\frac{\dd}{2}}+\\+
	\|  \omega^k \|_{L^\infty_t\BB_{2,1}^\frac{\dd}{2}}
	\|\nabla  \m^{k+1}\|_{L^2_t\BB_{2,1}^\frac{\dd}{2}})
	\|  \delta \m^{k} \|_{L^\infty_t\BB_{2,1}^{\frac{\dd}{2}-2}\cap \BB_{2,1}^\frac{\dd}{2}}
	\lesssim 
	\ee^3 \|\,\delta \m^k\,\|_{\X_2}\,\lesssim \ee^{k+3},
\end{align*}
together with
\begin{align*}
	\| \Div\{ \tilde \lambda^k_{2}([\bar \nn+\m^{k}]\cdot \nabla \delta \w^{k})[\bar{\nn} +\m^{k+1}]\} \|_{L^1_t\BB_{2,1}^{\frac{\dd}{2}-2}\cap\BB_{2,1}^{\frac{\dd}{2}}}
	\lesssim 
	\|  \omega^k \|_{L^\infty_t\BB_{2,1}^\frac{\dd}{2}}{\scriptstyle\times}\\
	{\scriptstyle\times}
	(1+\| (\m^k,\, \m^{k+1})\|_{L^\infty_t\BB_{2,1}^\frac{\dd}{2}}^2)
	\| \Delta \delta\omega^{k} \|_{L^1_t\BB_{2,1}^{\frac{\dd}{2}-2}\cap \BB_{2,1}^\frac{\dd}{2}}
	+
	\| \nabla \delta \omega^k \|_{L^2_t\BB_{2,1}^\frac{\dd}{2}}
	(1+\| (\m^k,\, \m^{k+1})\|_{L^\infty_t\BB_{2,1}^\frac{\dd}{2}}^2){\scriptstyle\times}\\
	{\scriptstyle\times}
	\| \nabla\delta \omega^{k} \|_{L^2_t\BB_{2,1}^{\frac{\dd}{2}-2}\cap \BB_{2,1}^\frac{\dd}{2}}+
	\|  \omega^k \|_{L^\infty_t\BB_{2,1}^\frac{\dd}{2}}
	(1+\|(\m^k,\, \m^{k+1})\|_{L^\infty_t\BB_{2,1}^\frac{\dd}{2}})
	\| \nabla (\m^k,\, \m^{k+1}) \|_{L^2_t\BB_{2,1}^\frac{\dd}{2}}{\scriptstyle\times}\\
	{\scriptstyle\times}
	\| \nabla\delta \omega^{k} \|_{L^2_t\BB_{2,1}^{\frac{\dd}{2}-2}\cap \BB_{2,1}^\frac{\dd}{2}}
	\lesssim 
	\ee^3 \|\,\delta \omega^k\,\|_{\X_3}\,\lesssim \ee^{k+3}.
\end{align*}
Moreover, 
\begin{align*}
	\| \Div\{ \tilde \lambda^k_{2}([\bar \nn+\m^{k}]\cdot \nabla  \w^{k})\delta\m^{k}\}\|_{L^1_t\BB_{2,1}^{\frac{\dd}{2}-2}\cap\BB_{2,1}^{\frac{\dd}{2}}}
	\lesssim
	\|  \omega^k \|_{L^\infty_t\BB_{2,1}^\frac{\dd}{2}}
	(1+\| \m^{k}\|_{L^\infty_t\BB_{2,1}^\frac{\dd}{2}}){\scriptstyle\times}\\
	{\scriptstyle\times}
	\|\delta  \m^{k}\|_{L^\infty_t\BB_{2,1}^\frac{\dd}{2}}
	\| \Delta \omega^{k} \|_{L^1_t\BB_{2,1}^{\frac{\dd}{2}-2}\cap \BB_{2,1}^\frac{\dd}{2}}
	+
	\| \nabla  \omega^k \|_{L^2_t\BB_{2,1}^\frac{\dd}{2}}
	(1+\| \m^{k}\|_{L^\infty_t\BB_{2,1}^\frac{\dd}{2}})
	\| \delta \m^{k}\|_{L^\infty_t\BB_{2,1}^\frac{\dd}{2}}{\scriptstyle\times}\\
	{\scriptstyle\times}
	\| \nabla \omega^{k} \|_{L^2_t\BB_{2,1}^{\frac{\dd}{2}-2}\cap \BB_{2,1}^\frac{\dd}{2}}+
	\|  \omega^k \|_{L^\infty_t\BB_{2,1}^\frac{\dd}{2}}
	(1+\| \m^{k}\|_{L^\infty_t\BB_{2,1}^\frac{\dd}{2}})
	\| \nabla \delta \m^{k} \|_{L^2_t\BB_{2,1}^{\frac{\dd}{2}-2}\cap \BB_{2,1}^\frac{\dd}{2}}+\\+
	\|  \omega^k \|_{L^\infty_t\BB_{2,1}^\frac{\dd}{2}}
	\|\nabla  \m^{k}\|_{L^2_t\BB_{2,1}^\frac{\dd}{2}})
	\|  \delta \m^{k} \|_{L^\infty_t\BB_{2,1}^{\frac{\dd}{2}-2}\cap \BB_{2,1}^\frac{\dd}{2}}
	\lesssim 
	\ee^3 \|\,\delta \m^k\,\|_{\X_2}\,\lesssim \ee^{k+3}.
\end{align*}
It then remains to control the difference between the two consecutive viscous dissipation, more precisely
\begin{align*}
	\delta \sigma^{{\rm L},k}:\mathbb{D}^{k+1}+\sigma^{{\rm L},k}:\delta \mathbb{D}^k +
	\delta {\rm g}^k\cdot \mcN^{k+1}+{\rm g}^k\cdot \delta \mcN^k =
	\delta \alpha_1^k |\nn^{k+1}\cdot { \mathbb{D}^{k+1}}\nn^{k+1}|^2 +\\+
	\alpha_1^k (\delta \m^{k}\cdot { \mathbb{D}^{k+1}}\nn^{k+1})
				(\nn^{k+1}\cdot { \mathbb{D}^{k+1}}\nn^{k+1})+
	\alpha_1^k ( \nn^{k}\cdot { \delta \mathbb{D}^{k}}\nn^{k+1})
				(\nn^{k+1}\cdot { \mathbb{D}^{k+1}}\nn^{k+1})+\\+		
	\alpha_1^k ( \nn^{k}\cdot { \mathbb{D}^{k}}\delta \m^{k})
				(\nn^{k+1}\cdot { \mathbb{D}^{k+1}}\nn^{k+1})+				
	\alpha_1^k ( \nn^{k}\cdot { \mathbb{D}^{k}} \nn^{k})
				(\delta \m^{k}\cdot { \mathbb{D}^{k+1}}\nn^{k+1})+\\+
	\alpha_1^k ( \nn^{k}\cdot { \mathbb{D}^{k}} \nn^{k})
				(\nn^{k}\cdot { \delta \mathbb{D}^{k}}\nn^{k+1})	+	
	\alpha_1^k ( \nn^{k}\cdot { \mathbb{D}^{k}} \nn^{k})
				(\nn^{k}\cdot {  \mathbb{D}^{k}}\delta\m^{k})	+				
	(\delta \alpha_2^k+\delta \alpha_3^k+\delta \alpha_5^k-\delta \alpha_6^k){\scriptstyle\times}\\
	{\scriptstyle\times}\mcN^{k+1}\cdot { \mathbb{D}}^{k+1}\nn^{k+1}+
	( \alpha_2^k+ \alpha_3^k+ \alpha_5^k- \alpha_6^k)\delta \mcN^{k}\cdot { \mathbb{D}}^{k+1}\nn^{k+1}+
	( \alpha_2^k+ \alpha_3^k+ \alpha_5^k- \alpha_6^k){\scriptstyle\times}\\
	{\scriptstyle\times}\mcN^{k}\cdot {\delta  \mathbb{D}}^k\nn^{k+1}+	
	( \alpha_2^k+ \alpha_3^k+ \alpha_5^k- \alpha_6^k)
	\mcN^{k}\cdot { \mathbb{D}}^k\delta \m^{k}+
	\delta \alpha_4^k|{ \mathbb{D}}^{k+1}|^2+
	\alpha_4^k{\delta \mathbb{D}}^k:{\mathbb{D}}^{k+1}+\\+
	\alpha_4^k{\mathbb{D}}^k:{\delta \mathbb{D}}^{k}+
	(\delta \alpha_5^k+\delta \alpha_6^k) |{ \mathbb{D}}^{k+1}\nn^{k+1}|^2+
	( \alpha_5^k+ \alpha_6^k) \delta \mathbb{D}^k\nn^{k+1}\cdot  \mathbb{D}^{k+1}\nn^{k+1}+	
	( \alpha_5^k+ \alpha_6^k) {\scriptstyle\times}\\
	{\scriptstyle\times} \mathbb{D}^k\delta\m^{k}\cdot  \mathbb{D}^{k+1}\nn^{k+1}+
	( \alpha_5^k+ \alpha_6^k) \mathbb{D}^k\nn^{k}\cdot \delta  \mathbb{D}^{k}\nn^{k+1}+	
	( \alpha_5^k+ \alpha_6^k) \mathbb{D}^k\nn^{k}\cdot \mathbb{D}^{k}\delta \m^k+		\\+
	(\delta \alpha_3^k -\delta \alpha_2^k)\mcN^{k+1}\cdot \mcN^k+
	( \alpha_3^k - \alpha_2^k)\delta \mcN^{k}\cdot \mcN^k+
	( \alpha_3^k - \alpha_2^k) \mcN^{k}\cdot \delta\mcN^k.
\end{align*}
We first remark that
\begin{align*}
	\| \delta \mcN^k 			\|_{L^2_t\BB_{2,1}^{\frac{\dd}{2}-1}\cap \BB_{2,1}^\frac{\dd}{2}} \lesssim
	\| \partial_t \delta \m^k 	\|_{L^2_t\BB_{2,1}^{\frac{\dd}{2}-1}\cap \BB_{2,1}^\frac{\dd}{2}}		+
	\| \delta \uu^k				\|_{L^\infty_t\BB_{2,1}^{\frac{\dd}{2}-1}\cap \BB_{2,1}^\frac{\dd}{2}} 
	\| \nabla \m^{k+1}			\|_{L^2_t\BB_{2,1}^{\frac{\dd}{2}-1}\cap \BB_{2,1}^\frac{\dd}{2}}		+\\+
	\| \uu^k						\|_{L^\infty_t\BB_{2,1}^{\frac{\dd}{2}-1}\cap \BB_{2,1}^\frac{\dd}{2}} 
	\| \nabla \delta \m^{k+1}	\|_{L^2_t\BB_{2,1}^{\frac{\dd}{2}-1}\cap \BB_{2,1}^\frac{\dd}{2}}		+
	\| \delta \Omega^k			\|_{L^2_t\BB_{2,1}^{\frac{\dd}{2}-1}\cap \BB_{2,1}^\frac{\dd}{2}}		+
	\|		  \Omega^{k+1}		\|_{L^2_t\BB_{2,1}^{\frac{\dd}{2}-1}\cap \BB_{2,1}^\frac{\dd}{2}}{\scriptstyle\times}\\
	{\scriptstyle\times}
	\| \delta	\m^k				\|_{L^\infty_t\BB_{2,1}^\frac{\dd}{2}}
	\lesssim
	\frac{1}{\bar{\alpha}_4}\|\delta \uu^k\|_{\X_1} + \| \delta \m^k \|_{\X_2}
	\lesssim
	\ee^3\|\delta \uu^k\|_{\X_1} + \| \delta \m^k \|_{\X_2}.
\end{align*}
Hence, a direct computation leads to the following inequality
\begin{align*}
	\| \delta \sigma^{{\rm L},k}:\mathbb{D}^{k+1}+\sigma^{{\rm L},k}:\delta \mathbb{D}^k \|_{L^1_t\BB_{2,1}^{\frac{\dd}{2}-2}\cap \BB_{2,1}^\frac{\dd}{2}}
	\lesssim 
	\| \delta \omega^k \|_{L^\infty_t\BB_{2,1}^\frac{\dd}{2}}
	\Big\{
		(1+\|(\m^k,\,\m^{k+1})\|_{L^\infty_t\BB_{2,1}^\frac{\dd}{2}}^2){\scriptstyle\times}\\
	{\scriptstyle\times}
	(1+\| (\mcN^k,\,\mcN^{k+1} ) \|_{L^2_t\BB_{2,1}^{\frac{\dd}{2}-1}\cap \BB_{2,1}^\frac{\dd}{2}}^2)
	\| (\nabla \uu^k,\,\nabla \uu^{k+1})\|_{L^2_t\cap \BB_{2,1}^{\frac{\dd}{2}-1}\cap\BB_{2,1}^\frac{\dd}{2}}+
	\| (\mcN^k,\,\mcN^{k+1} ) \|_{L^2_t\BB_{2,1}^{\frac{\dd}{2}-1}\cap \BB_{2,1}^\frac{\dd}{2}}^2
	\Big\}+\\+
	\| \delta \m^k \|_{L^\infty_t\BB_{2,1}^\frac{\dd}{2}}
	(1+\|\omega^k\|_{L^\infty_t\BB_{2,1}^\frac{\dd}{2}})
	\Big\{
		(1+\| (\m^k,\,\m^{k+1} ) \|_{L^\infty_t \BB_{2,1}^\frac{\dd}{2}}^3)\| (\nabla \uu^k,\,\nabla \uu^{k+1})\|_{L^2_t\cap \BB_{2,1}^{\frac{\dd}{2}-2}\cap\BB_{2,1}^\frac{\dd}{2}}^2+\\+
		\| (\nabla \uu^k,\,\nabla \uu^{k+1})\|_{L^2_t\cap \BB_{2,1}^{\frac{\dd}{2}-1}\cap\BB_{2,1}^\frac{\dd}{2}}
		\| (\mcN^k,\,\mcN^{k+1} ) \|_{L^2_t\BB_{2,1}^{\frac{\dd}{2}-1}\cap \BB_{2,1}^\frac{\dd}{2}}
	\Big\}+	
	(1+\|\omega^k\|_{L^\infty_t\BB_{2,1}^\frac{\dd}{2}}){\scriptstyle\times}\\
	{\scriptstyle\times}
	\| \nabla \delta \uu^k \|_{L^2_t\cap \BB_{2,1}^{\frac{\dd}{2}-1}\cap\BB_{2,1}^\frac{\dd}{2}} (1+\| (\m^k,\,\m^{k+1} ) \|_{L^\infty_t \BB_{2,1}^\frac{\dd}{2}}^4)
	\| (\nabla \uu^k,\,\nabla \uu^{k+1},\,\mcN^k,\,\mcN^{k+1})\|_{L^2_t\cap \BB_{2,1}^{\frac{\dd}{2}-1}\cap\BB_{2,1}^\frac{\dd}{2}} + 	\\+
	\|  \delta \mcN^k \|_{L^2_t\cap \BB_{2,1}^{\frac{\dd}{2}-1}\cap\BB_{2,1}^\frac{\dd}{2}}
	(1+\|\omega^k\|_{L^\infty_t\BB_{2,1}^\frac{\dd}{2}})\Big\{
	\|(\mcN^k,\,\mcN^{k+1})\|_{L^2_t\cap \BB_{2,1}^{\frac{\dd}{2}-1}\cap\BB_{2,1}^\frac{\dd}{2}} +\\+
	(1+\| (\m^k,\,\m^{k+1} ) \|_{L^\infty_t \BB_{2,1}^\frac{\dd}{2}})
	\| (\nabla \uu^k,\,\nabla \uu^{k+1},\,\mcN^k,\,\mcN^{k+1})\|_{L^2_t\cap \BB_{2,1}^{\frac{\dd}{2}-1}\cap\BB_{2,1}^\frac{\dd}{2}}
	\lesssim\\\lesssim
	\ee^2\Big\{\ee^3\|\delta \uu^k\|_{\X_1} + \| \delta \m^k \|_{\X_2} + \| \delta \omega\|_{\X_3}\Big\}
	\lesssim
	\ee^{k+3}.
\end{align*}
Combining all the previous inequalities together with \eqref{ineq:the-end} finally leads to
\begin{equation*}
	\|\, \delta \omega^{k+1}\, \|_{\X_3}\lesssim \ee^{k+3},
\end{equation*}
which concludes the proof of inequalities in \eqref{ineq:main-temp-ineq}.
%%%%%%%%%%%%%%%%%%%%%%%%%%%%%%%%%%%%%%%%%%%%%%%%%%%%%%%%%%%%%%%%%%%%%%%%%%%%%%%%%%%%%%%%%%%%%%%%%%%%%%%%%%%%%
%%%%%%%%%%%%%%%%%%%%%%%%%%%%%%%%%%%%%%%%%%%%%%%%%%%%%%%%%%%%%%%%%%%%%%%%%%%%%%%%%%%%%%%%%%%%%%%%%%%%%%%%%%%%%
%					Passing to the limit
%%%%%%%%%%%%%%%%%%%%%%%%%%%%%%%%%%%%%%%%%%%%%%%%%%%%%%%%%%%%%%%%%%%%%%%%%%%%%%%%%%%%%%%%%%%%%%%%%%%%%%%%%%%%%
%%%%%%%%%%%%%%%%%%%%%%%%%%%%%%%%%%%%%%%%%%%%%%%%%%%%%%%%%%%%%%%%%%%%%%%%%%%%%%%%%%%%%%%%%%%%%%%%%%%%%%%%%%%%%
\subsection{Passage to the limit}\label{sec:passagetothelimit}$\,$

\smallskip\noindent
We claim that the uniform estimates given by Proposition \ref{prop:the_core} allow us to pass to the limit as $k$ goes to $\infty$. Indeed, thanks to inequality \eqref{prop:the_core_ineq2}, the sequence $(\uu^k,\,\m^k,\omega^k)_\NN$ is a Cauchy sequence in the Banach space $\X_1\times\X_2\times\X_3$. We then achieve the following strong convergences:
\begin{equation*}
\begin{alignedat}{800}
	\uu^k \;&\rightarrow \; \uu\quad&&\text{in}\quad \Big(\,L^\infty_t\BB_{2,1}^{\frac{\dd}{2}-1}&&&&\cap \BB_{2,1}^\frac{\dd}{2}\,&&&&&&&&\Big)
	,\;
	(\,\partial_t\uu^k,\,&&&&&&&&&&&&&&&&\Delta \uu^k\,&&&&&&&&&&&&&&&&&&&&&&&&&&&&&&&&)\;\rightarrow \;
	(\,\partial_t\uu,\,&&&&&&&&&&&&&&&&&&&&&&&&&&&&&&&&&&&&&&&&&&&&&&&&&&&&&&&&&&&&&&&&\Delta \uu\,&&&&&&&&&&&&&&&&&&&&&&&&&&&&&&&&&&&&&&&&&&&&&&&&&&&&&&&&&&&&&&&&&&&&&&&&&&&&&&&&&&&&&&&&&&&&&&&&&&&&&&&&&&&&&&&&&&&&&&&&&&&&&&&&)
	\quad \text{in}\quad \Big(\,L^1_t\BB_{2,1}^{\frac{\dd}{2}-1}
	&&&&&&&&&&&&&&&&&&&&&&&&&&&&&&&&&&&&&&&&&&&&&&&&&&&&&&&&&&&&&&&&&&&&&&&&&&&&&&&&&&&&&&&&&&&&&&&&&&&&&&&&&&&&&&&&&&&&&&&&&&&&&&&&&&&&&&&&&&&&&&&&&&&&&&&&&&&&&&&&&&&&&&&&&&&&&&&&&&&&&&&&&&&&&&&&&&&&&&&&&&&&&&&&&&&&&&&&&&&&&&&&&&&&&&&&&&&&&&&&&&&&&&&&&&&&&&&&\cap \BB_{2,1}^\frac{\dd}{2}\,
	&&&&&&&&&&&&&&&&&&&&&&&&&&&&&&&&&&&&&&&&&&&&&&&&&&&&&&&&&&&&&&&&&&&&&&&&&&&&&&&&&&&&&&&&&&&&&&&&&&&&&&&&&&&&&&&&&&&&&&&&&&&&&&&&&&&&&&&&&&&&&&&&&&&&&&&&&&&&&&&&&&&&&&&&&&&&&&&&&&&&&&&&&&&&&&&&&&&&&&&&&&&&&&&&&&&&&&&&&&&&&&&&&&&&&&&&&&&&&&&&&&&&&&&&&&&&&&&&&&&&&&&&&&&&&&&&&&&&&&&&&&&&&&&&&&&&&&&&&&&&&&&&&&&&&&&&&&&&&&&&&&&&&&&&&&&&&&&&&&&&&&&&&&&&&&&&&&&&&&&&&&&&&&&&&&&&&&&&&&&&&&&&&&&&&&&&&&&&&&&&&&&&&&&&&&&&&&&&&&&&&&&&&&&&&&&&&&&&&&&&&&&&&&&&&&&&&&&&&&&&&&&&&&&&&&&&&&&&&&&&&&&&&&&&&&&&&&&&&&&&&&&&&&&&&&&&\Big),\\
	\m^k \;&\rightarrow \; \m\quad&&\text{in}\quad \Big(\,L^\infty_t\BB_{2,1}^{\frac{\dd}{2}}&&&&\cap \BB_{2,1}^{\frac{\dd}{2}+1}\,&&&&&&&&\Big)
	,\;
	(\,\partial_t\m^k,\,&&&&&&&&&&&&&&&&\Delta \m^k\,&&&&&&&&&&&&&&&&&&&&&&&&&&&&&&&&)\;\rightarrow \;
	(\,\partial_t\m,\,&&&&&&&&&&&&&&&&&&&&&&&&&&&&&&&&&&&&&&&&&&&&&&&&&&&&&&&&&&&&&&&&\Delta \m\,&&&&&&&&&&&&&&&&&&&&&&&&&&&&&&&&&&&&&&&&&&&&&&&&&&&&&&&&&&&&&&&&&&&&&&&&&&&&&&&&&&&&&&&&&&&&&&&&&&&&&&&&&&&&&&&&&&&&&&&&&&&&&&&&)
	\quad \text{in}\quad \Big(\,L^1_t\BB_{2,1}^{\frac{\dd}{2}}
	&&&&&&&&&&&&&&&&&&&&&&&&&&&&&&&&&&&&&&&&&&&&&&&&&&&&&&&&&&&&&&&&&&&&&&&&&&&&&&&&&&&&&&&&&&&&&&&&&&&&&&&&&&&&&&&&&&&&&&&&&&&&&&&&&&&&&&&&&&&&&&&&&&&&&&&&&&&&&&&&&&&&&&&&&&&&&&&&&&&&&&&&&&&&&&&&&&&&&&&&&&&&&&&&&&&&&&&&&&&&&&&&&&&&&&&&&&&&&&&&&&&&&&&&&&&&&&&&\cap \BB_{2,1}^{\frac{\dd}{2}+1}\,&&&&&&&&&&&&&&&&&&&&&&&&&&&&&&&&&&&&&&&&&&&&&&&&&&&&&&&&&&&&&&&&&&&&&&&&&&&&&&&&&&&&&&&&&&&&&&&&&&&&&&&&&&&&&&&&&&&&&&&&&&&&&&&&&&&&&&&&&&&&&&&&&&&&&&&&&&&&&&&&&&&&&&&&&&&&&&&&&&&&&&&&&&&&&&&&&&&&&&&&&&&&&&&&&&&&&&&&&&&&&&&&&&&&&&&&&&&&&&&&&&&&&&&&&&&&&&&&&&&&&&&&&&&&&&&&&&&&&&&&&&&&&&&&&&&&&&&&&&&&&&&&&&&&&&&&&&&&&&&&&&&&&&&&&&&&&&&&&&&&&&&&&&&&&&&&&&&&&&&&&&&&&&&&&&&&&&&&&&&&&&&&&&&&&&&&&&&&&&&&&&&&&&&&&&&&&&&&&&&&&&&&&&&&&&&&&&&&&&&&&&&&&&&&&&&&&&&&&&&&&&&&&&&&&&&&&&&&&&&&&&&&&&&&&&&&&&&&&&&&&&&&&&&&&&&&\Big),\\
	\omega^k \;&\rightarrow \; \omega\quad&&\text{in}\quad \Big(\,L^\infty_t\BB_{2,1}^{\frac{\dd}{2}-2}&&&&\cap \BB_{2,1}^\frac{\dd}{2}\,&&&&&&&&\Big)
	,\;
	(\,\partial_t\omega,\,&&&&&&&&&&&&&&&&\Delta \omega^k\,&&&&&&&&&&&&&&&&&&&&&&&&&&&&&&&&)\;\rightarrow \; 
	(\,\partial_t\omega,\,&&&&&&&&&&&&&&&&&&&&&&&&&&&&&&&&&&&&&&&&&&&&&&&&&&&&&&&&&&&&&&&&\Delta \omega&&&&&&&&&&&&&&&&&&&&&&&&&&&&&&&&&&&&&&&&&&&&&&&&&&&&&&&&&&&&&&&&&&&&&&&&&&&&&&&&&&&&&&&&&&&&&&&&&&&&&&&&&&&&&&&&&&&&&&&&&&&&&&&&)	
	\quad \text{in}\quad \Big(\,L^1_t\BB_{2,1}^{\frac{\dd}{2}-2}
	&&&&&&&&&&&&&&&&&&&&&&&&&&&&&&&&&&&&&&&&&&&&&&&&&&&&&&&&&&&&&&&&&&&&&&&&&&&&&&&&&&&&&&&&&&&&&&&&&&&&&&&&&&&&&&&&&&&&&&&&&&&&&&&&&&&&&&&&&&&&&&&&&&&&&&&&&&&&&&&&&&&&&&&&&&&&&&&&&&&&&&&&&&&&&&&&&&&&&&&&&&&&&&&&&&&&&&&&&&&&&&&&&&&&&&&&&&&&&&&&&&&&&&&&&&&&&&&&\cap \BB_{2,1}^{\frac{\dd}{2}}\,&&&&&&&&&&&&&&&&&&&&&&&&&&&&&&&&&&&&&&&&&&&&&&&&&&&&&&&&&&&&&&&&&&&&&&&&&&&&&&&&&&&&&&&&&&&&&&&&&&&&&&&&&&&&&&&&&&&&&&&&&&&&&&&&&&&&&&&&&&&&&&&&&&&&&&&&&&&&&&&&&&&&&&&&&&&&&&&&&&&&&&&&&&&&&&&&&&&&&&&&&&&&&&&&&&&&&&&&&&&&&&&&&&&&&&&&&&&&&&&&&&&&&&&&&&&&&&&&&&&&&&&&&&&&&&&&&&&&&&&&&&&&&&&&&&&&&&&&&&&&&&&&&&&&&&&&&&&&&&&&&&&&&&&&&&&&&&&&&&&&&&&&&&&&&&&&&&&&&&&&&&&&&&&&&&&&&&&&&&&&&&&&&&&&&&&&&&&&&&&&&&&&&&&&&&&&&&&&&&&&&&&&&&&&&&&&&&&&&&&&&&&&&&&&&&&&&&&&&&&&&&&&&&&&&&&&&&&&&&&&&&&&&&&&&&&&&&&&&&&&&&&&&&&&&&&&\Big),
\end{alignedat}
\end{equation*}
for a set of functions $(\uu,\,\m,\,\omega)$ in $\X_1\times \X_2\times \X_3$.
Furthermore, any Oseen-Frank coefficient and Leslie viscosity strongly converges in $L^\infty_t\BB_{2,1}^{\dd/2}$ to the respective term.
More precisely, we remark that for any smooth function $f$, the sequence $(\,f(\omega^k)\,)_\NN$ strongly converges to 
$f(\omega)$ in $L^\infty_t\BB_{2,1}^{\dd/2}$, since there exists a positive constant $C$, which does not depend on $k$, such that
\begin{equation*}
	\| f(\,\omega^{k}\,) - f(\,\omega\,) \|_{L^\infty_t\BB_{2,1}^\frac{\dd}{2}}\leq C\| \omega^{k} - \omega \|_{L^\infty_t\BB_{2,1}^\frac{\dd}{2}}\quad \rightarrow 0\quad \text{as}\quad k\rightarrow \infty.
\end{equation*}
We gather that coefficients $\alpha_i(\omega^k)$ and $K_i(\omega^k)$ strongly converge in $L^\infty_t\BB_{2,1}^{\dd/2}$ to $\alpha_i(\omega)$ and $K_i(\omega)$, respectively and for any index $i$. 

\noindent\smallskip
Thus, by passing to the limit as $k\rightarrow+\infty$, it readily follows that $(\uu,\,\nn,\,\Tt) = (\uu,\,\bar{\nn}+\m,\,\bar{\Tt}+\omega)$ is a classical solution of \eqref{intro:main_system-wp}, with the exception of the unitary constraint $|\nn(t,x)|=1$. Indeed any non-linear term depending on $(\uu^k,\,\m^k,\,\omega^k)$ strongly converges to the corresponding non-linear term in $(\uu,\,\m,\,\omega)$, in the class of affinity given by $\X_1\times \X_2\times \X_3$.

\smallskip\noindent
It remains to prove the unitary constraint on the director field $|\nn(t,x)|=1$, for any $(t,x)\in \RR_+\times\RR^\dd$. Recalling that the Lagrangian multiplier $\beta$ in the angular momentum of \eqref{intro:main_system-wp} is defined as  
$\beta = {\rm h} \cdot \nn = \nn\cdot \delta \FF/\delta \nn$, we multiply the director equation in \eqref{intro:main_system-wp} by $\nn$, obtaining
\begin{equation}
	\gamma_1\Big(\partial_t \frac{|\nn|^2-1}{2} + \uu\cdot \nabla \frac{|\nn|^2-1}{2}\Big) - 2\beta \frac{|\nn|^2-1}{2}=0,\quad \text{with}\quad \big( |\nn|^2-1\big)\big|_{t=0} = 0.
\end{equation}
Since $|\nn|^2-1$ satisfies a linear transport equation with damping term, the uniqueness of the equation yields $|\nn(t,x)|^2-1=0$ everywhere, from which we recover the unitary constraint $\nn\cdot \nn=1$.
%%%%%%%%%%%%%%%%%%%%%%%%%%%%%%%%%%%%%%%%%%%%%%%%%%%%%%%%%%%%%%%%%%%%%%%%%%%%%%%%%%%%%%%%%%%%%%%%%%%%%%%%%%%%%
%%%%%%%%%%%%%%%%%%%%%%%%%%%%%%%%%%%%%%%%%%%%%%%%%%%%%%%%%%%%%%%%%%%%%%%%%%%%%%%%%%%%%%%%%%%%%%%%%%%%%%%%%%%%%
%					Uniqueness
%%%%%%%%%%%%%%%%%%%%%%%%%%%%%%%%%%%%%%%%%%%%%%%%%%%%%%%%%%%%%%%%%%%%%%%%%%%%%%%%%%%%%%%%%%%%%%%%%%%%%%%%%%%%%
%%%%%%%%%%%%%%%%%%%%%%%%%%%%%%%%%%%%%%%%%%%%%%%%%%%%%%%%%%%%%%%%%%%%%%%%%%%%%%%%%%%%%%%%%%%%%%%%%%%%%%%%%%%%%
\subsection{Uniqueness} $\,$

\noindent
The scope of the present paragraph is to state the uniqueness of solutions to system \eqref{intro:main_system-wp}, in the class of affinity provided by Theorem \ref{main-thm3}. A tedious but straightforward computation yields the following proposition to be satisfied:
\begin{prop}\label{prop:uniqueness}
	We consider a set of initial data $(\uu_0,\,\nn_0,\,\Tt_0)$ as in Theorem \ref{main-thm3} and we denote by $(\uu^j,\,\nn^j,\,\Tt^j)$  two classical solutions to system \eqref{intro:main_system-wp}, which belong to $\X_1\times (\bar{\nn}+\X_2)\times (\bar{\Tt}+\X_3)$, for a positive constant $\bar{\Tt}$ and a fixed unit 
	vector $\bar \nn$. Defining $\delta \uu= \uu^1-\uu^2$, $\delta \nn = \nn^1 - \nn^2$ and $\delta \Tt = \Tt^1 - \Tt^2$, the following inequality holds:
	\begin{equation}\label{ineq:uniqueness}
		\|\,\delta \uu\,\|_{\X_{1}} + \|\,\delta \nn\,\|_{\X_{2}} + \|\,\delta \Tt\,\|_{\X_{3}} \lesssim \ee^2\Big[\|\,\delta \uu\,\|_{\X_{1}} + \|\,\delta \nn\,\|_{\X_{2}} + \|\,\delta \Tt\,\|_{\X_{3}}\Big].
	\end{equation}
\end{prop}
\noindent
The proof of Proposition \ref{prop:uniqueness} is equivalent to that of inequality \eqref{prop:the_core_ineq2} in Proposition \ref{prop:the_core}. Indeed we remark that $(\delta \uu,\,\delta \m,\,\delta \omega)$ is a solution of
\begin{equation*}
	\begin{cases}
		\,\partial_t \delta\uu\,+\,\uu^1\cdot \nabla \delta \uu\,+\,\delta \uu\cdot \nabla  \uu^{2}\,-\Aa[\,\delta \uu\,]\,+\nabla \delta \pre\,=\,\delta F_1, \\
		\,\Div\,\delta\uu = 0,\\
		\,\partial_t \delta\m\,+\,\uu^1\cdot \nabla \delta \m\,+\,\delta \uu\cdot \nabla  \m^{2}\,-\Bb[\,\delta \m\,]\,=\,\delta F_2,\\
		\,\partial_t \delta \omega  \,+\,\uu^{1}\cdot \nabla \delta \omega\,+\,\delta \uu\cdot \nabla \omega^{2}-\,\Cc[\, \delta \omega\,]\,=\,\delta F_3,\\
		\,(\,\uu,\,\w,\,\m\,)
							\big|_{\,t=0}=(\,0,\,0,\,0\,),
	\end{cases}
\end{equation*}
where $\delta F_1=\Div\,\{\,\delta\sigma^{{\rm E}}\,+\,\delta \sigma^{{\rm L}}\}$ and the driving terms $\delta F_2$ and $\delta F_3$ are defined by means of \eqref{def:deltaF2k} and \eqref{def:deltaF3k}, respectively,
replacing $( \uu^{k+1},\,\m^{k+1},\, \omega^{k+1})$ by $( \uu^1,\, \m^1,\, \omega^1)$ and 
$( \uu^{k},\, \m^{k},\, \omega^{k})$ by $( \uu^{2},\, \m^{2},\, \omega^{2})$. With similar arguments as the ones used for proving inequality \eqref{prop:the_core_ineq2}, we gather
\begin{align*}
	\|\,\delta \uu\,\|_{\X_{1}}\,+\,\| \,\nabla \delta \pre\,\|_{L^1_t\BB_{2,1}^{\frac{\dd}{2}-1}\cap\BB_{2,1}^\frac{\dd}{2}} +\, \|\,\delta \nn\,\|_{\X_{2}} \,+\, \|\,\delta \Tt\,\|_{\X_{3}} \lesssim 
	\|\, \delta F_1 \, \|_{L^1_t\BB_{2,1}^{\frac{\dd}{2}-1}\cap \BB_{2,1}^\frac{\dd}{2}}\,+\\+
	\|\, \delta F_2 \, \|_{L^1_t\BB_{2,1}^{\frac{\dd}{2}}\cap \BB_{2,1}^{\frac{\dd}{2}-1}}\,+\,
	\|\, \delta F_3 \, \|_{L^1_t\BB_{2,1}^{\frac{\dd}{2}-1}\cap \BB_{2,1}^\frac{\dd}{2}}
	\lesssim \Big[ \|\, (\uu^1,\,\uu^2\,)\,\|_{\X_1}\,+\,\|\, (\m^1,\,\m^2\,)\,\|_{\X_2}\,+\\+\,\|\, (\omega^1,\,\omega^2\,)\,\|_{\X_3}\,\Big]
	\Big[
		\|\,\delta \uu\,\|_{\X_{1}}\, +\, \|\,\delta \nn\,\|_{\X_{2}} \,+\, \|\,\delta \Tt\,\|_{\X_{3}} 
	\Big].
\end{align*}
Since $(\uu^1,\,\m^1,\,\omega^1)=(\uu^1,\,\nn^1-\bar \nn,\,\Tt^1-\bar \Tt)$ and $(\uu^2,\,\m^2,\,\omega^2)=(\uu^2,\,\nn^2-\bar \nn,\,\Tt^2-\bar \Tt)$ satisfy the smallness condition \eqref{ineq:smallness-condition-main-thm}, then inequality \eqref{ineq:uniqueness} holds. Hence, assuming the positive constant $\ee$ small enough, we finally achieve
\begin{equation*}
		\|\,\delta \uu\,\|_{\X_{1}} + \|\,\delta \nn\,\|_{\X_{2}} + \|\,\delta \Tt\,\|_{\X_{3}} =0,
\end{equation*}
which finally leads to the uniqueness result of Theorem \ref{main-thm3}.
\addtocontents{toc}{\protect\setcounter{tocdepth}{1}}
%%%%%%%%%%%%%%%%%%%%%%%%%%%%%%%%%%%%%%%%%%%%%%%%%%%%%%%%%%%%%%%%%%%%%%%%%%%%%%%%%%%%%%%%%%%%%%%%%%%%%%%%%%%%%
%%%%%%%%%%%%%%%%%%%%%%%%%%%%%%%%%%%%%%%%%%%%%%%%%%%%%%%%%%%%%%%%%%%%%%%%%%%%%%%%%%%%%%%%%%%%%%%%%%%%%%%%%%%%%
%					Appendix
%%%%%%%%%%%%%%%%%%%%%%%%%%%%%%%%%%%%%%%%%%%%%%%%%%%%%%%%%%%%%%%%%%%%%%%%%%%%%%%%%%%%%%%%%%%%%%%%%%%%%%%%%%%%%
%%%%%%%%%%%%%%%%%%%%%%%%%%%%%%%%%%%%%%%%%%%%%%%%%%%%%%%%%%%%%%%%%%%%%%%%%%%%%%%%%%%%%%%%%%%%%%%%%%%%%%%%%%%%%
\section{Appendix}\label{Appendix}

\subsection{Proof of inequality \eqref{ineq:lemma1ELtensors}}$\,$

\noindent
This paragraph is devoted to  the proof of inequality \ref{ineq:lemma1ELtensors}. Thanks to Definition \eqref{def:E-stress}, the Ericksen stress tensor $\sigma^{{\rm E},\, k}$ can be formulated as follows:
\begin{equation*}
\begin{aligned}
	\sigma^{\,{\rm E},\, k}\,=\,\tr\Big[ \nabla \m^k\,\Big]\frac{\partial {\rm W}^k_{\rm F}}{\partial \nabla \nn}=\,
	&-\,K_1^k\, \nabla \m^k\odot \nabla \m^k\,-\,K_2^k\,\Div\,\m^k\,\tr \nabla \m^k\,-
	\\&-\,K_3^k \Big[\,\Big(\,\nabla \m^k\odot\nabla\m^k\,\Big)(\bar{\nn}\,+\,\m^k)\,\Big]\otimes (\bar{\nn}\,+\,\m^k)
	\,-\,K^k_4\,\, \tr\big[\,(\,\nabla \m^k\,)^2\big].
\end{aligned}
\end{equation*}
Hence, we infer that
\begin{equation*}
	\|\Div\,\sigma^{\rm\, E,\,k}\|_{L^1_t \BB_{2,1}^{\frac{\dd}{2}-1}}
	=
	\|\sigma^{\rm\, E,\,k}\|_{L^1_t \BB_{2,1}^{\frac{\dd}{2}}}
	\lesssim
	\big(\,1\,+\,\|\,\w^k\|_{L^\infty_t\BB_{2,1}^\frac{\dd}{2}}\big)
	\big(1\,+\,
	\|\,\m^k\,\|_{L^\infty_t	\BB_{2,1}^\frac{\dd}{2}}^2\big)
	\|\,\nabla \m^k\,\|_{L^2_t\BB_{2,1}^\frac{\dd}{2}}^2
	\lesssim
	\ee^2,
\end{equation*}
together with
\begin{equation*}
\begin{aligned}
	\|\Div\,\sigma^{\rm\, E,\,k}\|_{L^1_t \BB_{2,1}^{\frac{\dd}{2}}}
	\lesssim
	\big(1\,+\,
	\|\,\m^k\,\|_{L^\infty_t	\BB_{2,1}^\frac{\dd}{2}}\big)
	\|\,\nabla \w^k\,\|_{L^2_t\BB_{2,1}^\frac{\dd}{2}}
	\|\,\nabla \m^k\,\|_{L^2_t\BB_{2,1}^\frac{\dd}{2}}+\\
	+
	\big(\,1\,+\,\|\,\w^k\|_{L^\infty_t\BB_{2,1}^\frac{\dd}{2}}\big)
	\|\,\nabla\m^k\,\|_{L^\infty_t	\BB_{2,1}^\frac{\dd}{2}}
	\big(1\,+\,
	\|\,\m^k\,\|_{L^\infty_t	\BB_{2,1}^\frac{\dd}{2}}\big)
	\|\,\nabla \m^k\,\|_{L^2_t\BB_{2,1}^\frac{\dd}{2}}^2+\\+
	\big(\,1\,+\,\|\,\w^k\|_{L^\infty_t\BB_{2,1}^\frac{\dd}{2}}\big)
	\big(1\,+\,
	\|\,\m^k\,\|_{L^\infty_t	\BB_{2,1}^\frac{\dd}{2}}^2\big)
	\|\,\nabla \m^k\,\|_{L^2_t\BB_{2,1}^\frac{\dd}{2}}
	\|\,\Delta \m^k\,\|_{L^2_t\BB_{2,1}^\frac{\dd}{2}}
	\lesssim
	\ee^3.
\end{aligned}
\end{equation*}
We now take into account the term related to the Leslie stress tensor $\sigma^{\rm\,L,\,k}$. 
We decompose such a tensor into six terms $\sigma^{\rm\,L,\,k}\,=\,\sum_{i=1}^6\sigma^{\rm\,L,\,k}_i$, where we denote by $\sigma^{\rm\,L,\,k}_i$ the tensor related to the coefficient $\alpha_i$, for any $i=1,\dots,6$.
An explicit formulation of these tensors is provided in what follows. We begin with $\sigma^{\rm\,L,\,k}_1$ which stands for
\begin{equation*}
\begin{alignedat}{8}
	\tilde \sigma^{\rm L}_1\,=
	%------------------------------------------------------------------------------------------------------------------------------
	%------------------------------------------------------------------------------------------------------------------------------
	\,\bar \alpha_1 \Big[\,\bar \nn \cdot \mathbb{D}\,\m\,+\,\m \cdot \mathbb{D}\,\bar \nn+
	\m \cdot \mathbb{D}\, \m\,\Big]\nn\otimes \nn\,
	%------------------------------------------------------------------------------------------------------------------------------
	+\,\bar \alpha_1 \Big[\,\bar \nn \cdot \mathbb{D}\,\m\,+\,\m \cdot \mathbb{D}\,\bar \nn\,+
	\m \cdot \mathbb{D}\, \m\,\Big]\m\otimes \nn+\\
	%------------------------------------------------------------------------------------------------------------------------------
	+\,\bar \alpha_1 \Big[\, \bar \nn \cdot \mathbb{D}\,\m\,+\,\m \cdot \mathbb{D}\,\bar \nn\,+\,
	\m \cdot \mathbb{D}\, \m\,\Big]\nn\otimes \m\,
	%------------------------------------------------------------------------------------------------------------------------------
	+\,\bar \alpha_1 \Big[\,\bar \nn \cdot \mathbb{D}\,\m\,+\,\m \cdot \mathbb{D}\,\bar \nn\,+\,
	\m \cdot \mathbb{D}\, \m\,\Big]\m\otimes \m\,+\\
	%------------------------------------------------------------------------------------------------------------------------------
	+\,\tilde \alpha_1(\w) [\,\nn \cdot \mathbb{D}\,\nn\,]\nn\otimes\nn.\,
\end{alignedat}
\end{equation*}
replacing $(\uu,\,\omega,\,\m) $ by  $(\uu^k,\,\omega^k,\,\m^k)$. We recall that $\alpha_i(\omega) =\bar{\alpha}_i\,+\tilde{\alpha}_i(\omega)$, for any $i=1,\dots,6$, where $\bar{\alpha}_i$ is a 
real constant and $\tilde{\alpha}_i(\omega)$ is a smooth function depending on $\omega$ such that $\tilde{\alpha}_i(0)=0$.
Hence, we infer that the approximate stress tensor $\sigma^{\rm\,L,\,k}_1$ satisfies the following inequalities:
\begin{equation*}
\begin{alignedat}{8}
 	\|\,\sigma^{\rm\,L,\,k}_1\,\|_{L^1_t \BB_{2,1}^\frac{\dd}{2}}\,
 	\lesssim
 	\Big[\,1+\|\,\m^k\,\|_{L^\infty_t\BB_{2,1}^\frac{\dd}{2}}^2\,\Big]
 	\Big[\,
	\,\|\,\m^k\,\|_{L^\infty_t\BB_{2,1}^\frac{\dd}{2}}\,+\|\,\m^k\,\|_{L^\infty_t\BB_{2,1}^\frac{\dd}{2}}^2\,
	\Big]
	\|\,\nabla \uu^k\,\|_{L^1_t \BB_{2,1}^{\frac{\dd}{2}}}\,
	+\,\\+\,
	\Big[\,1+\|\,\m^k\,\|_{L^\infty_t\BB_{2,1}^{\frac{\dd}{2}}}\,\Big]^4
		\|\,\w^k\,\|_{L^\infty_t\BB_{2,1}^{\frac{\dd}{2}}}	\|\,\nabla \uu^k\,\|_{L^1_t \BB_{2,1}^{\frac{\dd}{2}}}
		\lesssim \ee^3,
\end{alignedat}
\end{equation*}
together with 
\begin{equation*}
\begin{alignedat}{8}
 	\|\,\Div\,\sigma^{\rm\,L,\,k}_1\,\|_{L^1_t \BB_{2,1}^\frac{\dd}{2}}\,
 	\lesssim
 	\Big[\,1\,+\,\|\,\m^k\,\|_{L^\infty_t\BB_{2,1}^\frac{\dd}{2}}^3\,
	\Big]
	\|\,\nabla \m^k\,\|_{L^\infty_t\BB_{2,1}^\frac{\dd}{2}}
	\|\,\nabla \uu^k\,\|_{L^1_t \BB_{2,1}^{\frac{\dd}{2}}}\,
	+\,\\
	\Big[\,1+\|\,\m^k\,\|_{L^\infty_t\BB_{2,1}^\frac{\dd}{2}}^2\,\Big]
 	\Big[\,
	\,\|\,\m^k\,\|_{L^\infty_t\BB_{2,1}^\frac{\dd}{2}}\,+\|\,\m^k\,\|_{L^\infty_t\BB_{2,1}^\frac{\dd}{2}}^2\,
	\Big]
	\|\,\Delta \uu^k\,\|_{L^1_t \BB_{2,1}^{\frac{\dd}{2}}}
	\\+\,
	\Big[\,1+\|\,\m^k\,\|_{L^\infty_t\BB_{2,1}^{\frac{\dd}{2}}}\,\Big]^3
		\|\,\w^k\,\|_{L^\infty_t\BB_{2,1}^{\frac{\dd}{2}}}	\|\,\nabla \m^k\,\|_{L^\infty_t\BB_{2,1}^\frac{\dd}{2}}\|\,\nabla \uu^k\,\|_{L^1_t \BB_{2,1}^{\frac{\dd}{2}}}	
		+\\+
		\Big[\,1+\|\,\m^k\,\|_{L^\infty_t\BB_{2,1}^{\frac{\dd}{2}}}\,\Big]^4
		\|\,\nabla \w^k\,\|_{L^\infty_t\BB_{2,1}^{\frac{\dd}{2}}}\|\,\nabla \uu^k\,\|_{L^1_t \BB_{2,1}^{\frac{\dd}{2}}}	+\,\\+
		\Big[\,1+\|\,\m^k\,\|_{L^\infty_t\BB_{2,1}^{\frac{\dd}{2}}}\,\Big]^4
		\|\,\w^k\,\|_{L^\infty_t\BB_{2,1}^{\frac{\dd}{2}}}	\|\,\Delta \uu^k\,\|_{L^1_t \BB_{2,1}^{\frac{\dd}{2}}}		
		\lesssim \ee^3,
\end{alignedat}
\end{equation*}
From the previous relations we conclude that
\begin{equation}\label{bound2}
	\|\,\Div\,\sigma^{\rm\,L,\,k}_1\,\|_{L^1_t \BB_{2,1}^{\frac{\dd}{2}-1}\cap \BB_{2,1}^{\frac{\dd}{2}}}
	\, \lesssim \, \ee^3.
\end{equation}
We now take into account the second stress tensor $\sigma_2^{\rm L,k}$ , which is defined by means of
\begin{equation*}
\begin{aligned}
	\sigma_2^{\rm L,k}\,&=\,\bar{\alpha}_2^k\,\Big[\partial_t \m^k\,+\,\uu^k\cdot \nabla \m^k\,-\Omega\m^k\,\Big]\otimes \Big[ \,\bar \nn\,+\,\m^k\,\Big]\,+\\&+
	\tilde \alpha_2(\omega^k)\Big[\,\partial_t\m^k\,+\,\uu^k\cdot \nabla	\m^k\,+\Omega^k\,\m^k\,\Big]\otimes\big[\,\bar \nn\,+\,\m^k\,\big]\,+\,
	\big[\bar{\alpha}_2^k\,+\,\tilde{\alpha}_2(\omega^k)\big]\Omega\bar \nn\otimes\m^k.
\end{aligned}
\end{equation*}
We observe that the following estimates are fulfilled:
\begin{equation*}
\begin{aligned}
	\|\,\sigma^{\rm\,L,\,k}_2\,\|_{L^1_t \BB_{2,1}^\frac{\dd}{2}}
	\lesssim
	\Big[\,
		\|\,\partial_t \m^k\,\|_{L^1_t\BB_{2,1}^\frac{\dd}{2}}\,&+\,
		\|\,\uu^k\,\|_{L^2_t\BB_{2,1}^\frac{\dd}{2}}
		\|\,\nabla \m^k\,\|_{L^2_t\BB_{2,1}^\frac{\dd}{2}}\,+\\&+\,
		\|\,\nabla \uu^k\,\|_{L^1_t\BB_{2,1}^\frac{\dd}{2}}
	\Big]
	\Big[\,1+\|\,\m^k\,\|_{L^\infty_t\BB_{2,1}^\frac{\dd}{2}}\,\Big]
	\lesssim \ee^3
\end{aligned}
\end{equation*}
together with
\begin{equation*}
\begin{aligned}
	\|\,\Div\,\sigma^{\rm\,L,\,k}_2\,\|_{L^1_t \BB_{2,1}^\frac{\dd}{2}}
	\lesssim
	\Big[\,
		\|\,\partial_t \m^k\,\|_{L^1_t\BB_{2,1}^\frac{\dd}{2}}+
		\|\,\uu^k\,\|_{L^2_t\BB_{2,1}^\frac{\dd}{2}}
		\|\,\nabla \m^k\,\|_{L^2_t\BB_{2,1}^\frac{\dd}{2}}
		\|\,\nabla \uu^k\,\|_{L^1_t\BB_{2,1}^\frac{\dd}{2}}
	\Big]\|\,\nabla \m^k\,\|_{L^\infty_t\BB_{2,1}^\frac{\dd}{2}}+\\+\,
	\Big[\,
		\|\,\partial_t \nabla \m^k\,\|_{L^1_t\BB_{2,1}^\frac{\dd}{2}}\,+\,
		\|\,\nabla \uu^k\,\|_{L^2_t\BB_{2,1}^\frac{\dd}{2}}
		\|\,\nabla \m^k\,\|_{L^2_t\BB_{2,1}^\frac{\dd}{2}}\,+\,
		\|\, \uu^k\,\|_{L^2_t\BB_{2,1}^\frac{\dd}{2}}
		\|\,\Delta \m^k\,\|_{L^2_t\BB_{2,1}^\frac{\dd}{2}}\,+\\+\,
		\|\,\nabla \uu^k\,\|_{L^1_t\BB_{2,1}^\frac{\dd}{2}}
	\Big]
	\Big[\,1+\|\,\m^k\,\|_{L^\infty_t\BB_{2,1}^\frac{\dd}{2}}\,\Big]
	\lesssim \ee^3.
\end{aligned}
\end{equation*}
Hence we conclude that
\begin{equation}\label{bound3}
	\|\,\Div\,\sigma^{\rm\,L,\,k}_2\,\|_{L^1_t \BB_{2,1}^{\frac{\dd}{2}-1}\cap \BB_{2,1}^{\frac{\dd}{2}}}\lesssim \ee^3
\end{equation}
The third stress tensor $\sigma^{{\rm\,L},\,k}_3$ is defined by means of
\begin{equation*}
	\begin{aligned}
	\sigma_3^{\rm L\,k}\,&=\,\bar{\alpha}_3^k\,\Big[ \,\bar \nn\,+\,\m^k\,\Big]\otimes\Big[\partial_t \m^k\,+\,\uu^k\cdot \nabla \m^k\,-\Omega^k\m^k\,\Big]\,+\\&+
	\tilde \alpha_3(\omega^k)\big[\,\bar \nn\,+\,\m^k\,\big]\otimes\Big[\,\partial_t\m^k\,+\,\uu^k\cdot \nabla	\m^k\,+\Omega^k\,\m^k\,\Big]\,+\,
	\big[\bar{\alpha}_3\,+\,\tilde{\alpha}_3(\omega^k)\big]\m^k \otimes\Omega^k\bar \nn.
	\end{aligned}
\end{equation*}
Proceeding as for proving \eqref{bound3}, we infer that
\begin{equation}\label{bound4}
	\|\,\Div\,\sigma^{\rm\,L,\,k}_3\,\|_{L^1_t \BB_{2,1}^{\frac{\dd}{2}-1}\cap \BB_{2,1}^{\frac{\dd}{2}}}\lesssim \ee^3.
\end{equation}
Furthermore, by the definition of $\sigma^{{\rm L},\,k}_4\,:=\,\tilde{\alpha}_4(\omega)\mathbb{D}$, with $\tilde{\alpha}_4(0)=0$, we gather that
\begin{equation*}
	\|\,\Div\,\sigma^{\rm\,L,\,k}_4\,\|_{L^1_t \BB_{2,1}^{\frac{\dd}{2}-1}\cap \BB_{2,1}^{\frac{\dd}{2}}}
	\lesssim \|\,\omega^k\,\|_{L^\infty_t \BB_{2,1}^\frac{\dd}{2}}\|\,\nabla \uu^k\,\|_{L^1_t \BB_{2,1}^{\frac{\dd}{2}-1}\cap \BB_{2,1}^{\frac{\dd}{2}}}
	\lesssim \ee^3.
\end{equation*}
We now take into account the thensor $\sigma^{{\rm L}, k}_5$ whose explicit formula is given by
\begin{equation*}
\begin{aligned}
	\sigma^{{\rm L}, k}_5\,:=\,\tilde{\alpha}_5(\omega^k)\,\big[\bar \nn\,+\m^k\,\big]\otimes\mathbb{D}\big[\bar \nn\,+\m^k\,\big]
	\,+\,\bar \alpha_5\Big[\,\m^k\otimes	\mathbb{D}^k\m^k\,+\,\bar{\nn}\otimes \mathbb{D}^k\m^k\,+\,\m^k\otimes \mathbb{D}^k\bar{\nn}\,\Big],
\end{aligned}
\end{equation*}
from which we deduce
\begin{equation*}
\begin{aligned}
	\|\,\Div\,\sigma^{\rm\,L,\,k}_5\,\|_{L^1_t \BB_{2,1}^{\frac{\dd}{2}-1}}\lesssim 
	\|\,\sigma^{\rm\,L,\,k}_5\,\|_{L^1_t \BB_{2,1}^{\frac{\dd}{2}}}\lesssim	
	\Big(\,1\,+\,\|\,\m^k\,\|_{L^\infty_t \BB_{2,1}^\frac{\dd}{2}}\Big)^2
	\|\,\omega^k\,\|_{L^\infty_t \BB_{2,1}^\frac{\dd}{2}}\|\,\nabla \uu^k\,\|_{L^1_t \BB_{2,1}^{\frac{\dd}{2}}}\,+\\
	\,+\,
	\Big(\,\|\,\m^k\,\|_{L^\infty_t\BB_{2,1}^\frac{\dd}{2}}\,+\,\|\,\m^k\,\|_{L^\infty_t\BB_{2,1}^\frac{\dd}{2}}^2\Big)\|\,\nabla \uu^k\,\|_{L^1_t \BB_{2,1}^{\frac{\dd}{2}}}
	\lesssim \ee^3,
\end{aligned}
\end{equation*}
together with
\begin{equation*}
\begin{aligned}
	\|\,\Div\,\sigma^{\rm\,L,\,k}_5\,\|_{L^1_t \BB_{2,1}^{\frac{\dd}{2}}}
	\lesssim
	\Big(\,1\,+\,\|\,\m^k\,\|_{L^\infty_t \BB_{2,1}^\frac{\dd}{2}}\Big)^2
	\Big(\,1\,+\,\|\,\omega^k\,\|_{L^\infty_t \BB_{2,1}^\frac{\dd}{2}}\Big)
	\| \,\nabla \omega^k\,\|_{L^\infty_t \BB_{2,1}^\frac{\dd}{2}}
	\|\,\nabla \uu^k\,\|_{L^1_t \BB_{2,1}^{\frac{\dd}{2}}}
	\,+\\+\,	
	\Big(\,1\,+\,\|\,\m^k\,\|_{L^\infty_t \BB_{2,1}^\frac{\dd}{2}}\Big)
		\|\,\nabla \m^k\,\|_{L^\infty_t \BB_{2,1}^\frac{\dd}{2}}
		\|\,\omega^k\,\|_{L^\infty_t \BB_{2,1}^\frac{\dd}{2}}\|\,\nabla \uu^k\,\|_{L^1_t \BB_{2,1}^{\frac{\dd}{2}}}\,+\\
	\Big(\,1\,+\,\|\,\m^k\,\|_{L^\infty_t \BB_{2,1}^\frac{\dd}{2}}\Big)^2
	\|\,\omega^k\,\|_{L^\infty_t \BB_{2,1}^\frac{\dd}{2}}\|\,\Delta \uu^k\,\|_{L^1_t \BB_{2,1}^{\frac{\dd}{2}}}
	\lesssim 
	\ee^3.
\end{aligned}
\end{equation*}
Summarizing the previous two inequalities we finally get that
\begin{equation}\label{bound5}
		\|	\,	\Div\,\sigma^{\rm\,L,\,k}_5	\,	\|_{L^1_t \BB_{2,1}^{\frac{\dd}{2}-1}\cap \BB_{2,1}^{\frac{\dd}{2}}} \,\lesssim \, \ee^3.
\end{equation}
Similarly as for proving the above inequality, we deduce that the tensor $\sigma^{\rm\,L,\,k}_6$ 
\begin{equation*}
	\begin{aligned}
		\sigma^{{\rm L}, k}_6\,:=\,\tilde{\alpha}_6(\omega^k)\,\big[\bar \nn\,+\m^k\,\big]\otimes\mathbb{D}^k\big[\bar \nn\,+\m^k\,\big]
		\,+\,\bar \alpha_5\Big[\,\mathbb{D}^k\m^k\otimes \m^k\,+\,\mathbb{D}^k\bar{\nn}\otimes \m^k\,+\,\mathbb{D}^k\bar{\nn}\otimes \m^k\,\Big].
	\end{aligned}
\end{equation*}
fulfills the following inequality
\begin{equation}\label{bound6}
		\|	\,	\Div\,\sigma^{\rm\,L,\,k}_6	\,	\|_{L^1_t \BB_{2,1}^{\frac{\dd}{2}-1}\cap \BB_{2,1}^{\frac{\dd}{2}}} \,\lesssim \, \ee^3.
\end{equation}
Finally, summarizing inequalities \eqref{bound2}, \eqref{bound3}, \eqref{bound4}, \eqref{bound5} and \eqref{bound6}, the approximate Leslie stress tensor $\sigma^{{\rm L},k}$ fulfills the following inequality
\begin{equation}\label{bound7}
	\|	\,	\Div\,\sigma^{\rm\,L,\,k}	\,	\|_{L^1_t \BB_{2,1}^{\frac{\dd}{2}-1}\cap \BB_{2,1}^{\frac{\dd}{2}}} \,\lesssim \, \ee^3.
\end{equation}
This concludes the proof of inequality \eqref{ineq:lemma1ELtensors}.
\subsection{Proof of inequality \eqref{est:momeqL}}$\,$

\noindent
This paragraph is devoted to  the proof of inequality \eqref{est:momeqL}.
We first take into account the difference between two consecutive Ericksen Leslie tensors, namely
\begin{equation*}
\begin{aligned}
		\delta \sigma^{\,{\rm E},\, k}\,=\,-
		\,\delta 	K_1^k\, \nabla 			\m^{k+1	}\odot \nabla 			\m^{k+1	}\,-
		\,		 	K_1^k\, \nabla \delta 	\m^{k 	}\odot \nabla 			\m^{k+1	}\,-
		\,			K_1^k\, \nabla 			\m^{k	}\odot \nabla \delta 	\m^{k	}\,\\
	%------------------------------------------------------------------------------------------------------------------------------------------------------------
		-\,\delta 	K_2^k\,\Div\,	 	\m^{k+1	}\,\tr \nabla 		\m^{k+1	}\,
		-\,			K_2^k\,\Div\,\delta	\m^{k	}\,\tr \nabla 		\m^{k+1	}\,
		-\,			K_2^k\,\Div\,		\m^{k	}\,\tr \nabla \delta	\m^{k	}\,	-\\
	%------------------------------------------------------------------------------------------------------------------------------------------------------------			
		-\,\delta 	K_3^k \Big[\,\Big(\,\nabla 			\m^{k+1	}	\odot\nabla			\m^{k+1	}\,	\Big)(\bar{\nn}\,+\,\m^{k+1	})\,\Big]\otimes (\bar{\nn}\,+\,\m^{k+1	})\\
		-\,			K_3^k \Big[\,\Big(\,\nabla \delta  	\m^k			\odot\nabla			\m^{k+1	}\,	\Big)(\bar{\nn}\,+\,\m^{k+1	})\,\Big]\otimes (\bar{\nn}\,+\,\m^{k+1	})\\
		-\,			K_3^k \Big[\,\Big(\,\nabla 			\m^k			\odot\nabla \delta 	\m^k\,		\Big)(\bar{\nn}\,+\,\m^{k+1	})\,\Big]\otimes (\bar{\nn}\,+\,\m^{k+1	})\\
		-\,			K_3^k \Big[\,\Big(\,\nabla 			\m^k			\odot\nabla			\m^k	\,		\Big)(\delta \m^k)\,\Big]\otimes (\bar{\nn}\,+\,\m^{k+1	})
		-\,			K_3^k \Big[\,\Big(\,\nabla 			\m^k			\odot\nabla			\m^k	\,		\Big)(\bar{\nn}\,+\,\m^{k})\,\Big]\otimes (\delta \m^{k	})\\
	%------------------------------------------------------------------------------------------------------------------------------------------------------------	
		-\,\delta 	K^k_4\,\, \tr\big[\,\nabla 			\m^{k+1}\,	\nabla 			\m^{k+1}\,\big]
		-\,			K^k_4\,\, \tr\big[\,\nabla\delta  	\m^k\,		\nabla 			\m^{k+1}\,\big]
		-\,			K^k_4\,\, \tr\big[\,\nabla 			\m^k\,		\nabla	\delta  \m^k\,\big].
\end{aligned}
\end{equation*}
Then a direct computation leads to
\begin{equation*}
\begin{aligned}
	\|\,\Div\,\delta \sigma^{\,{\rm E},\, k}\,\|_{L^1_t\BB_{2,1}^{\frac{\dd}{2}-1}}	
	\lesssim	
	\|\, \delta \omega^k\,\|_{L^\infty_t\BB_{2,1}^\frac{\dd}{2}}
	\|\, \nabla \m^{k+1}\,\|_{L^2_t\BB_{2,1}^\frac{\dd}{2}}^2\,+\,
	\|\, \omega^k\,\|_{L^\infty_t\BB_{2,1}^\frac{\dd}{2}}\|\, \nabla (\m^k,\,\m^{k+1})\,\|_{L^2_t\BB_{2,1}^\frac{\dd}{2}}{\scriptstyle\times}\\{\scriptstyle \times}
	\|\, \nabla \delta \m^{k}\,\|_{L^2_t\BB_{2,1}^\frac{\dd}{2}}\,+\,
	\|\,\delta \omega^k\,\|_{L^\infty\BB_{2,1}^\frac{\dd}{2}}\Big(\,1\,+\,\|\,(\,\m^k,\,\m^{k+1}\,)\,\|_{L^\infty_t\BB_{2,1}^\frac{\dd}{2}}\Big)^2\|\, \nabla \m^{k+1}\,\|_{L^2_t\BB_{2,1}^\frac{\dd}{2}}^2\,+\\
	+\,\|\,\omega^k\,\|_{L^\infty\BB_{2,1}^\frac{\dd}{2}}\Big(\,1\,+\,\|\,(\,\m^k,\,\m^{k+1}\,)\,\|_{L^\infty_t\BB_{2,1}^\frac{\dd}{2}}\Big)^2
	\|\, \nabla \delta \m^{k}\,\|_{L^2_t\BB_{2,1}^\frac{\dd}{2}}\|\, \nabla (\,\m^k,\,\m^{k+1}\,)\,\|_{L^2_t\BB_{2,1}^\frac{\dd}{2}}\,+\\
	+\,\|\,\omega^k\,\|_{L^\infty\BB_{2,1}^\frac{\dd}{2}}\Big(\,1\,+\,\|\,(\,\m^k,\,\m^{k+1}\,)\,\|_{L^\infty_t\BB_{2,1}^\frac{\dd}{2}}\Big)
	\|\,  \delta \m^{k}\,\|_{L^\infty_t\BB_{2,1}^\frac{\dd}{2}}\|\, \nabla (\,\m^k,\,\m^{k+1}\,)\,\|_{L^2_t\BB_{2,1}^\frac{\dd}{2}}^2,
\end{aligned}
\end{equation*}
which yields
\begin{equation*}
	\|\,\Div\,\delta \sigma^{\,{\rm E},\, k}\,\|_{L^1_t\BB_{2,1}^{\frac{\dd}{2}-1}}	\lesssim \ee^4 \Big[\,\|\,\delta \omega^k\,\|_{\X_2}\,+\,\|\,\delta \m^k\,\|_{\X_3}\,\Big]\leq \ee^{k+2}.
\end{equation*}
A similar result holds for the $L^1_t\BB_{2,1}^{\dd/2}$-norm of $\Div\,\delta \sigma^{\,{\rm E},\, k}$:
\begin{equation}\label{est:momeqE}
	\|\,\Div\,\delta \sigma^{\,{\rm E},\, k}\,\|_{L^1_t\BB_{2,1}^{\frac{\dd}{2}}}	\lesssim \ee^4 \Big[\,\|\,\delta \omega^k\,\|_{\X_2}\,+\,\|\,\delta \m^k\,\|_{\X_3}\,\Big]\leq \ee^{k+2}.
\end{equation}
We then focus on the difference between two consecutive Leslie stress tensor $\Div\,\tilde \sigma^{\,{\rm L},\,k}$. We proceed similarly as for proving \eqref{ineq:lemma1ELtensors} and we split $\tilde \sigma^{\,{\rm L},\,k}$ by $\tilde \sigma^{\,{\rm L},\,k}=\sum_{i=1}^6\tilde \sigma^{\,{\rm L},\,k}_i$. We first define $\delta \tilde  \sigma^{\,{\rm L},\,k}_1$ by
\begin{align*}
	\delta \tilde\sigma^{\rm L}_1\,=
	%------------------------------------------------------------------------------------------------------------------------------
	%------------------------------------------------------------------------------------------------------------------------------
	%------------------------------------------------------------------------------------------------------------------------------
	%------------------------------------------------------------------------------------------------------------------------------
	\,\bar \alpha_1 \Big[
	\,\bar \nn \cdot \delta \mathbb{D}^k\,\m^{k+1}\,+\,
	\,\bar \nn \cdot \mathbb{D}^k\,\delta \m^k\,+\,	
	%------------------------------------------------------------------------------------------------------------------------------	
	\delta \m^k \cdot \mathbb{D}^{k+1}\,\bar \nn+
	\m^{k} \cdot \delta  \mathbb{D}^k\,\bar \nn+
	%------------------------------------------------------------------------------------------------------------------------------
	\delta \m^k \cdot \mathbb{D}^{k+1}\, \m^{k+1}\,\,+\\+
	\m^k \cdot \delta \mathbb{D}^{k}\, \m^{k+1}\,\,+
	\delta \m^k \cdot \mathbb{D}^k\, \delta \m^{k}\,\,
	\Big]\nn^{k+1}\otimes \nn^{k+1}\,+
	%------------------------------------------------------------------------------------------------------------------------------
	\,\bar \alpha_1 \Big[\,\bar \nn \cdot \mathbb{D}^k\,\m^k\,+\,\m^k \cdot \mathbb{D}\,\bar \nn+\\+
	\m^k \cdot \mathbb{D}^k\, \m^k\,\Big]\delta \m^k\otimes \nn^{k+1}\,+
	%------------------------------------------------------------------------------------------------------------------------------
	\,\bar \alpha_1 \Big[\,\bar \nn \cdot \mathbb{D}^k\,\m^k\,+\,\m^k \cdot \mathbb{D}\,\bar \nn+
	\m^k \cdot \mathbb{D}^k\, \m^k\,\Big]\nn^k\otimes \delta \m^{k+1}\,+\\
	%------------------------------------------------------------------------------------------------------------------------------
	%------------------------------------------------------------------------------------------------------------------------------
	%------------------------------------------------------------------------------------------------------------------------------
	%------------------------------------------------------------------------------------------------------------------------------
	\,\bar \alpha_1 \Big[
	\,\bar \nn \cdot \delta \mathbb{D}^k\,\m^{k+1}\,+\,
	\,\bar \nn \cdot \mathbb{D}^k\,\delta \m^k\,+\,	
	%------------------------------------------------------------------------------------------------------------------------------	
	\delta \m^k \cdot \mathbb{D}^{k+1}\,\bar \nn+
	\m^{k} \cdot \delta  \mathbb{D}^k\,\bar \nn+
	%------------------------------------------------------------------------------------------------------------------------------
	\delta \m^k \cdot \mathbb{D}^{k+1}\, \m^{k+1}\,\,+\\+
	\m^k \cdot \delta \mathbb{D}^{k}\, \m^{k+1}\,\,+
	\delta \m^k \cdot \mathbb{D}^k\, \delta \m^{k}\,\,
	\Big]\m^{k+1}\otimes \nn^{k+1}\,+
	%------------------------------------------------------------------------------------------------------------------------------
	\,\bar \alpha_1 \Big[\,\bar \nn \cdot \mathbb{D}^k\,\m^k\,+\,\m^k \cdot \mathbb{D}\,\bar \nn+\\+
	\m^k \cdot \mathbb{D}^k\, \m^k\,\Big]\delta \m^k\otimes \nn^{k+1}\,+
	%------------------------------------------------------------------------------------------------------------------------------
	\,\bar \alpha_1 \Big[\,\bar \nn \cdot \mathbb{D}^k\,\m^k\,+\,\m^k \cdot \mathbb{D}\,\bar \nn+
	\m^k \cdot \mathbb{D}^k\, \m^k\,\Big]\m^k\otimes \delta \m^{k+1}\,+\\
	%------------------------------------------------------------------------------------------------------------------------------
	%------------------------------------------------------------------------------------------------------------------------------
	%------------------------------------------------------------------------------------------------------------------------------
	%------------------------------------------------------------------------------------------------------------------------------
	\,\bar \alpha_1 \Big[
	\,\bar \nn \cdot \delta \mathbb{D}^k\,\m^{k+1}\,+\,
	\,\bar \nn \cdot \mathbb{D}^k\,\delta \m^k\,+\,	
	%------------------------------------------------------------------------------------------------------------------------------	
	\delta \m^k \cdot \mathbb{D}^{k+1}\,\bar \nn+
	\m^{k} \cdot \delta  \mathbb{D}^k\,\bar \nn+
	%------------------------------------------------------------------------------------------------------------------------------
	\delta \m^k \cdot \mathbb{D}^{k+1}\, \m^{k+1}\,\,+\\+
	\m^k \cdot \delta \mathbb{D}^{k}\, \m^{k+1}\,\,+
	\delta \m^k \cdot \mathbb{D}^k\, \delta \m^{k}\,\,
	\Big]\nn^{k+1}\otimes \m^{k+1}\,+
	%------------------------------------------------------------------------------------------------------------------------------
	\,\bar \alpha_1 \Big[\,\bar \nn \cdot \mathbb{D}^k\,\m^k\,+\,\m^k \cdot \mathbb{D}^k\,\bar \nn+\\+
	\m^k \cdot \mathbb{D}^k\, \m^k\,\Big]\delta \m^k\otimes \m^{k+1}\,+
	%------------------------------------------------------------------------------------------------------------------------------
	\,\bar \alpha_1 \Big[\,\bar \nn \cdot \mathbb{D}^k\,\m^k\,+\,\m^k \cdot \mathbb{D}\,\bar \nn+
	\m^k \cdot \mathbb{D}^k\, \m^k\,\Big]\m^k\otimes \delta \m^{k+1}\,+\\
	%------------------------------------------------------------------------------------------------------------------------------
	%------------------------------------------------------------------------------------------------------------------------------
	%------------------------------------------------------------------------------------------------------------------------------
	%------------------------------------------------------------------------------------------------------------------------------
	\,\bar \alpha_1 \Big[
	\,\bar \nn \cdot \delta \mathbb{D}^k\,\m^{k+1}\,+\,
	\,\bar \nn \cdot \mathbb{D}^k\,\delta \m^k\,+\,	
	%------------------------------------------------------------------------------------------------------------------------------	
	\delta \m^k \cdot \mathbb{D}^{k+1}\,\bar \nn+
	\m^{k} \cdot \delta  \mathbb{D}^k\,\bar \nn+
	%------------------------------------------------------------------------------------------------------------------------------
	\delta \m^k \cdot \mathbb{D}^{k+1}\, \m^{k+1}\,\,+\\+
	\m^k \cdot \delta \mathbb{D}^{k}\, \m^{k+1}\,\,+
	\m^k \cdot \mathbb{D}^k\, \delta \m^{k}\,\,
	\Big]\m^{k+1}\otimes \m^{k+1}\,+
	%------------------------------------------------------------------------------------------------------------------------------
	\,\bar \alpha_1 \Big[\,\bar \nn \cdot \mathbb{D}^k\,\m^k\,+\,\m^k \cdot \mathbb{D}\,\bar \nn+\\+
	\m^k \cdot \mathbb{D}^k\, \m^k\,\Big]\delta \m^k\otimes \m^{k+1}\,+
	%------------------------------------------------------------------------------------------------------------------------------
	\,\bar \alpha_1 \Big[\,\bar \nn \cdot \mathbb{D}^k\,\m^k\,+\,\m^k \cdot \mathbb{D}\,\bar \nn+
	\m^k \cdot \mathbb{D}^k\, \m^k\,\Big]\m^k\otimes \delta \m^{k+1}\,+\\
	%------------------------------------------------------------------------------------------------------------------------------
	%------------------------------------------------------------------------------------------------------------------------------
	%------------------------------------------------------------------------------------------------------------------------------
	%------------------------------------------------------------------------------------------------------------------------------
	+\,\delta 	\tilde 	\alpha_1^k [\,			\nn^{k+1} \cdot			\mathbb{D}^{k+1}\,			\nn^{k+1}\,]			\nn^{k+1}\otimes			\nn^{k+1}\,
	+\, 			\tilde 	\alpha_1^k [\,	\delta 	\m^{k  } \cdot 			\mathbb{D}^{k+1}\,			\nn^{k+1}\,]			\nn^{k+1}\otimes			\nn^{k+1}\,+\\
	+\, 			\tilde 	\alpha_1^k [\,			\nn^{k  } \cdot \delta 	\mathbb{D}^{k  }\,			\nn^{k+1}\,]			\nn^{k+1}\otimes			\nn^{k+1}\,
	+\, 			\tilde 	\alpha_1^k [\,			\nn^{k  } \cdot 			\mathbb{D}^{k  }\, \delta 	\m^{k  }\,]			\nn^{k+1}\otimes			\nn^{k+1}\,+\\
	+\, 			\tilde 	\alpha_1^k [\,			\nn^{k  } \cdot 			\mathbb{D}^{k  }\,			\nn^{k  }\,]	\delta 	\m^{k  }\otimes			\nn^{k+1}\,
	+\,			\tilde 	\alpha_1^k [\,			\nn^{k  } \cdot 			\mathbb{D}^{k  }\,			\nn^{k  }\,]			\nn^{k  }\otimes	\delta 	\m^{k  }.
\end{align*}
A direct computation then leads to
\begin{equation*}
\begin{split}
 \|\,\Div\,\sigma^{\,{\rm L}, k}_1\,\|_{L^1_t\BB_{2,1}^{\frac{\dd}{2}-1}}\,\lesssim\,
 \|\,\sigma^{\,{\rm L}, k}_1\,\|_{L^1_t\BB_{2,1}^{\frac{\dd}{2}}}\,\lesssim\,
 \|\,\nabla \delta \uu^k\,\|_{L^1_t\BB_{2,1}^\frac{\dd}{2}}\Big[\,
 \|\,(\,\m^k,\,\m^{k+1}\,)\,\|_{L^\infty_t\BB_{2,1}^\frac{\dd}{2}}\,+\\+
 \|\,(\,\m^k,\,\m^{k+1}\,)\,\|_{L^\infty_t\BB_{2,1}^\frac{\dd}{2}}^4
 \Big]
 \,+
 \|\,\nabla (\,\uu^k,\,\uu^{k+1}\,)\,\|_{L^1_t\BB_{2,1}^\frac{\dd}{2}}
 \Big(\,
 	1+\|\,(\,\m^k,\,\m^{k+1}\,)\,\|_{L^\infty_t\BB_{2,1}^\frac{\dd}{2}}^3\,
 \Big)
 \|\,\delta \m^k\,\|_{L^\infty_t\BB_{2,1}^\frac{\dd}{2}}\,+\\+\, 
 \|\,\delta \omega^k\,\|_{L^\infty_t\BB_{2,1}^\frac{\dd}{2}}
 \|\,\nabla \uu^{k+1}\,\|_{L^1_t\BB_{2,1}^\frac{\dd}{2}}
 \lesssim \ee^2\Big[\,\|\,\delta \uu^k\,\|_{\X_1}\,+\,\|\,\delta \omega^k\,\|_{\X_2}\,+\,\|\,\delta \m^k\,\|_{\X_3}\,\Big]
 \lesssim \ee^{k+2}.
\end{split}
\end{equation*}
Similarly, the $L^1_t\BB_{2,1}^{\frac{\dd}{2}}$-norm of $\Div\,\sigma^{\,{\rm L}, k}_1$ is treated by
\begin{align*}
 \|\,\Div\,\sigma^{\,{\rm L}, k}_1\,\|_{L^1_t\BB_{2,1}^{\frac{\dd}{2}}}\,\lesssim\,
 \|\,\Delta \delta \uu^k\,\|_{L^1_t\BB_{2,1}^\frac{\dd}{2}}\|\,(\,\m^k,\,\m^{k+1}\,)\,\|_{L^\infty_t\BB_{2,1}^\frac{\dd}{2}}
 \Big[
 	\,1\,+\|\,(\,\m^k,\,\m^{k+1}\,)\,\|_{L^\infty_t\BB_{2,1}^\frac{\dd}{2}}^3
 \,\Big]
 \,+\\+\,								
 \|\,\nabla \delta \uu^k\,\|_{L^1_t\BB_{2,1}^\frac{\dd}{2}}
 \|\,\nabla (\,\m^k,\,\m^{k+1}\,)\,\|_{L^\infty_t\BB_{2,1}^\frac{\dd}{2}}
 \Big[
 	\,1\,+\|\,(\,\m^k,\,\m^{k+1}\,)\,\|_{L^\infty_t\BB_{2,1}^\frac{\dd}{2}}^3
 \,\Big]\,+\\+
 \|\,\Delta (\,\uu^k,\,\uu^{k+1}\,)\,\|_{L^1_t\BB_{2,1}^\frac{\dd}{2}}
 \Big(\,
 	1\,+\,\|\,(\,\m^k,\,\m^{k+1}\,)\,\|_{L^\infty_t\BB_{2,1}^\frac{\dd}{2}}^3\,
 \Big)
 \|\,\delta \m^k\,\|_{L^\infty_t\BB_{2,1}^\frac{\dd}{2}}\,+\,
 \|\,\nabla (\,\uu^k,\,\uu^{k+1}\,)\,\|_{L^1_t\BB_{2,1}^\frac{\dd}{2}}
%----- 
	{\scriptstyle\times}\\{\scriptstyle\times}
%-----  
 \Big(\,
 	1\,+\,\|\,(\,\m^k,\,\m^{k+1}\,)\,\|_{L^\infty_t\BB_{2,1}^\frac{\dd}{2}}^3\,
 \Big)
 \|\,\nabla (\,\m^k,\,\m^{k+1}\,)\,\|_{L^\infty_t\BB_{2,1}^\frac{\dd}{2}}
 \|\,\delta \m^k\,\|_{L^\infty_t\BB_{2,1}^\frac{\dd}{2}}\,+\,
 \|\,\nabla (\,\uu^k,\,\uu^{k+1}\,)\,\|_{L^1_t\BB_{2,1}^\frac{\dd}{2}}
%----- 
	{\scriptstyle\times}\\{\scriptstyle\times}
%----- 
 \Big(\,
 	1\,+\,\|\,(\,\m^k,\,\m^{k+1}\,)\,\|_{L^\infty_t\BB_{2,1}^\frac{\dd}{2}}^3\,
 \Big)
 \|\,\nabla\delta \m^k\,\|_{L^\infty_t\BB_{2,1}^\frac{\dd}{2}}\,+\,
 \|\, \nabla \delta \omega^k\,\|_{L^\infty_t\BB_{2,1}^\frac{\dd}{2}}
 \|\,\nabla \uu^{k+1}\,\|_{L^1_t\BB_{2,1}^\frac{\dd}{2}}\,+\\+\,
 \|\, \delta \omega^k\,\|_{L^\infty_t\BB_{2,1}^\frac{\dd}{2}}
 \|\,\Delta \uu^{k+1}\,\|_{L^1_t\BB_{2,1}^\frac{\dd}{2}}\,+\,
 \|\, \delta \omega^k\,\|_{L^\infty_t\BB_{2,1}^\frac{\dd}{2}}
 \|\,\nabla \m^k\,\|_{L^\infty_t\BB_{2,1}^\frac{\dd}{2}}
 \|\,\nabla \uu^{k+1}\,\|_{L^1_t\BB_{2,1}^\frac{\dd}{2}}\\
 \lesssim \ee^2\Big[\,\|\,\delta \uu^k\,\|_{\X_1}\,+\,\|\,\delta \omega^k\,\|_{\X_2}\,+\,\|\,\delta \m^k\,\|_{\X_3}\,\Big]
 \lesssim \ee^{k+2}.
\end{align*}

\smallskip\noindent
The second stress tensor $\delta \sigma^{{\rm L},\,k}_2$ is defined as
\begin{equation*}
\begin{aligned}
	\delta \sigma_2^{\rm L,k}\,=
%------------------------------------------------------------------------------------------------------
%------------------------------------------------------------------------------------------------------	
	\,\bar{\alpha}_2\,
	\Big[\,
			\partial_t \delta \m^{k}+\delta \uu^{k}\cdot \nabla \m^{k+1}	+
			\uu^{k}\cdot \nabla \delta  \m^{k}-\delta \Omega^{k}  \m^{k+1}-\Omega^{k}\delta\m^{k}
	\Big]
	\otimes 
	\Big[ \,\bar \nn\,+\,\m^{k+1}\,\Big]\,+
%------------------------------------------------------------------------------------------------------
%------------------------------------------------------------------------------------------------------	
	\\+
	\delta \tilde \alpha_2^{k}\Big[\,\partial_t		   \m^{k+1}+		   \uu^{k+1}\cdot \nabla	\m^{k+1}+\Omega^{k+1}\,\m^{k+1}\,\Big]\otimes\big[\,\bar \nn\,+\,\m^{k+1}\,\big]\,+
	\tilde \alpha_2^{k}\Big[\,\partial_t \delta  \m^k+	\delta  \uu^{k}\cdot \nabla	\m^{k+1}+\\ + 
	\uu^{k}\cdot \nabla\delta\m^{k}+\delta \Omega^k\m^{k+1}+\,\Omega^{k}\,\delta \m^{k}\,\Big]\otimes\big[\bar \nn+\m^{k+1}\big]\,+
	\tilde \alpha_2^{k}\Big[\partial_t		   \m^k	 +	 \uu^k\cdot \nabla	\m^k+\Omega^k\m^k\Big]\otimes\delta \m^k+\\+
%------------------------------------------------------------------------------------------------------
%------------------------------------------------------------------------------------------------------
	\delta \tilde{\alpha}^k\Omega^{k+1}\bar \nn\otimes\m^{k+1}\,+\,
	\big[\bar{\alpha}_2\,+\,\tilde{\alpha}_2(\omega^k)\big]\delta \Omega^k\bar \nn\otimes\m^{k+1}\,+
	\big[\bar{\alpha}_2\,+\,\tilde{\alpha}_2(\omega^k)\big]\Omega^k\bar \nn\otimes\delta \m^k.
\end{aligned}
\end{equation*}
Thus, we first observe that
\begin{equation*}
\begin{aligned}
	\|\,\Div\,\delta \sigma_2^{\rm L,k}\,\|_{L^1_t\BB_{2,1}^{\frac{\dd}{2}-1}}
	\lesssim
	\Big[\,1+\|  \omega^k			\|_{L^\infty_t\BB_{2,1}^\frac{\dd}{2}}\,\Big]
	\Big[
	\|\,\delta \sigma_2^{\rm L,k}\,\|_{L^1_t\BB_{2,1}^{\frac{\dd}{2}}}
	\lesssim
	\|\partial_t \delta \m^{k}\|_{L^1_t\BB_{2,1}^\frac{\dd}{2}}+
	\| \delta \uu^{k}			\|_{L^2_t\BB_{2,1}^\frac{\dd}{2}}{\scriptstyle\times}\\{\scriptstyle\times}
	\| \nabla 		 \m^{k+1}	\|_{L^2_t\BB_{2,1}^\frac{\dd}{2}}+	
	\| \uu^{k}					\|_{L^2_t\BB_{2,1}^\frac{\dd}{2}}
	\| \nabla \delta  \m^{k}		\|_{L^2_t\BB_{2,1}^\frac{\dd}{2}}+
	\| \nabla \delta \uu^{k}		\|_{L^1_t\BB_{2,1}^\frac{\dd}{2}}
	\|  		 \m^{k+1}			\|_{L^\infty_t\BB_{2,1}^\frac{\dd}{2}}+	
	\| \nabla \uu^{k}			\|_{L^1_t\BB_{2,1}^\frac{\dd}{2}}{\scriptstyle\times}\\{\scriptstyle\times}
	\|		  \delta  \m^{k}		\|_{L^\infty_t\BB_{2,1}^\frac{\dd}{2}}
	\Big]
	+
	\| \delta \omega^k			\|_{L^\infty_t\BB_{2,1}^\frac{\dd}{2}}
	\Big[
		\|\partial_t 	\m^{k+1}\|_{L^1_t\BB_{2,1}^\frac{\dd}{2}}\hspace{-0.1cm}+
		\| 		 \uu^{k+1}		\|_{L^2_t\BB_{2,1}^\frac{\dd}{2}}
		\| \nabla 	 \m^{k+1}	\|_{L^2_t\BB_{2,1}^\frac{\dd}{2}}+
		\| \nabla  \uu^{k+1}		\|_{L^1_t\BB_{2,1}^\frac{\dd}{2}}{\scriptstyle\times}\\{\scriptstyle\times}
		\|  		 \m^{k+1}		\|_{L^\infty_t\BB_{2,1}^\frac{\dd}{2}}
	\Big]+
	\|  \omega^k			\|_{L^\infty_t\BB_{2,1}^\frac{\dd}{2}}
	\Big[\,
		\|\partial_t 	\m^{k}\|_{L^1_t\BB_{2,1}^\frac{\dd}{2}}+
		\| 		 \uu^{k}		\|_{L^2_t\BB_{2,1}^\frac{\dd}{2}}
		\| \nabla 	 \m^{k}	\|_{L^2_t\BB_{2,1}^\frac{\dd}{2}}+
		\| \nabla  \uu^{k}		\|_{L^1_t\BB_{2,1}^\frac{\dd}{2}}{\scriptstyle\times}\\{\scriptstyle\times}
		\|  		 \m^{k}		\|_{L^\infty_t\BB_{2,1}^\frac{\dd}{2}}
	\Big]
	\|\,\delta \m^k\,\|_{L^\infty_t\BB_{2,1}^\frac{\dd}{2}}
	\,+\,
	\|\,\nabla \delta \uu^k\,\|_{L^1_t\BB_{2,1}^\frac{\dd}{2}}
	\|\,\m^{k+1}\,\|_{L^\infty_t\BB_{2,1}^\frac{\dd}{2}}\,+\,
	\|\,\nabla  \uu^k\,\|_{L^1_t\BB_{2,1}^\frac{\dd}{2}}
	\|\,\delta \m^{k}\,\|_{L^\infty_t\BB_{2,1}^\frac{\dd}{2}}\\
	\lesssim
	\ee^2
	\Big[\,\|\,\delta \uu^k\,\|_{\X_1}\,+\,\|\,\delta \omega^k\,\|_{\X_2}\,+\,\|\,\delta \m^k\,\|_{\X_3}\,\Big]
	 \lesssim \ee^{k+2}.
\end{aligned}
\end{equation*}
Moreover
\begin{align*}
	\|\,\Div\,\delta \sigma_2^{\rm L,k}\,\|_{L^1_t\BB_{2,1}^{\frac{\dd}{2}}}
	\Big[\,1+\|  \omega^k			\|_{L^\infty_t\BB_{2,1}^\frac{\dd}{2}}\,\Big]
	\Big[
	\|\,\delta \sigma_2^{\rm L,k}\,\|_{L^1_t\BB_{2,1}^{\frac{\dd}{2}}}
	\lesssim
	\|\partial_t \nabla \delta \m^{k}\|_{L^1_t\BB_{2,1}^\frac{\dd}{2}}+
	\| \nabla \delta \uu^{k}			\|_{L^2_t\BB_{2,1}^\frac{\dd}{2}}{\scriptstyle\times}\\{\scriptstyle\times}
	\| \nabla 		 \m^{k+1}	\|_{L^2_t\BB_{2,1}^\frac{\dd}{2}}+	
	\|  \delta \uu^{k}			\|_{L^2_t\BB_{2,1}^\frac{\dd}{2}}
	\| \Delta 		 \m^{k+1}	\|_{L^2_t\BB_{2,1}^\frac{\dd}{2}}+	
	\| \nabla \uu^{k}					\|_{L^2_t\BB_{2,1}^\frac{\dd}{2}}
	\| \nabla \delta  \m^{k}		\|_{L^2_t\BB_{2,1}^\frac{\dd}{2}}+\\
	\| \uu^{k}					\|_{L^2_t\BB_{2,1}^\frac{\dd}{2}}
	\| \Delta \delta  \m^{k}		\|_{L^2_t\BB_{2,1}^\frac{\dd}{2}}+
	\| \Delta \delta \uu^{k}		\|_{L^1_t\BB_{2,1}^\frac{\dd}{2}}
	\|  		 \m^{k+1}			\|_{L^\infty_t\BB_{2,1}^\frac{\dd}{2}}+	
	\| \nabla \delta \uu^{k}		\|_{L^1_t\BB_{2,1}^\frac{\dd}{2}}
	\|  \nabla	 \m^{k+1}			\|_{L^\infty_t\BB_{2,1}^\frac{\dd}{2}}+	\\+
	\| \Delta \uu^{k}			\|_{L^1_t\BB_{2,1}^\frac{\dd}{2}}
	\|		  \delta  \m^{k}		\|_{L^\infty_t\BB_{2,1}^\frac{\dd}{2}}+
	\| \nabla 	\uu^{k}			\|_{L^1_t\BB_{2,1}^\frac{\dd}{2}}
	\|	\nabla  \delta  \m^{k}		\|_{L^\infty_t\BB_{2,1}^\frac{\dd}{2}}
	\Big]
	%---------------------
	+\\+
	\| \delta \omega^k			\|_{L^\infty_t\BB_{2,1}^\frac{\dd}{2}}
	\Big[\,
		\|\partial_t 	\m^{k+1}\|_{L^1_t\BB_{2,1}^\frac{\dd}{2}}+
		\| 		 \uu^{k+1}		\|_{L^2_t\BB_{2,1}^\frac{\dd}{2}}
		\| \nabla 	 \m^{k+1}	\|_{L^2_t\BB_{2,1}^\frac{\dd}{2}}+
		\| \nabla  \uu^{k+1}		\|_{L^1_t\BB_{2,1}^\frac{\dd}{2}}{\scriptstyle\times}\\{\scriptstyle\times}
		\|  		 \m^{k+1}		\|_{L^\infty_t\BB_{2,1}^\frac{\dd}{2}}
	\Big]+
	\|  \omega^k			\|_{L^\infty_t\BB_{2,1}^\frac{\dd}{2}}
	\Big[\,
		\|\partial_t 	\m^{k}\|_{L^1_t\BB_{2,1}^\frac{\dd}{2}}+
		\| 		 \uu^{k}		\|_{L^2_t\BB_{2,1}^\frac{\dd}{2}}
		\| \nabla 	 \m^{k}	\|_{L^2_t\BB_{2,1}^\frac{\dd}{2}}+
		\| \nabla  \uu^{k}		\|_{L^1_t\BB_{2,1}^\frac{\dd}{2}}{\scriptstyle\times}\\{\scriptstyle\times}
		\|  		 \m^{k}		\|_{L^\infty_t\BB_{2,1}^\frac{\dd}{2}}
	\Big]
	\|\,\delta \m^k\,\|_{L^\infty_t\BB_{2,1}^\frac{\dd}{2}}
	\,+\,
	\|\,\nabla \delta \uu^k\,\|_{L^1_t\BB_{2,1}^\frac{\dd}{2}}
	\|\,\m^{k+1}\,\|_{L^\infty_t\BB_{2,1}^\frac{\dd}{2}}\,+\,
	\|\,\nabla  \uu^k\,\|_{L^1_t\BB_{2,1}^\frac{\dd}{2}}
	\|\,\delta \m^{k}\,\|_{L^\infty_t\BB_{2,1}^\frac{\dd}{2}}.\\
\end{align*}
Hence, we finally achieve that
\begin{equation}\label{est:1deltasigma2}
	\|\,\Div\,\delta \sigma_2^{{\,\rm L},\,k}\,\|_{L^1_t\BB_{2,1}^{\frac{\dd}{2}-1}\cap \BB_{2,1}^{\frac{\dd}{2}}}\,\lesssim \,\ee^2
	\Big[\,\|\,\delta \uu^k\,\|_{\X_1}\,+\,\|\,\delta \omega^k\,\|_{\X_2}\,+\,\|\,\delta \m^k\,\|_{\X_3}\,\Big]
	 \lesssim \ee^{k+2}.
\end{equation}
The tensor $\delta \sigma^{{\rm\,\,L},k}_3$ corresponds to the transpose tensor of $\delta \sigma^{{\rm\,L},k}_2$, replacing $\alpha_2$ by $\alpha_3$. Thus with proceeding as for proving \eqref{est:1deltasigma2}, 
we deduce that
\begin{equation}\label{est:2deltasigma2}
	\|\,\Div\,\delta \sigma_3^{{\,\rm L},\,k}\,\|_{L^1_t\BB_{2,1}^{\frac{\dd}{2}-1}\cap \BB_{2,1}^{\frac{\dd}{2}}}\,\lesssim\,
	\ee^2
	\Big[\,\|\,\delta \uu^k\,\|_{\X_1}\,+\,\|\,\delta \omega^k\,\|_{\X_2}\,+\,\|\,\delta \m^k\,\|_{\X_3}\,\Big]
	 \lesssim \ee^{k+2}.
\end{equation}
We now consider $\delta \sigma^{\,{\rm L},\,k}_4$, defined by $\delta \sigma^{\,{\rm L},\,k}_4\,:=\,\delta \tilde \alpha_4^{k}\,\mathbb{D}^{k+1}\,+\, \tilde \alpha_4(\omega^k)\,\delta \mathbb{D}^k$, hence
\begin{equation}\label{est:deltasigma3}
	\begin{aligned}
		\|\,\delta \sigma_4^{{\,\rm L},\,k}\|_{L^1_t\BB_{2,1}^{\frac{\dd}{2}}}+
		\|\,\Div\,\delta \sigma_4^{{\,\rm L},\,k}\|_{L^1_t\BB_{2,1}^{\frac{\dd}{2}}}\,
		\lesssim\,
		\|\delta 	\omega^k\|_{L^\infty_t\BB_{2,1}^\frac{\dd}{2}}\|\,\nabla  \uu^{k+1}		\|_{L^1_t\BB_{2,1}^\frac{\dd}{2}}+
		\|\, 		\omega^k\|_{L^\infty_t\BB_{2,1}^\frac{\dd}{2}}\|\,\nabla \delta \uu^k	\|_{L^1_t\BB_{2,1}^\frac{\dd}{2}}+\\
		+
		\|\nabla \delta 	\omega^k	\|_{L^\infty_t\BB_{2,1}^\frac{\dd}{2}}
		\|\nabla  \uu^{k+1}		\|_{L^1_t\BB_{2,1}^\frac{\dd}{2}}+
		\|\,		\delta 	\omega^k\,		\|_{L^\infty_t\BB_{2,1}^\frac{\dd}{2}}\|\Delta  \uu^{k+1}		\|_{L^1_t\BB_{2,1}^\frac{\dd}{2}}+
		\|\, \nabla	\omega^k\,\|_{L^\infty_t\BB_{2,1}^\frac{\dd}{2}}
		\|\nabla \delta \uu^k\,	\|_{L^1_t\BB_{2,1}^\frac{\dd}{2}}+\\+
		\|\,			 	\omega^k\,		\|_{L^\infty_t\BB_{2,1}^\frac{\dd}{2}}\|\,\Delta \delta  \uu^{k}\,		\|_{L^1_t\BB_{2,1}^\frac{\dd}{2}}
		\,\lesssim\,\ee^2
		\Big[\,
			\|\,\delta \uu^k		\,\|_{\X_1}	\,+\,
			\|\,\delta \omega^k	\,\|_{\X_2}	\,+\,
			\|\,\delta \m^k		\,\|_{\X_3}	\,
		\Big]\,\leq\,\ee^{k+2}.
	\end{aligned}
\end{equation}
It remains to control $\delta \tilde \sigma^{{\rm\,L}\,k}_5$ together with $\delta \tilde \sigma^{{\rm\,L}\,k}_6$. We focus on $\delta \tilde \sigma^{{\rm\,L}\,k}_5$, since a similar result holds for $\delta \tilde \sigma^{{\rm\,L}\,k}_6$. Such a tensor is defined by means of
\begin{equation*}
\begin{aligned}
	\delta \sigma^{{\rm L}, k}_5\,:=\,
	\delta \tilde{\alpha}_5^k\,\big[\bar \nn\,+\m^{k+1}\,\big]\otimes\mathbb{D}^{k+1}\big[\bar \nn\,+\m^{k+1}\,\big]\,+\,
		   \tilde{\alpha}_5^k\,\delta \m^{k}\otimes\mathbb{D}^{k+1}\big[\bar \nn\,+\m^{k+1}\,\big]\,+\\+\,
		   \tilde{\alpha}_5^k\,\big[\bar \nn\,+\m^{k}\,\big]\otimes\delta \mathbb{D}^{k}\big[\bar \nn\,+\m^{k+1}\,\big]\,+
		   \tilde{\alpha}_5^k\,\big[\bar \nn\,+\m^{k}\,\big]\otimes\mathbb{D}^{k}\delta \m^k\,+
	%------------------------------------------------------------------------------------------------------------------------
	%------------------------------------------------------------------------------------------------------------------------\\
	%------------------------------------------------------------------------------------------------------------------------\\
	\bar \alpha_5
	\Big[\,
		\delta 	\m^k\otimes			\mathbb{D}^{k+1}		\m^{k+1}\,+\\\,
				\m^k\otimes	\delta 	\mathbb{D}^k			\m^{k+1}\,+
				\m^k\otimes			\mathbb{D}^k \delta 	\m^k\,+\,
	%------------------------------------------------------------------------------------------------------------------------
	%------------------------------------------------------------------------------------------------------------------------
	%------------------------------------------------------------------------------------------------------------------------
	\bar{\nn}\otimes \delta 	\mathbb{D}^{k}		\m^{k+1}\,+\,
	\bar{\nn}\otimes  		\mathbb{D}^{k}\delta	\m^{k  }\,+\,
	\delta \m^{k}\otimes \mathbb{D}^{k+1}\bar{\nn}\,\Big].
\end{aligned}
\end{equation*}
We then deduce that
\begin{equation*}
\begin{aligned}
	\|\,		\,\sigma^{{\rm L}, k}_5\,\|_{L^1_t\BB_{2,1}^{\frac{\dd}{2}	}}\lesssim
	\|\,\delta \omega^k\,\|_{L^\infty_t\BB_{2,1}^{\frac{\dd}{2}}}
	\|\,\nabla \uu^{k+1}\,\|_{L^1_t\BB_{2,1}^\frac{\dd}{2}}+
	\|\,	 \omega^k\,\|_{L^\infty_t\BB_{2,1}^{\frac{\dd}{2}}}
	\|\, \delta \m^k\,\|_{L^\infty_t\BB_{2,1}^\frac{\dd}{2}}{\scriptstyle\times}
	\|\,\nabla (\uu^k,\,\uu^{k+1})\,\|_{L^1_t\BB_{2,1}^\frac{\dd}{2}}\\
	\,+\,
	\|\,	 \omega^k\,\|_{L^\infty_t\BB_{2,1}^{\frac{\dd}{2}}}
	\Big(	1\,+\,\|\, 	 \m^k\,\|_{L^\infty_t\BB_{2,1}^\frac{\dd}{2}}\Big)
	\|\,\nabla \delta \uu^k\,\|_{L^1_t\BB_{2,1}^\frac{\dd}{2}}+
	\Big(1+\|	 \omega^k\,\|_{L^\infty_t\BB_{2,1}^{\frac{\dd}{2}}}\,\Big)
	\|\, (	 \m^k,\,\m^{k+1})\,\|_{L^\infty_t\BB_{2,1}^\frac{\dd}{2}}{\scriptstyle\times}\\{\scriptstyle\times}
	\|\,\nabla (\uu^k,\uu^{k+1})\|_{L^1_t\BB_{2,1}^\frac{\dd}{2}}
	\|\delta \m^k\|_{L^\infty_t\BB_{2,1}^\frac{\dd}{2}}\hspace{-0.1cm}+
	\|\delta \m^k\|_{L^\infty_t\BB_{2,1}^{\frac{\dd}{2}}}
	\Big(1+\|\,( \m^k,\m^{k+1})\|_{L^\infty_t\BB_{2,1}^{\frac{\dd}{2}}}\Big)
	\|\nabla (\uu^k,\uu^{k+1})\|_{L^1_t\BB_{2,1}^\frac{\dd}{2}}\\
	\lesssim
	\ee^2\Big[\,
			\|\,\delta \uu^k		\,\|_{\X_1}	\,+\,
			\|\,\delta \omega^k	\,\|_{\X_2}	\,+\,
			\|\,\delta \m^k		\,\|_{\X_3}	\,
		\Big]\,\leq\,\ee^{k+2}.
\end{aligned}
\end{equation*}
Proceeding as for proving \eqref{est:1deltasigma2}, the $L^1_t\BB_{2,1}^{\frac{\dd}{2}	}$-norm of $\Div\,\sigma^{{\rm L}, k}_5$ is bounded by
\begin{equation}
	\|\,		\Div\,\sigma^{{\rm L}, k}_5\,\|_{L^1_t\BB_{2,1}^{\frac{\dd}{2}	}}
	\lesssim
	\ee^2\Big[\,
			\|\,\delta \uu^k		\,\|_{\X_1}	\,+\,
			\|\,\delta \omega^k	\,\|_{\X_2}	\,+\,
			\|\,\delta \m^k		\,\|_{\X_3}	\,
	\Big]\,\leq\,\ee^{k+2},
\end{equation}
which concludes the proof of inequality \eqref{est:momeqL}.

%%%%%%%%%%%%%%%%%%%%%%%%%%%%%%%%%%%%%%%%%%%%%%%%%%%%%%%%%%%%%%%%%%%%%%%%%%%%%%%%%%%%%%%%%%%%%%%%%%%%
%%%%%%%%%%%%%%%%%%%%%%%%%%%%%%%%%%%%%%%%%%%%%%%%%%%%%%%%%%%%%%%%%%%%%%%%%%%%%%%%%%%%%%%%%%%%%%%%%%%%
%																	Aknowledgment
%%%%%%%%%%%%%%%%%%%%%%%%%%%%%%%%%%%%%%%%%%%%%%%%%%%%%%%%%%%%%%%%%%%%%%%%%%%%%%%%%%%%%%%%%%%%%%%%%%%%
%%%%%%%%%%%%%%%%%%%%%%%%%%%%%%%%%%%%%%%%%%%%%%%%%%%%%%%%%%%%%%%%%%%%%%%%%%%%%%%%%%%%%%%%%%%%%%%%%%%%
\vspace{0.5cm}
\noindent
{\bf Acknowledgment} 
The authors express their sincere appreciation to Professor Marius Paicu and Professor  Arghir Zarnescu for constructive
suggestions and discussions. The work proceeded substantially at the the Department of
Mathematics of the Penn State University. We thank deeply the Department of Mathematics for their generous support and for providing 
a stimulating environment in which to work. 
The work of the second author has been partially supported by the NSF (grants DMS-1714401 and DMS-1412005).

%%%%%%%%%%%%%%%%%%%%%%%%%%%%%%%%%%%%%%%%%%%%%%%%%%%%%%%%%%%%%%%%%%%%%%%%%%%%%%%%%%%%%%%%%%%%%%%%%%%%
%%%%%%%%%%%%%%%%%%%%%%%%%%%%%%%%%%%%%%%%%%%%%%%%%%%%%%%%%%%%%%%%%%%%%%%%%%%%%%%%%%%%%%%%%%%%%%%%%%%%
%																	Bibliography
%%%%%%%%%%%%%%%%%%%%%%%%%%%%%%%%%%%%%%%%%%%%%%%%%%%%%%%%%%%%%%%%%%%%%%%%%%%%%%%%%%%%%%%%%%%%%%%%%%%%
%%%%%%%%%%%%%%%%%%%%%%%%%%%%%%%%%%%%%%%%%%%%%%%%%%%%%%%%%%%%%%%%%%%%%%%%%%%%%%%%%%%%%%%%%%%%%%%%%%%%

{

}
%%%%%%%%%%%%%%%%%%%%%%%%%%%%%%%%%%%%%%%%%%%%%%%%%%%%%%%%%%%%%%%%%%%%%%%%%%%%%%%%%%%%%%%%%%%%%%%%%%%%
%%%%%%%%%%%%%%%%%%%%%%%%%%%%%%%%%%%%%%%%%%%%%%%%%%%%%%%%%%%%%%%%%%%%%%%%%%%%%%%%%%%%%%%%%%%%%%%%%%%%
%																	End Document
%%%%%%%%%%%%%%%%%%%%%%%%%%%%%%%%%%%%%%%%%%%%%%%%%%%%%%%%%%%%%%%%%%%%%%%%%%%%%%%%%%%%%%%%%%%%%%%%%%%%
%%%%%%%%%%%%%%%%%%%%%%%%%%%%%%%%%%%%%%%%%%%%%%%%%%%%%%%%%%%%%%%%%%%%%%%%%%%%%%%%%%%%%%%%%%%%%%%%%%%%

\begin{thebibliography}{xxx}
%---------------------------------------------------------------------------------------------------
\bibitem{B-C-D} H. Bahouri, J.-Y. Chemin and R. Danchin: 
{\it ``Fourier Analysis and Nonlinear Partial Differential Equations''},
Grundlehren der Mathematischen Wissenschaften (Fundamental Principles of Mathematical Sciences),
{343}, Springer, Heidelberg (2011).
%---------------------------------------------------------------------------------------------------
\bibitem{Cavaterra-Rocca-Wu} C. Cavaterra, E. Rocca, H. Wu:
{\it ``Global weak solution and blow--up criterion of the general Ericksen--Leslie system for
nematic liquid crystal flows''}, J. Differ. Equ., {255}, 1432--1807 (2013).
%---------------------------------------------------------------------------------------------------
\bibitem{D2}R. Danchin:
{\it ``Global existence in critical spaces for flows of Compressible viscous and heat-conductive gases''},  
Arch. Rational Mech. Anal., {160} (2001).
%---------------------------------------------------------------------------------------------------
\bibitem{D1}R. Danchin:
{\it ``Local theory in critical spaces for compressible viscous and heat-conductive gases''},  
Comm. Partial Differential Equations, {26}, 7--8, 2183--1233 (2011).
%---------------------------------------------------------------------------------------------------
\bibitem{DM}R. Danchin and P. B. Mucha:
{\it ``A Lagrangian Approach for the Incompressible Navier--Stokes Equations with Variable Density''},  Comm. Pure Appl. Math.,  
{65}, 1458--1480 (2012).
%---------------------------------------------------------------------------------------------------
\bibitem{Dea} F. De Anna:
{\it ``Global solvability of the inhomogeneous Ericksen--Leslie system with only bounded density''},
Anal. Appl., 0, 1--51 (2016).
%---------------------------------------------------------------------------------------------------
\bibitem{Ericksen2}J. L. Ericksen:
{\it ``Conservation laws for liquid crystals''},  Trans. Soc. Rheology 
{5}, 23--34 (1961).
%---------------------------------------------------------------------------------------------------
\bibitem{Ericksen} J. L. Ericksen:
{\it ``Hydrostatic theory of liquid crystals''}, 
Arch. Rational Mech. Anal. {9}, 371--378 (1962).
%---------------------------------------------------------------------------------------------------
\bibitem{FFRS}E. Feireisl, M. Fr\'emond E. Rocca and G. Schimperna:
{\it ``A New Approach to Non-Isothermal Models for Nematic Liquid Crystals''},  Arch. Ration. Mech. Anal., 
{205}, 651--672 (2012).
%---------------------------------------------------------------------------------------------------
\bibitem{F-R-S}E. Feireisl, E. Rocca and G. Schimperna:
{\it ``On a non-isothermal model for the nematic liquid crystals''},  Nonlinearity, 
{24}, 243--257 (2011).
%---------------------------------------------------------------------------------------------------
\bibitem{FRSZ}E. Feireisl, E. Rocca, G. Schimperna and A. Zarnescu:
{\it ``Nonisothermal nematic liquid crystal flows with the Ball--Majumdar free energy''},  Ann. Mat. Pur. Appl., 
{194}, 1269--1299 (2015).
%---------------------------------------------------------------------------------------------------
\bibitem{FRSZ2}E. Feireisl, E. Rocca, G. Schimperna and A. Zarnescu:
{\it ``Evolution of non--isothermal Landau--de Gennes
nematic liquid crystals flows with singular potential''},  Comm. Math. Sci., 
{12}, 317--343 (2014).
%---------------------------------------------------------------------------------------------------
\bibitem{F}E. Feireisl,:
{\it ``On a non-isothermal model for the nematic liquid crystals''},  Nonlinearity, 
{24}, 243--257 (2011).
%---------------------------------------------------------------------------------------------------
\bibitem{Frank}F.C. Frank:
{\it ``On the theory of liquid crystals''},  Disc. Faraday Soc.,  
{25}, (1958).
%---------------------------------------------------------------------------------------------------
\bibitem{Hineman-wang}J.L. Hineman, and C. Wang:
{\it ``Well-Posedness of Nematic Liquid Crystal Flow in $L^{3}_{\rm uloc}(\mathbb{\RR}^3)$''},  Arch. Rational Mech. Anal.  
{210}, 210--177 (2013).
%---------------------------------------------------------------------------------------------------
\bibitem{Jeffery} G.. Jeffery: 
{\it ``The motion of ellipsolidal particles immersed in a viscous fluid''},
Roy. Soc. Proc., {102}, 102--161 (1922).
%---------------------------------------------------------------------------------------------------
\bibitem{Wu-Xu-Liu} C. Liu, H. Wu and X. Xu, :
{\it ``On the general Ericksen--Leslie system: Parodi's relation, well--posedness and stability''},  Arch. Rational Mech. Anal., { 208}, 59-107 (2013).
%---------------------------------------------------------------------------------------------------
\bibitem{Malek-Prusa}J. M{\'a}lek and V. Pr{\r{u}}{\v{s}}a : 
{\it ``Derivation of Equations for Continuum Mechanics and Thermodynamics of Fluids''},
Handbook of Mathematical Analysis in Mechanics of Viscous Fluids, Springer, Cham, 1--70 (2016).
%---------------------------------------------------------------------------------------------------
\bibitem{H1} M. Hieber and J. Pr\"uss:
{\it ``Dynamics of the Ericksen--Leslie equations with general Leslie stress I: the incompressible 
isotropic case''}, Math. Ann.,  1432-1807 (2016).
%---------------------------------------------------------------------------------------------------
\bibitem{H-N-P-S} M. Hieber, M. Nesensohn, J. Pr\"uss, and K. Schade:
{\it ``Dynamics of nematic liquid crystal flows:
The quasilinear approach''},
Annales de l'institut Henri Poincare (C) Non Linear Analysis,
{ 0}, (2014).
%---------------------------------------------------------------------------------------------------
\bibitem{H2} M. Hieber and J. Pr\"uss:
{\it ``Thermodynamic Consistent Modeling and Analysis of Nematic Liquid Crystal Flows''}, 
Springer Proceedings in Mathematics \& Statistics, (2016, to appear).
%---------------------------------------------------------------------------------------------------
\bibitem{Leslie}F. Leslie:
{\it ``Some constitutive equations for liquid crystals''},  Arch. Rational Mech. Anal.  
28, no. { 4}, 265--283 (1968).
%---------------------------------------------------------------------------------------------------
\bibitem{Lin-Liu}F.-H. Lin and C. Liu:
{\it ``Existence of solutions for the Ericksen--Leslie system''},  Arch. Rational Mech. Anal.  
{ 154}, 135--156 (2000).
%---------------------------------------------------------------------------------------------------
\bibitem{Lin-Liu2}F.-H. Lin and C. Liu:
{\it ``Nonparabolic dissipative systems modeling the flow of liquid crystals''}.  Comm. Pure Appl. Math.,  
{ 48}, 501--537 (1995).
%---------------------------------------------------------------------------------------------------
\bibitem{Lin-Wang} F.-H Lin and C. Wang: 
{\it ``Global Existence of Weak Solutions of the Nematic Liquid Crystal Flow in Dimension Three''},
Comm. Pure Appl. Math., { 69}, 1532--1571 (2016).
%---------------------------------------------------------------------------------------------------
\bibitem{S-V} A.M. Sonnet and E.G. Virga:
{\it ``Theory of flow phenomena in liquid crystals''},
Surveys and Tutorials in the Applied Mathematical Sciences,
{ 343}, Springer, New York, (2012).
%---------------------------------------------------------------------------------------------------
\bibitem{Stewart} I. W. Stewart: 
{\it ``The Static and Dynamic Continuum Theory of Liquid Crystals: A Mathematical Introduction''},
Liquid Crystals Book Series, CRC Press, (2004).
%---------------------------------------------------------------------------------------------------

\end{thebibliography}
\end{document}